\documentclass[letterpaper,12pt]{article}
\usepackage[top=1in, bottom=1in, left=1.25in, right=1.25in]{geometry}
\usepackage{setspace}
\usepackage[subfigure]{tocloft} 
\usepackage{float}

\usepackage{tocstyle}

\newtocstyle[standard][leaders]{mytocstyle}{\settocfeature[1]{entryhook}{\normalfont\bfseries}}
\usetocstyle{mytocstyle}
\usepackage{sectsty}
\sectionfont{\normalsize}
\usepackage{enumitem}

\renewcommand{\thesubsection}{\arabic{section}.\arabic{subsection}}
\usepackage{amsmath,amssymb,url}

\usepackage{lipsum} 
\usepackage{dsfont}
\newcommand{\diag}{\mathop{\mathrm{diag}}}
\newcommand{\sat}{\mathop{\mathrm{sat}}}
\newcommand{\satr}{\mathop{\mathrm{sat}}}
\newcommand{\norm}[1]{\ensuremath{\left\| #1 \right\|}}

\newcommand{\bracket}[1]{\ensuremath{\left[ #1 \right]}}
\newcommand{\braces}[1]{\ensuremath{\left\{ #1 \right\}}}

\newcommand{\refeqn}[1]{(\ref{eqn:#1})}

\newcommand{\tr}[1]{\mbox{tr}\ensuremath{\negthickspace\bracket{#1}}}
\newcommand{\trs}[1]{\mathrm{tr}\ensuremath{[#1]}}

\newcommand{\SO}{\ensuremath{\mathsf{SO(3)}}}

\newcommand{\so}{\ensuremath{\mathfrak{so}(3)}}

\newcommand{\SE}{\ensuremath{\mathsf{SE(3)}}}

\renewcommand{\Re}{\ensuremath{\mathbb{R}}}
\newcommand{\Sph}{\ensuremath{\mathsf{S}}}

\newcommand{\D}{\ensuremath{\mathbf{D}}}

\newcommand{\g}{\ensuremath{\mathfrak{g}}}

\newcommand{\CCC}{\ensuremath{\mathfrak{C}}}

\newcommand{\Mb}{\mathbf{M}}

\newcommand{\Bb}{\mathbf{B}}
\newcommand{\xb}{\mathbf{x}}

\newcommand{\Gb}{\mathbf{G}}

\newtheorem{prop}{Proposition}

\newtheorem{proof}{Proof}
\usepackage{graphicx,subfigure}

\begin{document}
\title{\singlespacing\normalsize\bf Geometric Nonlinear Controls for Multiple Cooperative Quadrotor UAVs Transporting a Rigid Body}
\author{\normalsize Farhad A. Goodarzi}
\date{}
\maketitle
\thispagestyle{empty}
\begin{center}
B.S. in Mechanical Engineering, December 2009, Sharif University of Technology \\
M.S. in Mechanical Engineering, December 2011, Santa Clara University
\\[\baselineskip]
A Dissertation submitted to\\[\baselineskip]
The Faculty of\\The School of Engineering and Applied Science\\ of The George
Washington University\\ in partial satisfaction of the requirements\\ for the degree
of Doctor of Philosophy\\[\baselineskip]
August 31, 2015\\[\baselineskip]
Dissertation directed by\\[\baselineskip]
Taeyoung Lee\\Professor of Engineering and Applied Science 
\end{center}
\pagestyle{plain}
\setcounter{page}{1}
\pagenumbering{roman}
\newpage
\doublespacing
\noindent The School of Engineering and Applied Science of The George Washington
University certifies that Farhad A. Goodarzi has passed the Final Examination for the
degree of Doctor of Philosophy as of August 31, 2015. This is the
final and approved form of the dissertation.
\begin{center}
\singlespacing
{\normalsize Geometric Nonlinear Control for Multiple Cooperative Quadrotor UAVs Transporting a Rigid Body}\\[\baselineskip]
{\normalsize Farhad A. Goodarzi}
\end{center}\vspace{1.0cm}
\noindent Dissertation Research Committee:\\ 
\hfill\begin{minipage}{5in}\vspace{1.0cm}
{Taeyoung Lee, Associate Professor, Dissertation Director}\\
{Azim Eskandarian, Professor, Committee Chairman}\\
{Adam Wickenheiser, Assistant Professor, Committee Member}\\
{Evan Drumwright, Assistant Professor, Committee Member} \\
{Jinglai Shen, Associate Professor, Committee Member}
\end{minipage}
\newpage
\vspace*{\fill}
\begin{center}
\singlespacing
$\copyright$ Copyright 2015 by Farhad A. Goodarzi\\ All rights reserved
\end{center}
\vspace*{\fill}
\newpage
\section{\protect\centering{Dedication}}
\indent I dedicate my dissertation work to my family and many friends. A special feeling of gratitude to my loving mother Giti Moradi who has always loved me unconditionally and whose good examples have taught me to work hard for the things that I aspire to achieve. This work is also dedicated to my deceased father who I lost at the begging of my Ph.D. program and his words of encouragement and push for tenacity ring in my ears. \\

\indent My twin sister Shirin has never left my side and is very special. I also dedicate this dissertation to my many friends who have supported me throughout the process. I will always appreciate all they have done, especially Naeem Masnadi and Parisa Abdollahi for a constant source of support and encouragement during the challenges of graduate school and life. I am truly thankful for having you in my life. \\

\indent To Dr. Taeyoung Lee, my advisor, whose passion for excellence is apparent in all that he does. His motivational guidance, intellectual tenacity, and editing skills are unsurpassed. 

\doublespacing
\newpage
\section{\protect \centering Abstract}
\begin{center}
\doublespacing{{\normalsize Geometric Nonlinear Control for Multiple Cooperative Quadrotor UAVs Transporting a Rigid Body}}
\end{center}
This dissertation presents nonlinear tracking control systems for quadrotor unmanned aerial vehicles (UAV) under the influence of uncertainties. Assuming that there exist unstructured disturbances in the translational dynamics and the attitude dynamics, a geometric nonlinear adaptive controller is developed directly on the special Euclidean group. In particular, a new form of an adaptive control term is proposed to guarantee stability while compensating the effects of uncertainties in quadrotor dynamics. Next, we derived a coordinate-free form of equations of motion for a complete model of a quadrotor UAV with a payload which is connected via a flexible cable according to Lagrangian mechanics on a manifold. The flexible cable is modeled as a system of serially-connected links and has been considered in the full dynamic model. A geometric nonlinear control system is presented to asymptotically stabilize the position of the quadrotor while aligning the links to the vertical direction below the quadrotor. Finally, we focused on the dynamics and control of arbitrary number of quadrotor UAVs transporting a rigid body payload. The rigid body payload is connected to quadrotors via flexible cables. It is shown that a coordinate-free form of equations of motion can be derived for arbitrary numbers of quadrotors and links according to Lagrangian mechanics on a manifold. A geometric nonlinear controller is presented to transport the rigid body to a fixed desired position while aligning all of the links along the vertical direction. Numerical simulation and experimental results are presented and rigorous mathematical stability analysis are provided. These results will be particularly useful for aggressive load transportation that involves large deformation of the cable.

\doublespacing
\newpage
\newgeometry{top=1in, bottom=1in, left=1.125in, right=1.125in}
\tableofcontents
\restoregeometry
\newpage
\cleardoublepage
\phantomsection \label{listoffig}

%%%%%%%%%%%%%%%%%%%%%%%%%%%%%%%%%%%%%%%%%%%%%%%%%%%%%
%%%%%%%%%%%%%%%%%%%%%%%%%%%%%%%%%%%%%%%%%%%%%%%%%%%%%

\addcontentsline{toc}{section}{\hspace{1.5pt} List of Figures}
\begin{center}
\listoffigures
\end{center}

\newpage
\setcounter{page}{1}
\setcounter{section}{1}
\pagenumbering{arabic}

%%%%%%%%%%%%%%%%%%%%%%%%%%%%%%%%%%%%%%%%%%%%%%%%%%%%%
%%%%%%%%%%%%%%%%%%%%%%%%%%%%%%%%%%%%%%%%%%%%%%%%%%%%%

\newpage
\begin{singlespace}
\section{\protect \centering Chapter 1: Introduction}
\end{singlespace}
\doublespacing
\setcounter{section}{1}
\subsection {\normalsize Motivation}
{\addtolength{\leftskip}{0.5in}
Quadrotor unmanned aerial vehicles (UAVs) are becoming increasingly popular. They offer flight characteristics comparable to traditional helicopters, namely stationary, vertical, and lateral flights in a wide range of speeds, with a much simpler mechanical structure. With their small-diameter rotors driven by electric motors, these multi-rotor platforms are safer to operate than helicopters in indoor environments. Also, they have sufficient payload transporting capability and flight endurance for various missions~\cite{Mahony2012,gooddaewontaeyoungacc14,IJCAS2015,farhadacc15}. UAVs have been studied for different applications such as surveillance or mobile sensor networks as well as for educational purposes. Quadrotors are very popular due to their dynamic simplicity, maneuverability and high performance. Areal transportation of a cable-suspended load has been studied traditionally for helicopters~\cite{CicKanJAHS95,BerPICRA09}. 
Recently, small-size single or multiple autonomous vehicles are considered for load transportation and deployment~\cite{PalCruIRAM12,MicFinAR11,MazKonJIRS10,MelShoDARSSTAR13}, and trajectories with minimum swing and oscillation of payload are generated~\cite{ ZamStaJDSMC08, SchMurIICRA12, PalFieIICRA12}.\\ 

Safe cooperative transportation of possibly large or bulky payloads is extremely important in various missions, such as military operations, search and rescue, mars surface explorations and personal assistance.
\begin{figure}[h]
\centerline{
	\subfigure{
        \includegraphics[width=0.32\columnwidth]{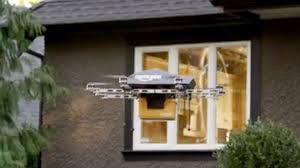}}
        \subfigure{
        \includegraphics[width=0.32\columnwidth]{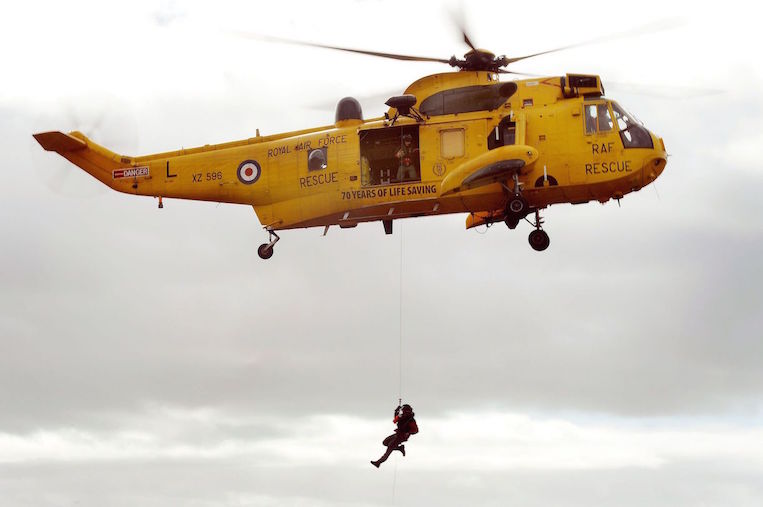}}
	\subfigure{
	\includegraphics[width=0.32\columnwidth]{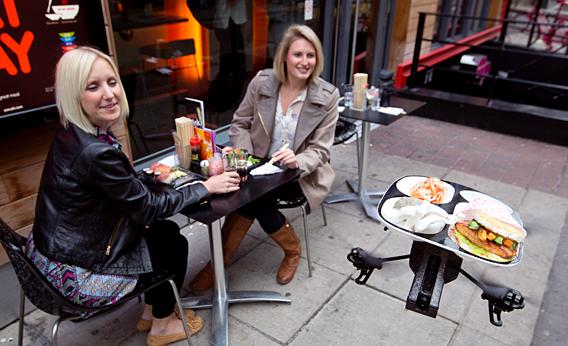}}
}
\caption{Arial payload transportation}\label{fig:exampleso}
\end{figure}

However, these results are based on the common and restrictive assumption that the cable connecting the payload to the quadrotor UAV is always taut and rigid. Also, the dynamic of the cable and payload are ignored and they are considered as bounded disturbances to the transporting vehicle. Therefore, they cannot be applied to aggressive, rapid load transportations where the cable is deformed or the tension along the cable is low, thereby restricting its applicability. As such, it is impossible to guarantee safety operations.\\

It is challenging to incorporate the effects of a deformable cable, since the dimension of the configuration space becomes infinite. Finite element approximation of a cable often yields complicated equations of motion that make dynamic analysis and controller design extremely difficult.

}
\subsection {\normalsize Literature Review}
{\addtolength{\leftskip}{0.5in}
Several control systems have been proposed for quadrotors. In many cases, disturbances and uncertainties are eliminated in the model for simplicity. There are other limitations of quadrotor control systems, such as complexities in controller structures or lack of stability proof. For example, tracking control of a quadrotor UAV has been considered in~\cite{CabCunPICDC09,MelKumPICRA11}, but the control system in~\cite{CabCunPICDC09} has a complex structure since it is based on a multiple-loop backstepping approach, and no stability proof is presented in~\cite{MelKumPICRA11}. Robust tracking control systems are studied in~\cite{NalMarPICDC09,HuaHamITAC09}, but the quadrotor dynamics is simplified by considering planar motion only~\cite{NalMarPICDC09}, or by ignoring the rotational dynamics by timescale separation assumption~\cite{HuaHamITAC09}.\\ 

In other studies, disturbances and uncertainties have been considered into the dynamics of the quadrotors~\cite{NormanKam2013,HassanFaiz2013,Besnard2007,Bolandi2013}. Several controllers have been designed and presented to eliminate these disturbances such as PID~\cite{Sharma2012}, sliding mode~\cite{Liu2013}, or robust controllers~\cite{Wahyudie2013}. In one example, proportional-derivative controllers are developed with consideration of blade flapping for operations under wind disturbances~\cite{HofHuaAGNCC07}. A backstepping control method is proposed in~\cite{Castillo2006} by considering an aggressive perturbation with bounded signals. These approaches have certain limitations on handling uncertainties. For example, it is well known that sliding mode controller causes chattering problems that may excite high-frequency unmodeled dynamics. Nonlinear robust tracking control systems in~\cite{LeeLeoPACC12,LeeLeoAJC13} guarantee ultimate boundedness of tracking errors only, and they are also prone to chattering if the required ultimate bound is smaller. PID controllers are adopted widely, but it is required that the uncertainties are fixed. \\

Due to the limitations mentioned above, adaptive controllers have been developed~\cite{dydekAnnnaLav12,Justin2014,Antonelli2013,HZhenchina13}. Although adaptive controllers compensate the effects of disturbances and uncertainties for quadrotor controls, most of the studies are based on linearization~\cite{dydekAnnnaLav12,Justin2014} or simplification~\cite{Antonelli2013,HZhenchina13}. In~\cite{Antonelli2013} only the constant external disturbances is considered into the system dynamics and the stability analysis. In~\cite{HZhenchina13} an adaptive block backstepping controller is presented to stabilize the attitude of a quadrotor, however this method only guarantees the boundedness of errors. A nonlinear adaptive state feedback controller is also presented in~\cite{Silvestre2014}, where the proposed controller only assumes constant known disturbance forces. An adaptive sliding mode controller is developed for under-actuated quadrotor dynamics in~\cite{DWLee2009}. This controller uses slack variables to overcome the under-actuated property of a quadrotor system while simplifying the dynamics to reduce the higher-order derivative terms which makes it very sensitive to the noise. A robust adaptive control of a quadrotor is also presented in~\cite{Bialy2013}, where linear-in-the-parameter uncertainties and bounded disturbances are considered. These simplifications~\cite{Antonelli2013,Silvestre2014}, linearization~\cite{dydekAnnnaLav12,Justin2014}, and assumptions~\cite{HZhenchina13} in the dynamics and controller design process of adaptive controllers restrict the quadrotor to maintain complex or aggressive missions such as a flipping maneuver~\cite{Bialy2013}.\\ 

The other critical issue in designing controllers for quadrotors is that they are mostly based on local coordinates. Local coordinates such as Euler angles and quaternions produce singularities and the dynamic model needs to be examined and studied more explicitly~\cite{farhadayoubi2011}. For quadrotor UAV's dynamics, some aggressive maneuvers are demonstrated at~\cite{MelMicIJRR12} which are based on Euler angles. Therefore they involve complicated expressions for trigonometric functions, and they exhibit singularities in representing quadrotor attitudes, thereby restricting their ability to achieve complex rotational maneuvers significantly. A quaternion-based feedback controller for attitude stabilization was shown in~\cite{TayMcGITCSTI06}. By considering the Coriolis and gyroscopic torques explicitly, this controller guarantees exponential stability. Quaternions do not have singularities but, as the three-sphere double-covers the special orthogonal group, one attitude may be represented by two antipodal points on the three-sphere. This ambiguity should be carefully resolved in quaternion-based attitude control systems, otherwise they may exhibit unwinding, where a rigid body unnecessarily rotates through a large angle even if the initial attitude error is small~\cite{BhaBerSCL00}. To avoid these, an additional mechanism to lift attitude onto the unit-quaternion space is introduced~\cite{MaySanITAC11}.\\

The dynamics of a quadrotor UAV is globally expressed on the special Euclidean group, $\SE$, and nonlinear control systems are developed to track outputs of several flight modes~\cite{LeeLeoPICDC10}. There are also several studies using the estimations for dynamical objects developed on the special Euclidean group~\cite{Misra2015coupled,ACC2014,Space2014}. Several aggressive maneuvers of a quadrotor UAV or spacecrafts are demonstrated based on a hybrid control architecture, and a nonlinear robust control system is also considered in~\cite{LeeLeoPACC12,Farhad2013}. As they are directly developed on the special Euclidean/Orthogonal group, complexities, singularities, and ambiguities associated with minimal attitude representations or quaternions are completely avoided~\cite{ChaSanICSM11,Automatica_2014,ICRA2015}.\\

In most of the prior works, the dynamics of aerial transportation has been simplified due to the inherent dynamic complexities. For example, it is assumed that the dynamics of the payload is considered completely decoupled from quadrotors, and the effects of the payload and the cable are regarded as arbitrary external forces and moments exerted to the quadrotors~\cite{ ZamStaJDSMC08, SchMurIICRA12, PalFieIICRA12}, thereby making it challenging to suppress the swinging motion of the payload actively, particularly for agile aerial transportations.\\

Recently, the coupled dynamics of the payload or cable has been explicitly incorporated into control system design~\cite{LeeSrePICDC13}. In particular, a complete model of a quadrotor transporting a payload modeled as a point mass, connected via a flexible cable is presented, where the cable is modeled as serially connected links to represent the deformation of the cable~\cite{gooddaewontaeyoungacc14}. In another distinct study, multiple quadrotors transporting  a rigid body payload has been studied~\cite{LeeMultipleRigid14}, but it is assume that the cables connecting the rigid body payload and quadrotors are always taut. These assumptions and simplifications in the dynamics of the system reduce the stability of the controlled system, particularly in rapid and aggressive load transportation where the motion of the cable and payload is excited nontrivially.

}
\subsection {\normalsize Proposed Approches}
{\addtolength{\leftskip}{0.5in}
Geometric nonlinear controllers are developed to follow an attitude tracking command and a position tracking command. In particular, a new form of an adaptive control term is proposed to guarantee asymptotical convergence of tracking error variables when there exist uncertainties at the translational dynamics and the rotational dynamics of quadrotors. The corresponding stability properties are analyzed mathematically, and they are verified by several experiments~\cite{farhadASME15}. The robustness of the proposed tracking control systems are critical in generating complex maneuvers, as the impact of the several aerodynamic effects resulting from the variation in air speed is significant even at moderate velocities~\cite{HofHuaAGNCC07}.\\

A coordinate-free form of the equations of motion for a chain pendulum connected a cart that moves on a horizontal plane is presented according to Lagrangian mechanics on a manifold~\cite{LeeLeoPICDC12}. The cable is modeled as an arbitrary number of links with different sizes and masses that are serially-connected by spherical joints. The resulting configuration manifold is the product of the special Euclidean group for the position and the attitude of the quadrotor, and a number of two-spheres that describe the direction of each link. We present Euler-Lagrange equations of the presented quadrotor model that are globally defined on the nonlinear configuration manifold.\\

The coupled dynamics of the payload or cable has been explicitly incorporated into control system design~\cite{LeeSrePICDC13}. In particular, a complete model of a quadrotor transporting a payload modeled as a point mass, connected via a flexible cable is presented. In another distinct study, multiple quadrotors transporting  a rigid body payload has been studied~\cite{LeeMultipleRigid14}, but it is assumed that the cables connecting the rigid body payload and quadrotors are always taut. These assumptions and simplifications in the dynamics of the system reduce the stability of the controlled system, particularly in rapid and aggressive load transportation where the motion of the cable and payload is excited nontrivially.\\

Quadrotor UAV is under-actuated as the direction of the total thrust is always fixed relative to its body. By utilizing geometric control systems for quadrotor~\cite{LeeLeoPICDC10,LeeLeoAJC13,Farhad2013}, we show that the hanging equilibrium of the links can be asymptotically stabilized while translating the quadrotor to a desired position. In contrast to existing papers where the force and the moment exerted by the payload to the quadrotor are considered as disturbances, the control systems proposed in this work explicitly consider the coupling effects between the cable/load dynamics and the quadrotor dynamics.

}
\subsection {\normalsize Contributions}
{\addtolength{\leftskip}{0.5in}
The first distinct feature of the proposed approach is that the equations of motion and the control systems are developed directly on the nonlinear configuration manifold in a coordinate-free fashion. This yields remarkably compact expressions for the dynamic model and controllers, compared with local coordinates that often require symbolic computational tools due to complexity of multi-body systems. Furthermore, singularities of local parameterization are completely avoided to generate agile maneuvers in a uniform way.\\

This work presents a rigorous Lyapunov stability analysis to establish stability properties without any timescale separation assumptions or singular perturbation, and a new nonlinear integral and adaptive control term are designed to guarantee robustness against unstructured uncertainties in both rotational and translational dynamics.\\

The other distinct contribution of this dissertation is presenting the complete dynamic model of an arbitrary number of quadrotors transporting a rigid body where each quadrotor is connected to the rigid body via a flexible cable. Each flexible cable is modeled as an arbitrary number of serially connected links, and it is valid for various masses and lengths. A coordinate free form of equations of motion is derived according to Lagrange mechanics on a nonlinear manifold for the full dynamic model. These sets of equations of motion are presented in a complete and organized manner without any simplification.\\

Another contribution of this study is designing a control system to stabilize the rigid body at desired position. Geometric nonlinear controller is utilized~\cite{LeeLeoPICDC10,LeeLeoAJC13,Farhad2013}, and they are generalized for the presented model. More explicitly, we show that the rigid body payload is asymptotically transported into a desired location, while aligning all of the links along the vertical direction corresponding to a hanging equilibrium.\\

The unique property of the proposed control system is that the nontrivial coupling effects between the dynamics of rigid payload, flexible cables, and multiple quadrotors are explicitly incorporated into control system design, without any simplifying assumption. Also, the equations of motion and the control systems are developed directly on the nonlinear configuration manifold intrinsically. Therefore, singularities of local parameterization are completely avoided to generate agile maneuvers of the payload in a uniform way. In short, the proposed control system is particularly useful for rapid and safe payload transportation in complex terrain, where the position of the payload should be controlled concurrently while suppressing the deformation of the cables.

}

\subsection {\normalsize Publications}
{\addtolength{\leftskip}{0.5in}
The contributions in this dissertation have been published in the following journals and conference proceedings.
\begin{itemize}
\item Farhad A. Goodarzi, Daewon Lee, Taeyoung Lee, Geometric Control of a Quadrotor UAV Transporting a Payload Connected via Flexible Cable, International Journal of Control, Automation and Systems, Vol. 13, No. 6, December 2015.
\item Farhad A. Goodarzi, Daewon Lee, Taeyoung Lee, Geometric Adaptive Tracking Control of a Quadrotor UAV on SE(3) for Agile Maneuvers, ASME Journal of Dynamic Systems, Measurement and Control, 137(9), 091007-12, Sep 01, 2015
\item Farhad A. Goodarzi, Taeyoung Lee, Dynamic and Control of Quadrotor UAVs Transporting a Rigid Body Connected via Flexible Cables, in Proceeding of American Control Conference, pp. 4677-4682, Chicago, IL, July 2015
\item Farhad A. Goodarzi, Daewon Lee, Taeyoung Lee, Geometric Stabilization of a Quadrotor UAV with a Payload Connected by flexible Cable, in Proceeding of American Control Conference, pp. 4925-4930, Portland, OR, June 2014
\item Farhad A. Goodarzi, Daewon Lee, Taeyoung Lee, Geometric Nonlinear PID Control of a Quadrotor UAV on SO(3), in Proceeding of European Control Conference, pp. 3845-3850, Zurich, Switzerland, February 2013
\end{itemize}

}

%%%%%%%%%%%%%%%%%%%%%%%%%%%%%%%%%%%%%%%%%%%%%%%%%%%%%
%%%%%%%%%%%%%%%%%%%%%%%%%%%%%%%%%%%%%%%%%%%%%%%%%%%%%

\newpage
\begin{singlespace}
\section{\protect \centering Chapter 2-Quadrotor Autonomous Trajectory Tracking}
\end{singlespace}
\setcounter{section}{2}
\doublespacing
In this chapter, we present a complete model of a quadrotor transporting a payload connecting with a flexible cable. The flexible cable is modeled as an arbitrary number of serially connected links. Geometric nonlinear controller presented for the full dynamic model considering the coupling effects of payload and the cable.

New contributions and the unique features of the control system proposed in this chapter compared with other studies are as follows: (i) it is developed for the full six degrees of freedom dynamic model of a quadrotor UAV on $\SE$, including the coupling effects between the translational dynamics and the rotational dynamics on a nonlinear manifold without any simplification or assumptions, (ii) the control systems are developed directly on the nonlinear configuration manifold in a coordinate-free fashion.  Thus, singularities of local parameterization are completely avoided to generate agile maneuvers in a uniform way, (iii) a rigorous Lyapunov analysis is presented to establish stability properties without any timescale separation assumption, and (iv) a new form of an adaptive control term is proposed to guarantee asymptotical convergence of tracking error variables in the presence of uncertainties, (v) in contrast to hybrid control systems~\cite{GilHofIJRR11}, complicated reachability set analysis is not required to guarantee safe switching between different flight modes, as the region of attraction for each flight mode covers the configuration space almost globally, (vi) the proposed algorithm is validated with experiments for agile maneuvers. A rigorous mathematical analysis of nonlinear adaptive controllers for a quadrotor UAV on $\SE$ with experimental validations for complex and aggressive maneuvers is unprecedented. 

The chapter is organized as follows. We develop a globally defined model for a quadrotor UAV in Section \ref{sec:chap2dynamic}. Adaptive tracking control systems are developed in Sections \ref{sec:chap2controlattitude} and \ref{sec:chap2controlposition}, followed by numerical examples in Section \ref{sec:chap2numerical}.

%%%%%%%%%%%%%%%%%%%%%%%%%%%%%%%%%%%%%%%%%%%%%%%%%%%%%

\subsection {\normalsize Quadrotor Dynamic Model}\label{sec:chap2dynamic}
{\addtolength{\leftskip}{0.5in}
Consider a quadrotor UAV model illustrated in Figure~\ref{fig:QM}. We choose an inertial reference frame $\{\vec e_1,\vec e_2,\vec e_3\}$ and a body-fixed frame $\{\vec b_1,\vec b_2,\vec b_3\}$. The origin of the body-fixed frame is located at the center of mass of this vehicle. The first and the second axes of the body-fixed frame, $\vec b_1,\vec b_2$, lie in the plane defined by the centers of the four rotors.

\begin{figure}[h]
\setlength{\unitlength}{0.8\columnwidth}\footnotesize
\centerline{
\begin{picture}(1,0.8)(0,0)
\put(0,0){\includegraphics[width=0.8\columnwidth]{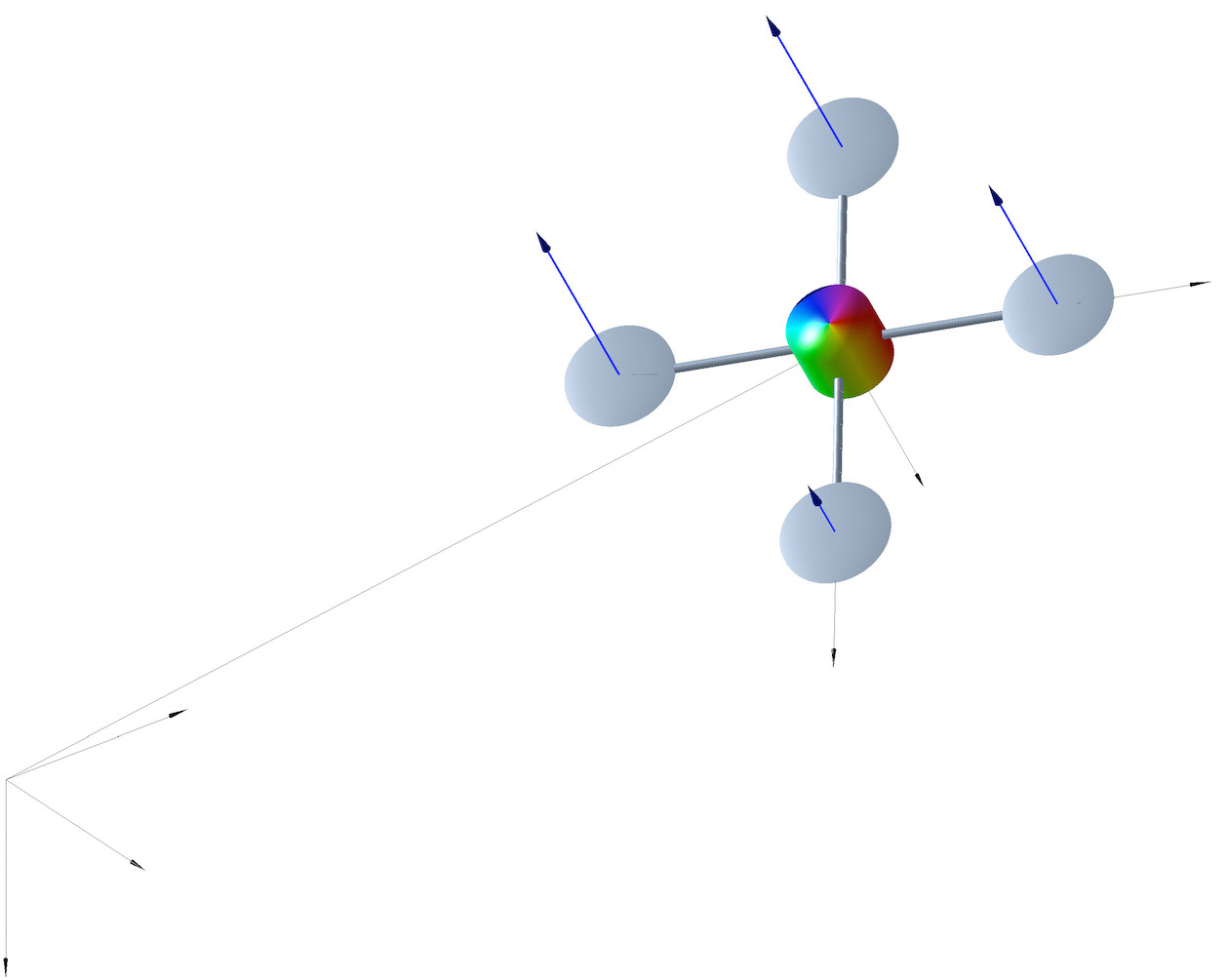}}
\put(0.16,0.18){\shortstack[c]{$\vec e_1$}}
\put(0.13,0.06){\shortstack[c]{$\vec e_2$}}
\put(0.02,0.0){\shortstack[c]{$\vec e_3$}}
\put(0.98,0.5){\shortstack[c]{$\vec b_1$}}
\put(0.70,0.22){\shortstack[c]{$\vec b_2$}}
\put(0.76,0.37){\shortstack[c]{$\vec b_3$}}
\put(0.78,0.66){\shortstack[c]{$f_1$}}
\put(0.56,0.76){\shortstack[c]{$f_2$}}
\put(0.40,0.63){\shortstack[c]{$f_3$}}
\put(0.61,0.42){\shortstack[c]{$f_4$}}
\put(0.30,0.35){\shortstack[c]{$x$}}
\put(0.90,0.35){\shortstack[c]{$R$}}
\end{picture}}
\caption{Quadrotor model}\label{fig:QM}
\end{figure}
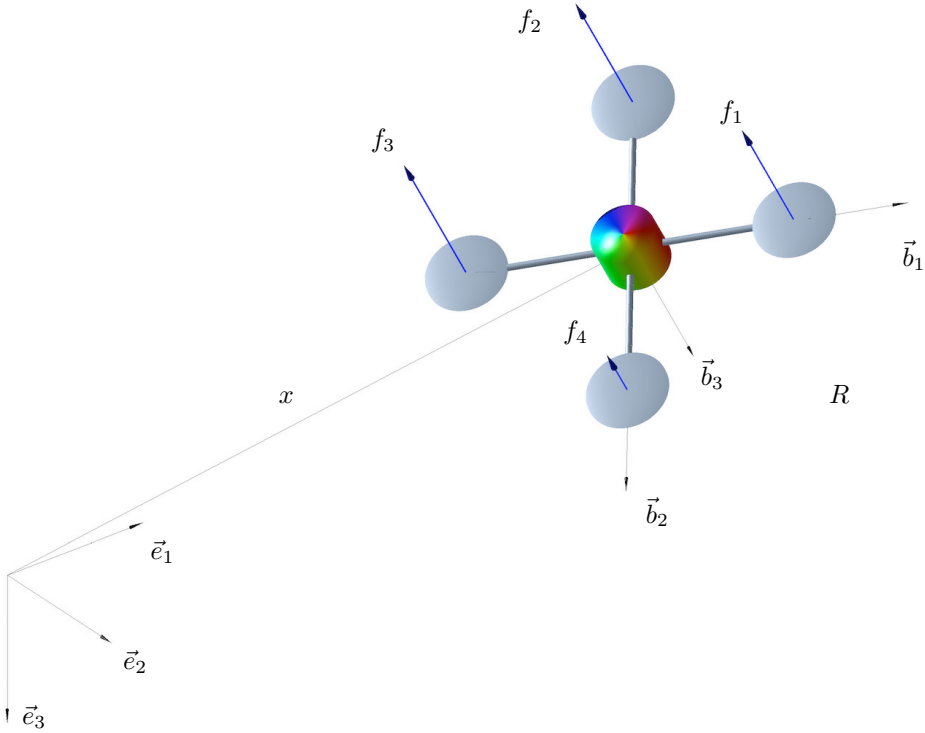

The configuration of this quadrotor UAV is defined by the location of the center of mass and the attitude with respect to the inertial frame. Therefore, the configuration manifold is the special Euclidean group $\SE$, which is the semi-direct product of $\Re^3$ and the special orthogonal group $\SO=\{R\in\Re^{3\times 3}\,|\, R^TR=I,\, \det{R}=1\}$. 

The mass and the inertial matrix of a quadrotor UAV are denoted by $m\in\Re$ and $J\in\Re^{3\times 3}$. Its attitude, angular velocity, position, and velocity are defined by $R\in\SO$, $\Omega,x,v\in\Re^3$, respectively, where the rotation matrix $R$ represents the linear transformation of a vector from the body-fixed frame to the inertial frame and the angular velocity $\Omega$ is represented with respect to the body-fixed frame. The distance between the center of mass to the center of each rotor is $d\in\Re$, and the $i$-th rotor generates a thrust $f_i$ and a reaction torque $\tau_i$ along $-\vec b_3$ for $1\leq i \leq 4$. The magnitude of the total thrust and the total moment in the body-fixed frame are denoted by $f\in\Re$, $M\in\Re^3$, respectively. 

The following conventions are assumed for the rotors and propellers, and the thrust and moment that they exert on the quadrotor UAV. We assume that the thrust of each propeller is directly controlled, and the direction of the thrust of each propeller is normal to the quadrotor plane. The first and third propellers are assumed to generate a thrust along the direction of $-\vec b_3$ when rotating clockwise; the second and fourth propellers are assumed to generate a thrust along the same direction of $-\vec b_3$ when rotating counterclockwise. Thus, the thrust magnitude is $f=\sum_{i=1}^4 f_i$, and it is positive when the total thrust vector acts  along $-\vec b_3$, and it is negative when the total thrust vector acts along $\vec b_3$. By the definition of the rotation matrix $R\in\SO$, the direction of the $i$-th body-fixed axis $\vec b_i$ is given by $Re_i$ in the inertial frame, where $e_1=[1;0;0],e_2=[0;1;0],e_3=[0;0;1]\in\Re^3$. Therefore, the total thrust vector is given by $-fRe_3\in\Re^3$ in the inertial frame.

We also assume that the torque generated by each propeller is directly proportional to its thrust. Since it is assumed that the first and the third propellers rotate clockwise and the second and the fourth propellers rotate counterclockwise to generate a positive thrust along the direction of $-\vec b_3$, %
the torque generated by the $i$-th propeller about $\vec b_3$ can be written as $\tau_i=(-1)^{i} c_{\tau f} f_i$  for a fixed constant $c_{\tau f}$.    All of these assumptions are fairly common in many quadrotor control systems~\cite{TayMcGITCSTI06,CasLozICSM05}.

Under these assumptions, the thrust of each propeller $f_1, f_2, f_3, f_4$ is directly converted into $f$ and $M$, or vice versa. In this paper, the thrust magnitude $f\in\Re$ and the moment vector $M\in\Re^3$ are viewed as control inputs. The corresponding equations of motion are given by
\begin{gather}
\dot x  = v,\label{eqn:EL1}\\
m \dot v = mge_3 - f R e_3 + \mathds{W}_{x}(x,v,R,\Omega)\theta_{x},\label{eqn:EL2}\\
\dot R = R\hat\Omega,\label{eqn:EL3}\\
J\dot \Omega + \Omega\times J\Omega = M + \mathds{W}_{R}(x,v,R,\Omega)\theta_{R},\label{eqn:EL4}
\end{gather}
where the \textit{hat map} $\hat\cdot:\Re^3\rightarrow\SO$ is defined by the condition that $\hat x y=x\times y$ for all $x,y\in\Re^3$. 
More explicitly, for a vector $x=[x_1,x_2,x_3]^{T}\in\Re^3$, the matrix $\hat x$ is given by
\begin{align}
    \hat x = \begin{bmatrix} 0 & -x_3 & x_2\\
                                x_3 & 0 & -x_1\\
                                -x_2 & x_1 & 0 \end{bmatrix}\label{eqn:hat}.
\end{align}
This identifies the Lie algebra $\SO$ with $\Re^3$ using the vector cross product in $\Re^3$. The inverse of the hat map is denoted by the \textit{vee} map, $\vee:\SO\rightarrow\Re^3$. 

The modeling error and uncertainties in the translational dynamics and the rotational dynamics are given by $\mathds{W}_{x}(x,v,R,\Omega)\theta_{x}$, and $\mathds{W}_{R}(x,v,R,\Omega)\theta_{R}$, respectively. Where $\mathds{W}_{x}(x,v,R,\Omega),\mathds{W}_{R}(x,v,R,\Omega)\in\Re^{3\times P}$ are known functions of the state, and $\theta_{x},\theta_{R}\in\Re^{P\times 1}$ are fixed unknown parameters. It is assumed that the bounds of unknown parameters are given by
\begin{align}
\|\mathds{W}_{x}\| \leq B_{W_{x}},\quad \|\theta_{x}\| \leq B_{\theta},\quad \|\theta_{R}\| \leq B_{\theta},
\end{align}
for $B_{W_{x}},B_{\theta}>0$. Throughout this chapter, $\lambda_m (A)$ and $\lambda_{M}(A)$ denote the minimum eigenvalue and the maximum eigenvalue of a square matrix $A$, respectively, and $\lambda_m$ and $\lambda_M$ are the minimum eigenvalue and the maximum eigenvalue of the inertia matrix $J$. i.e., $\lambda_m=\lambda_m(J)$ and $\lambda_M=\lambda_M(J)$. The two-norm of a matrix $A$ is denoted by $\|A\|$. The standard dot product in $\Re^n$ is denoted by $\cdot$, i.e., $x\cdot y =x^Ty$ for any $x,y\in\Re^n$.

}

%%%%%%%%%%%%%%%%%%%%%%%%%%%%%%%%%%%%%%%%%%%%%%%%%%%%%

\subsection {\normalsize Attitude Controlled Flight Mode}\label{sec:chap2controlattitude}
{\addtolength{\leftskip}{0.5in}
Since the quadrotor UAV has four inputs, it is possible to achieve asymptotic output tracking for at most four quadrotor UAV outputs. The quadrotor UAV has three translational and three rotational degrees of freedom; it is not possible to achieve asymptotic output tracking of both attitude and position of the quadrotor UAV. This  motivates us to introduce two flight modes, namely (1) an attitude controlled flight mode, and (2) a position controlled flight mode. While a quadrotor UAV is under-actuated, a complex flight maneuver can be defined by specifying a concatenation of flight modes together with conditions for switching between them. This will be further illustrated by a numerical example and experimental results later. In this section, the attitude controlled flight mode is considered. 

%%%%%%%%%%%%%%%%%%%%%%%%%%%%%%%%%%%%%%%%%%%%%%%%%%%%%

\subsubsection {\normalsize Attitude Tracking Errors}
Suppose that a smooth attitude command $R_d(t)\in\SO$ satisfying the following kinematic equation is given:
\begin{align}
\dot R_d = R_d \hat\Omega_d,
\end{align}
where $\Omega_d(t)\in\Re^3$ is the desired angular velocity, which is assumed to be uniformly bounded. We first define errors associated with the attitude dynamics as follows~\cite{BulLew05,LeeITCST13}.
\begin{prop}\label{prop:propchap1_first}
For a given tracking command $(R_d,\Omega_d)$, and the current attitude and angular velocity $(R,\Omega)$, we define an attitude error function $\Psi:\SO\times\SO\rightarrow\Re$, an attitude error vector $e_R\in\Re^3$, and an angular velocity error vector $e_\Omega\in \Re^3$ as follows~\cite{LeeITCST13}:
\begin{gather}
\Psi (R,R_d) = \frac{1}{2}\tr{G(I-R_d^TR)},\\
e_R =\frac{1}{2} (GR_d^TR-R^TR_dG)^\vee,\label{eqn:errr}\\
e_\Omega = \Omega - R^T R_d\Omega_d\label{eqn:eomega},
\end{gather}
where the matrix $G\in\Re^{3\times3}$ is given by $G = \diag[g_1, g_2, g_3]$ for distinct, positive constants $g_1$, $g_2$, $g_3$ $\in\Re$. Then, the following properties hold:
\begin{itemize}
\item[(i)] $\Psi$ is positive-definite about $R=R_d$.
\item[(ii)] The directional derivative of $\Psi$ with respect to $R$ along $\delta R = R\hat\eta$ for any $\eta\in\Re^3$ is given by
\begin{align}
\frac{d}{d\epsilon}\bigg|_{\epsilon=0} \Psi(R\exp(\epsilon\hat\eta),R_d) = e_R \cdot \eta. 
\end{align}
\item[(iii)] The critical points of $\Psi$, where $e_R=0$, are $\{R_d\}\cup\{R_d\exp (\pi \hat s),\,s\in\Sph^2 \}$, where $\Sph^2=\{s\in\Re^{3}\; |\; \|s\|=1\}$.
\item[(iv)] A lower bound of $\Psi$ is given as follows:
\begin{align}
b_{1}\|e_R\|^2 \leq \Psi(R,R_d),\label{eqn:PsiLB}
\end{align}
where the constant $b_1$ is given by $b_{1}=\frac{h_1}{h_2+h_3}$ for
\begin{gather}
h_{1}=\min\{g_{1}+g_{2},g_{2}+g_{3},g_{3}+g_{1}\}\nonumber,\\
h_{2}=\max\{(g_{1}-g_{2})^2,(g_{2}-g_{3})^2,(g_{3}-g_{1})^2\}\nonumber,\\
h_{3}=\max\{(g_{1}+g_{2})^2,(g_{2}+g_{3})^2,(g_{3}+g_{1})^2\}\nonumber.
\end{gather} 

\item[(v)] Let $\psi$ be a positive constant that is strictly less than $2$. If $\Psi(R,R_d)< \psi<h_1$, then an upper bound of $\Psi$ is given by
\begin{align}
\Psi(R,R_d)\leq b_2 \|e_R\|^2,\label{eqn:PsiUB}
\end{align}
where the constant $b_2$ is given by $b_{2}=\frac{h_{1}h_{4}}{h_{5}(h_{1}-\psi)}$ for
\begin{gather}
h_{4}=\max\{g_{1}+g_{2},g_{2}+g_{3},g_{3}+g_{1}\}\nonumber,\\
h_{5}=\min\{(g_{1}+g_{2})^2,(g_{2}+g_{3})^2,(g_{3}+g_{1})^2\}\nonumber.
\end{gather}

\item[(vi)] The time-derivative of $\Psi$ and $e_R$ satisfies:
\begin{align}
\dot\Psi = e_R\cdot e_\Omega,\quad \|\dot e_R\|\leq\|e_\Omega\|.\label{eqn:Psidot00}
\end{align}
\end{itemize}
\end{prop}
\begin{proof} 
See Appendix \ref{sec:pfchap1_first}.
\end{proof}

%%%%%%%%%%%%%%%%%%%%%%%%%%%%%%%%%%%%%%%%%%%%%%%%%%%%%

\subsubsection {\normalsize Attitude Tracking Controller}
We now introduce a nonlinear controller for the attitude controlled flight mode:
\begin{align}\label{eqn:aM}
M  =& -k_R e_R -k_\Omega e_\Omega -\mathds{W}_{R}\bar{\theta}_{R}+(R^TR_d\Omega_d)^\wedge J R^T R_d \Omega_d+ J R^T R_d\dot\Omega_d,
\end{align}
and adaptive law
\begin{align}\label{eqn:eI}
\dot{\bar{\theta}}_{R}=\gamma_{R}\mathds{W}_{R}^{T}(e_{\Omega}+c_{2}e_{R}),
\end{align}
where $k_R,k_\Omega,c_2,\gamma_{R}$ are positive constants and $\bar{\theta}_{R}\in\Re^{p}$ denotes the estimated value of $\theta_{R}$. The adaptive gains for $e_\Omega$ and $e_R$ correspond to $\gamma_R$ and $c_2\gamma_R$, respectively.

The control moment is composed of proportional, derivative, and adaptive terms, augmented with additional terms to cancel out the angular acceleration caused by the desired angular velocity. 
\begin{prop}{(Attitude Controlled Flight Mode)}\label{prop:propchap2_2}
Consider the control moment $M$ defined in \refeqn{aM}-\refeqn{eI}. For positive constants $k_R,k_\Omega$, the constants $c_2,B_2$ are chosen such that
\begin{gather}
\|(2J-\trs{J}I)\| \|\Omega_d\| \leq B_2,\label{eqn:B_2}\\
c_2 < \min\bigg\{  \frac{\sqrt{k_R\lambda_m}}{\lambda_M}, \frac{4k_\Omega}{8k_R\lambda_M+(k_\Omega+B_2)^2}\bigg\},\label{eqn:c2}
\end{gather}
Then, the zero equilibrium of tracking errors and the estimation errors is stable in the sense of Lyapunov, and $e_R,e_\Omega\rightarrow 0$ as $t\rightarrow\infty$. Furthermore the estimation error, $\tilde{\theta}_{R}=\theta_{R}-\bar{\theta}_{R}$, is uniformly bounded.
\end{prop}
\begin{proof}
See Appendix \ref{sec:pfchap2_2}.
\end{proof}

While these results are developed for the attitude dynamics of a quadrotor UAV, they can be applied to the attitude dynamics of any rigid body. Nonlinear adaptive controllers have been developed for attitude stabilization in terms of modified Rodriguez parameters~\cite{SubJAS04} and quaternions~\cite{SubAkeJGCD04}, and for attitude tracking in terms of Euler-angles~\cite{ShoJuaPACC02}. The proposed tracking control system is developed on $\SO$, therefore it avoids singularities of Euler-angles and Rodriguez parameters, as well as unwinding of quaternions. 

Asymptotic tracking of the quadrotor attitude does not require specification of the thrust magnitude. As an auxiliary problem, the thrust magnitude can be chosen in many different ways to achieve an additional translational motion objective. For example, it can be used to asymptotically track a quadrotor altitude command~\cite{LeeLeoAJC13}. Since the translational motion of the quadrotor UAV can only be partially controlled; this flight mode is most suitable for short time periods where an attitude maneuver is to be completed. 

}

%%%%%%%%%%%%%%%%%%%%%%%%%%%%%%%%%%%%%%%%%%%%%%%%%%%%%

\subsection {\normalsize Position Controlled Flight Mode}\label{sec:chap2controlposition}
{\addtolength{\leftskip}{0.5in}
We introduce a nonlinear controller for the position controlled flight mode in this section.
\subsubsection {\normalsize Position Tracking Errors}
Suppose that an arbitrary smooth position tracking command $x_d (t) \in \Re^3$ is given. The position tracking errors for the position and the velocity are given by:
\begin{align}
e_x  = x - x_d,\quad
e_v  = \dot e_x = v - \dot x_d.
\end{align}
In the position controlled tracking mode, the attitude dynamics is controlled to follow the computed attitude $R_c(t)\in\SO$ as illustrated in Figure~\ref{fig:controlfellow} and the computed angular velocity $\Omega_c(t)$ defined as
\begin{align}
R_c=[ b_{1_c};\, b_{3_c}\times b_{1_c};\, b_{3_c}],\quad \hat\Omega_c = R_c^T \dot R_c\label{eqn:RdWc},
\end{align}
where $b_{3_c}\in\Sph^2$ is given by
\begin{align}
 b_{3_c} = -\frac{-k_x e_x - k_v e_v -\mathds{W}_{x}\bar{\theta}_{x} - mg e_3 +m\ddot x_d}{\norm{-k_x e_x - k_v e_v -\mathds{W}_{x}\bar{\theta}_{x}- mg e_3 + m\ddot x_d}},\label{eqn:Rd3}
\end{align}
\begin{figure}[h]
\centerline{
	\subfigure{
        \includegraphics[width=0.9\columnwidth]{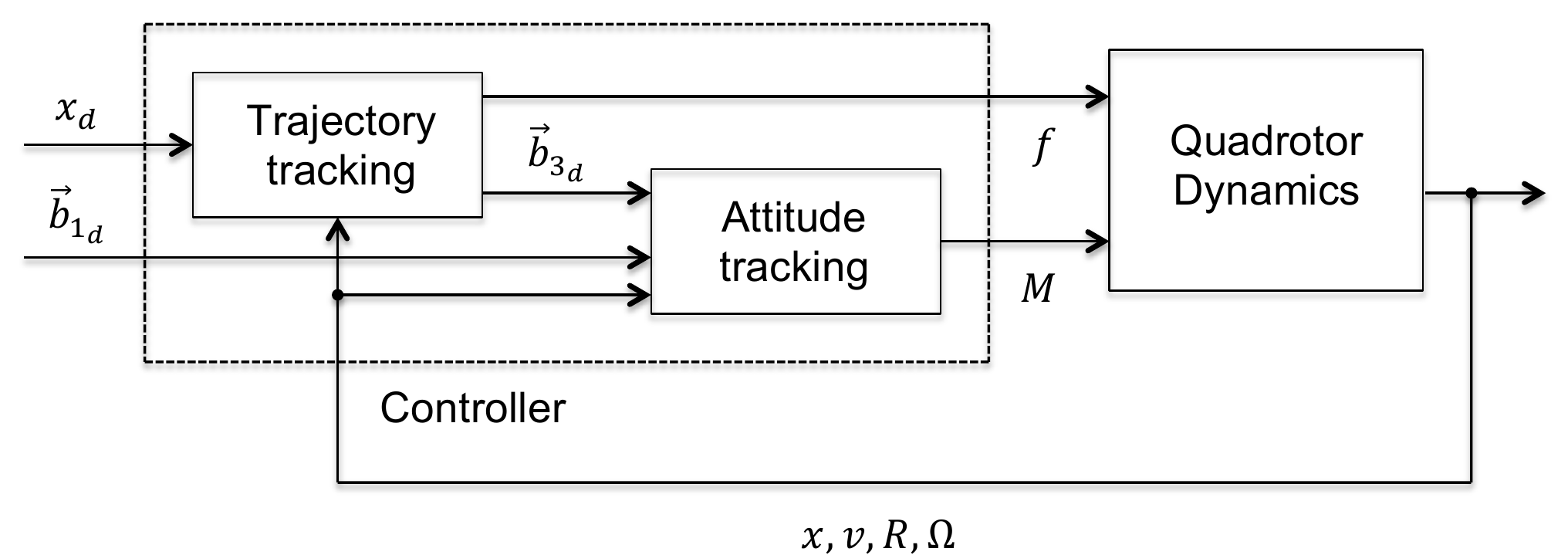}}
}
\caption{Controller loops}\label{fig:controlfellow}
\end{figure}
for positive constants $k_x,k_v$ and $\bar{\theta}_{x}\in\Re^{p}$ denotes the estimate value of $\theta_{x}$. The unit vector $b_{1_c}\in\Sph^2$ is selected to be orthogonal to $b_{3_c}$, thereby guaranteeing that $R_c\in\SO$. It can be chosen to specify the desired heading direction, and the detailed procedure to select $b_{1c}$ is described later at Section \ref{sec:b1c}.  
Following the prior definition of the attitude error and the angular velocity error given at \refeqn{errr} and \refeqn{eomega}, and assume that the commanded acceleration is uniformly bounded:
\begin{align}
\|-mge_3+m\ddot x_d\| < B_1\label{eqn:B}
\end{align}
for a given positive constant $B_1$. These imply that the given desired position command is distinctive from free-fall, where no control input is required.

%%%%%%%%%%%%%%%%%%%%%%%%%%%%%%%%%%%%%%%%%%%%%%%%%%%%%

\subsubsection {\normalsize Position Tracking Controller}
The nonlinear controller for the position controlled flight mode, described by control expressions for the thrust magnitude and the moment vector, are:
\begin{align}
f  =& ( k_x e_x + k_v e_v +\mathds{W}_{x}\bar{\theta}_{x}+ mg e_3-m\ddot x_d)\cdot Re_3,\label{eqn:f}\\
M  =& -k_R e_R -k_\Omega e_\Omega -\mathds{W}_{R}\bar{\theta}_{R}+(R^TR_c\Omega_c)^\wedge J R^T R_c \Omega_c + J R^T R_c\dot\Omega_c.\label{eqn:M}
\end{align}
In \ref{eqn:f}, we can show that the total thrust is always positive. The first term in this equation is the projection into $b_{3}$ and while controller works perfectly we can guarantee that the angle between two terms in \ref{eqn:f} is always less than 90 degree and it happens when $R$ is close to $R_{c}$. Similar with \refeqn{eI}, an adaptive control law for the position tracking controller is defined as
\begin{align}\label{eqn:adaptivelawx}
\dot{\bar{\theta}}_{x}&=\begin{cases}
\gamma_{x}\mathds{W}_{x}^{T}(e_{v}+c_{1}e_{x}) & \hspace{-0.4cm}\mbox{if } \|\bar{\theta}_{x}\|<B_{\theta}\\
& \hspace{-0.4cm}\mbox{or } \|\bar{\theta}_{x}\|=B_{\theta}\\
& \hspace{-0.4cm}\mbox{and } \bar{\theta}_{x}^{T}\mathds{W}_{x}^{T}(e_{v}+c_{1}e_{x})\leq 0\\
&\\
\gamma_{x}(I-\frac{\bar{\theta}_{x}\bar{\theta}_{x}^{T}}{\bar{\theta}_{x}^{T}\bar{\theta}_{x}})\mathds{W}_{x}^{T}(e_{v}+c_{1}e_{x}) & \mbox{Otherwise}
\end{cases},
\end{align}
for a positive adaptive gain $\gamma_{x}$ and a weighting parameter $c_1$. The corresponding effective adaptive gain for $e_x$ corresponds to $c_1\gamma_x$. The above expression correspond adaptive controls with projection~\cite{Ioannou96} and it is included to restrict the effects of attitude tracking errors on the translational dynamics as shown in Figure~\ref{fig:acl}

\begin{figure}[h]
\centerline{
	\subfigure{
        \includegraphics[width=0.36\columnwidth]{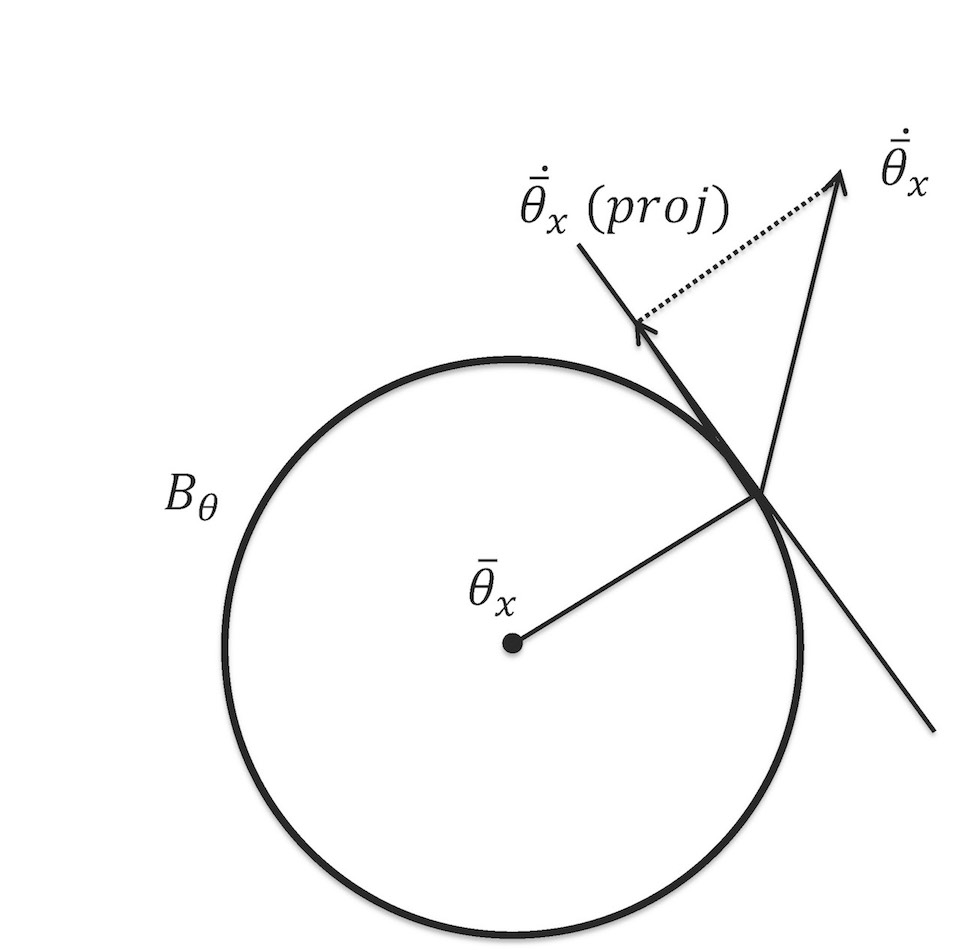}}
}
\caption{Adaptive control law}\label{fig:acl}
\end{figure}

The nonlinear controller given by equations \refeqn{f}, \refeqn{M} can be given a backstepping interpretation. The computed attitude $R_c$ given in equation \refeqn{RdWc} is selected so that the thrust axis $-b_3$ of the quadrotor UAV tracks the computed direction given by $-b_{3_c}$ in \refeqn{Rd3}, which is a direction of the thrust vector that achieves position tracking. The moment expression \refeqn{M} causes the attitude of the quadrotor UAV to asymptotically track $R_c$ and the thrust magnitude expression \refeqn{f} achieves asymptotic position tracking. 

\begin{prop}{(Position Controlled Flight Mode)}\label{prop:propchap2_3}
Suppose that the initial conditions satisfy
\begin{align}\label{eqn:Psi0}
\Psi(R(0),R_c(0)) < \psi_1 < 1,
\end{align}
for positive constant $\psi_1$. Consider the control inputs $f,M$ defined in \refeqn{f}-\refeqn{M}.
For positive constants $k_x,k_v$, we choose positive constants $c_1,c_2,k_R,k_\Omega$ such that
\begin{gather}
c_1 < \min\braces{\frac{4k_xk_v(1-\alpha)^2}{k_v^2(1+\alpha)^2+4m k_x(1-\alpha)},\; \sqrt{\frac{k_x}{m}} },\label{eqn:c1b}\\
\lambda_{m}(W_2) > \frac{\|W_{12}\|^2}{4\lambda_{m}(W_1)},\label{eqn:kRkWb}
\end{gather}
and \refeqn{c2} is satisfied, where $\alpha=\sqrt{\psi_1(2-\psi_1)}$, and the matrices $W_1,W_{12},W_2\in\Re^{2\times 2}$ are given by
\begin{align}
W_1 &= \begin{bmatrix} {c_1k_x}(1-\alpha) & -\frac{c_1k_v}{2}(1+\alpha)\\
-\frac{c_1k_v}{2}(1+\alpha) & k_v(1-\alpha)-mc_1\end{bmatrix},\label{eqn:W1}\\
W_{12}&=\begin{bmatrix}
{c_1}(B_{W_{x}} B_{\theta}+B_1) & 0 \\ B_{W_{x}} B_{\theta}+B_1+k_xe_{x_{\max}} & 0\end{bmatrix},\label{eqn:W12}\\
W_2 & = \begin{bmatrix} c_2k_R & -\frac{c_2}{2}(k_\Omega+B_2) \\ 
-\frac{c_2}{2}(k_\Omega+B_2) & k_\Omega-2c_2\lambda_M \end{bmatrix}.\label{eqn:W2}
\end{align}
This implies that the zero equilibrium of the tracking errors and the estimation errors is stable in the sense of Lyapunov and all of the tracking error variables asymptotically converge to zero. Also, the estimation errors are uniformly bounded.
\end{prop}
\begin{proof}
See Appendix \ref{sec:pfchap2_3}.
\end{proof}

Proposition \ref{prop:propchap2_3} requires that the initial attitude error is less than $90^\circ$ in \refeqn{Psi0}. Suppose that this is not satisfied, i.e. $1\leq\Psi(R(0),R_c(0))<2$. We can still apply Proposition \ref{prop:propchap2_2}, which states that the attitude error is asymptotically decreases to zero for almost all cases, and it satisfies \refeqn{Psi0} in a finite time. Therefore, by combining the results of Proposition \ref{prop:propchap2_2} and \ref{prop:propchap2_3}, we can show attractiveness of the tracking errors when $\Psi(R(0),R_c(0))<2$.

\begin{prop}{(Position Controlled Flight Mode with a Larger Initial Attitude Error)}\label{prop:propchap2_4}
Suppose that the initial conditions satisfy
\begin{gather}
1\leq \Psi(R(0),R_c(0)) < 2\label{eqn:eRb3},\quad \|e_x(0)\| < e_{x_{\max}},
\end{gather}
for a constant $e_{x_{\max}}$. Consider the control inputs $f,M$ defined in \refeqn{f}-\refeqn{M}, where the control parameters satisfy \refeqn{Psi0}-\refeqn{kRkWb} for a positive constant $\psi_1<1$. Then the zero equilibrium of the tracking errors is attractive, i.e., $e_x,e_v,e_R,e_\Omega\rightarrow 0$ as $t\rightarrow\infty$. 
\end{prop}
\begin{proof}
See Appendix \ref{sec:pfchap2_4}.
\end{proof}

Linear or nonlinear PID and adaptive controllers have been widely used for a quadrotor UAV. But, they have been applied in an \textit{ad-hoc} manner without stability analysis. This paper provides a new form of nonlinear adaptive controller on $\SE$ that guarantees almost global attractiveness in the presence of uncertainties. 

Compared with other PID controllers requiring that uncertain terms should be fixed, the proposed adaptive control system can handle more general class of uncertainties. The nonlinear robust tracking control system in~\cite{LeeLeoPACC12,LeeLeoAJC13} provides ultimate boundedness of tracking errors, and the control input may be prone to chattering if the required ultimate bound is smaller. Compared with~\cite{LeeLeoAJC13}, the control system in this paper guarantees stronger asymptotic stability, and there is no concern for chattering.

%%%%%%%%%%%%%%%%%%%%%%%%%%%%%%%%%%%%%%%%%%%%%%%%%%%%%

\subsubsection {\normalsize Direction of the First Body-Fixed Axis}\label{sec:b1c}
As described above, the construction of the orthogonal matrix $R_c$ involves having its third column $b_{3_c}$ specified by \refeqn{Rd3}, and its first column $b_{1_c}$ is arbitrarily chosen to be orthogonal to the third column, which corresponds to a one-dimensional degree of choice. 

By choosing $b_{1_c}$ properly, we constrain the asymptotic direction of the first body-fixed axis. Here, we propose to specify the \textit{projection} of the first body-fixed axis onto the plane normal to $b_{3_c}$. In particular, we choose a desired direction $b_{1_d}\in\Sph^2$, that is not parallel to $b_{3_c}$, and $b_{1_c}$ is selected as $b_{1_c}=\mathrm{Proj}[b_{1_d}]$, where $\mathrm{Proj}[\cdot]$ denotes the normalized projection onto the plane perpendicular to $b_{3_c}$. In this case, the first body-fixed axis does not converge to $b_{1_d}$, but it converges to the projection of $b_{1_d}$, i.e. $b_1\rightarrow b_{1_c}=\mathrm{Proj}[b_{1_d}]$ as $t\rightarrow\infty$. This can be used to specify the heading direction of a quadrotor UAV in the horizontal plane~\cite{LeeLeoAJC13}.

}

%%%%%%%%%%%%%%%%%%%%%%%%%%%%%%%%%%%%%%%%%%%%%%%%%%%%%

\subsection {\normalsize Numerical Example}\label{sec:chap2numerical}
{\addtolength{\leftskip}{0.5in}
An aggressive flipping maneuver is considered for a numerical simulation. The quadrotor parameters are chosen as
\begin{gather*}
J =
\begin{bmatrix}
    5.5711 & 0.0618 & -0.0251\\
    0.06177 & 5.5757 & 0.0101\\
    -0.02502 & 0.01007 & 1.05053
\end{bmatrix}\times 10^{-2}\,\mathrm{kgm}^2 ,\\ m = 0.755\,\mathrm{kg},\quad
d = 0.169\,\mathrm{m},\quad c_{\tau f} = 0.1056.
 \end{gather*}
Also, controller parameters are selected as follows: $k_x=6.0$, $k_v=3.0$, $k_R=0.7$, $k_\Omega = 0.12$, $c_1=0.1$, $c_2=0.1$.

In this simulation, initial state of the quadrotor UAV is at a hovering condition: $x(0)=v(0)=\Omega(0)=0_{3\times 1}$, and $R(0)=I_{3\times 3}$. The desired trajectory is a flipping maneuver where the quadrotor rotates about rotation axis $e_r = [\frac{\sqrt{2}}{2}, \frac{\sqrt{2}}{2}, 0]$ by $360^\circ$. This is a complex maneuver combining a nontrivial pitching maneuver with a yawing motion. It is achieved by concatenating the following two control modes:
\begin{itemize}
\item[(i)] Attitude tracking to rotate the quadrotor ($t\leq 0.375$)
\begin{align*}
&R_d(t)= I+\sin(4 \pi t)\hat{e}_r+(1-\cos(4 \pi t))(e_r e_r^T-I),\; \Omega_d= 4\pi\cdot e_r.
\end{align*}
\item[(ii)] Trajectory tracking to make it hover after completing the preceding rotation ($0.375<t \leq 2$)
\begin{align*}
x_d (t)=[0,0,0]^T,\quad b_{1_d} = [1,0,0]^T.
\end{align*}
\end{itemize}

We considered two cases for this numerical simulation to compare the effect of the proposed adaptive term in the presence of disturbances. In this numerical simulation we considered a special case of $W_{x}=I_{3\times 3}$ and $W_{R}=I_{3\times 3}$. Two cases are as follows: (i) with adaptive term and (ii) without the adaptive term, where constant disturbances are defined as
\begin{align*}
\theta_R=[0.03,-0.06,0.09]^T,\quad \theta_x=[0.25,0.125,0.2]^T
\end{align*}

The switching time is determined as follows. The desired angular velocity of the first attitude tracking mode is $4\pi$, which means it requires $0.5\;\mathrm{sec}$ for one revolution but the control mode is switched to trajectory tracking mode at $t=0.375\;\mathrm{sec}$ to compensate rotational inertia. It is also assumed that the maximum thrust at each rotor is given by $f_{i_{\max}}=3.2\,\mathrm{N}$, and any thrust command above the maximum thrust is saturated to represent the actual motor limitation in the numerical simulation.
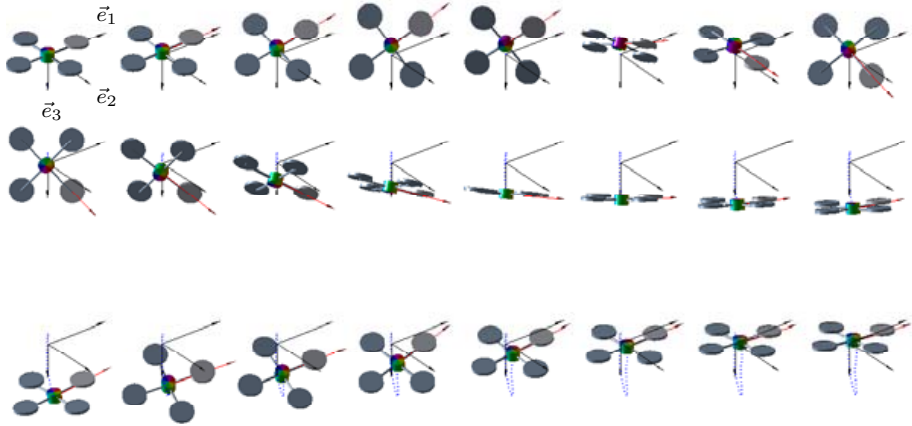
\begin{figure}[h]
\setlength{\unitlength}{0.08\columnwidth}\scriptsize
\centerline{
\begin{picture}(9.9,4.5)(0,0)
	\put(0,0){\includegraphics[width=0.8\columnwidth]{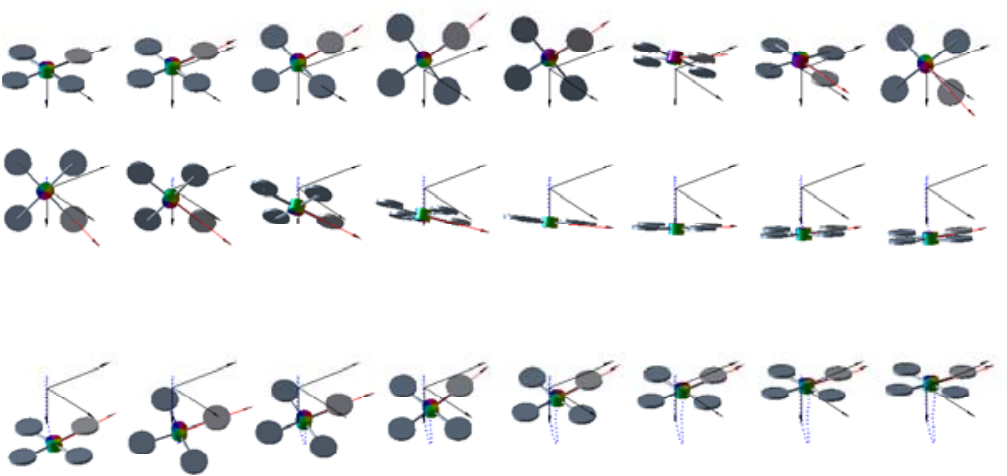}}
	\put(1.00,4.4){$\vec e_1$}
	\put(1.0,3.5){$\vec e_2$}
	\put(0.4,3.35){$\vec e_3$}
\end{picture}}
\caption{Snapshots of a flipping maneuver: the red line denotes the rotation axis $e_r=[\frac{1}{\sqrt{2}},\; \frac{1}{\sqrt{2}},\;0]$. The quadrotor UAV rotates about the $e_r$ axis by $360^\circ$. The trajectory of its mass center is denoted by blue, dotted lines.}\label{fig:Int3D}
\end{figure}
\begin{figure}
\centerline{
	\subfigure[Attitude error function $\Psi$]{
		\includegraphics[width=0.3\columnwidth]{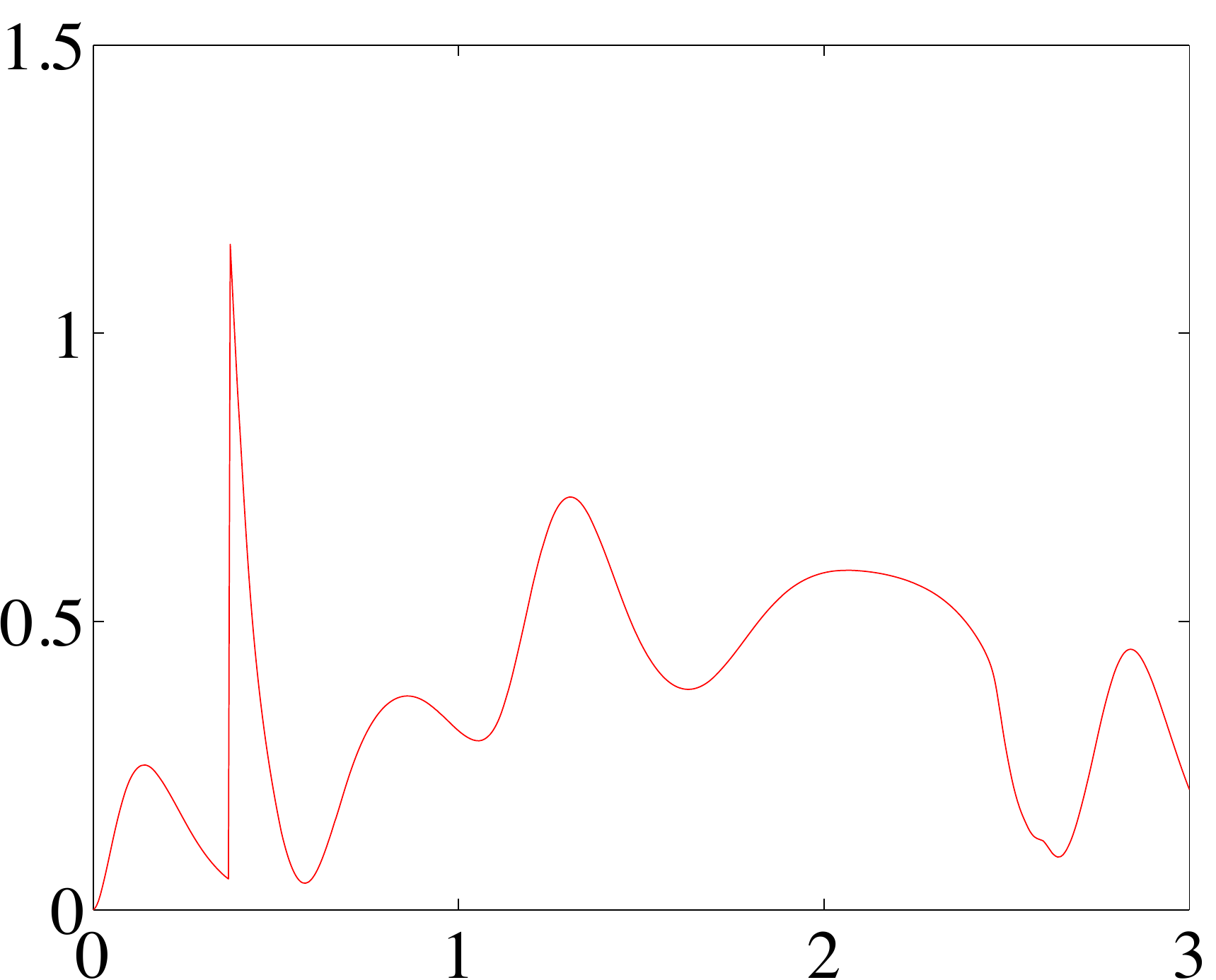}\label{fig:sim_error_WO}}
	\subfigure[Thrust at each rotor $f_i$ ($\mathrm{N}$)]{
		\includegraphics[width=0.3\columnwidth]{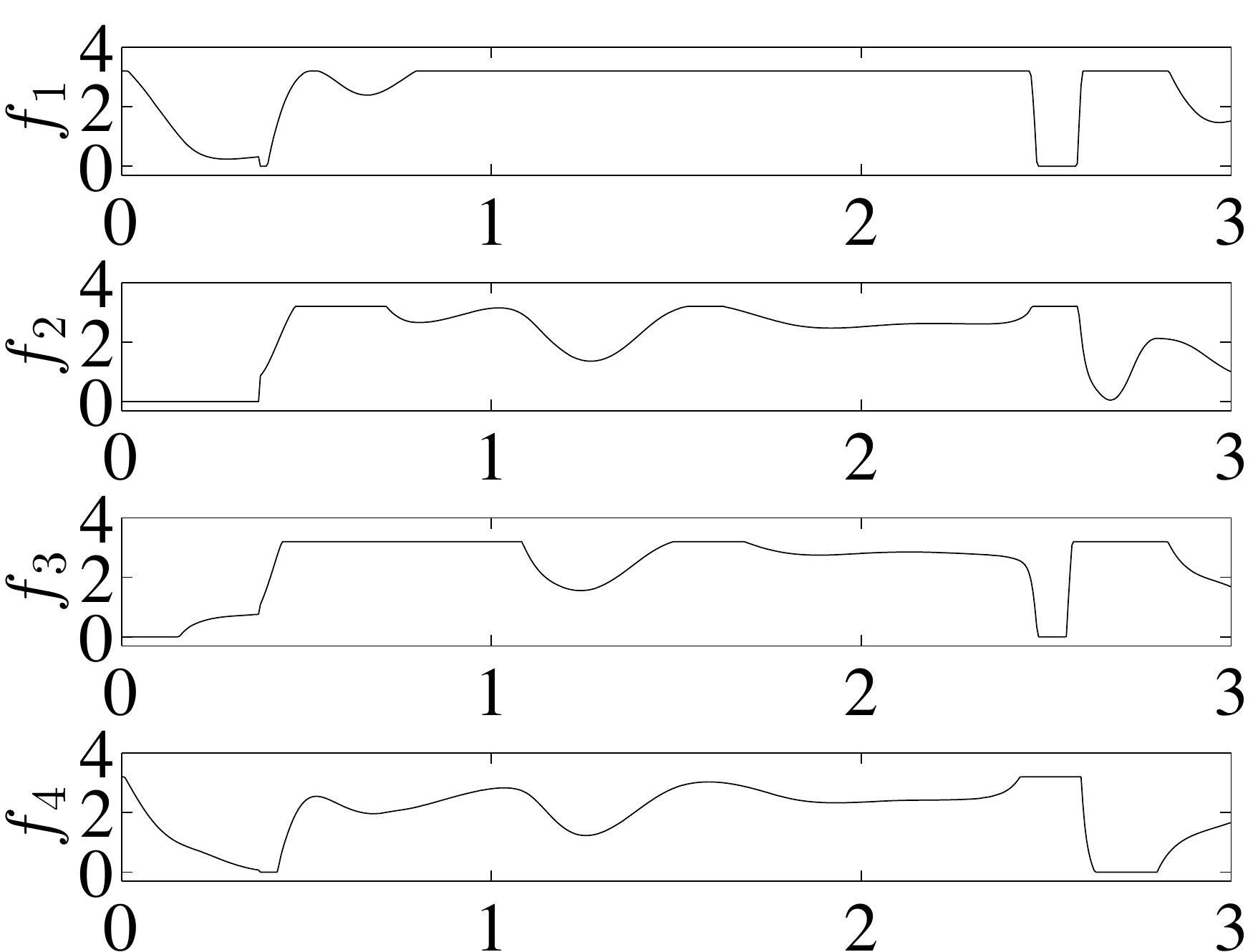}\label{fig:sim_force_WO}}
		\subfigure[Attitude error $e_{R}$ ($\mathrm{rad}$)]{
		\includegraphics[width=0.3\columnwidth]{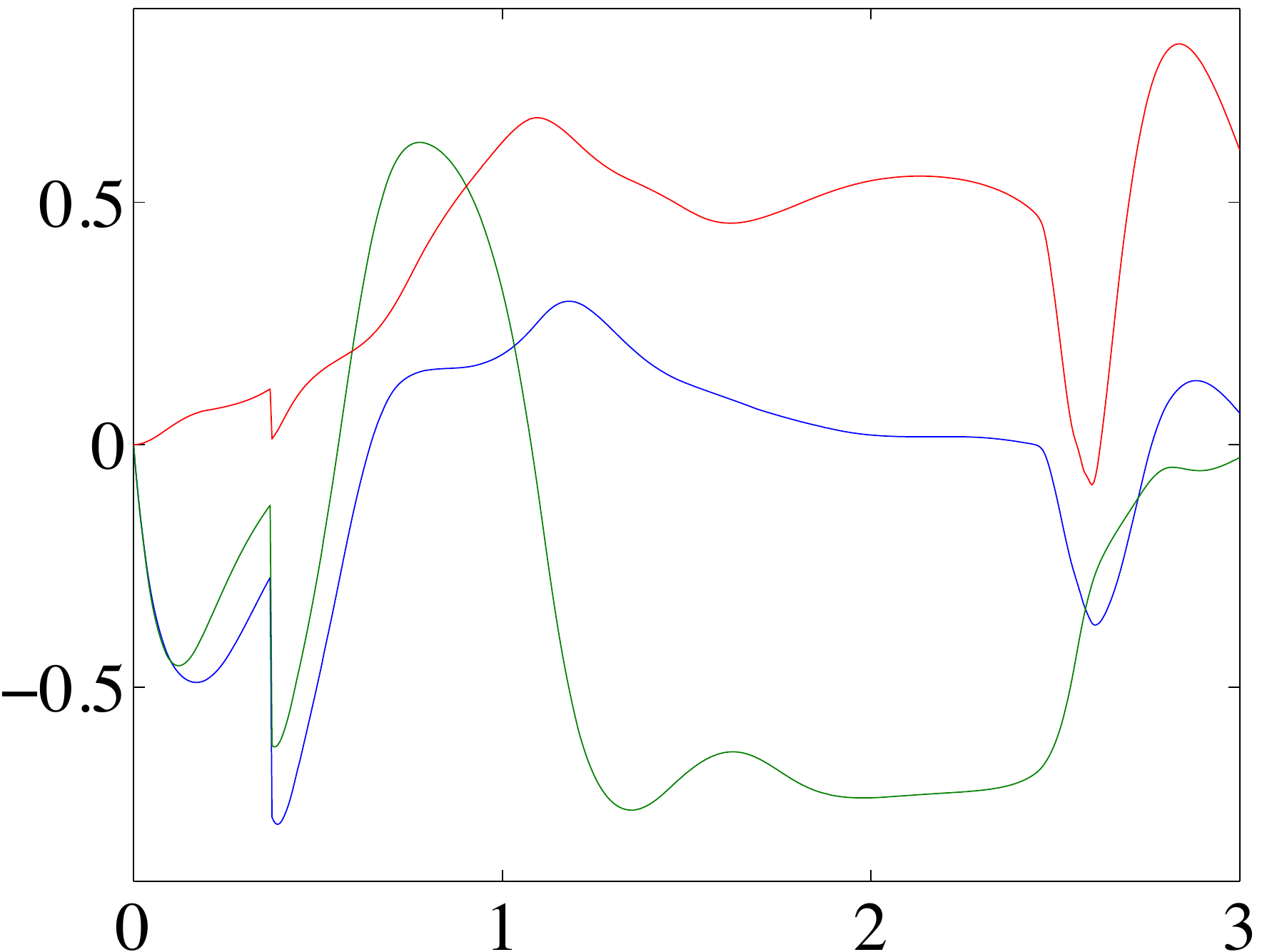}\label{fig:sim_eR_WO}}
}
\centerline{
	\subfigure[Angular Velocity error $e_{\Omega}$($\mathrm{rad}/\mathrm{sec}$)]{
		\includegraphics[width=0.3\columnwidth]{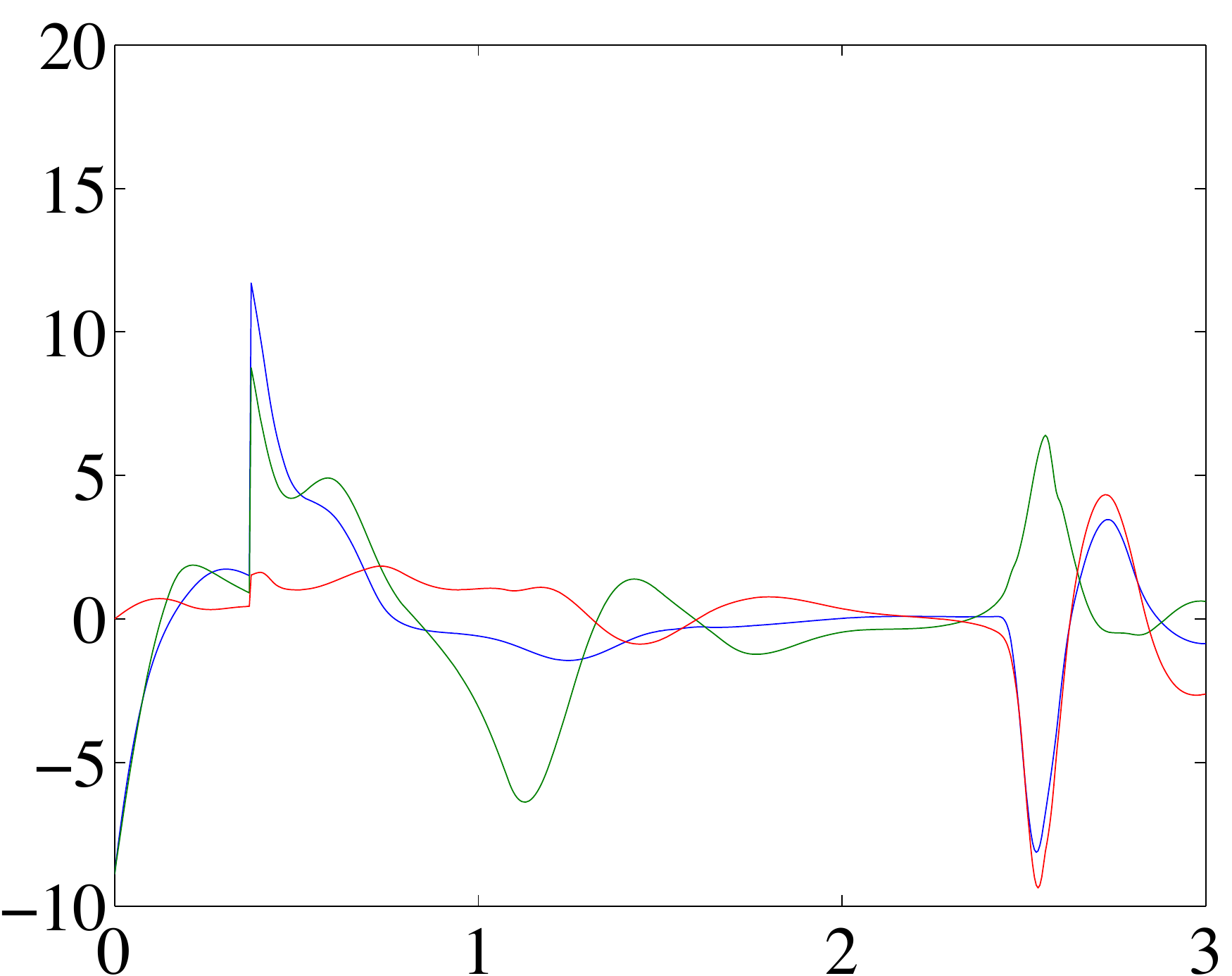}\label{fig:sim_eW_WO}}
		\subfigure[Angular Velocity $\Omega$,$\Omega_{d}$ ($\mathrm{rad/s}$)]{
		\includegraphics[width=0.3\columnwidth]{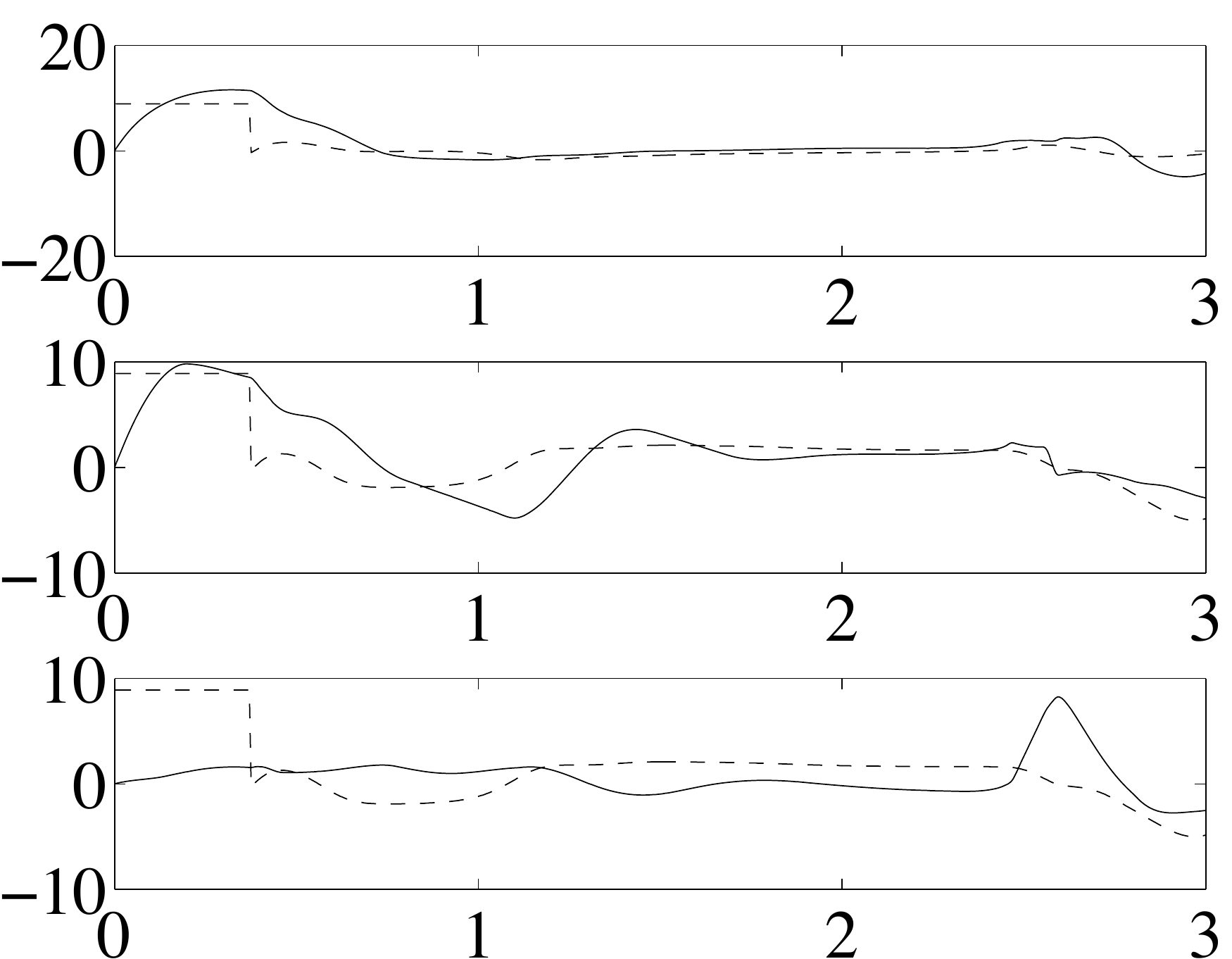}\label{fig:sim_ang_WO}}
	\subfigure[Position $x$,$x_{d}$($\mathrm{m}$)]{
		\includegraphics[width=0.3\columnwidth]{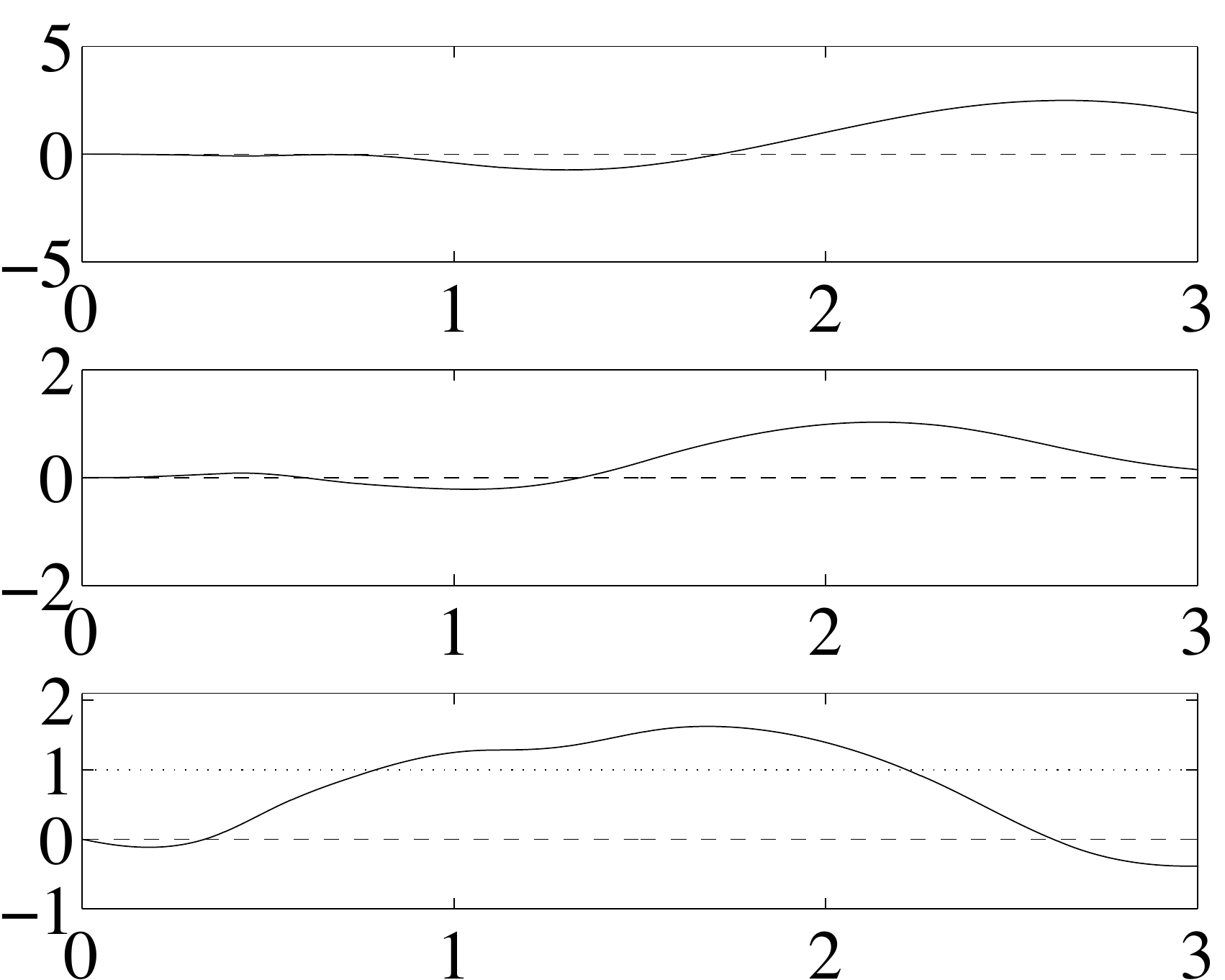}\label{fig:sim_x_WO}}
}
\centerline{
}
\caption{Flipping without adaptive term (dotted:desired, solid:actual)}\label{fig:sim_WO}
\end{figure}
\begin{figure}
\centerline{
	\subfigure[Attitude error function $\Psi$]{
		\includegraphics[width=0.3\columnwidth]{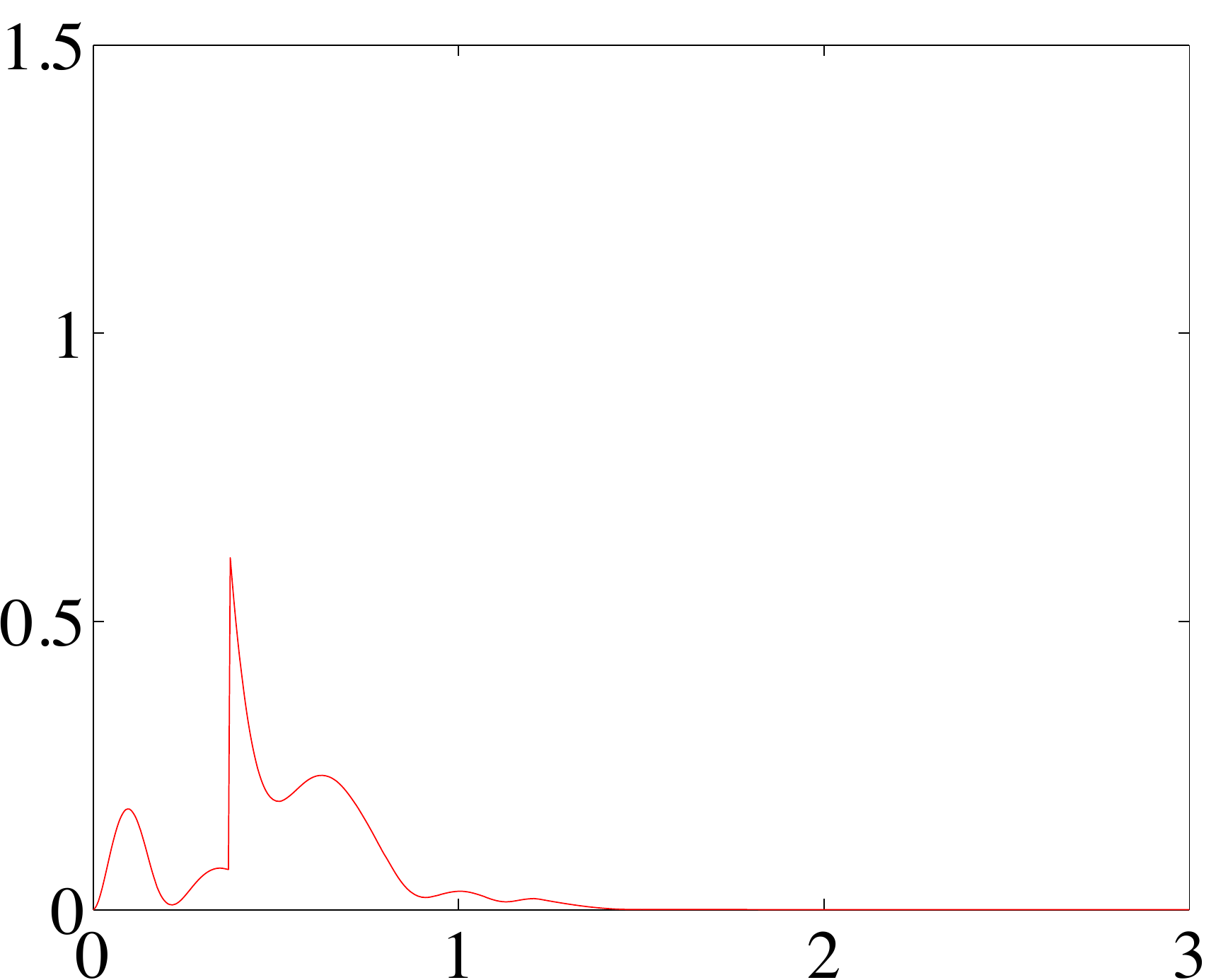}\label{fig:sim_error_W}}
	\subfigure[Thrust at each rotor $f_i$ ($\mathrm{N}$)]{
		\includegraphics[width=0.3\columnwidth]{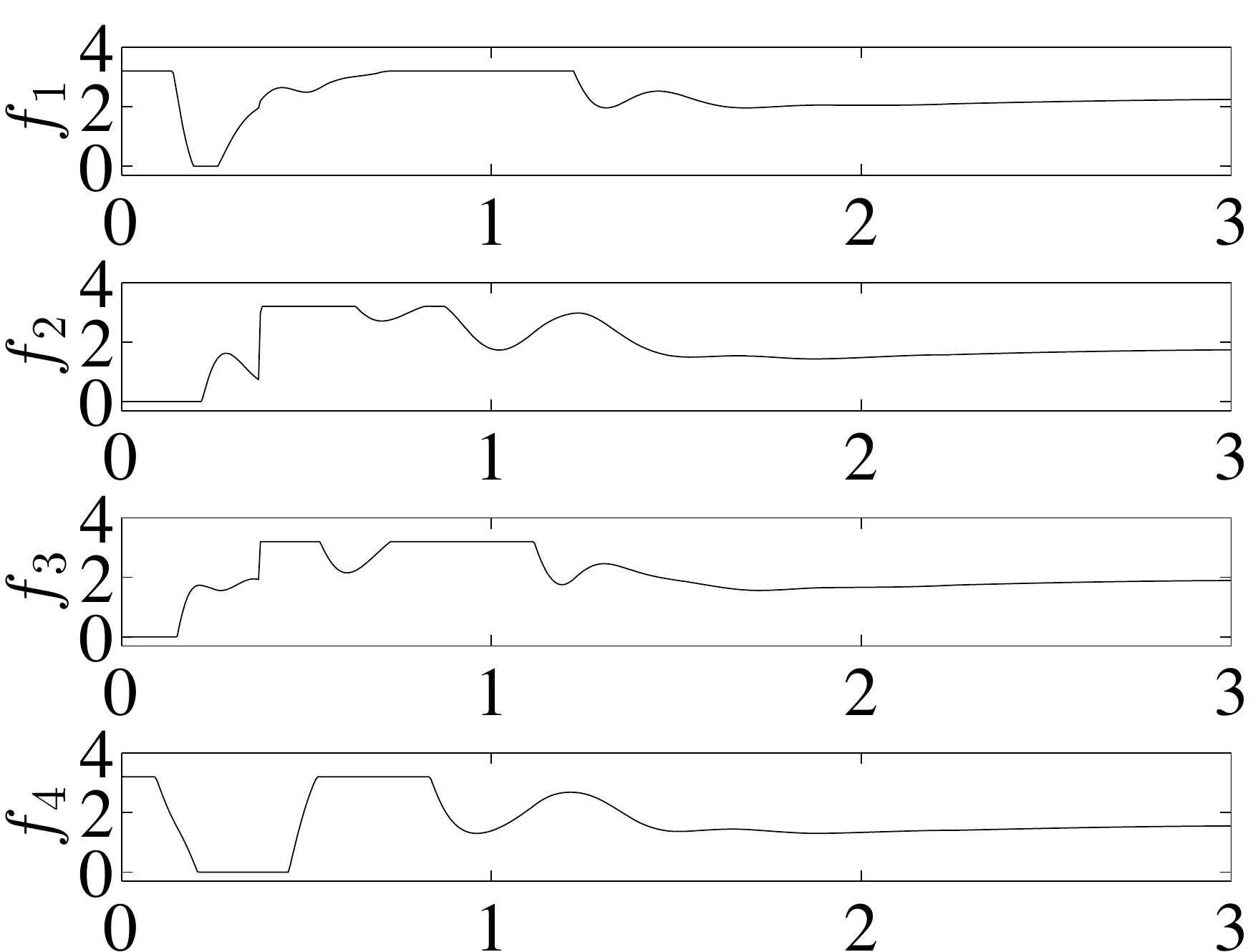}\label{fig:sim_force_W}}
		\subfigure[Attitude error $e_{R}$ ($\mathrm{rad}$)]{
		\includegraphics[width=0.3\columnwidth]{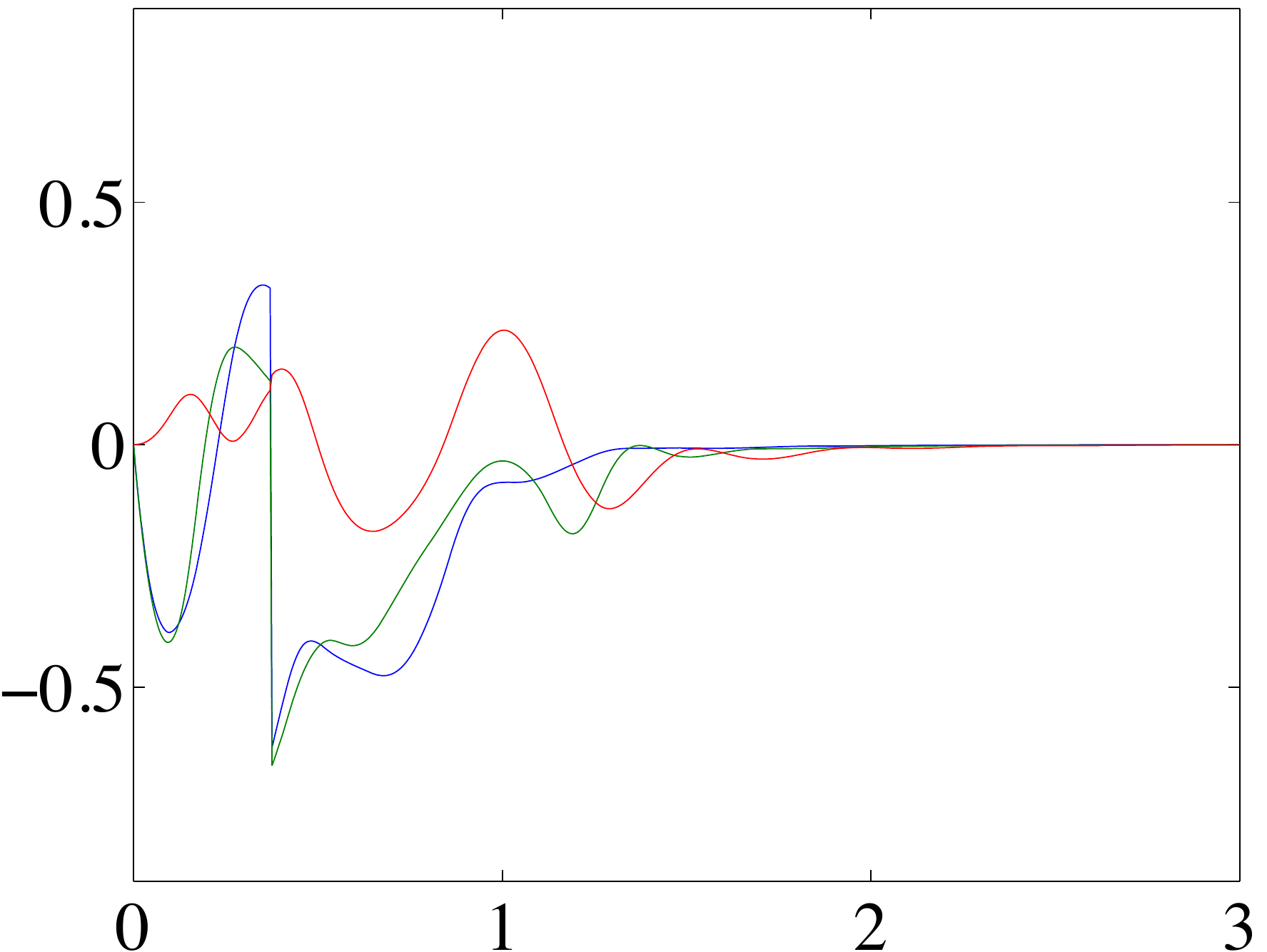}\label{fig:sim_eR_W}}
}
\centerline{
	\subfigure[Angular Velocity error $e_{\Omega}$($\mathrm{rad}/\mathrm{sec}$)]{
		\includegraphics[width=0.3\columnwidth]{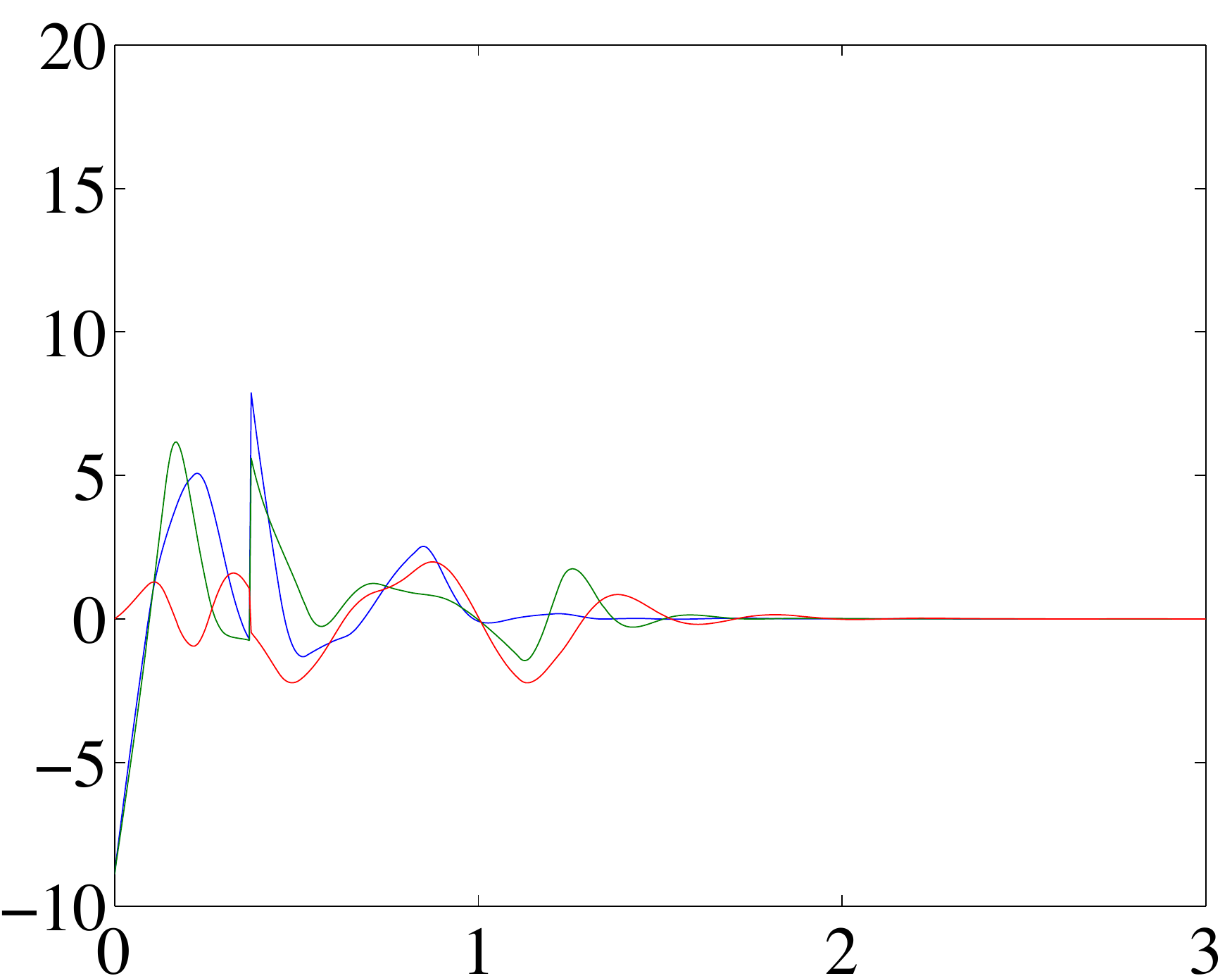}\label{fig:sim_eW_W}}
		\subfigure[Angular Velocity $\Omega$,$\Omega_{d}$ ($\mathrm{rad/s}$)]{
		\includegraphics[width=0.3\columnwidth]{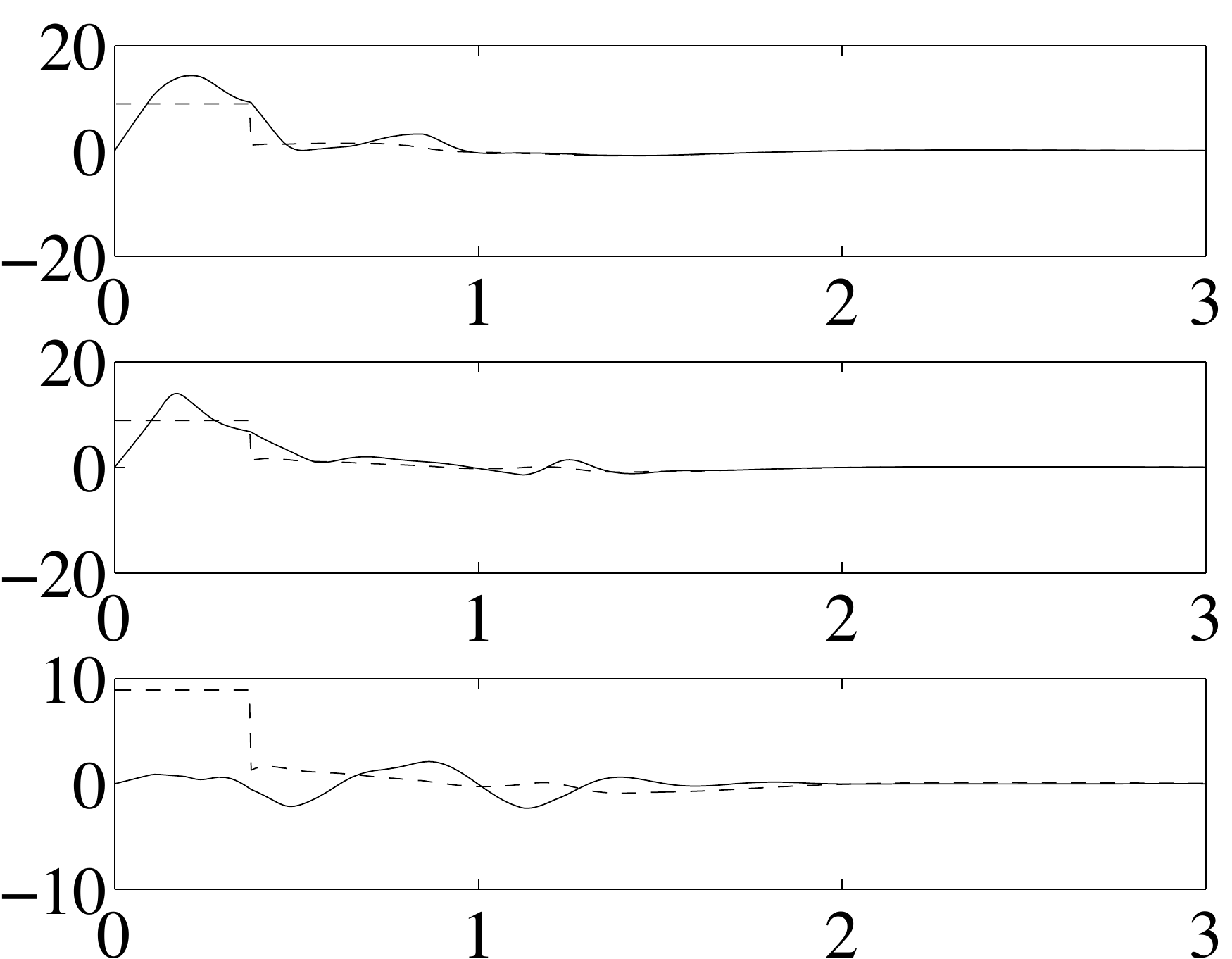}\label{fig:sim_ang_W}}
	\subfigure[Position $x$,$x_{d}$($\mathrm{m}$)]{
		\includegraphics[width=0.3\columnwidth]{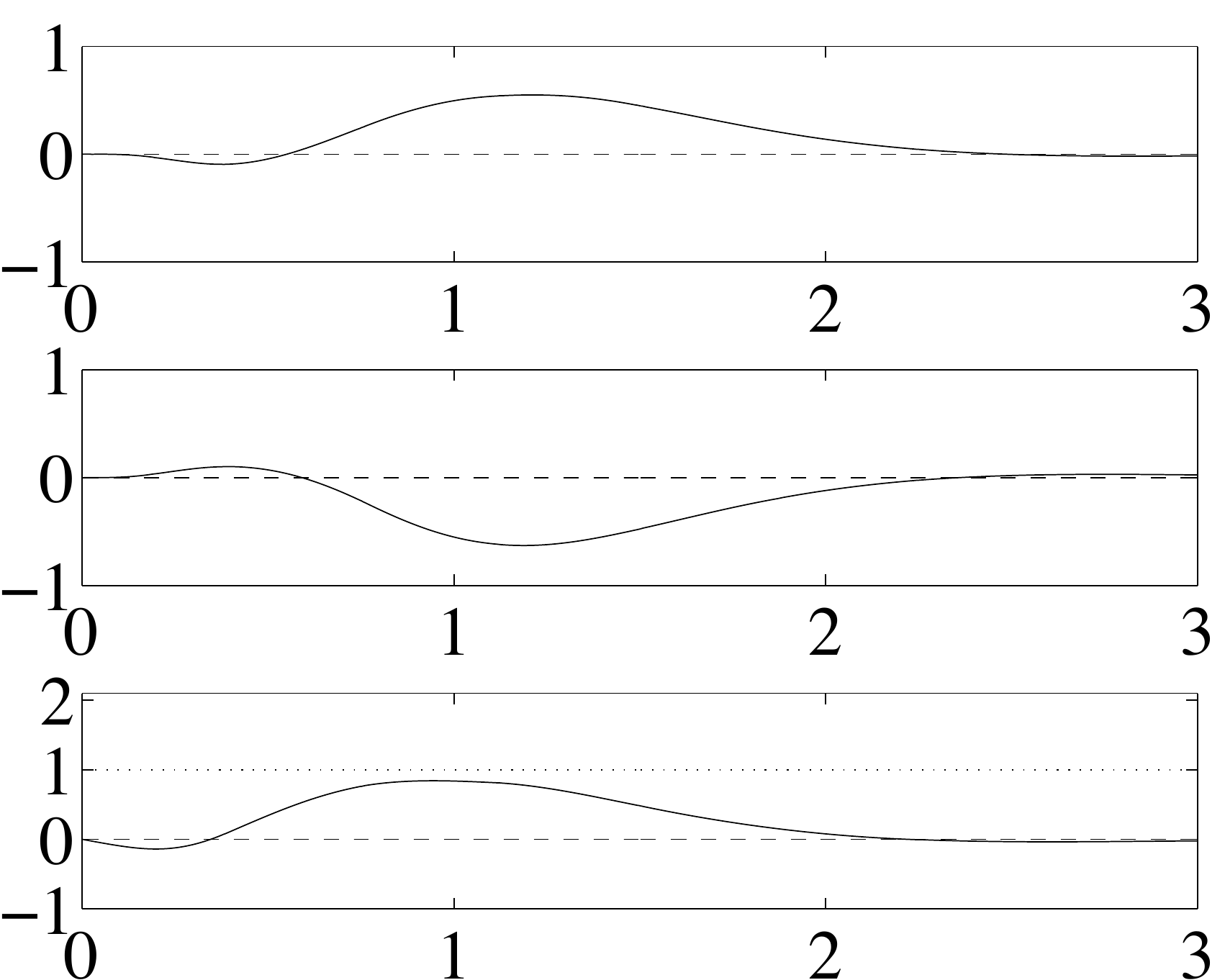}\label{fig:sim_x_W}}
}
\centerline{
}
\caption{Flipping with adaptive term (dotted:desired, solid:actual)}\label{fig:sim_W}
\end{figure}
Snapshots of the maneuver are presented at Figure \ref{fig:Int3D} and numerical results are illustrated at Figures \ref{fig:sim_WO} and \ref{fig:sim_W}, where it is shown that the adaptive terms eliminate the steady state error. The altitude drop during the flipping maneuver is reduced noticeably.

}
%%%%%%%%%%%%%%%%%%%%%%%%%%%%%%%%%%%%%%%%%%%%%%%%%%%%%
%%%%%%%%%%%%%%%%%%%%%%%%%%%%%%%%%%%%%%%%%%%%%%%%%%%%%
%%%%%%%%%%%%%%%%%%%%%%%%%%%%%%%%%%%%%%%%%%%%%%%%%%%%%

\newpage
\begin{singlespace}
\section{\protect \centering Chapter 3: Quadrotor UAV Transporting a Payload with Flexible Cable}
\end{singlespace}
\setcounter{section}{3}
\doublespacing
The main contribution of this chapter is presenting a nonlinear dynamic model and a control system for a quadrotor UAV with a cable-suspended load, that explicitly incorporate the effects of deformable cable. 

This chapter is organized as follows. In section \ref{sec:chap3dynamic}, equations of motion in 3D space for a chain pendulum on a quadrotor are derived. Section \ref{sec:chap3linear}, talks about the dynamics equation, linearization, and linear control design for a chain pendulum on a cart moving in 3D. Section \ref{sec:chap3control} presents the proposed geometric nonlinear controller and stability analysis to asymptotically stabilize the chain pendulum on a quadrotor. In section \ref{sec:chap3numerical}, we provide numerical simulations results to validate our work.

\subsection {\normalsize Quadrotor Dynamic Model}\label{sec:chap3dynamic}
{\addtolength{\leftskip}{0.5in}
Consider a quadrotor UAV with a payload that is connected via a chain of $n$ links, as illustrated at Figure \ref{fig:Quad}. The inertial frame is defined by the unit vectors $e_{1}=[1;0;0]$, $e_{2}=[0;1;0]$, and $e_{3}=[0;0;1]\in \Re^{3}$, and the third axis $e_{3}$ corresponds to the direction of gravity. Define a body-fixed frame $\{\vec{b}_{1},\vec{b}_{2},\vec{b}_{3}\}$ whose origin is located at the center of mass of the quadrotor, and its third axis $\vec b_3$ is aligned to the axis of symmetry. 

The location of the mass center, and the attitude of the quadrotor are denoted by $x\in\Re^3$ and $R\in\SO$, respectively, where the special orthogonal group is $\SO=\{R\in\Re^{3\times 3}\,|\, R^T R=I_{3\times 3},\;\mathrm{det}[R]=1\}$. A rotation matrix represents the linear transformation of a representation of a vector from the body-fixed frame to the inertial frame. 

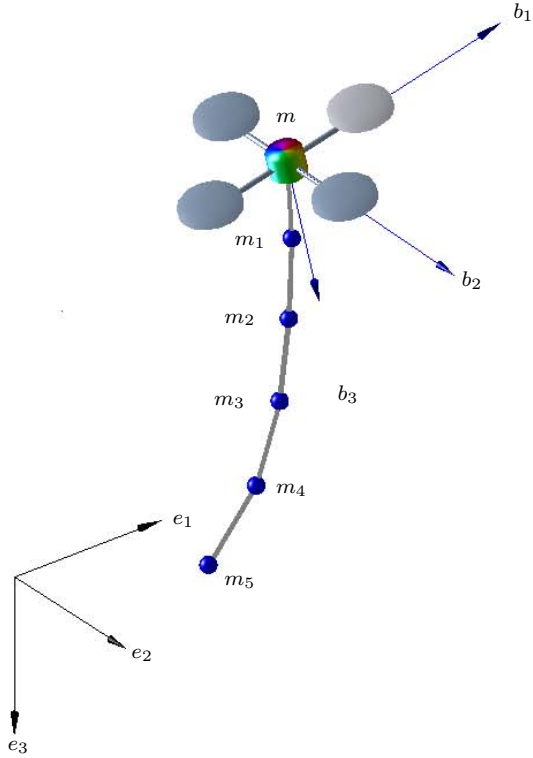
\begin{figure}[h]
\centerline{
	\setlength{\unitlength}{0.09\columnwidth}\scriptsize
\begin{picture}(5,7.5)(0,0)
\put(0,0){\includegraphics[width=0.45\columnwidth]{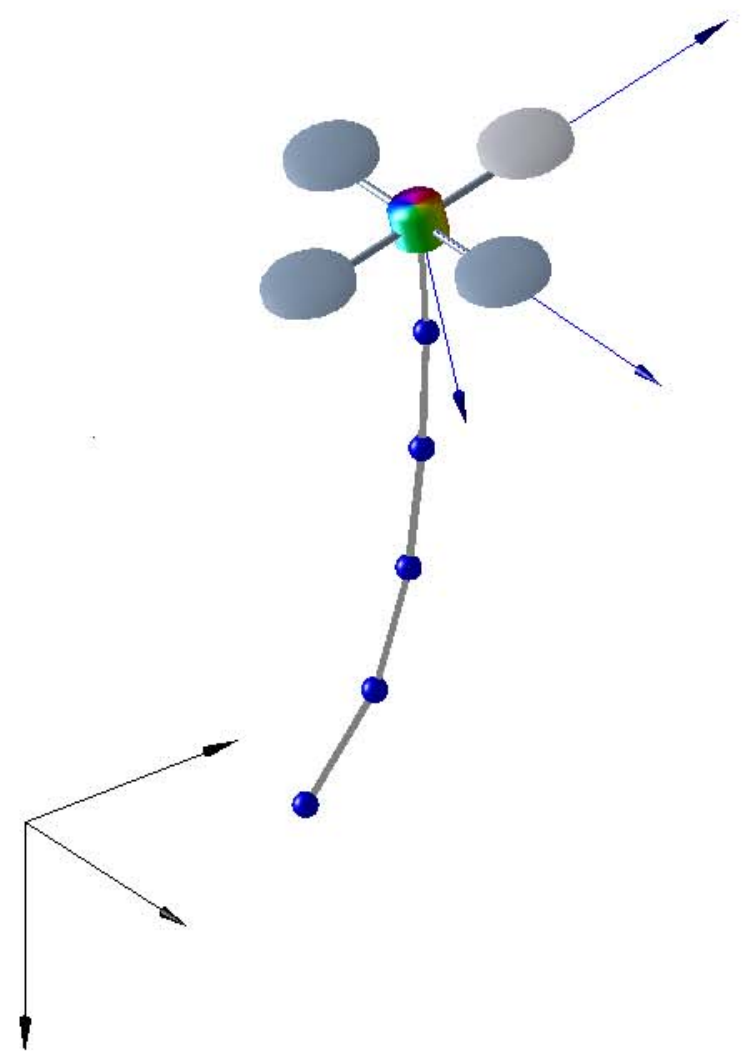}}
\put(2.7,6){\shortstack[c]{$m$}}
\put(2.3,4.8){\shortstack[c]{$m_{1}$}}
\put(2.2,4.05){\shortstack[c]{$m_{2}$}}
\put(2.1,3.25){\shortstack[c]{$m_{3}$}}
\put(2.7,2.4){\shortstack[c]{$m_{4}$}}
\put(2.2,1.5){\shortstack[c]{$m_{5}$}}
\put(1.3,0.8){\shortstack[c]{$e_{2}$}}
\put(1.7,2.1){\shortstack[c]{$e_{1}$}}
\put(0.1,-0.1){\shortstack[c]{$e_{3}$}}
\put(5,7){\shortstack[c]{$b_{1}$}}
\put(4.5,4.4){\shortstack[c]{$b_{2}$}}
\put(3.3,3.3){\shortstack[c]{$b_{3}$}}
\end{picture}
}
\caption{Quadrotor UAV with a cable-suspended load. Cable is modeled as a serial connection of arbitrary number of links (only 5 are illustrated).}\label{fig:Quad}
\end{figure}

The dynamic model of the quadrotor is identical to~\cite{LeeLeoPICDC10}. The mass and the inertia matrix of the quadrotor are denoted by $m\in\Re$ and $J\in\Re^{3\times 3}$, respectively. The quadrotor can generates a thrust $-fRe_3\in\Re^3$ with respect to the inertial frame, where $f\in\Re$ is the total thrust magnitude. It also generates a moment $M\in\Re^3$ with respect to its body-fixed frame. The pair $(f,M)$ is considered as control input of the quadrotor. 

Let $q_i\in\Sph^2$ be the unit-vector representing the direction of the $i$-th link, measured outward from the quadrotor toward the payload, where the two-sphere is the manifold of unit-vectors in $\Re^3$, i.e., $\Sph^2=\{q\in\Re^3\,|\, \|q\|=1\}$. For simplicity, we assume that the mass of each link is concentrated at the outboard end of the link, and the point where the first link is attached to the quadrotor corresponds to the mass center of the quadrotor. The mass and the length of the $i$-th link are defined by $m_i$ and $l_i\in\Re$, respectively. Thus, the mass of the payload corresponds to $m_n$. The corresponding configuration manifold of this system is given by $\SO\times\Re^3\times (\Sph^2)^n$.

Next, we show the kinematics equations. Let $\Omega\in\Re^3$ be the angular velocity of the quadrotor represented with respect to the body fixed frame, and let $\omega_i\in\Re^3$ be the angular velocity of the $i$-th link represented with respect to the inertial frame. The angular velocity is normal to the direction of the link, i.e., $q_i\cdot\omega_i=0$. The kinematics equations are given by
\begin{align}
\dot R & = R\hat\Omega,\label{eqn:Rdot}\\
\dot{q}_{i} & =\omega_{i}\times q_{i}=\hat{\omega}_{i}q_{i}.\label{eqn:qidot}
\end{align}

Throughout this chapter, the 2-norm of a matrix $A$ is denoted by $\|A\|$, and the dot product is denoted by $x \cdot y = x^Ty$. Also $\lambda_{\min}(\cdot)$ and $\lambda_{\max}(\cdot)$ denotes the minimum and maximum eigenvalue of a square matrix respectively, and $\lambda_{m}$ and $\lambda_{M}$ are shorthand for $\lambda_{m}=\lambda_{m}(J)$ and $\lambda_{M}=\lambda_{M}(J)$. 

%%%%%%%%%%%%%%%%%%%%%%%%%%%%%%%%%%%%%%%%%%%%%%%%%%%%%

\subsubsection {\normalsize Lagrangian}
We derive the equations of motion according to Lagrangian mechanics. The kinetic energy of the quadrotor is given by
\begin{align}
T_Q = \frac{1}{2}m\|\dot x\|^2 + \frac{1}{2} \Omega\cdot J\Omega.\label{eqn:TQ}
\end{align}
Let $x_i\in\Re^3$ be the location of $m_i$ in the inertial frame. It can be written as
\begin{align}\label{posvec33}
x_{i}=x+\sum^{i}_{a=1}{l_{a}q_{a}}.
\end{align}
Then, the kinetic energy of the links are given by
\begin{align}
T_L & = \frac{1}{2} \sum_{i=1}^n m_i \|\dot x+\sum^{i}_{a=1}{l_{a}\dot q_{a}}\|^2\nonumber\\
& = \frac{1}{2}\sum_{i=1}^n m_i \|\dot x\| + \dot x\cdot \sum_{i=1}^n\sum_{a=i}^n m_a l_i \dot q_i
+\frac{1}{2}\sum_{i=1}^n m_i \|\sum_{a=1}^i l_a \dot q_a\|^2.
\label{eqn:TL}
\end{align}
From \refeqn{TQ} and \refeqn{TL}, the total kinetic energy can be written as
\begin{align}
T & =\frac{1}{2}M_{00}\|\dot{x}\|^{2}+\dot{x}\cdot\sum^{n}_{i=1}{M_{0i}\dot{q}_{i}}+\frac{1}{2}\sum^{n}_{i,j=1}{M_{ij}\dot{q}_{i}\cdot\dot{q}_{j}}+\frac{1}{2}\Omega^{T}J\Omega,\label{eqn:KE}
\end{align}
where the inertia values $M_{00},M_{0i},M_{ij}\in\Re$ are given by
\begin{gather}
M_{00}=m+\sum_{i=1}^n m_i,\quad M_{0i}=\sum_{a=i}^n m_a l_i,\quad M_{i0}=M_{0i},\nonumber\\
M_{ij}=\braces{\sum_{a=\max\{i,j\}}^n m_a} l_i l_j,\label{eqn:Mij}
\end{gather}
for $1\leq i,j\leq n$. The gravitational potential energy is given by
\begin{align}
V & = -mgx\cdot e_3 - \sum_{i=1}^n m_i g x_i\cdot e_3\nonumber\\
& = -\sum^{n}_{i=1}\sum^{n}_{a=i}m_{a}gl_{i}e_{3}\cdot q_{i}-M_{00}ge_{3}\cdot x,\label{eqn:PE}
\end{align}
From \refeqn{KE} and \refeqn{PE}, the Lagrangian is $L=T-V$.

%%%%%%%%%%%%%%%%%%%%%%%%%%%%%%%%%%%%%%%%%%%%%%%%%%%%%

\subsubsection {\normalsize Euler-Lagrange equations}
Coordinate-free form of Lagrangian mechanics on the two-sphere $\Sph^2$ and the special orthogonal group $\SO$ for various multi-body systems has been studied in~\cite{Lee08,LeeLeoIJNME08}. The key idea is representing the infinitesimal variation of $R\in\SO$ in terms of the exponential map
\begin{align}
\delta R = \frac{d}{d\epsilon}\bigg|_{\epsilon = 0} \exp R(\epsilon \hat\eta) = R\hat\eta,\label{eqn:delR}
\end{align}
for $\eta\in\Re^3$. The corresponding variation of the angular velocity is given by $\delta\Omega=\dot\eta+\Omega\times\eta$. Similarly, the infinitesimal variation of $q_i\in\Sph^2$ is given by
\begin{align}
\delta q_i = \xi_i\times q_i,\label{eqn:delqi}
\end{align}
for $\xi_i\in\Re^3$ satisfying $\xi_i\cdot q_i=0$. This lies in the tangent space as it is perpendicular to $q_{i}$. Using these, we obtain the following Euler-Lagrange equations.
\begin{prop}\label{prop:propchap3_1}
Consider a quadrotor with a cable suspended payload whose Lagrangian is given by \refeqn{KE} and \refeqn{PE}. The Euler-Lagrange equations on $\Re^3\times\SO\times(\Sph^2)^n$ are as follows
\begin{gather}
M_{00}\ddot{x}+\sum^{n}_{i=1}{M_{0i}\ddot{q}_{i}}=-fRe_{3}+M_{00}ge_{3}+\Delta_{x},\label{eqn:xddot}\\
M_{ii}\ddot q_i  -\hat q_i^2 (M_{i0}\ddot x + \sum_{\substack{j=1\\j\neq i}}^n M_{ij}\ddot q_j)=- M_{ii}\|\dot q_i\|^2 q_i-\sum_{a=i}^n m_a gl_i\hat q_i^2 e_3,\label{eqn:qddot}\\
J\dot{\Omega}+\hat{\Omega}J\Omega=M+\Delta_{R},\label{eqn:Wdot}
\end{gather}
where $M_{ij}$ is defined at \refeqn{Mij}. Therefore $\Delta_{x}$ and $\Delta_{R}\in\Re^3$ are fixed disturbances applied to the translational and rotational dynamics of the quadrotor respectively. Equations \refeqn{xddot} and \refeqn{qddot} can be rewritten in a matrix form as follows:
\begin{align}
&\begin{bmatrix}%
    M_{00} & M_{01} & M_{02} & \cdots & M_{0n} \\
    -\hat q_1^2 M_{10} & M_{11}I_{3} & -M_{12} \hat q_1^2 & \cdots & -M_{1n}\hat q_1^2\\%
    -\hat q_2^2 M_{20} & -M_{21} \hat q_2^2 & M_{22} I_{3} & \cdots & -M_{2n} \hat q_2^2\\%
    \vdots & \vdots & \vdots & & \vdots\\
    -\hat q_n^2 M_{n0} & -M_{n1} \hat q_n^2 & -M_{n2}\hat q_n^2 & \cdots & M_{nn} I_{3}
    \end{bmatrix}%
    \begin{bmatrix}
    \ddot x \\ \ddot q_1 \\ \ddot q_2 \\ \vdots \\ \ddot q_n
    \end{bmatrix}\nonumber\\
 &=   \begin{bmatrix}
    -fRe_3 +M_{00}ge_3+\Delta_{x}\\
    -\|\dot q_1\|^2M_{11} q_1 -\sum_{a=1}^n m_a gl_1\hat q_1^2 e_3\\
    -\|\dot q_2\|^2M_{22} q_2 -\sum_{a=2}^n m_a gl_2\hat q_2^2 e_3\\
    \vdots\\
    -\|\dot q_n\|^2M_{nn} q_n - m_n gl_n\hat q_n^2 e_3
    \end{bmatrix}.\label{eqn:ELm}
\end{align}
Or equivalently, it can be written in terms of the angular velocities as
\begin{gather}
\begin{bmatrix}%
    M_{00} & -M_{01}\hat q_1 & -M_{02}\hat q_2 & \cdots & -M_{0n}\hat q_n\\
    \hat q_1 M_{10} & M_{11}I_{3} & -M_{12} \hat q_1 \hat q_2 & \cdots & -M_{1n}\hat q_1 \hat q_n\\%
    \hat q_2 M_{20} &-M_{21} \hat q_2\hat q_1 & M_{22} I_{3} & \cdots & -M_{2n} \hat q_2 \hat q_n\\%
    \vdots & \vdots & \vdots & & \vdots\\
    \hat q_n M_{n0} &-M_{n1} \hat q_n \hat q_1 & -M_{n2}\hat q_n \hat q_2 & \cdots & M_{nn} I_{3}
    \end{bmatrix}%
    \begin{bmatrix}
    \ddot x \\ \dot \omega_1 \\ \dot \omega_2 \\ \vdots \\ \dot \omega_n
    \end{bmatrix}\nonumber\\
    =
    \begin{bmatrix}
    \sum_{j=1}^n M_{0j} \|\omega_j\|^2 q_j-fRe_3+M_{00}ge_3+\Delta_{x}\\
    \sum_{j=2}^n M_{1j}\|\omega_j\|^2\hat q_1 q_j +\sum_{a=1}^n m_a gl_1\hat q_1 e_3\\
    \sum_{j=1,j\neq 2}^n M_{2j}\|\omega_j\|^2\hat q_2 q_j +\sum_{a=2}^n m_a gl_2\hat q_2 e_3\\
    \vdots\\
    \sum_{j=1}^{n-1} M_{nj}\|\omega_j\|^2\hat q_n q_j + m_n gl_n\hat q_n e_3\\
    \end{bmatrix},\label{eqn:ELwm}\\
\dot q_i = \omega_i\times q_i.\label{eqn:ELwm2}
\end{gather}
\end{prop}
\begin{proof}
See Appendix \ref{sec:pfchap3_1}
\end{proof}
These provide a coordinate-free form of the equations of motion for the presented quadrotor UAV that is uniformly defined for any number of links $n$, and that is globally defined on the nonlinear configuration manifold. Compared with equations of motion derived in terms of local coordinates, such as Euler-angles, these avoid singularities completely, and they provide a compact form of equations that are suitable for control system design.

However, the presented finite element model may not capture the certain dynamic characteristics of the actual cable dynamics represented by partial differential equations. Designing a control system for such infinite-dimensional system is beyond the scope of this dissertation. 

}

%%%%%%%%%%%%%%%%%%%%%%%%%%%%%%%%%%%%%%%%%%%%%%%%%%%%%

\subsection {\normalsize Control System Design for a Simplified Dynamic Model}\label{sec:chap3linear}
{\addtolength{\leftskip}{0.5in}
Let $x_d\in\Re^3$ be a fixed desired location of the quadrotor UAV. Assuming that all of the links are pointing downward, i.e., $q_i=e_3$, the resulting location of the payload is given by 
\begin{align}
x_n=x_d +\sum_{i=1}^n l_i e_3. 
\end{align}
We wish to design the control force $f$ and the control moment $M$ such that this hanging equilibrium configuration at the desired location becomes asymptotically stable. 

%%%%%%%%%%%%%%%%%%%%%%%%%%%%%%%%%%%%%%%%%%%%%%%%%%%%%

\subsubsection {\normalsize Control Problem Formulation}
For the given equations of motion \refeqn{xddot} for $x$, the control force is given by $-fRe_3$. This implies that the total thrust magnitude $f$ can be arbitrarily chosen, but the direction of the thrust vector is always along the third body-fixed axis. Also, the rotational attitude dynamics of the quadrotor is not affected by the translational dynamics of the quadrotor or the dynamics of links.

To overcome the under-actuated property of a quadrotor, in this section, we first replace the term  $-fRe_3$ of \refeqn{xddot} by a fictitious control input $u\in\Re^3$, and design an expression for $u$ to asymptotically stabilize the desired equilibrium. This is equivalent to assuming that the attitude $R$ of the quadrotor can be instantaneously controlled. The effects of the attitude dynamics are incorporated at the next section. Also $\Delta_{x}$ is ignored in the simplified dynamic model.
In short, the equations of motion for the simplified dynamic model considered in the section are given by
\begin{align}
M_{00}\ddot{x}+\sum^{n}_{i=1}{M_{0i}\ddot{q}_{i}}=u+M_{00}ge_{3},\label{eqn:xddot_sim}
\end{align}
and \refeqn{qddot}.
%%%%%%%%%%%%%%%%%%%%%%%%%%%%%%%%%%%%%%%%%%%%%%%%%%%%%

\subsubsection {\normalsize Linear Control System}\label{sec:LCS}
The fictitious control input is designed from the linearized dynamics about the desired hanging equilibrium. The variation of $x$ and $u$ are given by
\begin{align}
\delta x = x - x_d,\quad \delta u = u - M_{00}g e_3.\label{eqn:delxLin}
\end{align}
From \refeqn{delqi}, the variation of $q_i$ from the equilibrium can be written as
\begin{align}
\delta q_i = \xi_i\times e_3,\label{eqn:delqLin}
\end{align}
where $\xi_i\in\Re^3$ with $\xi_i\cdot e_3=0$. The variation of $\omega_i$ is given by $\delta\omega\in\Re^3$ with $\delta\omega_i \cdot e_3=0$. Therefore, the third element of each of $\xi_i$ and $\delta\omega_i$ for any equilibrium configuration is zero, and they are omitted in the following linearized equation, i.e., the state vector of the linearized equation is composed of $C^T\xi_i\in\Re^2$, where $C=[e_1,e_2]\in\Re^{3\times 2}$. 

\begin{prop}\label{prop:propchap3_2}
The linearized equations of the simplified dynamic model \refeqn{xddot_sim} and \refeqn{qddot} can be written as follows
\begin{gather}
\Mb\ddot \xb  + \Gb\xb = \Bb \delta u+ \g(\xb,\dot{\xb}),\label{eqn:Lin}
\end{gather}
where $\g(\xb,\dot{\xb})$ corresponds to the higher order terms where $\xb=[\delta x,\; \xb_{q}]^{T}\in\Re^{2n+3}$, $\Mb\in\Re^{2n+3\times 2n+3}$, $\Gb\in\Re^{2n+3\times 2n+3}$, $\Bb\in\Re^{2n+3\times 3}$, and \refeqn{Lin} can equivalently be written as
\begin{align*}
\begin{bmatrix} \Mb_{xx} & \Mb_{xq}\\ \Mb_{qx} & \Mb_{qq} \end{bmatrix}
\begin{bmatrix}  \delta\ddot x \\ \ddot \xb_q\end{bmatrix}
&+
\begin{bmatrix} 0_3 & 0_{3\times 2n}\\ 0_{2n\times 3} & \Gb_{qq}\end{bmatrix}
\begin{bmatrix}  \delta x \\ \xb_q\end{bmatrix}=
\begin{bmatrix} I_3 \\ 0_{2n\times 3}\end{bmatrix}
\delta u+ \g(\xb,\dot{\xb}),
\end{align*}
where the corresponding sub-matrices are defined as
\begin{align*}
\xb_q & = [C^T \xi_1;\,\ldots\,;\,C^T \xi_n],\\
\Mb_{xx} &= M_{00}I_{3},\\
\Mb_{xq} &= \begin{bmatrix}
-M_{01}\hat e_3C & -M_{02}\hat e_3C & \cdots & -M_{0n}\hat e_3C
\end{bmatrix},\\
\Mb_{qx} & = \Mb_{xq}^T,\\
\Mb_{qq} &=
    \begin{bmatrix}%
M_{11}I_{2} & M_{12} I_2 & \cdots & M_{1n}I_2\\%
M_{21} I_2 & M_{22} I_{2} & \cdots & M_{2n}I_2\\%
\vdots & \vdots & & \vdots\\
M_{n1}I_2 & M_{n2}I_2 & \cdots & M_{nn} I_{2}
    \end{bmatrix},\\
\Gb_{qq}  = \mathrm{diag}[&\sum_{a=1}^n {m_a gl_1 I_2},\cdots,m_ngl_nI_2].
\end{align*}
\end{prop}

\begin{proof}
See Appendix \ref{sec:pfchap3_2}
\end{proof}
For the linearized dynamics \refeqn{Lin}, the following control system is chosen
\begin{align}
\delta u & = -k_{x}\delta{x}-k_{\dot{x}}\delta\dot{x}-\sum_{a=1}^{n}{k_{q_{i}}C^{T}(e_3\times q_{i})}-k_{\omega_{i}}C^{T}\delta\omega_{i}\nonumber\\
& = -K_x \xb - K_{\dot x} \dot \xb,\label{eqn:delu}
\end{align}
for controller gains $K_x =[k_xI_3,k_{q_1}I_{3\times 2},\ldots k_{q_n}I_{3\times 2}]\in\Re^{3\times (3+2n)}$ and $K_{\dot x} =[k_{\dot x}I_3,k_{\omega_1}I_{3\times 2},\ldots k_{\omega_n}I_{3\times 2}]\in\Re^{3\times (3+2n)}$. Provided that \refeqn{Lin} is controllable, we can choose the controller gains $K_x,K_{\dot x}$ such that the equilibrium is asymptotically stable for the linearized equation \refeqn{Lin}. Then, the equilibrium becomes asymptotically stable for the nonlinear Euler Lagrange equation \refeqn{xddot_sim} and \refeqn{qddot}~\cite{Kha96}. The controlled linearized system can be written as
\begin{align}
\dot{z}_{1}=&\mathds{A} z_{1}+\mathds{B}\g(\xb,\dot{\xb}),
\end{align}
where $z_{1}=[\xb,\; \dot{\xb}]^{T}\in\Re^{4n+6}$ and the matrices $\mathds{A}\in\Re^{4n+6\times 4n+6}$ and $\mathds{B}\in\Re^{4n+6\times 2n+3}$ are defined as
\begin{align}
\mathds{A}=\begin{bmatrix}
0&I\\
-\Mb^{-1}(\Gb+\Bb K_{x})&-{(\Mb^{-1}\Bb K_{\dot{x}})}
\end{bmatrix}, \mathds{B}=\begin{bmatrix}
0\\
\Mb^{-1}
\end{bmatrix}.
\end{align}
We can also choose $K_{\xb}$ and $K_{\dot{\xb}}$ such that $\mathds{A}$ is Hurwitz. Then for any positive definite matrix $Q\in\Re^{4n+6\times4n+6}$, there exist a positive definite and symmetric matrix $P\in\Re^{4n+6\times4n+6}$ such that $\mathds{A}^{T}P+P\mathds{A}=-Q$ according to~\cite[Thm 3.6]{Kha96}.

}

\subsection {\normalsize Controller Design for a Quadrotor with a Flexible Cable}\label{sec:chap3control}
{\addtolength{\leftskip}{0.5in}
The control system designed in the previous section is generalized to the full dynamic model that includes the attitude dynamics. The central idea is that the attitude $R$ of the quadrotor is controlled such that its total thrust direction $-Re_3$ that corresponds to the third body-fixed axis asymptotically follows the direction of the fictitious control input $u$. By choosing the total thrust magnitude properly, we can guarantee that the total thrust vector $-fRe_{3}$ asymptotically converges to the fictitious ideal force $u$, thereby yielding asymptotic stability of the full dynamic model.

%%%%%%%%%%%%%%%%%%%%%%%%%%%%%%%%%%%%%%%%%%%%%%%%%%%%%

\subsubsection {\normalsize Controller Design}
Consider the full nonlinear equations of motion, let $A\in\Re^3$ be the ideal total thrust of the quadrotor system that asymptotically stabilize the desired equilibrium. From \refeqn{delxLin}, we have 
\begin{align}
A= M_{00}ge_3 + \delta u = -K_{x} \xb-K_{\dot{x}}\dot\xb -K_{z}\satr_{\sigma}(e_{\xb})+ M_{00}ge_3,\label{eqn:A}
\end{align}
where the following integral term $e_{\xb}\in\Re^{2n+3}$ is added to eliminate the effect of disturbance $\Delta_x$ in the full dynamic model 
\begin{align}\label{eqn:exterm}
e_{\xb}=\int^{t}_{0}{(P\mathds{B})^{T}z_{1}(\tau)\;d\tau},
\end{align}
where $K_z =[k_{z}I_3,k_{z_1}I_{3\times 2},\ldots k_{z_n}I_{3\times 2}]\in\Re^{3\times (3+2n)}$ is an integral gain. For a positive constant $\sigma\in\Re$, a saturation function $\sat_\sigma:\Re\rightarrow [-\sigma,\sigma]$ is introduced as
\begin{align*}
\sat_{\sigma}(y) = \begin{cases}
\sigma & \mbox{if } y >\sigma\\
y & \mbox{if } -\sigma \leq y \leq\sigma\\
-\sigma & \mbox{if } y <-\sigma\\
\end{cases}.
\end{align*}
If the input is a vector $y\in\Re^n$, then the above saturation function is applied element by element to define a saturation function $\sat_\sigma(y):\Re^n\rightarrow [-\sigma,\sigma]^n$ for a vector. It is also assumed that an upper bound of the infinite norm of the uncertainty is known
\begin{align}\label{eqn:disturbancecond}
\|\Delta_{x}\|_{\infty}\leq \delta,
\end{align}
for positive constant $\delta$. The desired direction of the third body-fixed axis $b_{3_c}\in\Sph^2$ is given by
\begin{align}
b_{3_c} = - \frac{A}{\|A\|}.\label{eqn:b3c}
\end{align}
This provides a two-dimensional constraint for the desired attitude of quadrotor, and there is additional one-dimensional degree of freedom that corresponds to rotation about the third body-fixed axis, i.e., yaw angle. A desired direction of the first body-fixed axis, namely $b_{1_d}\in\Sph^2$ is introduced to resolve it, and it is projected onto the plane normal to $b_{3_c}$. The desired direction of the second body-fixed axis is chosen to constitute an orthonormal frame. More explicitly, the desired attitude is given by
\begin{align}
R_c = \bracket{-\frac{\hat b_{3_c}^2 b_{1_d}}{\|\hat b_{3_c}^2 b_{1_d}\|},\;
 \frac{\hat b_{3_c}b_{1_d}}{\|\hat b_{3_c}b_{1_d}\|},\; b_{3_c}},
\end{align}
which is guaranteed to lie in $\SO$ by construction, assuming that $b_{1_d}$ is not parallel to $b_{3_c}$~\cite{LeeLeoAJC13}. The desired angular velocity $\Omega_{c}\in\Re^{3}$ is obtained by the attitude kinematics equation
\begin{align}
\Omega_c = (R_c^T \dot R_c)^\vee.
\end{align}
Next, we introduce the tracking error variables for the attitude and the angular velocity $e_{R}$, $e_{\Omega}\in\Re^{3}$ as follows~\cite{TFJCHTLeeHG}
\begin{align}
&e_{R}=\frac{1}{2}(R_{c}^{T}R-R^{T}R_{c})^{\vee},\\
&e_{\Omega}=\Omega-R^{T}R_{c}\Omega_{c}.
\end{align}
The thrust magnitude and the moment vector of quadrotor are chosen as 
\begin{align}
f = -A\cdot Re_3,\label{eqn:fi}
\end{align}
and
\begin{align}
M =-{k_R}e_{R} -{k_{\Omega}}e_{\Omega} -k_{I}e_{I}+\Omega\times J\Omega-J(\hat\Omega R^T R_{c} \Omega_{c} - R^T R_{c}\dot\Omega_{c}),\label{eqn:Mi}
\end{align}
where $k_R,k_\Omega$, and $k_{I}$ are positive constants and the following integral term is introduced to eliminate the effect of fixed disturbance $\Delta_{R}$
\begin{align}\label{eqn:integralterm}
e_{I}=\int^{t}_{0}{e_{\Omega}(\tau)+c_{2}e_{R}(\tau)d\tau},
\end{align}
where $c_{2}$ is a positive constant.

%%%%%%%%%%%%%%%%%%%%%%%%%%%%%%%%%%%%%%%%%%%%%%%%%%%%%

\subsubsection {\normalsize Stability Analysis}
\begin{prop}\label{prop:propchap3_3}
Consider control inputs $f$, $M$ defined in \refeqn{fi} and \refeqn{Mi}. There exist controller parameters and gains such that, (i) the zero equilibrium of tracking error is stable in the sense of Lyapunov; (ii) the tracking errors $e_{R}$, $e_{\Omega}$, $\xb$, $\dot{\xb}$ asymptotically converge to zero as $t\rightarrow\infty$; (iii) the integral terms $e_{I}$ and $e_{\xb}$ are uniformly bounded.
\end{prop}
\begin{proof}
See Appendix \ref{sec:pfchap3_3}
\end{proof}
By utilizing geometric control systems for quadrotor, we show that the hanging equilibrium of the links can be asymptotically stabilized while translating the quadrotor to a desired position. The control systems proposed explicitly consider the coupling effects between the cable/load dynamics and the quadrotor dynamics. We presented a rigorous Lyapunov stability analysis to establish stability properties without any timescale separation assumptions or singular perturbation, and a new nonlinear integral control term is designed to guarantee robustness against unstructured uncertainties in both rotational and translational dynamics.

}

%%%%%%%%%%%%%%%%%%%%%%%%%%%%%%%%%%%%%%%%%%%%%%%%%%

\subsection {\normalsize Numerical Example}\label{sec:chap3numerical}
{\addtolength{\leftskip}{0.5in}
The desirable properties of the proposed control system are illustrated by a numerical example. Properties of a quadrotor are chosen as
\begin{align*}
m=0.5\,\mathrm{kg},\quad J=\mathrm{diag}[0.557,\,0.557,\,1.05]\times 10^{-2}\,\mathrm{kgm^2}.
\end{align*}
Five identical links with $n=5$, $m_i=0.1\,\mathrm{kg}$, and $l_i=0.1\,\mathrm{m}$ are considered. Controller parameters are selected as follows: $k_x=12.8$, $k_v=4.22$, ${k_R}=0.65$, ${k_\Omega}= 0.11$, $k_{I}=1.5$, $c_{1}=c_{2}=0.7$. Also $k_{q}$ and $k_{\omega}$ are defined as 
\begin{align*}
&k_q=[11.01,\,6.67,\,1.97,\,0.41,\,0.069],\\
&k_\omega=[0.93,\,0.24,\,0.032,\,0.030,\,0.025].
\end{align*}

\begin{figure}
\centerline{
	\subfigure[Attitude error function $\psi$]{
		\includegraphics[width=0.4\columnwidth]{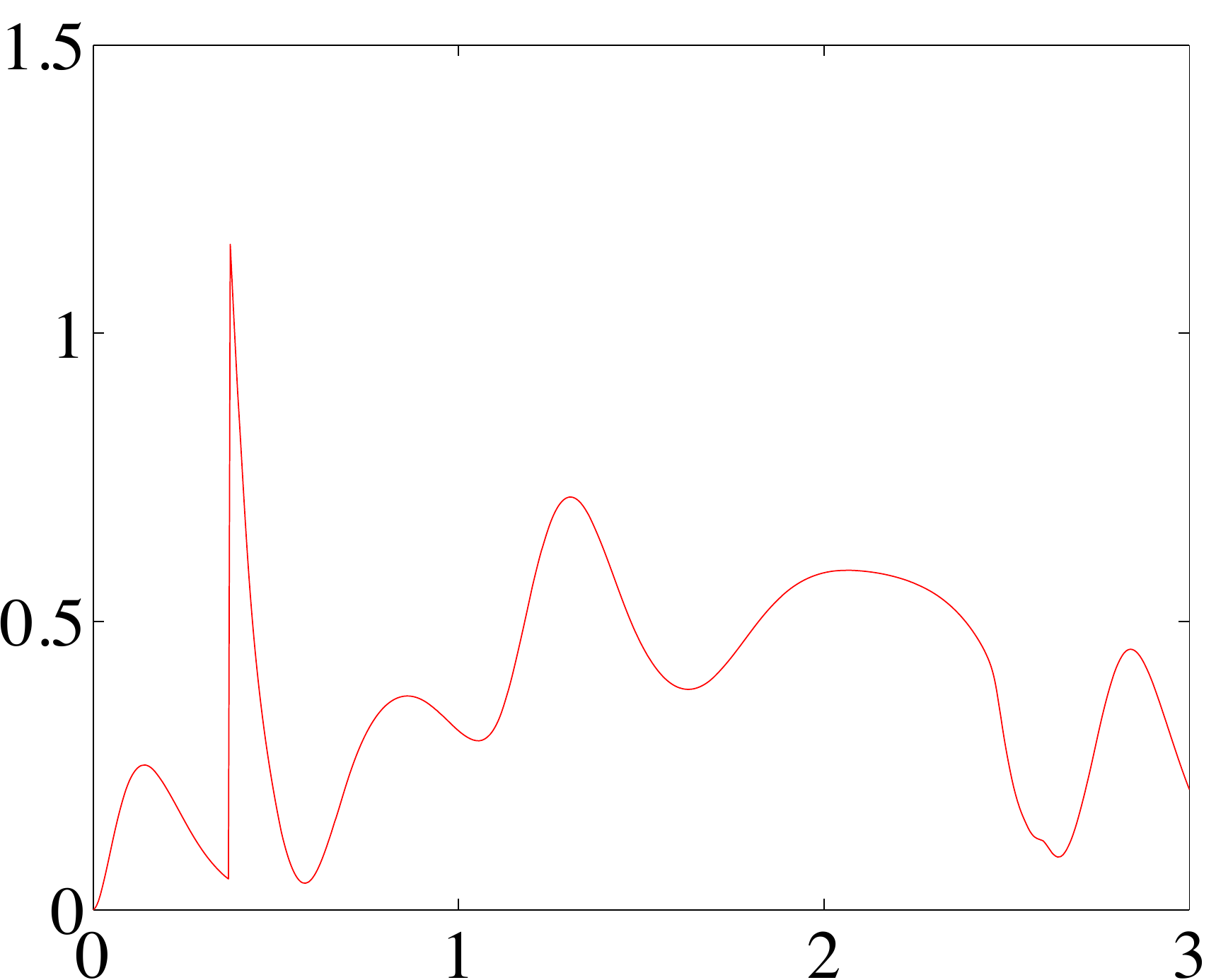}\label{fig:psi_without}}
	\subfigure[Direction error $e_{q}$ and angular velocity error $e_{\omega}$ for links]{
		\includegraphics[width=0.4\columnwidth]{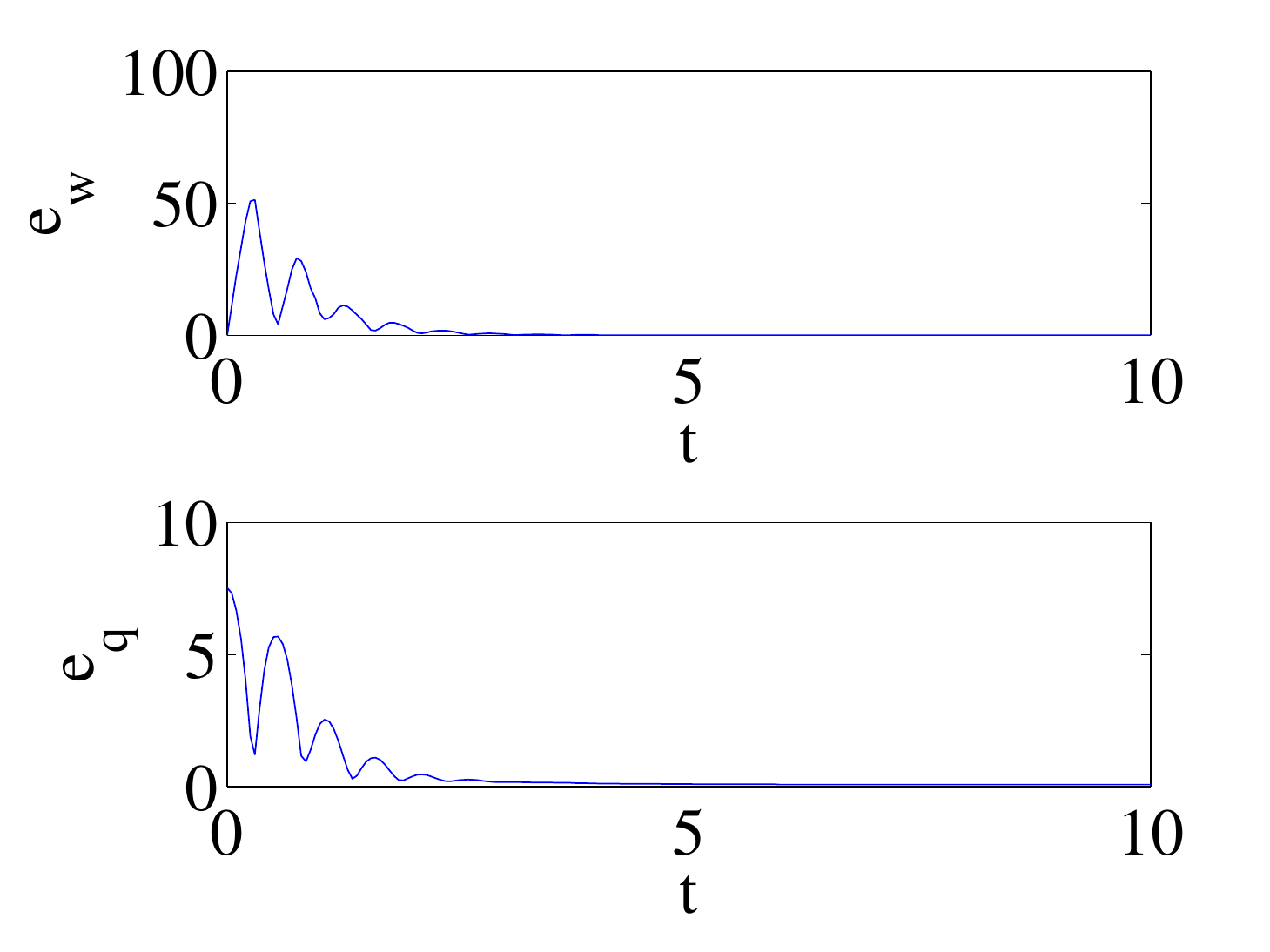}\label{fig:errors_without}}
}
\centerline{
	\subfigure[Quadrotor angular velocity $\Omega$:blue, $\Omega_{d}$:red]{
		\includegraphics[width=0.4\columnwidth]{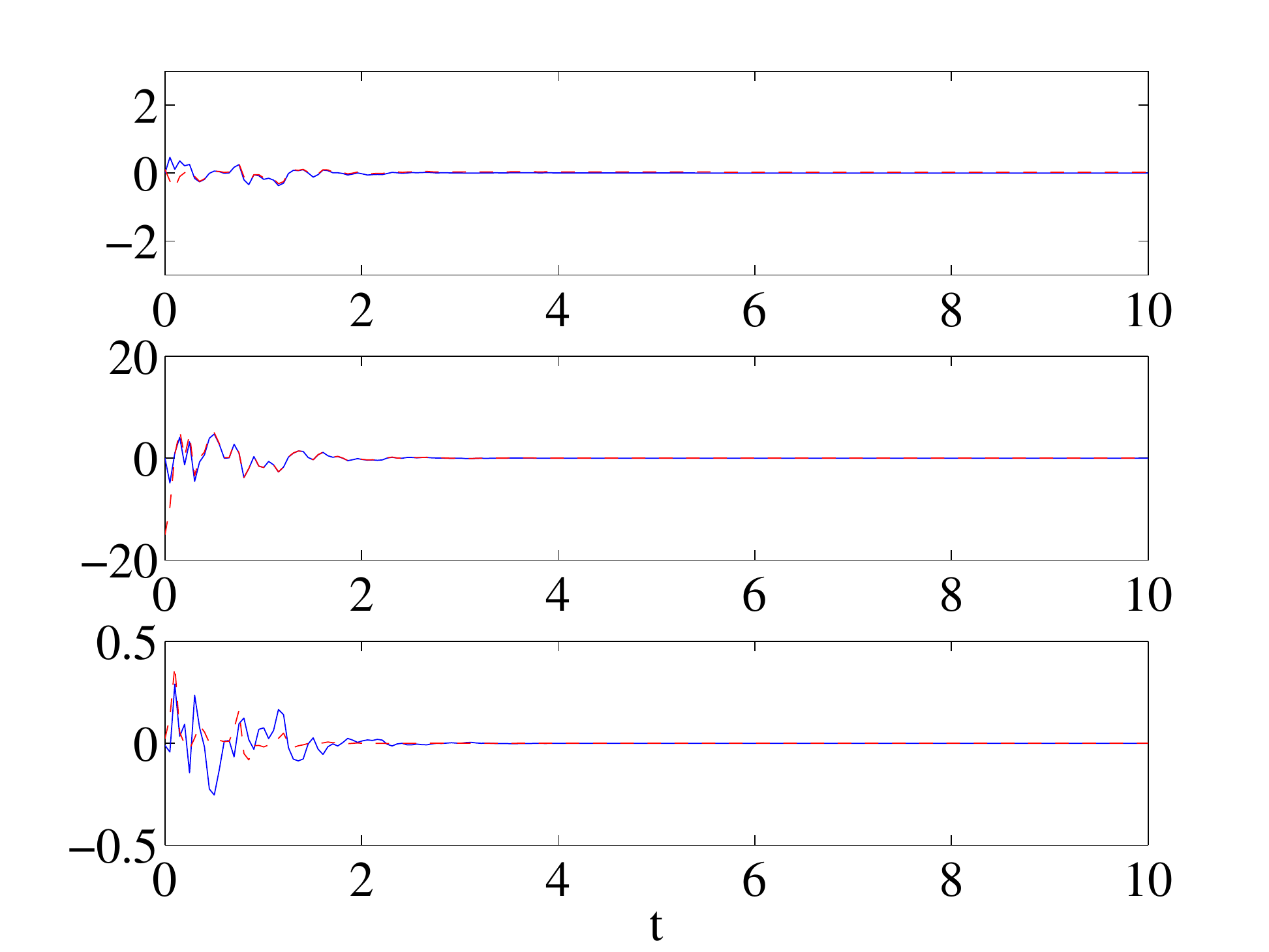}}
	\subfigure[Control force $u$]{
		\includegraphics[width=0.4\columnwidth]{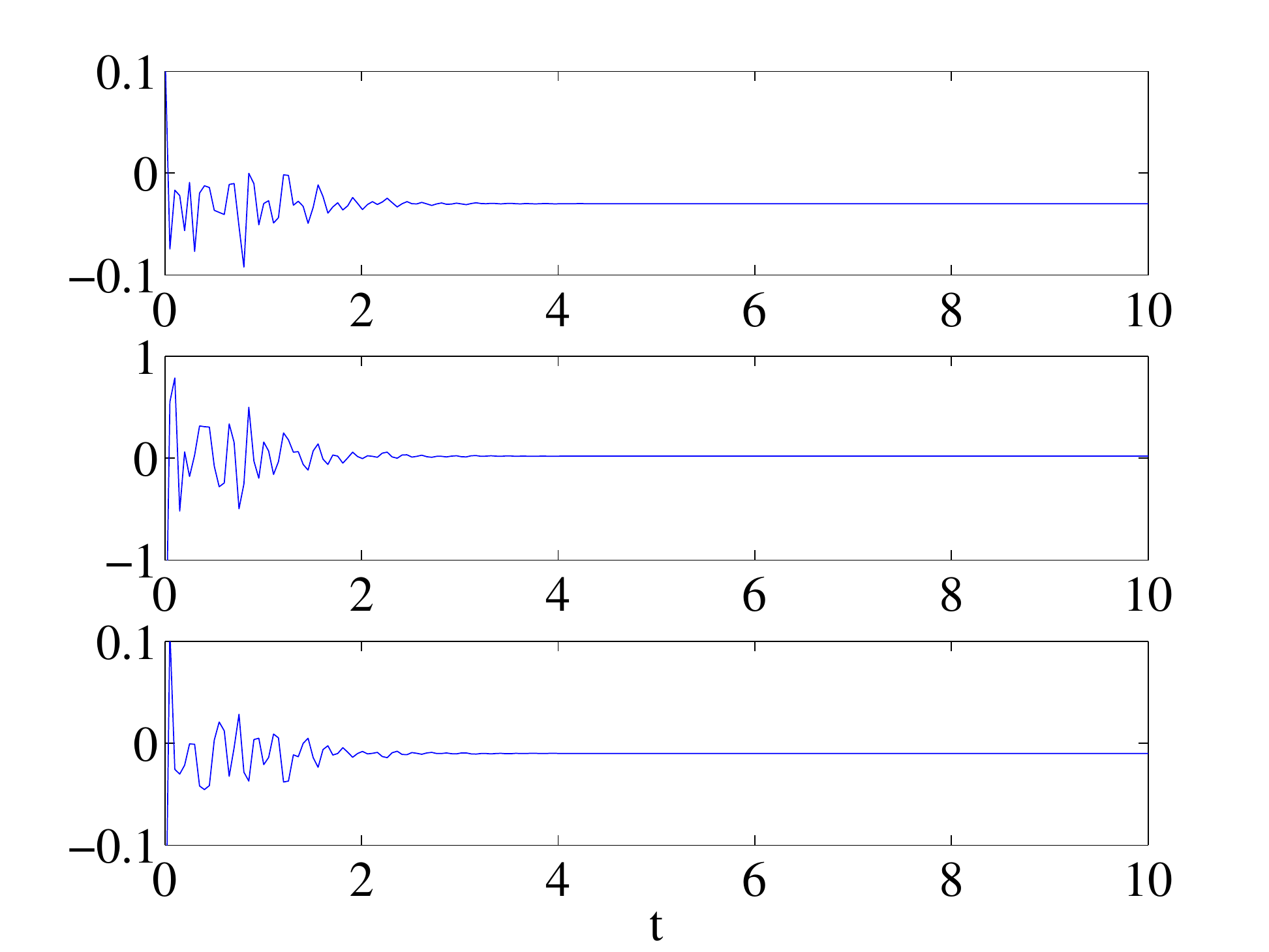}}
}
\centerline{
	\subfigure[Quadrotor position]{
		\includegraphics[width=0.4\columnwidth]{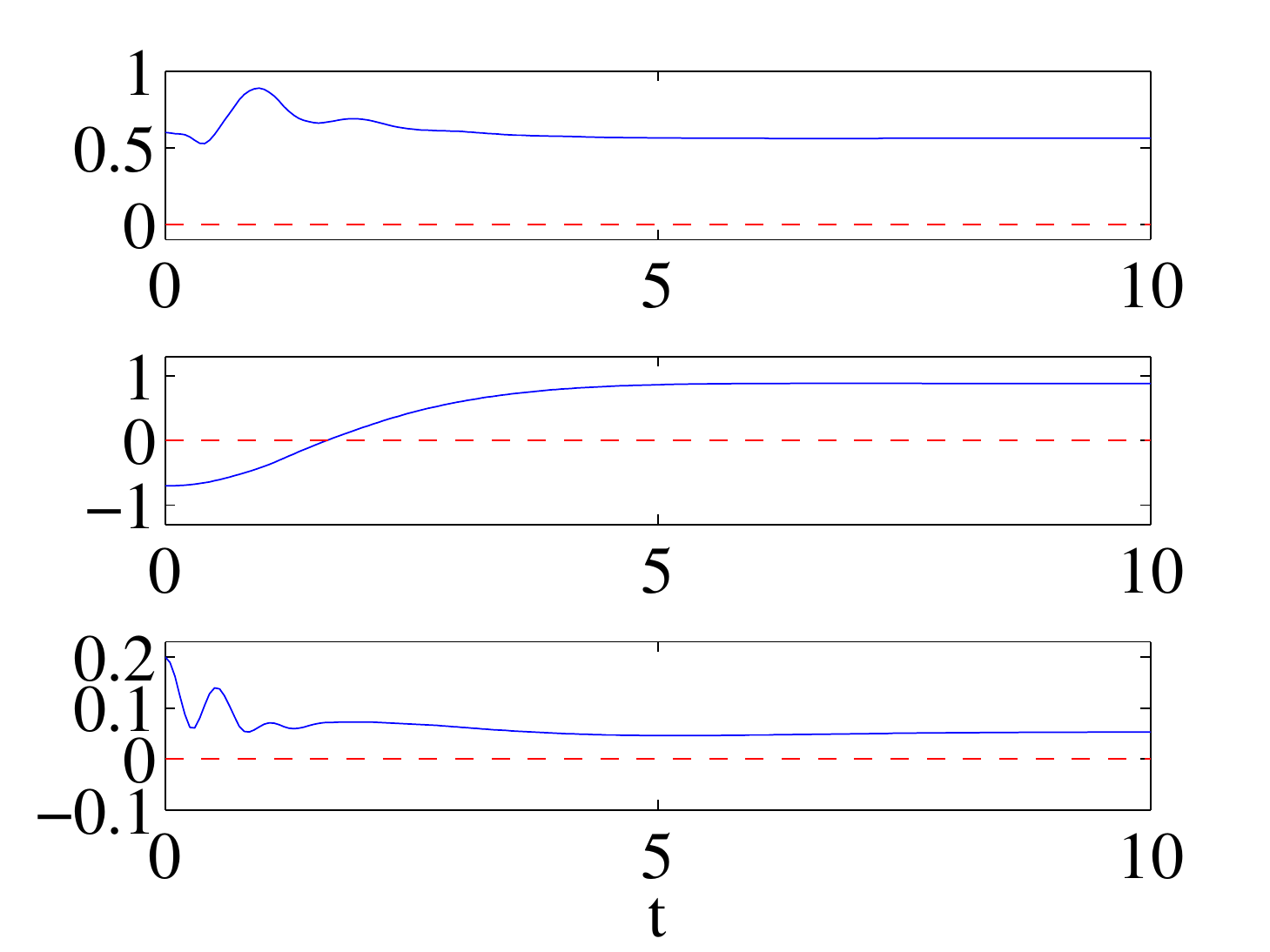}\label{fig:pos_without}}
	\subfigure[Quadrotor velocity]{
		\includegraphics[width=0.4\columnwidth]{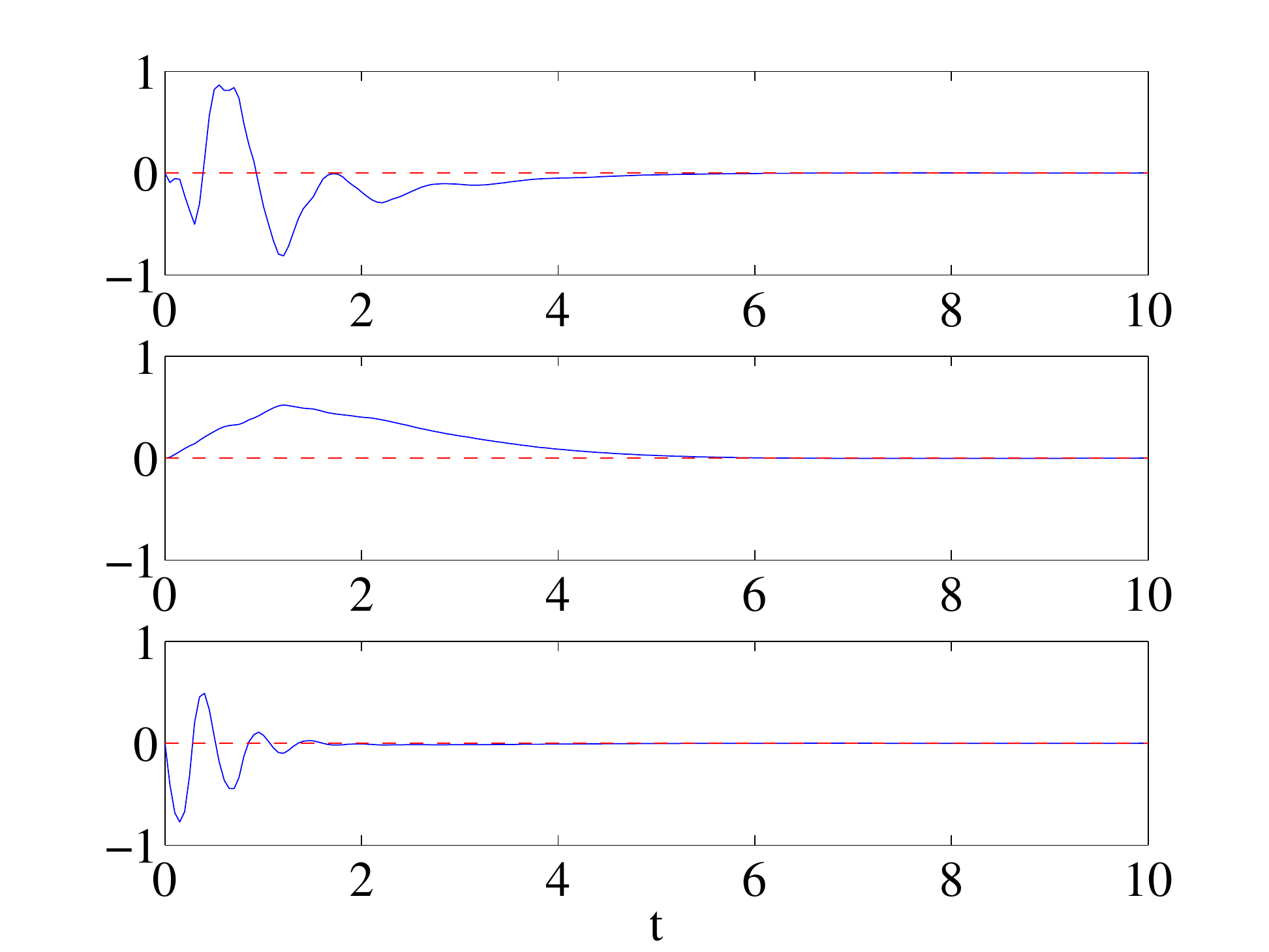}}
}
\caption{Stabilization of a payload connected to a quadrotor with 5 links without Integral term}\label{fig:simresultsWO}
\end{figure}
The desired location of the quadrotor is selected as $x_d=0_{3\times 1}$. The initial conditions for the quadrotor are given by
\begin{gather*}
x(0)=[0.6;-0.7;0.2], \ \dot{x}(0)=0_{3\times 1},\\
R(0)=I_{3\times 3},\quad \Omega(0)=0_{3\times 1}.
\end{gather*}
The initial direction of the links are chosen such that the cable is curved along the horizontal direction, as illustrated at Figure \ref{fig:fisrt_sub11}, and the initial angular velocity of each link is chosen as zero. 
The following two fixed disturbances are included in the numerical simulation.
\begin{align*}
&\Delta_R=[0.03,-0.02,0.01]^T,\\
&\Delta_x=[-0.0125,0.0125,0.01]^T.
\end{align*}

\begin{figure}
\centerline{
	\subfigure[Attitude error function $\psi$]{
		\includegraphics[width=0.4\columnwidth]{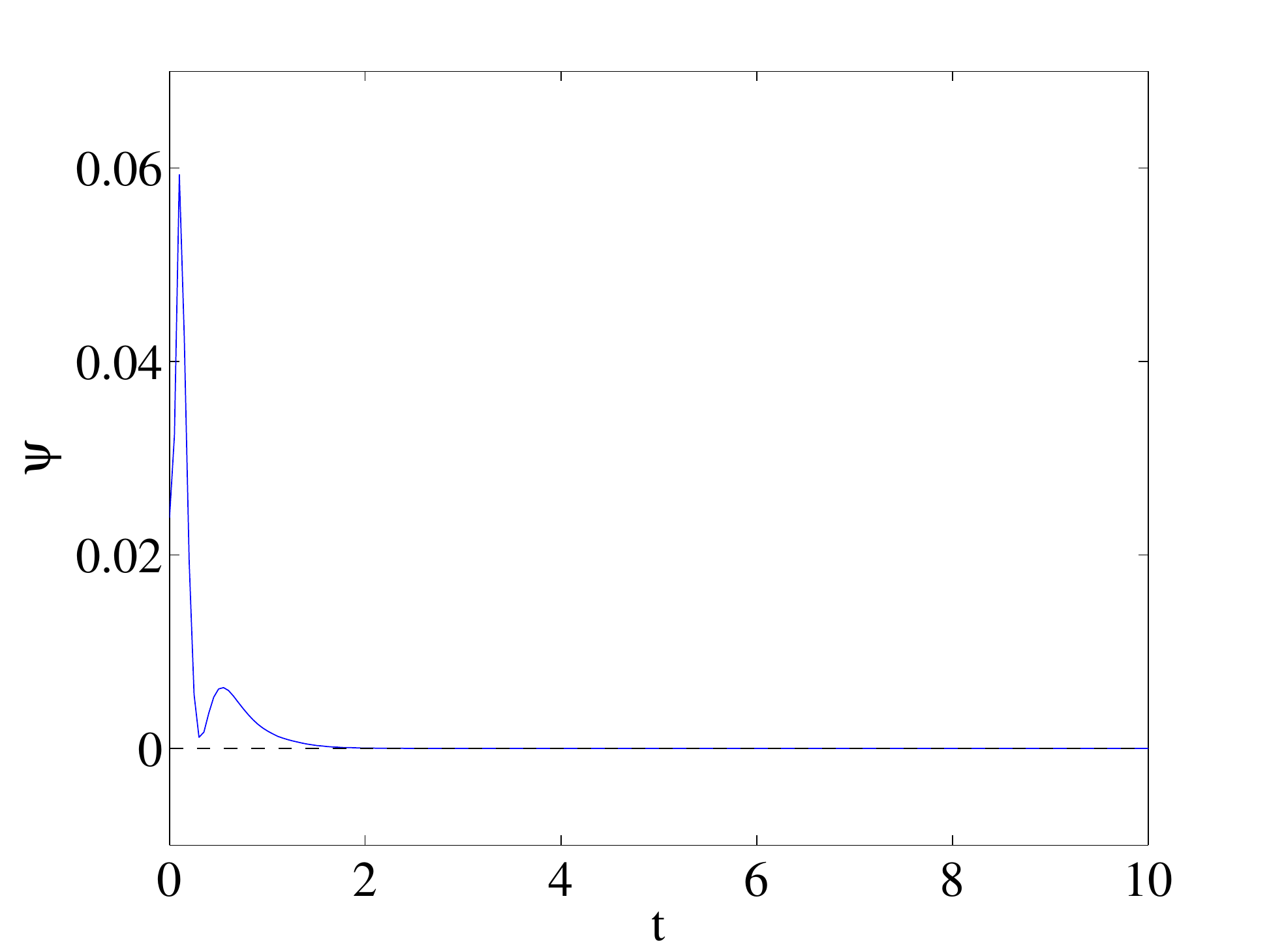}\label{fig:psi_with}}
	\subfigure[Direction error $e_{q}$ and angular velocity error $e_{\omega}$ for links]{
		\includegraphics[width=0.4\columnwidth]{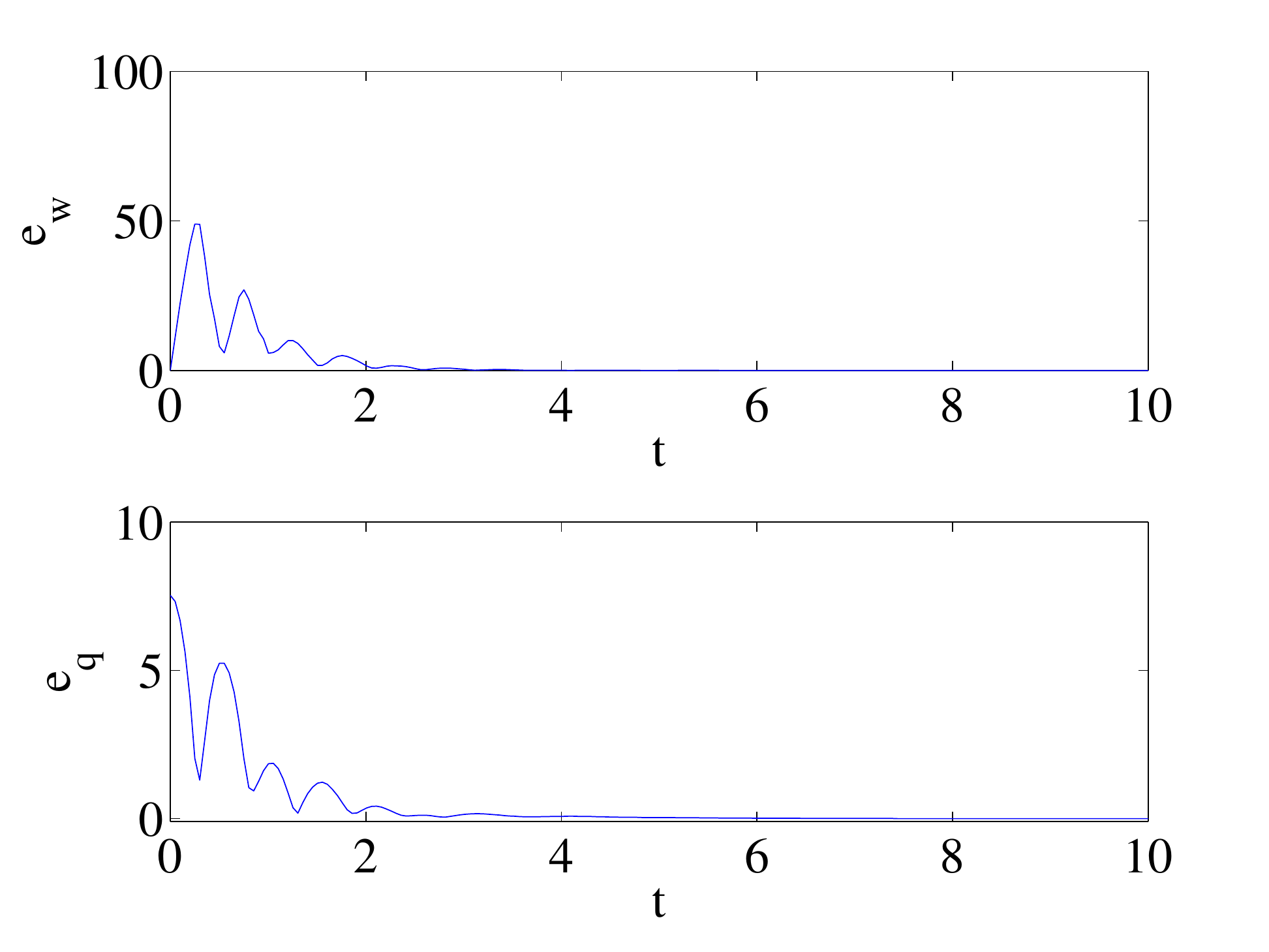}\label{fig:errors_with}}
}
\centerline{
	\subfigure[Quadrotor angular velocity $\Omega$:blue, $\Omega_{d}$:red]{
		\includegraphics[width=0.4\columnwidth]{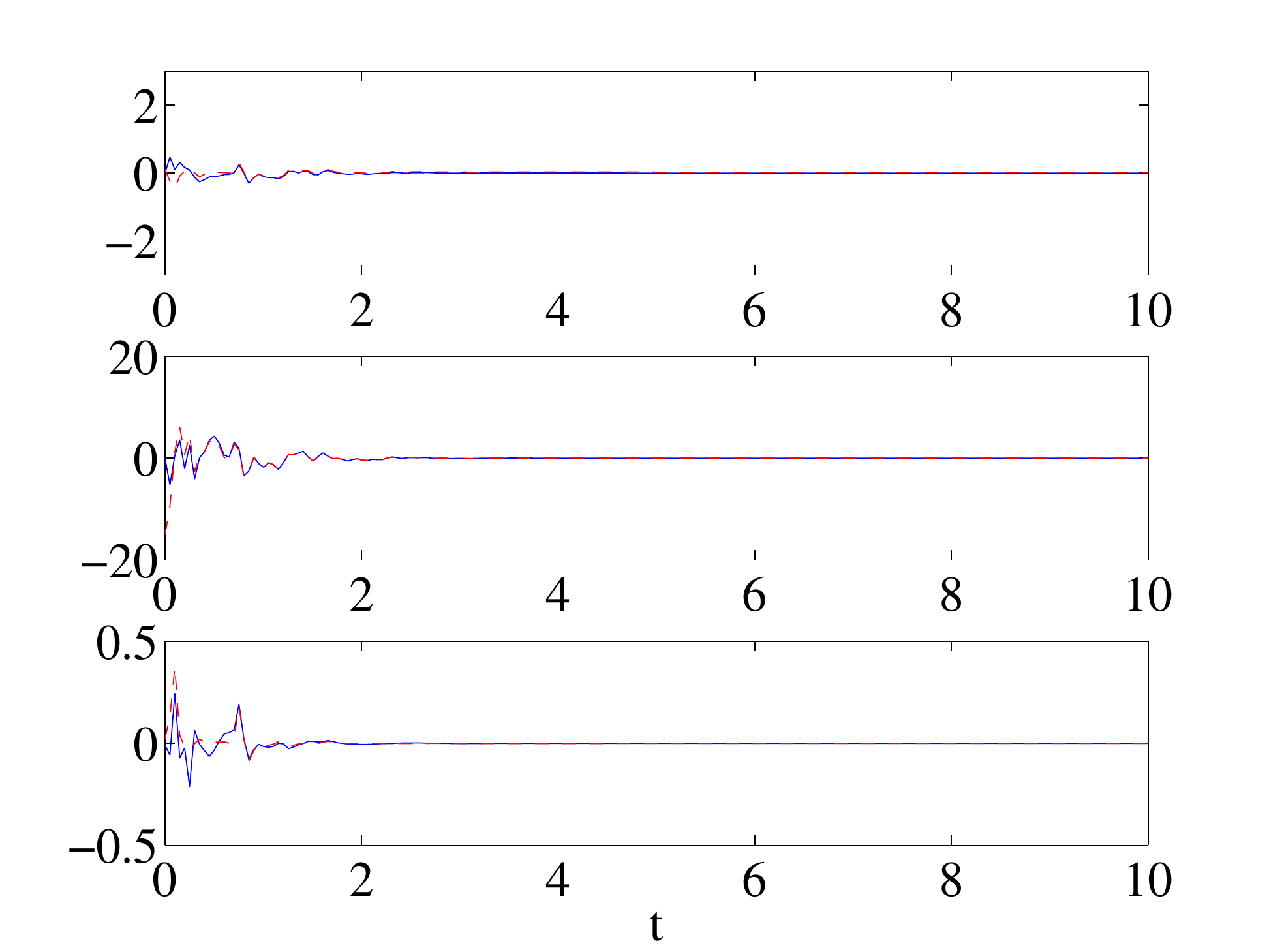}}
	\subfigure[Control force $u$]{
		\includegraphics[width=0.4\columnwidth]{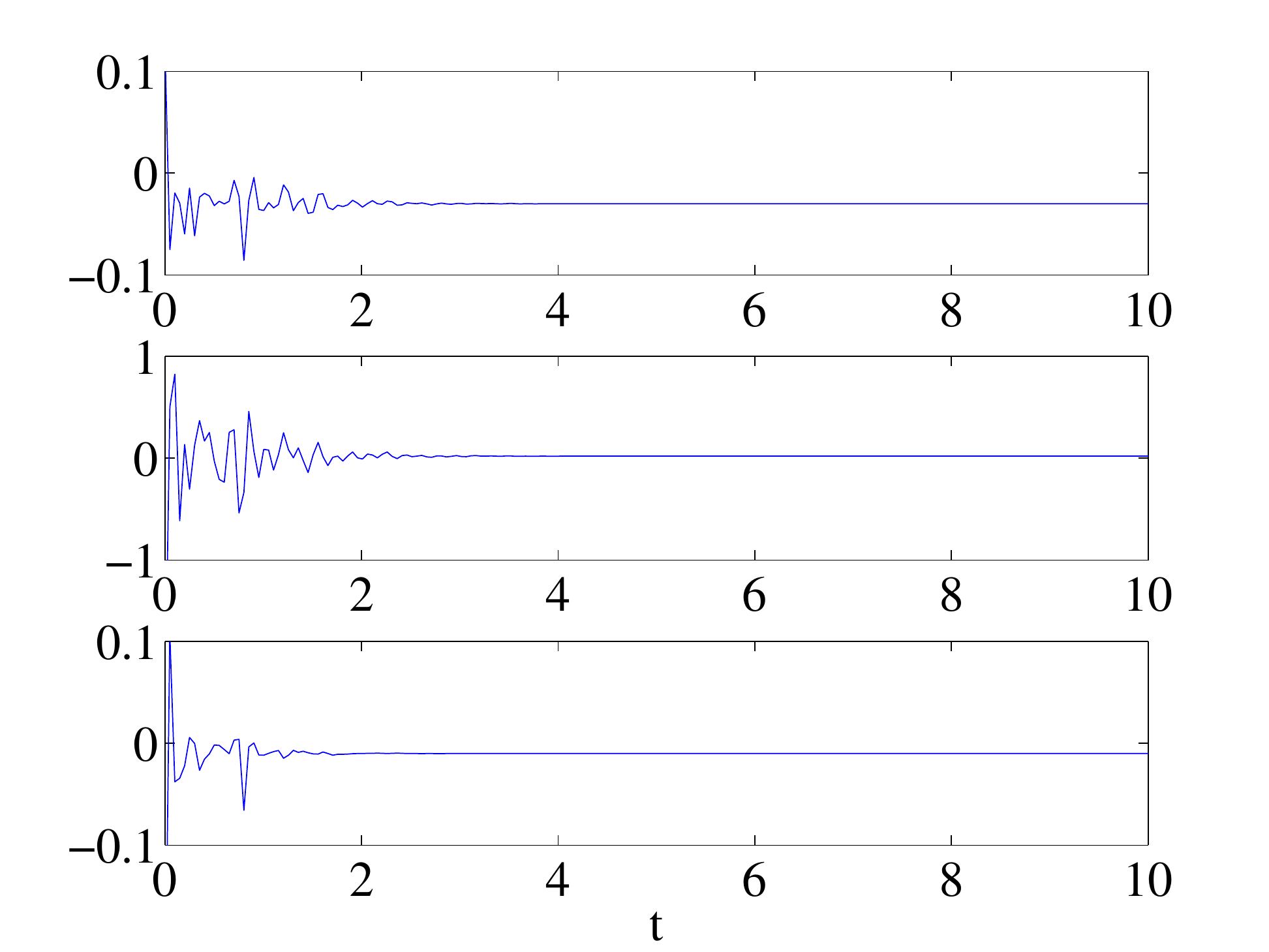}}
}
\centerline{
	\subfigure[Quadrotor position]{
		\includegraphics[width=0.4\columnwidth]{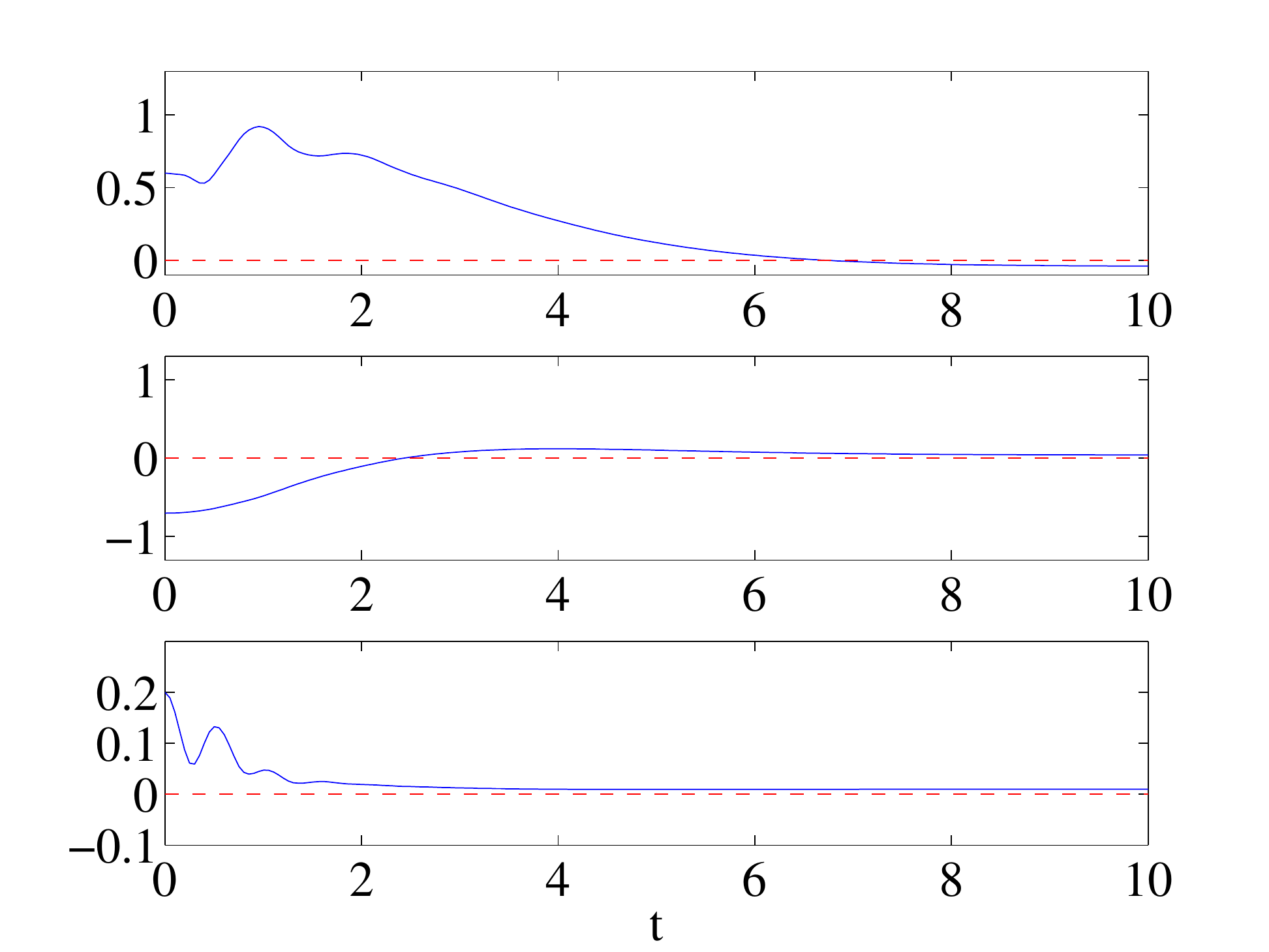}\label{fig:pos_with}}
	\subfigure[Quadrotor velocity]{
		\includegraphics[width=0.4\columnwidth]{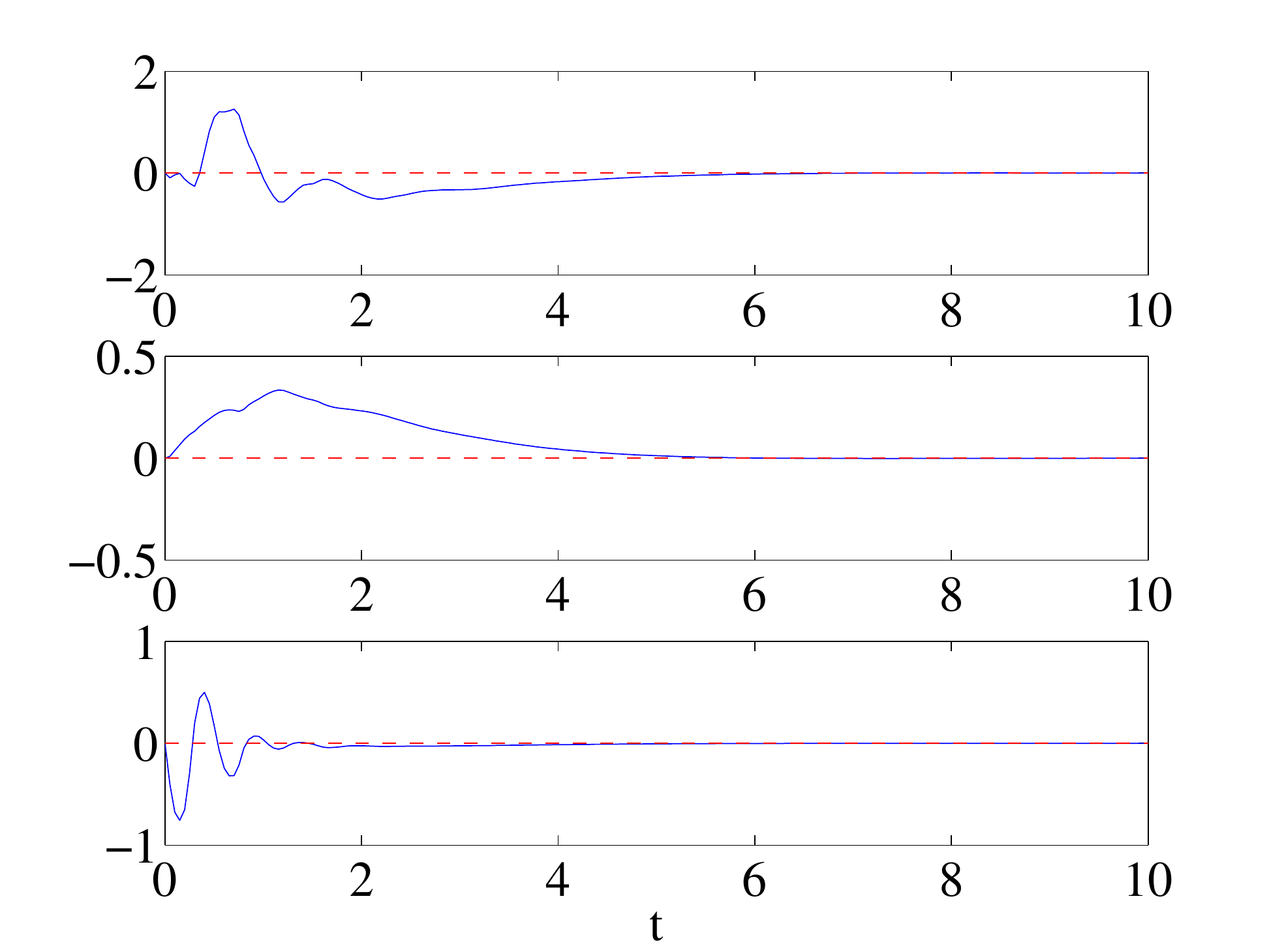}}
}
\caption{Stabilization of a payload connected to a quadrotor with 5 links with Integral term}\label{fig:simresultsW}
\end{figure}

We considered two cases for this numerical simulation to compare the effect of the proposed integral term in the presence of disturbances as follows: (i) with integral term and (ii) without integral term, to emphasize the effect of the integral term. The simulation results for each case are illustrated  at Figures \ref{fig:simresultsWO} and \ref{fig:simresultsW}, respectively. The corresponding maneuvers of the quadrotor and the links for the second case are illustrated at Figure \ref{animationsim}. 

Figures \ref{fig:psi_without} and \ref{fig:psi_with} show the attitude error function defined as
\begin{align}
\psi(R,R_{d})=\frac{1}{2}\trs{[I-R_{c}^{T}R]},
\end{align}
which represents the difference between the desired attitude and the actual attitude of the quadrotor. It is shown that there is a steady state error for the first case without the integral control term, but the error is eliminated at Figure \ref{fig:psi_with} for the second case. Next, we define the two following error cumulative variables to show the stabilizing performance for all links:
\begin{align}
e_{q}=\sum_{i=1}^{n}{\|q_{i}-e_3\|},\quad e_{\omega}=\sum_{i=1}^{n}{\|\omega_{i}\|},
\end{align}
and the results of numerical simulation for these error functions are presented in \ref{fig:errors_without} and \ref{fig:errors_with} for each case. The initial errors for the links are quite large, but they converge to zero nicely. For the first case, there exist a small steady state error in both $e_w$ and $e_q$.  Figures \ref{fig:pos_without} and \ref{fig:pos_with} show the desired location of the quadrotor (dashed line) and the actual location (solid like) for each case.

These numerical examples verifies that under the proposed control system, all of the position and attitude of the quadrotor, the direction of links, and the location of the payload asymptotically converge to their desire values, and the presented integral terms are effective in eliminating steady state errors caused by disturbances. 
 
\begin{figure}
\centering
\subfigure[$ t =0 $]
{
\includegraphics[width=0.15\columnwidth]{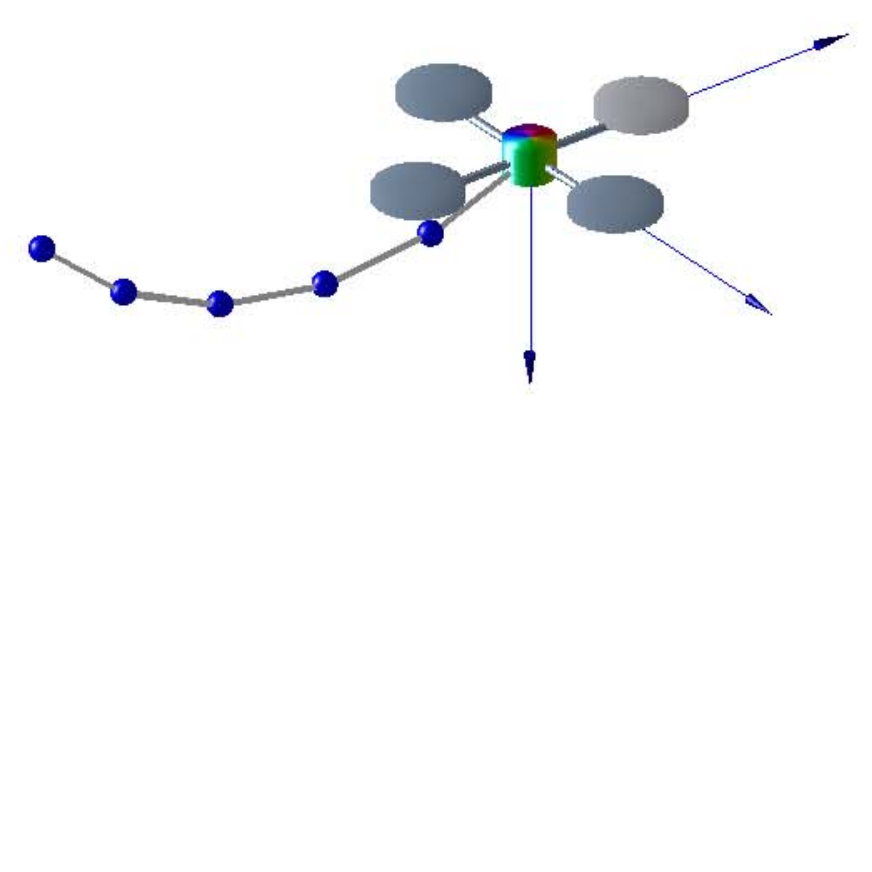}
\label{fig:fisrt_sub11}
}
\subfigure[$ t =0.2 $]
{
\includegraphics[width=0.15\columnwidth]{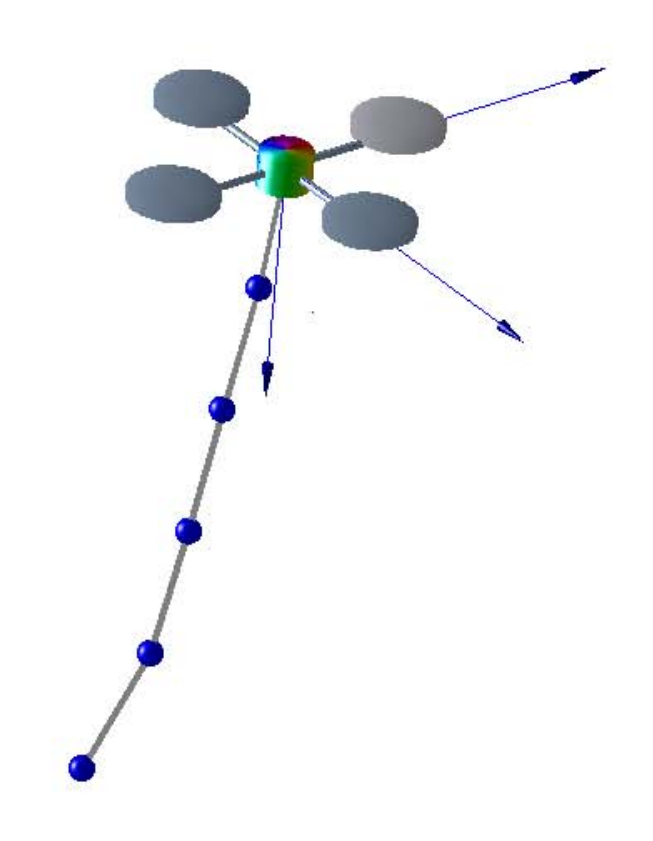}
}
\subfigure[$ t =0.35 $]
{
\includegraphics[width=0.15\columnwidth]{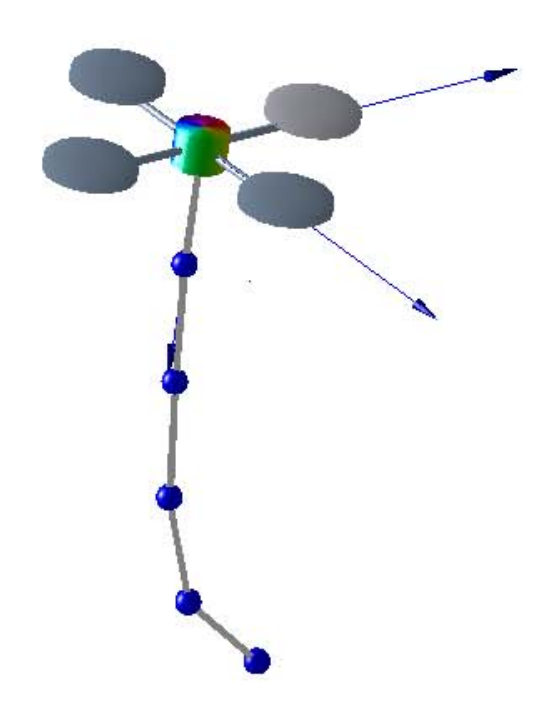}
}
\subfigure[$ t =0.40 $]
{
\includegraphics[width=0.15\columnwidth]{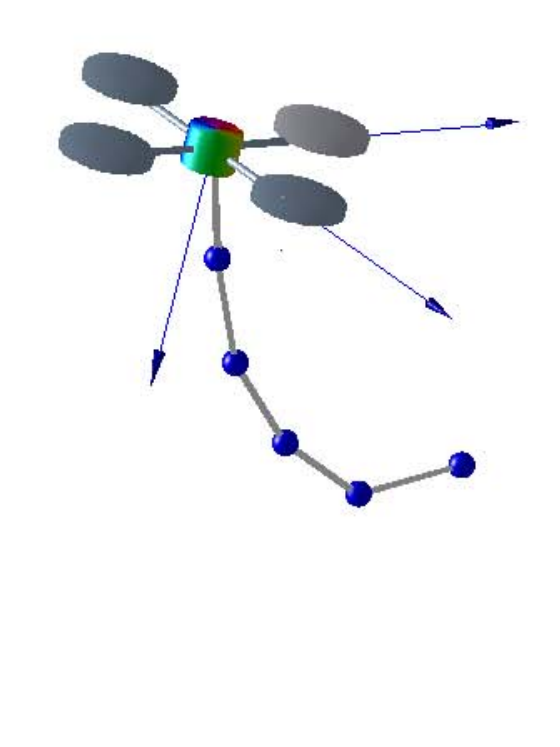}
}
\subfigure[$ t =0.42 $]
{
\includegraphics[width=0.15\columnwidth]{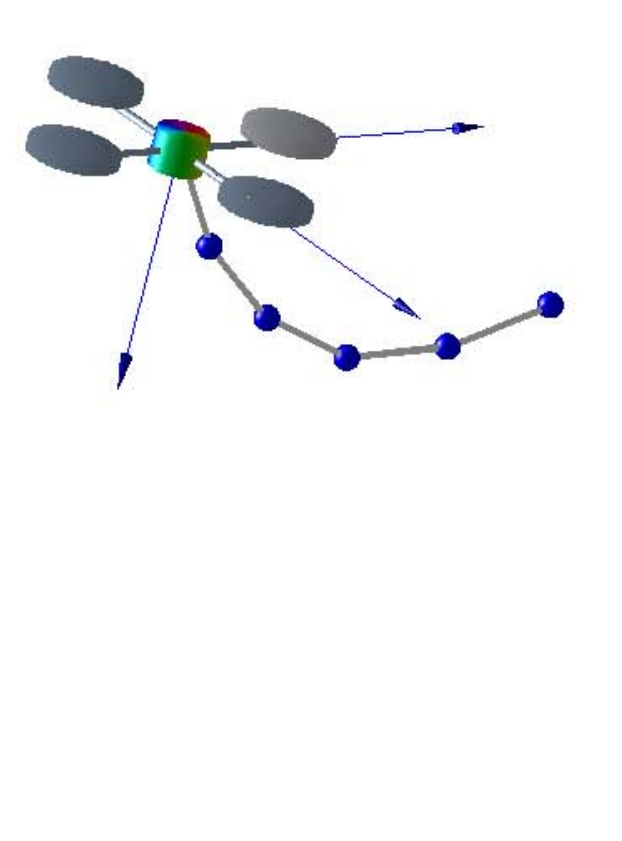}
}\\
\subfigure[$ t =0.45 $]
{
\includegraphics[width=0.15\columnwidth]{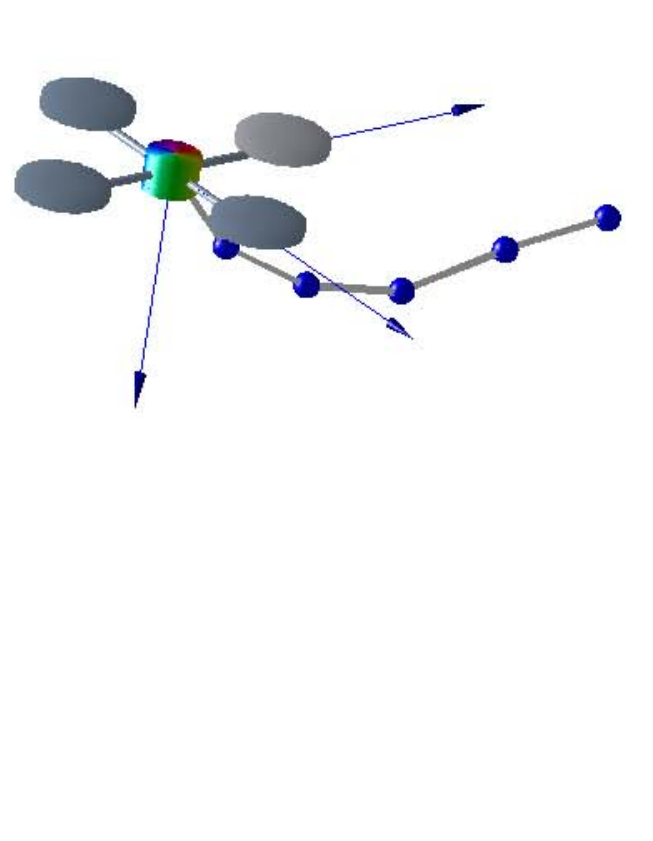}
}
\subfigure[$ t =0.5 $]
{
\includegraphics[width=0.15\columnwidth]{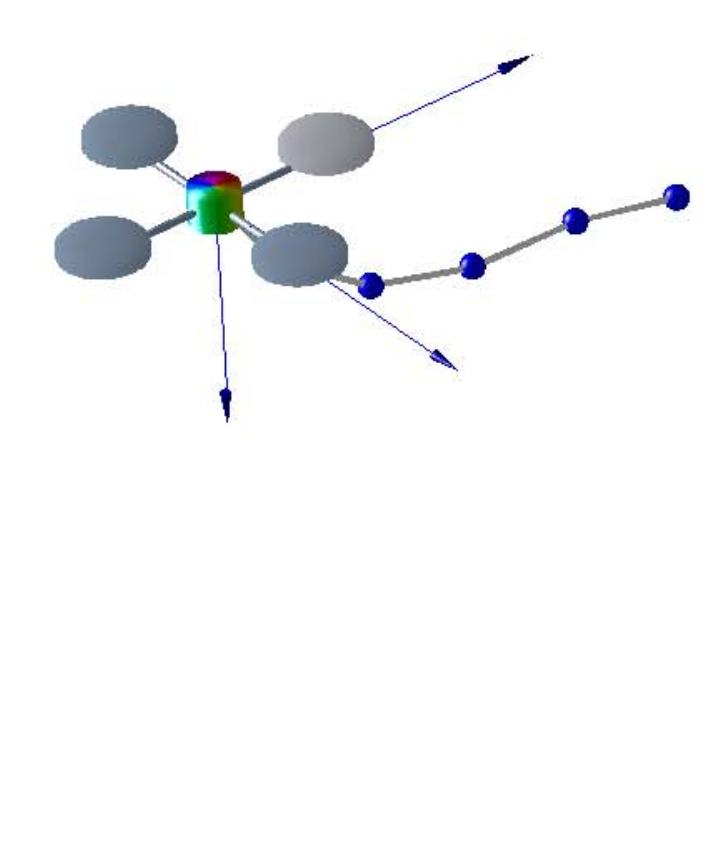}
}
\subfigure[$ t =0.6 $]
{
\includegraphics[width=0.15\columnwidth]{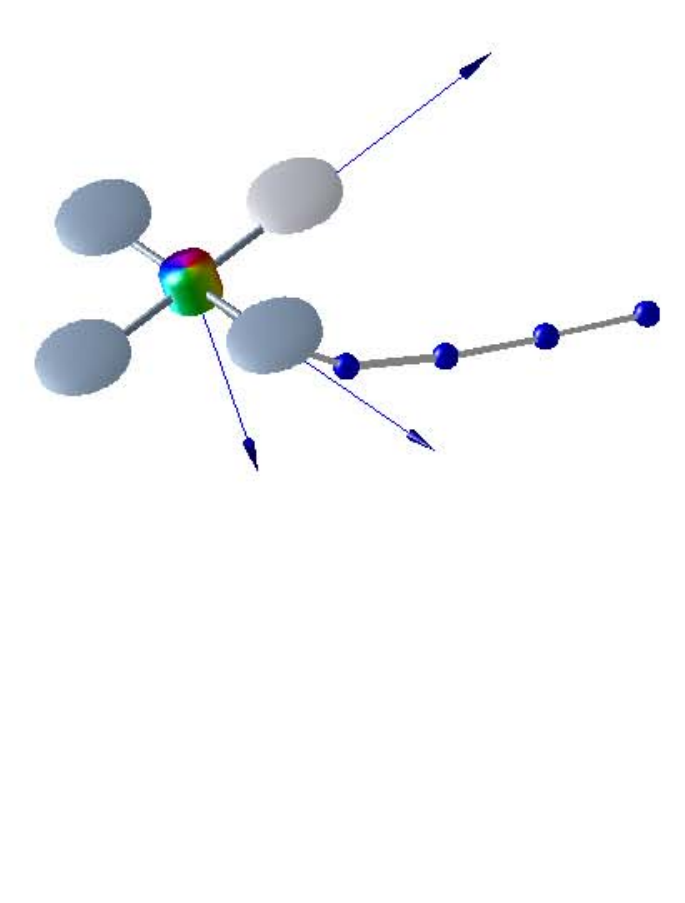}
}
\subfigure[$ t =0.7 $]
{
\includegraphics[width=0.15\columnwidth]{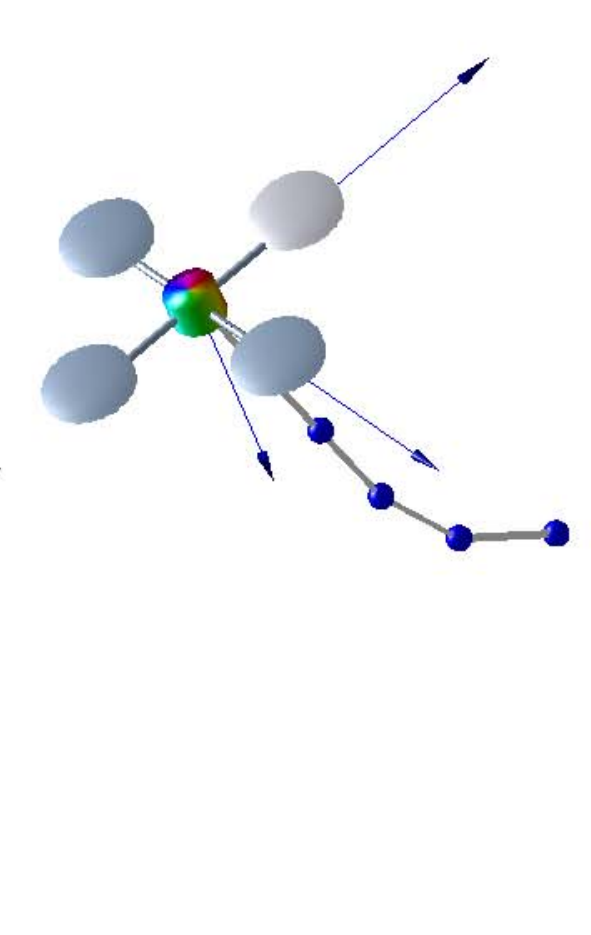}
}
\subfigure[$ t =0.8 $]
{
\includegraphics[width=0.15\columnwidth]{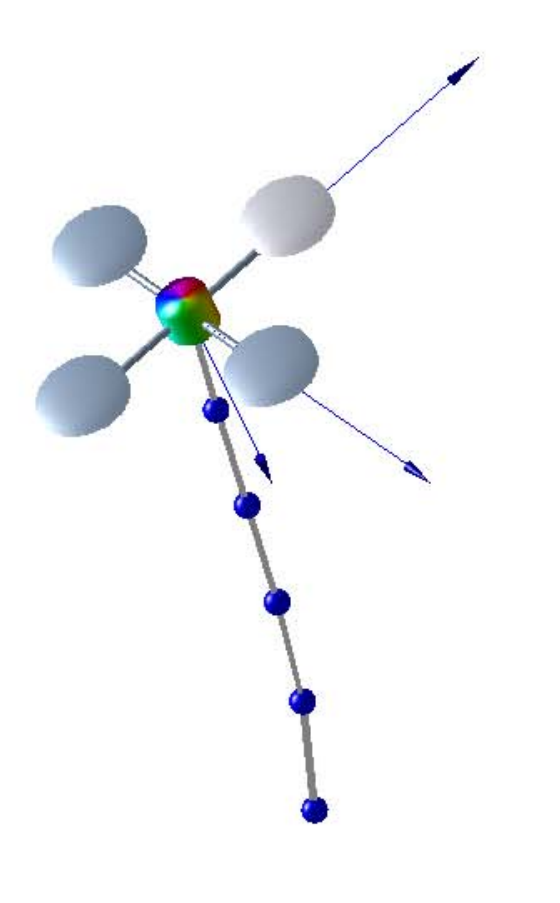}
}\\\subfigure[$ t =0.9 $]
{
\includegraphics[width=0.15\columnwidth]{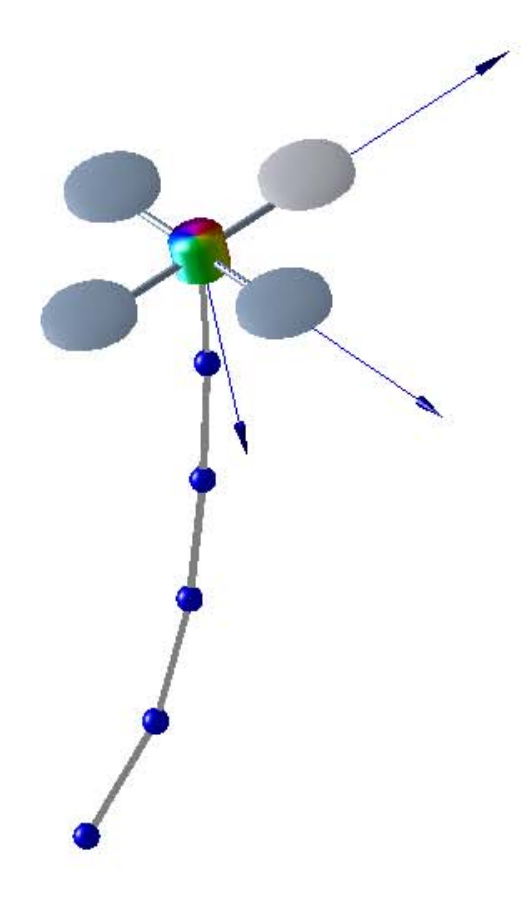}
}
\subfigure[$ t =1.3 $]
{
\includegraphics[width=0.15\columnwidth]{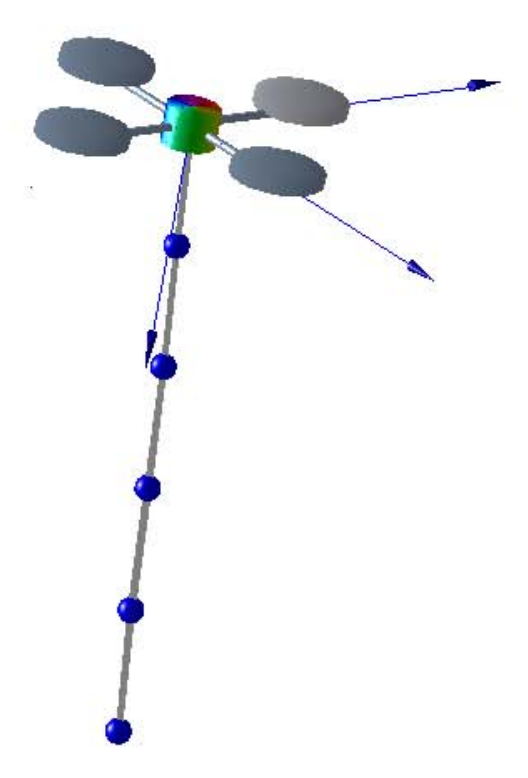}
}
\subfigure[$ t =1.5 $]
{
\includegraphics[width=0.15\columnwidth]{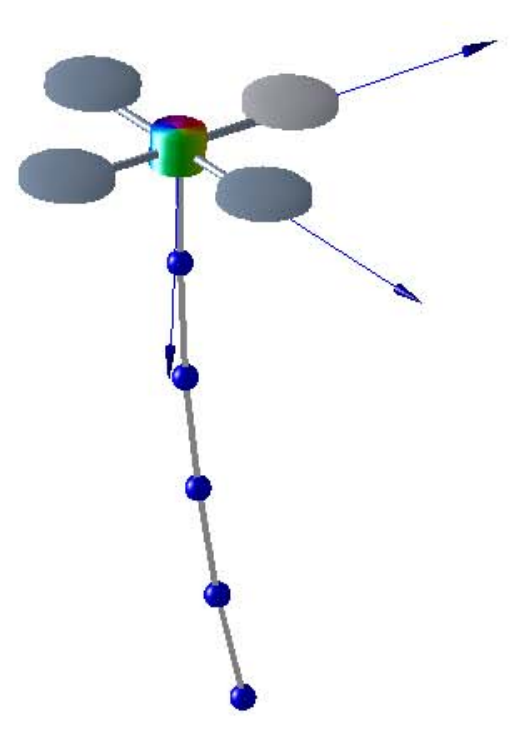}
}
\subfigure[$ t =2.0 $]
{
\includegraphics[width=0.15\columnwidth]{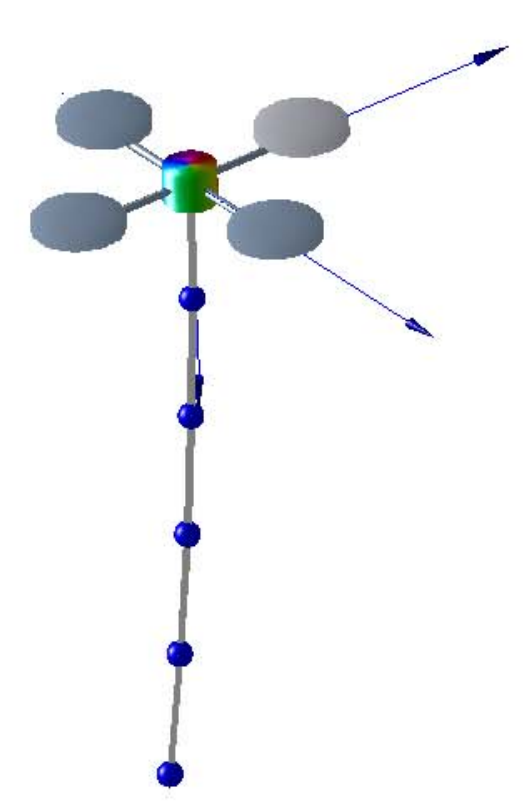}
}
\subfigure[$ t =10.0 $]
{
\includegraphics[width=0.15\columnwidth]{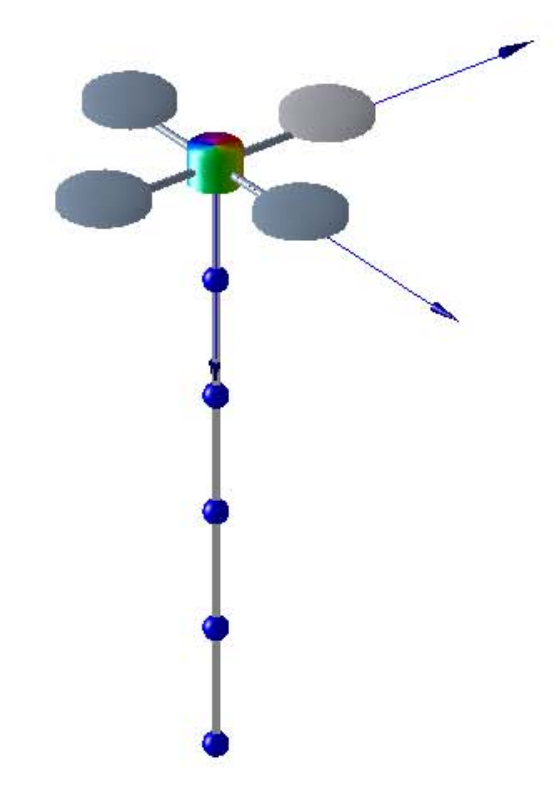}
}
\caption{Snapshots of the controlled maneuver}
\label{animationsim}
\end{figure}
}

%%%%%%%%%%%%%%%%%%%%%%%%%%%%%%%%%%%%%%%%%%%%%%%%%%%%%
%%%%%%%%%%%%%%%%%%%%%%%%%%%%%%%%%%%%%%%%%%%%%%%%%%%%%
%%%%%%%%%%%%%%%%%%%%%%%%%%%%%%%%%%%%%%%%%%%%%%%%%%%%%
\newcommand{\Pb}{\mathbf{P}}
\newcommand{\Nb}{\mathbf{N}}
\newcommand{\zb}{\mathbf{z}}

\newpage
\begin{singlespace}
\section{\protect \centering Chapter 4: Multiple Quadrotor UAVs Transporting a Rigid Body}
\end{singlespace}
\setcounter{section}{4}
\doublespacing
There are various applications for aerial load transportation such as usage in construction, military operations, emergency response, or delivering packages. Load transportation with the cable-suspended load has been studied traditionally for a helicopter~\cite{CicKanJAHS95,BerPICRA09} or for small unmanned aerial vehicles such as quadrotor UAVs~\cite{PalCruIRAM12,MicFinAR11,MazKonJIRS10}. 

In most of the prior works, the dynamics of aerial transportation has been simplified due to the inherent dynamic complexities. For example, it is assumed that the dynamics of the payload is considered completely decoupled from quadrotors, and the effects of the payload and the cable are regarded as arbitrary external forces and moments exerted to the quadrotors~\cite{ ZamStaJDSMC08, SchMurIICRA12, PalFieIICRA12}, thereby making it challenging to suppress the swinging motion of the payload actively, particularly for agile aerial transportations.

The complete dynamic model of an arbitrary number of quadrotors transporting a rigid body presented in this chapter where each quadrotor is connected to the rigid body via a flexible cable. Each flexible cable is modeled as an arbitrary number of serially connected links, and it is valid for various masses and lengths. A coordinate free form of equations of motion is derived according to Lagrange mechanics on a nonlinear manifold for the full dynamic model. These sets of equations of motion are presented in a complete and organized manner without any simplification.

Another contribution of this chapter is designing a control system to stabilize the rigid body at desired position. Geometric nonlinear controllers is utilized~\cite{LeeLeoPICDC10,LeeLeoAJC13,Farhad2013}, and they are generalized for the presented model. More explicitly, we show that the rigid body payload is asymptotically transported into a desired location, while aligning all of the links along the vertical direction corresponding to a hanging equilibrium.

The unique property of the proposed control system is that the nontrivial coupling effects between the dynamics of rigid payload, flexible cables, and multiple quadrotors are explicitly incorporated into control system design, without any simplifying assumption. Another distinct feature is that the equations of motion and the control systems are developed directly on the nonlinear configuration manifold intrinsically. Therefore, singularities of local parameterization are completely avoided to generate agile maneuvers of the payload in a uniform way. In short, the proposed control system is particularly useful for rapid and safe payload transportation in complex terrain, where the position of the payload should be controlled concurrently while suppressing the deformation of the cables.

\begin{figure}
\centerline{
	\setlength{\unitlength}{0.09\columnwidth}\scriptsize
\begin{picture}(5,8.8)(0,0)
\put(0,0){\includegraphics[width=.6\columnwidth]{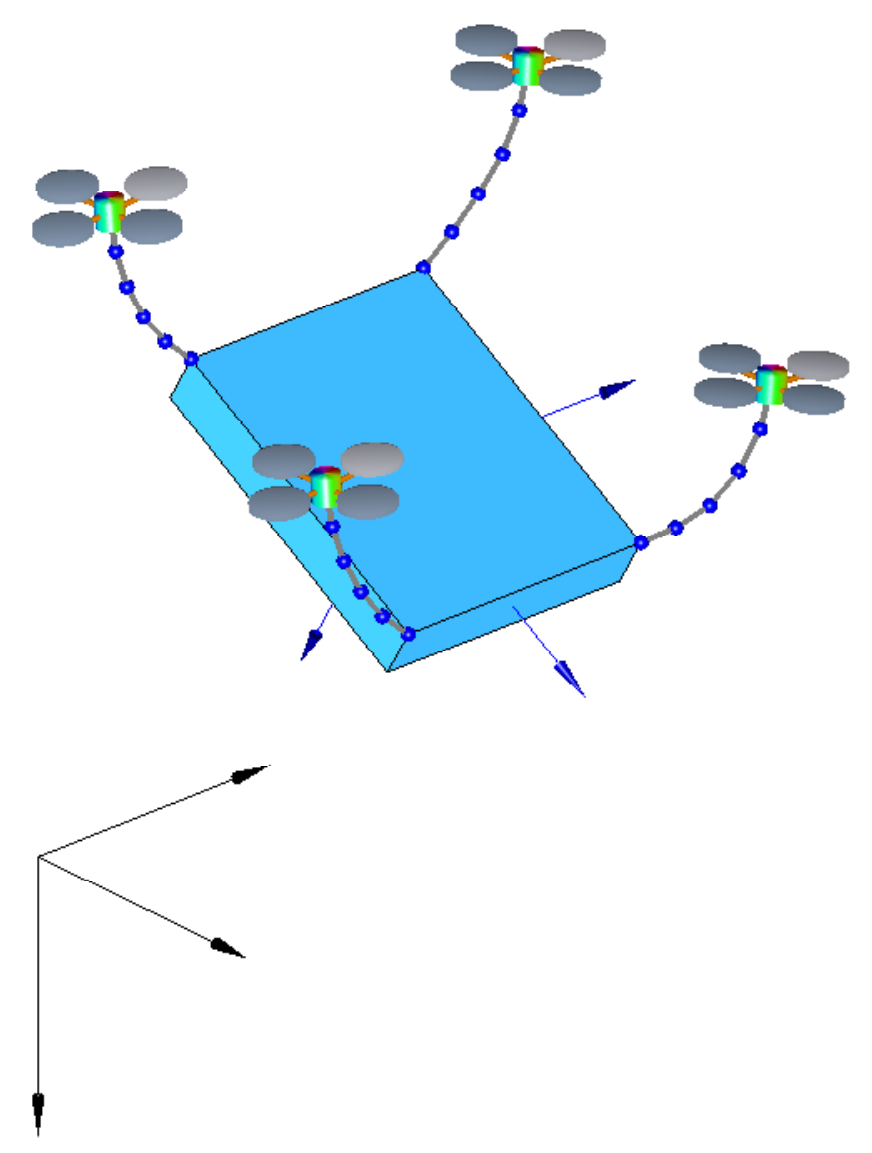}}
\put(3.6,5.3){\shortstack[c]{$m_{0}$}}
\put(3.6,4.9){\shortstack[c]{$J_{0}$}}
\put(-0.4,7.3){\shortstack[c]{$m_{1}$}}
\put(0.3,6.2){\shortstack[c]{$m_{1j}$}}
\put(-0.4,6.9){\shortstack[c]{$J_{1}$}}
\put(4.9,8.5){\shortstack[c]{$m_{2}$}}
\put(3.9,7.1){\shortstack[c]{$m_{2j}$}}
\put(4.9,8.1){\shortstack[c]{$J_{2}$}}
\put(6.7,6.0){\shortstack[c]{$m_{3}$}}
\put(5.9,5.0){\shortstack[c]{$m_{3j}$}}
\put(6.7,5.6){\shortstack[c]{$J_{3}$}}
\put(2.2,2.8){\shortstack[c]{$e_{1}$}}
\put(1.9,1.3){\shortstack[c]{$e_{2}$}}
\put(0.1,-0.1){\shortstack[c]{$e_{3}$}}
\put(4.9,5.9){\shortstack[c]{$b_{1}$}}
\put(4.6,3.2){\shortstack[c]{$b_{2}$}}
\put(2.0,3.4){\shortstack[c]{$b_{3}$}}
\end{picture}
}
\caption{Quadrotor UAVs with a rigid body payload. Cables are modeled as a serial connection of an arbitrary number of links (only 4 quadrotors with 5 links in each cable are illustrated).}\label{fig:fig1}
\end{figure}
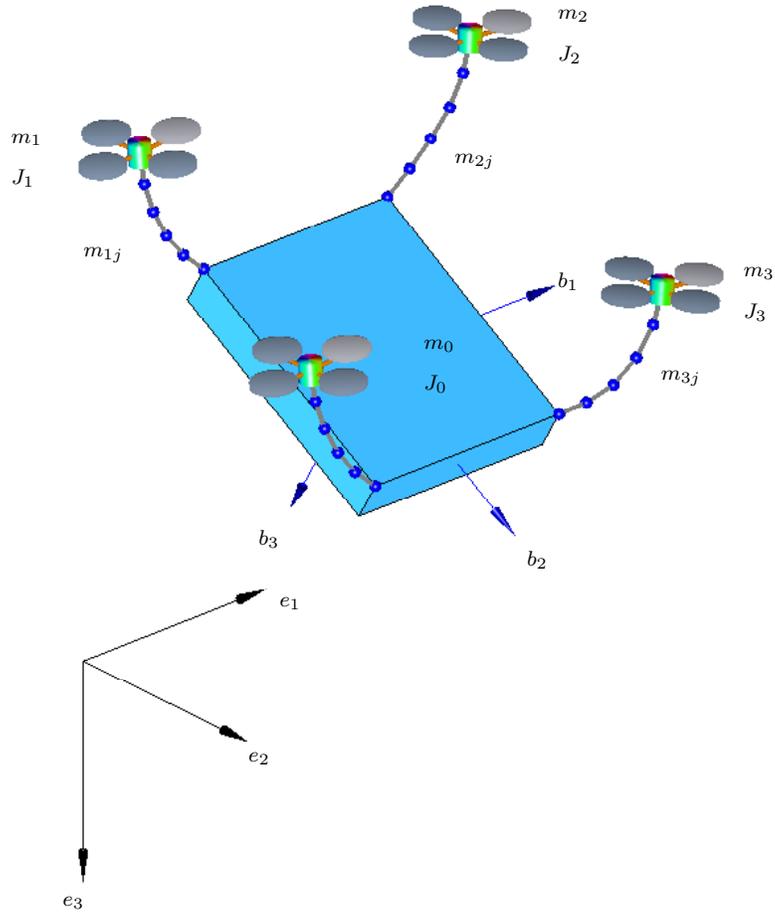

This chapter is organized as follows. A dynamic model is presented and the problem is formulated at Section \ref{sec:chap4dynamic}. Control systems are constructed at Sections \ref{sec:chap4control1} and \ref{sec:chap4control2}, which are followed by numerical examples in Section \ref{sec:chap4numerical}.

%%%%%%%%%%%%%%%%%%%%%%%%%%%%%%%%%%%%%%%%%%%%%%%%%%%%%

\subsection {\normalsize Quadrotor Dynamic Model}\label{sec:chap4dynamic}
{\addtolength{\leftskip}{0.5in}
Consider a rigid body with mass $m_{0}\in\Re$ and moment of inertia $J_{0}\in\Re^{3\times 3}$, being transported with arbitrary $n$ numbers of quadrotors. The location of the mass center of the rigid body is denoted by $x_{0}\in\Re^{3}$, and its attitude is given by $R_{0}\in\SO$, where the special orthogonal group is given by $\SO=\{R\in\Re^{3\times 3} \mid R^{T}R=I,\det(R)=1\}$. Figure \ref{fig:fig1} illustrates the system with an inertial frame. We choose an inertial frame $\{\vec{e}_{1},\vec{e}_{2},\vec{e}_{3}\}$ and body fixed frame $\{\vec{b}_{1},\vec{b}_{2},\vec{b}_{3}\}$ attached to the payload. We also consider a body fixed frame attached to the $i$-th quadrotor $\{\vec{b}_{1_i},\vec{b}_{2_i},\vec{b}_{3_i}\}$. In the inertial frame, the third axes $\vec{e}_{3}$ points downward with gravity and the other axes are chosen to form an orthonormal frame. 

The mass and the moment of inertia of the $i$-th quadrotor are denoted by $m_{i}\in\Re$ and $J_{i}\in\Re^{3\times 3}$ respectively. The cable connecting each quadrotor to the rigid body is modeled as an arbitrary numbers of links for each quadrotor with varying masses and lengths. The direction of the $j$-th link of the $i$-th quadrotor, measured outward from the quadrotor toward the payload is defined by the unit vector $q_{ij}\in\Sph^2$, where $\Sph^2=\{q\in\Re^{3}\mid\|q\|=1\}$, where the mass and length of that link is denoted with $m_{ij}$ and $l_{ij}$ respectively. The number of links in the cable connected to the $i$-th quadrotor is defined as $n_{i}$.

The configuration manifold for this system is given by $\SO\times \Re^{3}\times (\SO^{n})\times (\Sph^2)^{\sum_{i=1}^{n}n_{i}}$. The $i$-th quadrotor can generate a thrust force of $-f_{i}R_{i}e_{3}\in\Re^{3}$ with respect to the inertial frame, where $f_{i}\in\Re$ is the total thrust magnitude of the $i$-th quadrotor. It also generates a moment $M_{i}\in\Re^{3}$ with respect to its body-fixed frame.
Throughout this chapter, the two norm of a matrix $A$ is denoted by $\|A\|$. The standard dot product is denoted by $x\cdot y=x^{T}y$ for any $x,y\in\Re^{3}$.
\subsubsection {\normalsize Lagrangian}
The kinematics equations for the links, payload, and quadrotors are given by
\begin{gather}
\dot{q}_{ij}=\omega_{ij}\times q_{ij}=\hat{\omega}_{ij}q_{ij},\\
\dot{R}_{0}=R_{0}\hat{\Omega}_{0},\\
\dot{R}_{i}=R_{i}\hat{\Omega}_{i},
\end{gather}
where $\omega_{ij}\in\Re^{3}$ is the angular velocity of the $j$-th link in the $i$-th cable satisfying $q_{ij}\cdot\omega_{ij}=0$. Also, $\Omega_{0}\in\Re^{3}$ is the angular velocity of the payload and $\Omega_{i}\in\Re^{3}$ is the angular velocity of the $i$-th quadrotor, expressed with respect to the corresponding body fixed frame. The hat map $\hat{\cdot}:\Re^{3}\rightarrow\so$ is defined by the condition that $\hat{x}y=x\times y$ for all $x,y\in\Re^{3}$, and the inverse of the hat map is denoted by the vee map $\vee:\so\rightarrow\Re^{3}$.

The position of the $i$-th quadrotor is given by
\begin{align}\label{eqn:xi}
x_{i}=x_{0}+R_{0}\rho_{i}-\sum_{a=1}^{n_{i}}{l_{ia}q_{ia}},
\end{align}
where $\rho_{i}\in\Re^{3}$ is the vector from the center of mass of the rigid body to the point that $i$-th quadrotor is connected to rigid body via the cable. Similarly the position of the $j$-th link in the cable connecting the $i$-th quadrotor to the rigid body is given by
\begin{align}\label{eqn:xij}
x_{ij}=x_{0}+R_{0}\rho_{i}-\sum_{a=j+1}^{n_{i}}{l_{ia}q_{ia}}.
\end{align}

We derive equations of motion according to Lagrangian mechanics. Total kinetic energy of the system is given by
\begin{align}
T=&\frac{1}{2}m_{0}\|\dot{x}_{0}\|^{2}+\sum_{i=1}^{n}\sum_{j=1}^{n_{i}}{\frac{1}{2}m_{ij}\|\dot{x}_{ij}\|^{2}}+\frac{1}{2}\sum_{i=1}^{n}{m_{i}\|\dot{x}_{i}\|^{2}}\nonumber\\
&+\frac{1}{2}\sum_{i=1}^{n}{\Omega_{i}\cdot J_{i}\Omega_{i}}+\frac{1}{2}\Omega_{0}\cdot J_{0}\Omega_{0}.
\end{align}
The gravitational potential energy is given by
\begin{align}
V=-m_{0}ge_{3}\cdot x_{0}-\sum_{i=1}^{n}{m_{i}ge_{3}}\cdot x_{i}-\sum_{i=1}^{n}\sum_{j=1}^{n_{i}}{m_{ij}ge_{3}}\cdot x_{ij},
\end{align}
where it is assumed that the unit-vector $e_{3}$ points downward along the gravitational acceleration as shown at Figure \ref{fig:fig1}. The corresponding Lagrangian of the system is $L=T-V$.

\subsubsection {\normalsize Euler-Lagrange equations}
Coordinate-free form of Lagrangian mechanics on the two-sphere $\Sph^2$ and the special orthogonal group $\SO$ for various multi-body systems has been studied in~\cite{Lee08,LeeLeoIJNME08}. The key idea is representing the infinitesimal variation of $R_i\in\SO$ in terms of the exponential map
\begin{align}
\delta R_{i} = \frac{d}{d\epsilon}\bigg|_{\epsilon = 0} R_{i}\exp(\epsilon \hat\eta_{i}) = R_{i}\hat\eta_{i},\label{eqn:delR}
\end{align}
for $\eta_{i}\in\Re^3$. The corresponding variation of the angular velocity is given by $\delta\Omega_{i}=\dot\eta_{i}+\Omega_{i}\times\eta_{i}$. Similarly, the infinitesimal variation of $q_{ij}\in\Sph^2$ is given by
\begin{align}
\delta q_{ij} = \xi_{ij}\times q_{ij},\label{eqn:delqi}
\end{align}
for $\xi_{ij}\in\Re^3$ satisfying $\xi_{ij}\cdot q_{ij}=0$. This lies in the tangent space as it is perpendicular to $q_{i}$. Using these, we obtain the following Euler-Lagrange equations.

\begin{prop}\label{prop:propchap4_1}
By using the above expressions, the equations of motion can be obtained from Hamilton's principle:
\begin{gather}\label{eqn:EOM}
M_{T}\ddot{x}_{0}-\sum_{i=1}^{n}\sum_{j=1}^{n_{i}}{M_{0ij}l_{ij}\ddot{q}_{ij}}-\sum_{i=1}^{n}{M_{iT}R_{0}\hat{\rho}_{i}\dot{\Omega}_{0}} \nonumber\\
=M_{T}ge_{3}+\sum_{i=1}^{n}{-f_{i}R_{i}e_{3}}-\sum_{i=1}^{n}{M_{iT}R_{0}\hat{\Omega}_{0}^2 \rho_{i}},\label{eqn:EOMM1}\\
\bar{J}_{0}\dot{\Omega}_{0}+\sum_{i=1}^{n}{M_{iT}\hat{\rho}_{i}R_{0}^{T}\ddot{x}_{0}}-\sum_{i=1}^{n}\sum_{j=1}^{n_{i}}{M_{0ij}l_{ij}\hat{\rho}_{i}R_{0}^{T}\ddot{q}_{ij}} \nonumber\\
=\sum_{i=1}^{n}{\hat{\rho}_{i}R_{0}^{T}(-f_{i}R_{i}e_{3}+M_{iT}ge_{3})}-\hat{\Omega}_{0}\bar{J}_{0}\Omega_{0},\label{eqn:EOMM2}\\
\sum_{k=1}^{n_{i}}{M_{0ij}l_{ik}\hat{q}_{ij}^{2}\ddot{q}_{ik}}-M_{0ij}\hat{q}_{ij}^{2}\ddot{x}_{0}+M_{0ij}\hat{q}_{ij}^{2}R_{0}\hat{\rho}_{i}\dot{\Omega}_{0} \nonumber\\
=M_{0ij}\hat{q}_{ij}^{2}R_{0}\hat{\Omega}_{0}^{2}\rho_{i}-\hat{q}_{ij}^{2}(M_{0ij}ge_{3}-f_{i}R_{i}e_{3}),\label{eqn:EOMM3}\\
J_{i}\Omega_{i}+\Omega_{i}\times J_{i}\Omega_{i}=M_{i}\label{eqn:EOMM4}.
\end{gather}
Here the total mass $M_{T}$ of the system and the mass of the $i$-th quadrotor and its flexible cable $M_{iT}$ are defined as
\begin{gather}
M_{T}=m_{0}+\sum_{i=1}^{n}M_{iT},\; M_{iT}=\sum_{j=1}^{n_{i}}{m_{ij}}+m_{i},\label{eqn:def1}
\end{gather}
and the constants related to the mass of links are given as
\begin{align}
M_{0ij}&=m_{i}+\sum_{a=1}^{j-1}{m_{ia}}\label{eqn:def3},
\end{align}
The equations of motion can be rearranged in a matrix form as follow
\begin{align}
\Nb\ddot{X}=\Pb
\end{align}
where the state vector $X\in\Re^{D_{X}}$ with $D_{X}=6+3\sum_{i=1}^{n}n_{i}$ is given by
\begin{align}
X=[{x}_{0},\; {\Omega}_{0},\; {q}_{1j},\; {q}_{2j},\; \cdots,\; {q}_{nj}]^{T},
\end{align}
and matrix $\Nb\in\Re^{D_{X}\times D_{X}}$ is defined as
\begin{align}\label{eqn:EOM11}
\Nb=\begin{bmatrix}
M_{T}I_{3}&\Nb_{x_{0}\Omega_{0}}&\Nb_{x_{0}1}&\Nb_{x_{0}2}&\cdots&\Nb_{x_{0}n}\\
\Nb_{\Omega_{0} x_{0}}&\bar{J}_{0}&\Nb_{\Omega_{0}1}&\Nb_{\Omega_{0}2}&\cdots&\Nb_{\Omega_{0}n}\\
\Nb_{1 x_{0}}&\Nb_{1\Omega_{0}}&\Nb_{qq1}&0&\cdots&0\\
\Nb_{2 x_{0}}&\Nb_{2\Omega_{0}}&0&\Nb_{qq2}&\cdots&0\\
\vdots&\vdots&\vdots&\vdots&\vdots&\vdots\\
\Nb_{n x_{0}}&\Nb_{n\Omega_{0}}&0&0&\cdots&\Nb_{qqn}
\end{bmatrix},
\end{align}
where the sub-matrices are defined as
\begin{gather}
\Nb_{x_{0}\Omega_{0}}=-\sum_{i=1}^{n}{M_{iT}R_{0}\hat{\rho}_{i}};\; \Nb_{\Omega_{0} x_{0}}=\Mb_{x_{0}\Omega_{0}}^{T},\nonumber\\
\Nb_{x_{0}i}=-[M_{0i1}l_{i1}{I}_{3},\; M_{0i2}l_{i2}{I}_{3},\; \cdots,\;M_{0in_{i}}l_{in_{i}}{I}_{3}],\nonumber\\
\Nb_{\Omega_{0}i}=-[M_{0i1}l_{i1}\hat{\rho}_{i}R_{0}^{T},\; M_{0i2}l_{i2}\hat{\rho}_{i}R_{0}^{T},\; \cdots,\; M_{0in_{i}}l_{in_{i}}\hat{\rho}_{i}R_{0}^{T}],\nonumber\\
\Nb_{ix_{0}}=-[M_{0i1}\hat{q}_{i1}^{2},\; M_{0i2}\hat{q}_{i2}^{2},\; \cdots,\;M_{0in_{i}}\hat{q}_{in_{i}}^{2}]^{T},\nonumber\\
\Nb_{i\Omega_{0}}=[M_{0i1}\hat{q}_{i1}^{2}R_{0}\hat{\rho}_{i},\; M_{0i2}\hat{q}_{i2}^{2}R_{0}\hat{\rho}_{i},\; \cdots,\; M_{0in_{i}}\hat{q}_{in_{i}}^{2}R_{0}\hat{\rho}_{i}]^{T},
\end{gather}
and the sub-matrix $\Nb_{qqi}\in\Re^{3n_i\times 3n_i}$ is given by
\begin{align}\Nb_{qqi}=
\begin{bmatrix}
-M_{011}l_{i1}I_{3}&M_{012}l_{i2}\hat{q}_{i2}^2&\cdots&M_{01n_{i}}l_{in_{i}}\hat{q}_{in_{i}}^2\\
M_{021}l_{i1}\hat{q}_{i1}^2&-M_{022}l_{i2}I_{3}&\cdots&M_{02n_{i}}l_{in_{i}}\hat{q}_{in_{i}}^2\\
\vdots&\vdots&&\vdots\\
M_{0n_{i}1}l_{i1}\hat{q}_{i1}^2&M_{0n_{i}2}l_{i2}\hat{q}_{i2}^2&\cdots&-M_{0n_{i}n_{i}}l_{in_{i}}I_{3}
\end{bmatrix}.
\end{align}
The $\Pb\in\Re^{D_{X}}$ matrix is
\begin{align}
\Pb=[P_{x_{0}},\; P_{\Omega_{0}},\; P_{1j},\; P_{2j},\; \cdots,\; P_{nj}]^{T},
\end{align}
and sub-matrices of $\Pb$ matrix are also defined as 
\begin{align*}
P_{x_{0}}&=M_{T}ge_{3}+\sum_{i=1}^{n}{-f_{i}R_{i}e_{3}}-\sum_{i=1}^{n}{M_{iT}R_{0}\hat{\Omega}_{0}^2\rho_{i}},\\
P_{\Omega_{0}}&=-\hat{\Omega}_{0}\bar{J}_{0}\Omega_{0}+\sum_{i=1}^{n}{\hat{\rho}_{i}R_{0}^T(M_{iT}ge_{3}-f_{i}R_{i}e_{3})},\nonumber\\
P_{ij}=&-\hat{q}_{ij}^2(-f_{i}R_{i}e_{3}+M_{0ij}ge_{3})+M_{0ij}\hat{q}_{ij}^2 R_{0}\hat{\Omega}_{0}^2 \rho_{i}+M_{0ij}\|\dot{q}_{ij}\|^{2}q_{ij}.\nonumber
\end{align*}
\end{prop}
\begin{proof}
See Appendix~\ref{sec:pfchap4_1}
\end{proof}
These equations are derived directly on a nonlinear manifold without any simplification. The dynamics of the payload, flexible cables, and quadrotors are considered explicitly, and they avoid singularities and complexities associated to local coordinates.   

}

%%%%%%%%%%%%%%%%%%%%%%%%%%%%%%%%%%%%%%%%%%%%%%%%%%%%%

\subsection {\normalsize Controlled System Design for Simplified Dynamic Model}\label{sec:chap4control1}
{\addtolength{\leftskip}{0.5in}
Let $x_{0_{d}}\in\Re^{3}$ be the desired position of the payload. The desired attitude of the payload is considered as $R_{0_d}=I_{3\times 3}$, and the desired direction of links is aligned along the vertical direction. The corresponding location of the $i$-th quadrotor at this desired configuration is given by
\begin{align}
x_{i_d}=x_{0_d}+\rho_{i}-\sum_{a=1}^{n_{i}}{l_{ia}e_{3}}.
\end{align}
We wish to design control forces $f_{i}$ and control moments $M_{i}$ of quadrotors such that this desired configuration becomes asymptotically stable.
%%%%%%%%%%%%%%%%%%%%%%%%%%%%%%%%%%%%%%%%%%%%%%%%%%%%%
\subsubsection {\normalsize Simplified Dynamic Model}
Control forces for each quadrotor is given by $-f_{i}R_{i}e_{3}$ for the given equations of motion \refeqn{EOMM1}, \refeqn{EOMM2}, \refeqn{EOMM3}, \refeqn{EOMM4}. As such, the quadrotor dynamics is under-actuated. The total thrust magnitude of each quadrotor can be arbitrary chosen, but the direction of the thrust vector is always along the third body fixed axis, represented by $R_ie_3$. But, the rotational attitude dynamics of the quadrotors are fully actuated, and they are not affected by the translational dynamics of the quadrotors or the dynamics of links. 

Based on these observations, in this section, we simplify the model by replacing the $-f_{i}R_{i}e_{3}$ term by a fictitious control input $u_{i}\in\Re^{3}$, and design an expression for $u$ to asymptotically stabilize the desired equilibrium. In another words, we assume that the attitude of the quadrotor can be instantaneously changed. The effects of the attitude dynamics are studied at the next section.

%%%%%%%%%%%%%%%%%%%%%%%%%%%%%%%%%%%%%%%%%%%%%%%%%%%%%
\subsubsection {\normalsize Linear Control System}
The control system for the simplified dynamic model is developed based on the linearized equations of motion. At the desired equilibrium, the position and the attitude of the payload are given by $x_{0_d}$ and $R_{0}^{*}=I_{3}$, respectively, where the superscript $*$ denotes the value of a variable at the desired equilibrium throughout this paper. Also, we have $q_{ij}^{*}=e_{3}$ and $R_{i}^{*}=I_{3}$. In this equilibrium configuration, the control input for the $i$-th quadrotor is
\begin{align}
u_{i}^{*}=-f_{i}^{*}R_{i}^{*}e_{3},
\end{align}
where the total thrust is $f_{i}^{*}=(M_{iT}+\frac{m_{0}}{n})g$.

The variation of $x_{0}$ is given by
\begin{align}\label{eqn:xlin}
\delta x_{0}=x_{0}-x_{0_{d}},
\end{align}
and the variation of the attitude of the payload is defined as
\begin{align*}
\delta R_0 = R_0^* \hat\eta_0 = \hat\eta_0,
\end{align*}
for $\eta_0\in\Re^3$. The variation of $q_{ij}$ can be written as
\begin{align}\label{eqn:qlin}
\delta q_{ij}=\xi_{ij}\times e_{3},
\end{align}
where $\xi_{ij}\in\Re^{3}$ with $\xi_{ij}\cdot e_{3}=0$. The variation of $\omega_{ij}$ is given by $\delta \omega_{ij}\in\Re^{3}$ with $\delta\omega_{ij}\cdot e_{3}=0$. Therefore, the third element of each of $\xi_{ij}$ and $\delta\omega_{ij}$ for any equilibrium configuration is zero, and they are omitted in the following linearized equations. The state vector of the linearized equation is composed of $C^{T}\xi_{ij}\in\Re^{2}$, where $C=[e_{1},\;e_{2}]\in\Re^{3\times 2}$. The variation of the control input $\delta u_{i}\in\Re^{3\times 1}$, is given as $\delta u_i = u_i-u_i^*$.

\begin{prop}\label{prop:propchap4_2}
The linearized equations of the simplified dynamic model are given by 
\begin{align}\label{eqn:EOMLin}
\Mb\ddot \xb  + \Gb\xb = \Bb \delta u,
\end{align}
where the state vector $\xb\in\Re^{D_{\xb}}$ with $D_{\xb}=6+2\sum_{i=1}^{n}n_{i}$ is given by
\begin{align*}
\xb=\begin{bmatrix}
\delta{x}_{0},\eta_{0},C^{T}{\xi}_{1j},C^T{\xi}_{2j},\cdots,C^{T}{\xi}_{nj}\\
\end{bmatrix},
\end{align*}
and $\delta u=[\delta u_{1}^T,\; \delta u_{2}^T,\;\cdots,\;\delta u_{n}^T]^{T}\in\Re^{3n\times 1}$. The matrix $\Mb\in\Re^{D_{\xb}\times D_{\xb}}$ are defined as
\begin{align*}\Mb=
\begin{bmatrix}
M_{T}I_{3}&\Mb_{x_{0}\Omega_{0}}&\Mb_{x_{0}1}&\Mb_{x_{0}2}&\cdots&\Mb_{x_{0}n}\\
\Mb_{\Omega_{0} x_{0}}&\bar{J}_{0}&\Mb_{\Omega_{0}1}&\Mb_{\Omega_{0}2}&\cdots&\Mb_{\Omega_{0}n}\\
\Mb_{1 x_{0}}&\Mb_{1\Omega_{0}}&\Mb_{qq1}&0&\cdots&0\\
\Mb_{2 x_{0}}&\Mb_{2\Omega_{0}}&0&\Mb_{qq2}&\cdots&0\\
\vdots&\vdots&\vdots&\vdots&\vdots&\vdots\\
\Mb_{n x_{0}}&\Mb_{n\Omega_{0}}&0&0&\cdots&\Mb_{qqn}
\end{bmatrix},
\end{align*}
where the sub-matrices are defined as
\begin{gather}
\Mb_{x_{0}\Omega_{0}}=-\sum_{i=1}^{n}{M_{iT}\hat{\rho}_{i}};\; \Mb_{\Omega_{0} x_{0}}=\Mb_{x_{0}\Omega_{0}}^{T},\nonumber\\
\Mb_{x_{0}i}=[M_{0i1}l_{i1}\hat{e}_{3}C,\; M_{0i2}l_{i2}\hat{e}_{3}C,\; \cdots,\;M_{0in_{i}}l_{in_{i}}\hat{e}_{3}C],\nonumber\\
\Mb_{\Omega_{0}i}=[M_{0i1}l_{i1}\hat{\rho}_{i}C,\; M_{0i2}l_{i2}\hat{\rho}_{i}C,\; \cdots,\; M_{0in_{i}}l_{in_{i}}\hat{\rho}_{i}C],\nonumber\\
\Mb_{ix_{0}}=-[M_{0i1}C^{T}\hat{e}_{3},\; M_{0i2}C^{T}\hat{e}_{3},\; \cdots,\; M_{0in_{i}}C^{T}\hat{e}_{3}],\\
\Mb_{i\Omega_{0}}=[M_{0i1}C^{T}\hat{e}_{3}\hat{\rho}_{i},\; M_{0i2}C^{T}\hat{e}_{3}\hat{\rho}_{i},\;\cdots,\; M_{0in_{i}}C^{T}\hat{e}_{3}\hat{\rho}_{i}],
\end{gather}
and the sub-matrix $\Mb_{qqi}\in\Re^{2n_i\times 2n_i}$ is given by
\begin{align}\Mb_{qqi}=
\begin{bmatrix}
M_{i11}l_{i1}I_{2}&M_{i12}l_{i2}I_{2}&\cdots&M_{i1n_{i}}l_{in_{i}}I_{2}\\
M_{i21}l_{i1}I_{2}&M_{i22}l_{i2}I_{2}&\cdots&M_{i2n_{i}}l_{in_{i}}I_{2}\\
\vdots&\vdots&&\vdots\\
M_{in_{i}1}l_{i1}I_{2}&M_{in_{i}2}l_{i2}I_{2}&\cdots&M_{in_{i}n_{i}}l_{in_{i}}I_{2}
\end{bmatrix}.
\end{align}
The matrix $\Gb\in\Re^{D_{\xb}\times D_{\xb}}$ is defined as
\begin{align*}
\Gb=\begin{bmatrix}
0&0&0&0&0&0\\
0&\Gb_{\Omega_{0}\Omega_{0}}&0&0&0&0\\
0&0&\Gb_{1}&0&0&0\\
0&0&0&\Gb_{2}&0&0\\
\vdots&\vdots&\vdots&\vdots&\vdots&\vdots\\
0&0&0&0&0&\Gb_{n}
\end{bmatrix},
\end{align*}
where $\Gb_{\Omega_{0}\Omega_{0}}=\sum_{i=1}^{n}\frac{m_{0}}{n}g\hat{\rho}_{i}\hat{e}_{3}$ and the sub-matrices $\Gb_{i}\in\Re^{2n_{i}\times 2n_{i}}$ are
\begin{align*}
\Gb_{i}= \diag[(-M_{iT}-\frac{m_{0}}{n}+M_{0ij})ge_{3}I_{2}].
\end{align*}
The matrix $\Bb\in\Re^{D_{\xb}\times 3n}$ is given by 
\begin{align*}
\Bb=\begin{bmatrix}
I_{3}&I_{3}&\cdots&I_{3}\\
\hat{\rho}_{1}&\hat{\rho}_{2}&\cdots&\hat{\rho}_{n}\\
\Bb_{\Bb}&0&0&0\\
0&\Bb_{\Bb}&0&0\\
\vdots&\vdots&\vdots&\vdots\\
0&0&0&\Bb_{\Bb}
\end{bmatrix},
\end{align*}
where $\Bb_{\Bb}=-[C^{T}\hat{e}_{3},\; C^{T}\hat{e}_{3},\; \cdots,\; C^{T}\hat{e}_{3}]^{T}$.
\end{prop}
\begin{proof}
See Appendix~\ref{sec:pfchap4_2}
\end{proof}
We present the following PD-type control system for the linearized dynamics
\begin{align}
\delta u_{i}%=&-k_{x_{i}}\delta x_{0}-k_{\dot{x}_{i}}\delta \dot{x}_{0}-k_{\eta_{{0}_{i}}}\eta_{0}-k_{\Omega_{{0}_{i}}}\delta\Omega_{0}\nonumber\\
%&-\sum_{j=1}^{n_{i}}k_{q_{ij}}C^{T}(e_{3}\times q_{ij})-k_{\omega_{ij}}C^{T}(\delta \omega_{ij})\nonumber\\
=&-K_{x_{i}}\xb-K_{\dot{x}_{i}}\dot{\xb},
\end{align}
for controller gains $K_{x_{i}},K_{\dot{x}_{i}}\in\Re^{3\times D_{\xb}}$. Provided that \refeqn{EOMLin} is controllable, we can choose the combined controller gains $K_{x}=[K_{x_{1}}^T,\,\ldots\;K_{x_{n}}^T]^{T}$, $K_{\dot{x}}=[K_{\dot{x}_{1}}^T,\,\ldots K_{\dot{x}_{n}}^T]^{T}\in\Re^{3n\times D_{\xb}}$ such that the equilibrium is asymptotically stable for the linearized equation \refeqn{EOMLin}. 

}

%%%%%%%%%%%%%%%%%%%%%%%%%%%%%%%%%%%%%%%%%%%%%%%%%%%%%

\subsection {\normalsize Controlled System Design for Full Dynamic Model}\label{sec:chap4control2}
{\addtolength{\leftskip}{0.5in}

The control system designed at the previous section is based on a simplifying assumption that each quadrotor can generates a thrust along any direction. In the full dynamic model, the direction of the thrust for each quadrotor is parallel to its third body-fixed axis always. In this section, the attitude of each quadrotor is controlled such that the third body-fixed axis becomes parallel to the direction of the ideal control force designed in the previous section. The central idea is that the attitude $R_{i}$ of the quadrotor is controlled such that its total thrust direction $-R_{i}e_{3}$, corresponding to the third body-fixed axis, asymptotically follows the direction of the fictitious control input $u_{i}$. By choosing the total thrust magnitude properly, we can guarantee asymptotical stability for the full dynamic model. 

Let $A_{i}\in\Re^{3}$ be the ideal total thrust of the $i$-th quadrotor that asymptotically stabilize the desired equilibrium. Therefor, we have
\begin{align}\label{eqn:Ai}
A_{i}=u_{i}^{*}+\delta u_{i}=-K_{x_{i}}\xb-K_{\dot{x}_{i}}\dot{\xb}+u_{i}^{*},
\end{align} 
where $f_{i}^{*}$ and $u_{i}^{*}$ are the total thrust and control input of each quadrotor at its equilibrium respectively. 

From the desired direction of the third body-fixed axis of the $i$-th quadrotor, namely $b_{3_{i}}\in\Sph^2$, is given by
\begin{align}
b_{3_{i}}=-\frac{A_{i}}{\|A_{i}\|}.
\end{align}
This provides a two-dimensional constraint on the three dimensional desired attitude of each quadrotor, such that there remains one degree of freedom. To resolve it, the desired direction of the first body-fixed axis $b_{1_{i}}(t)\in\Sph^2$ is introduced as a smooth function of time. Due to the fact that the first body-fixed axis is normal to the third body-fixed axis, it is impossible to follow an arbitrary command $b_{1_{i}}(t)$ exactly. Instead, its projection onto the plane normal to $b_{3_{i}}$ is followed, and the desired direction of the second body-fixed axis is chosen to constitute an orthonormal frame~\cite{LeeLeoAJC13}. More explicitly, the desired attitude of the $i$-th quadrotor is given by
\begin{align}
R_{i_{c}}=\begin{bmatrix}
-\frac{(\hat{b}_{3_{i}})^{2}b_{1_{i}}}{\|(\hat{b}_{3_{i}})^{2}b_{1_{i}}\|} & \frac{\hat{b}_{3_{i}}b_{1_{i}}}{\|\hat{b}_{3_{i}}b_{1_{i}}\|} & b_{3_{i}}\end{bmatrix},
\end{align}
which is guaranteed to be an element of $\so$. The desired angular velocity is obtained from the attitude kinematics equation, $\Omega_{i_{c}}=(R_{i_{c}}^{T}\dot{R}_{i_{c}})^\vee\in\Re^{3}$. Define the tracking error vectors for the attitude and the angular velocity of the $i$-th quadrotor as
\begin{align}
e_{R_{i}}=\frac{1}{2}(R_{i_{c}}^{T}R_{i}-R_{i}^{T}R_{i_{c}})^{\vee},\; e_{\Omega_{i}}=\Omega_{i}-R_{i}^{T}R_{i_{c}}\Omega_{i_{c}},
\end{align}
and a configuration error function on $\SO$ as follows
\begin{align}
\Psi_{i}= \frac{1}{2}\trs[I- R_{{i}_c}^T R_{i}].
\end{align}
The thrust magnitude is chosen as the length of $u_{i}$, projected on to $-R_{i}e_{3}$, and the control moment is chosen as a tracking controller on $\SO$:
\begin{align}
f_{i}=&-A_{i}\cdot R_{i}e_{3},\label{eqn:fi}\\
M_{i}=&-k_R e_{R_{i}} -k_\Omega e_{\Omega_{i}}+(R_{i}^TR_{c_i}\Omega_{c_{i}})^\wedge J_{i} R_{i}^T R_{c_i} \Omega_{c_i} + J_{i} R_{i}^T R_{c_i}\dot\Omega_{c_i},\label{eqn:Mi}
\end{align}
where $k_{R}$ and $k_{\Omega}$ are positive constants. Stability of the corresponding controlled systems for the full dynamic model can be studied by showing the the error due to the discrepancy between the desired direction $b_{3_{i}}$ and the actual direction $R_{i}e_{3}$. This stability is shown via a Lyapunov analysis.

\begin{prop}\label{prop:propchap4_3}
Consider the full dynamic model defined by \refeqn{EOMM1}, \refeqn{EOMM2}, \refeqn{EOMM3}, \refeqn{EOMM4}. For the command $x_{0_{d}}$ and the desired direction of the first body-fixed axis $b_{1_{i}}$, control inputs for quadrotors are designed as \refeqn{fi} and \refeqn{Mi}. Then, the equilibrium of zero tracking errors for $e_{x_{0}},\; \dot{e}_{x_{0}},\; e_{R_0},\;e_{\Omega_{0}},\; e_{q_{ij}},\; e_{\omega_{ij}},\; e_{R_i},\;e_{\Omega_{i}}$, is exponentially stable.
\end{prop}
\begin{proof}
See Appendix \ref{sec:pfchap4_3}
\end{proof}
}

%%%%%%%%%%%%%%%%%%%%%%%%%%%%%%%%%%%%%%%%%%%%%%%%%%%%%
\subsection {\normalsize Numerical Example}\label{sec:chap4numerical}
{\addtolength{\leftskip}{0.5in}
We demonstrate the desirable properties of the proposed control system with numerical examples. Two cases are presented. At the first case, a payload is transported to a desired position from the ground. The second case considers stabilization of a payload with large initial attitude errors.
%%%%%%%%%%%%%%%%%%%%%%%%%%%%%%%%%%%%%%%%%%%%%%%%%%%%%
\subsubsection {\normalsize Stabilization of the Rigid Body}
Consider four quadrotors $(n=4)$ connected via flexible cables  to a rigid body payload. Initial conditions are chosen as
\begin{gather*}
x_{0}(0)=[1.0,\; 4.8,\; 0.0]^{T}\,\mathrm{m},\; v_{0}(0)=0_{3\times 1},\\
q_{ij}(0)=e_{3},\; \omega_{ij}(0)=0_{3\times 1},\; R_{i}(0)=I_{3\times 3},\; \Omega_{i}(0)=0_{3\times 1}\\
R_{0}(0)=I_{3\times3},\; \Omega_{0}=0_{3\times 1}.
\end{gather*}
The desired position of the payload is chosen as
\begin{align}
x_{0_{d}}(t)=[0.44,\; 0.78,\; -0.5]^{T}\,\mathrm{m}.
\end{align}
The mass properties of quadrotors are chosen as
\begin{gather}
m_{i}=0.755\,\mathrm{kg},\nonumber\\ 
J_{i}=\diag[0.557,\; 0.557,\; 1.05]\times 10^{-2} \mathrm{kgm^2}.
\end{gather}
The payload is a box with mass $m_{0}=0.5\,\mathrm{kg}$, and its length, width, and height are $0.6$, $0.8$, and $0.2\,\mathrm{m}$, respectively. Each cable connecting the rigid body to the $i$-th quadrotor is considered to be $n_{i}=5$ rigid links. All the links have the same mass of $m_{ij}=0.01\,\mathrm{kg}$ and length of $l_{ij}=0.15\,\mathrm{m}$. Each cable is attached to the following points of the payload
\begin{gather*}
\rho_{1}=[0.3,\; -0.4,\; -0.1]^T\,\mathrm{m},\; \rho_{2}=[0.3,\; 0.4,\; -0.1]^T\,\mathrm{m},\\
\rho_{3}=[-0.3,\; -0.4,\; -0.1]^T\,\mathrm{m},\; \rho_{4}=[-0.3,\; 0.4,\; -0.1]^T\,\mathrm{m}.
\end{gather*}
\begin{figure}
\centerline{
	\subfigure[Payload position ($x_0$:blue, $x_{0_{d}}$:red)]{
		\includegraphics[width=0.4\columnwidth]{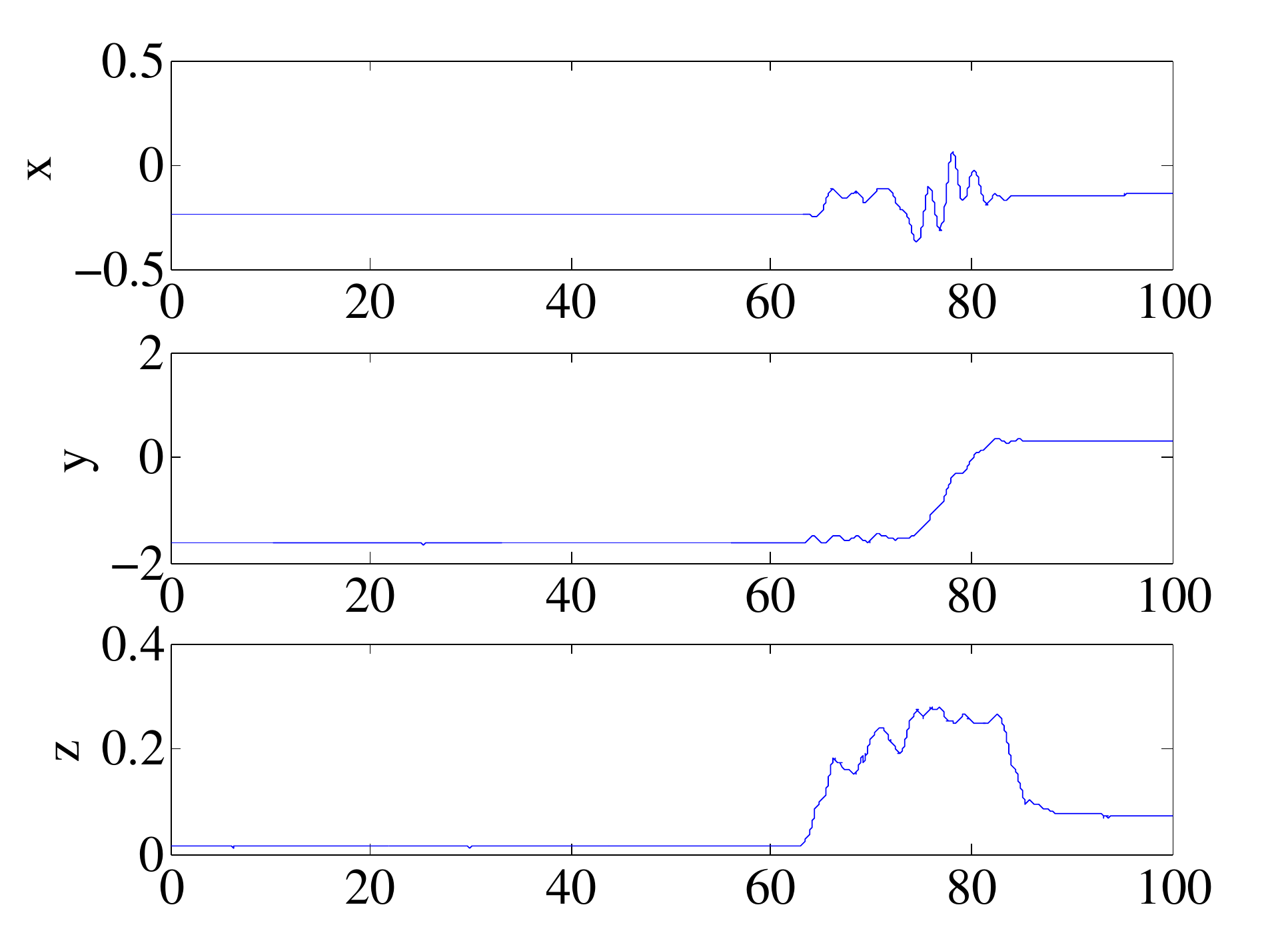}}
	\subfigure[Payload velocity ($v_0$:blue, $v_{0_{d}}$:red)]{
		\includegraphics[width=0.4\columnwidth]{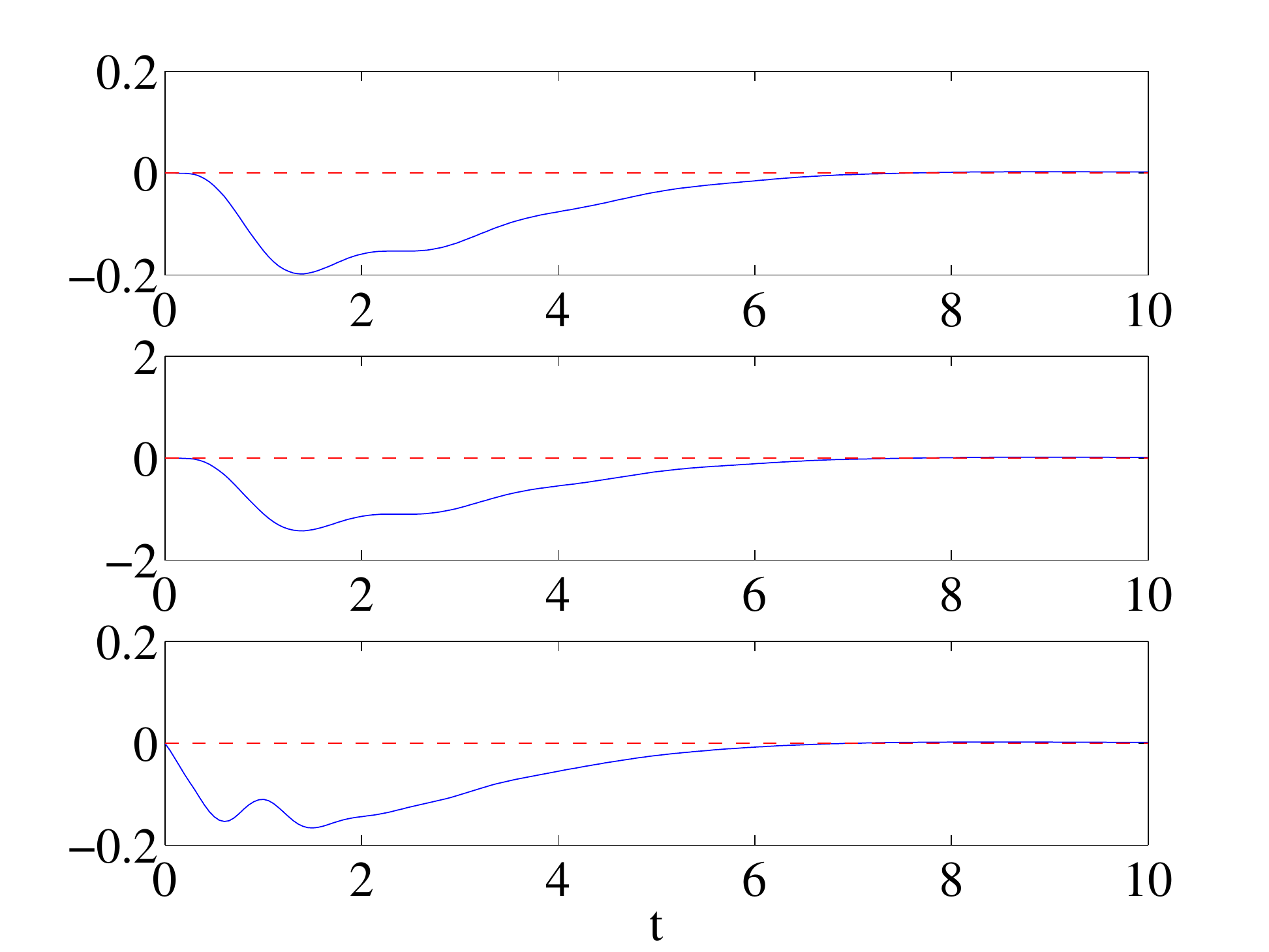}}
}
\centerline{
	\subfigure[Payload angular velocity $\Omega_{0}$]{
		\includegraphics[width=0.4\columnwidth]{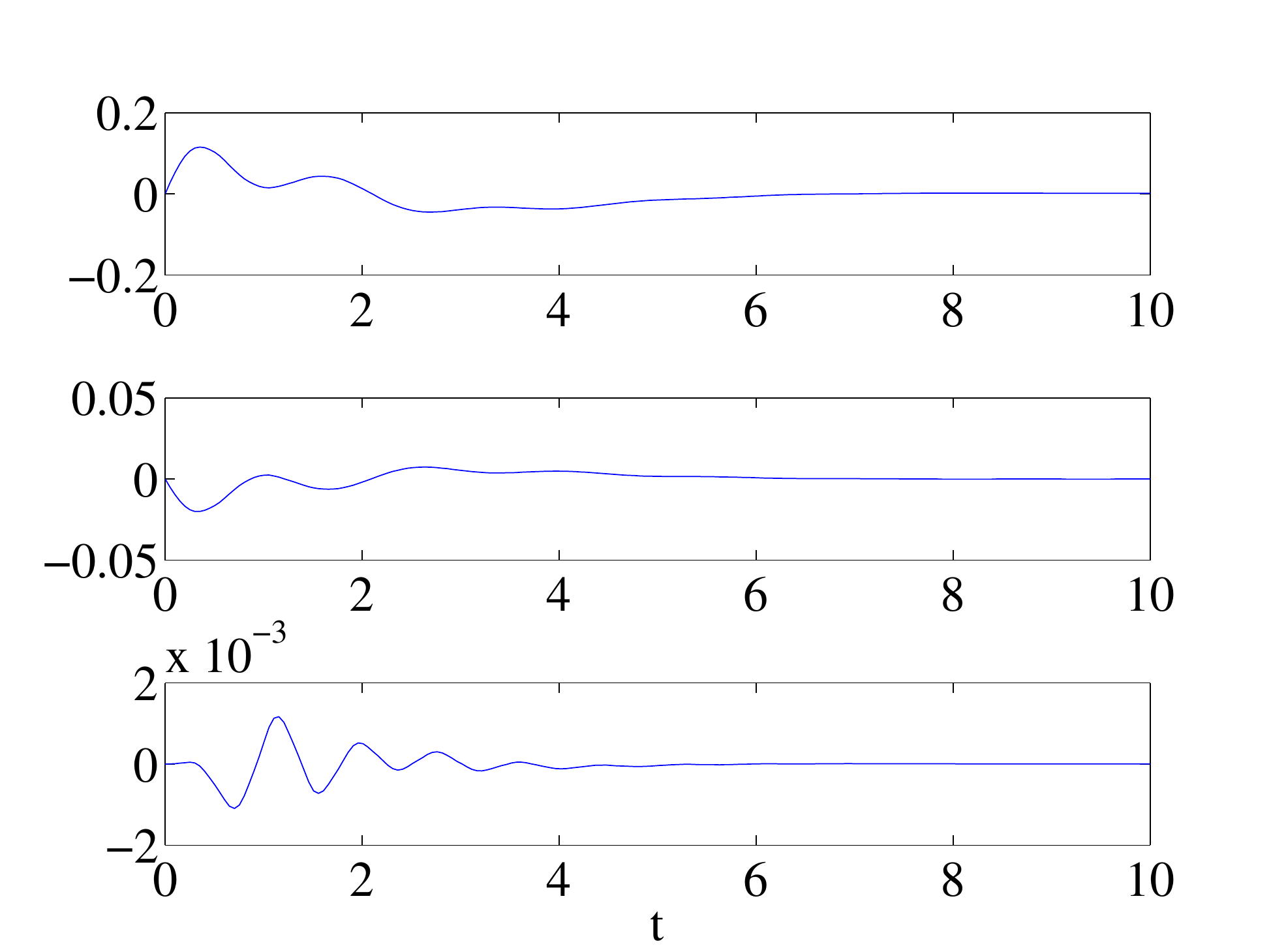}}
	\subfigure[Quadrotors angular velocity errors $e_{\Omega_{i}}$]{
		\includegraphics[width=0.4\columnwidth]{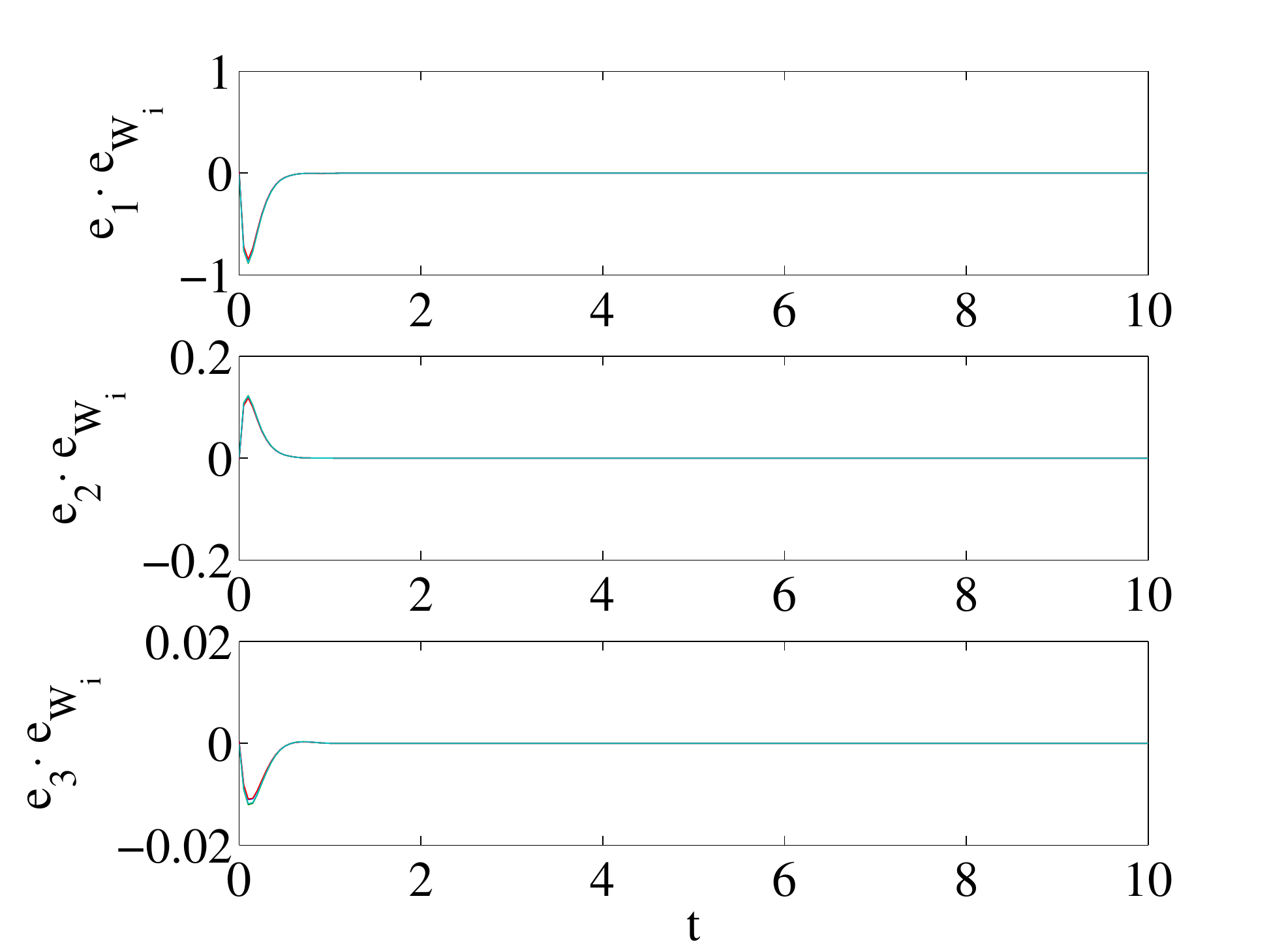}}
}
\centerline{
	\subfigure[Payload attitude error $\Psi_{0}$]{
		\includegraphics[width=0.4\columnwidth]{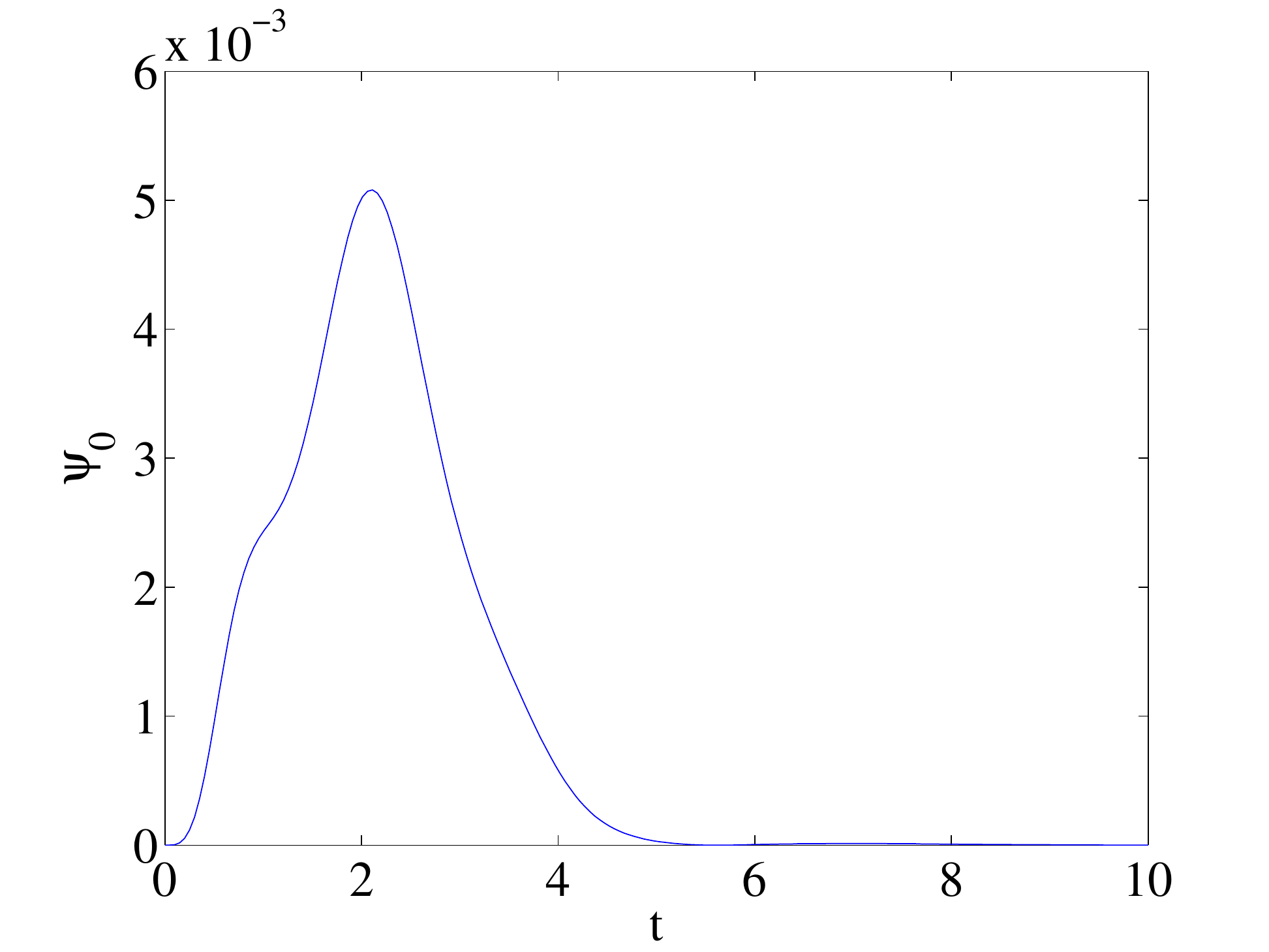}}
	\subfigure[Quadrotors attitude errors $\Psi_{i}$]{
		\includegraphics[width=0.4\columnwidth]{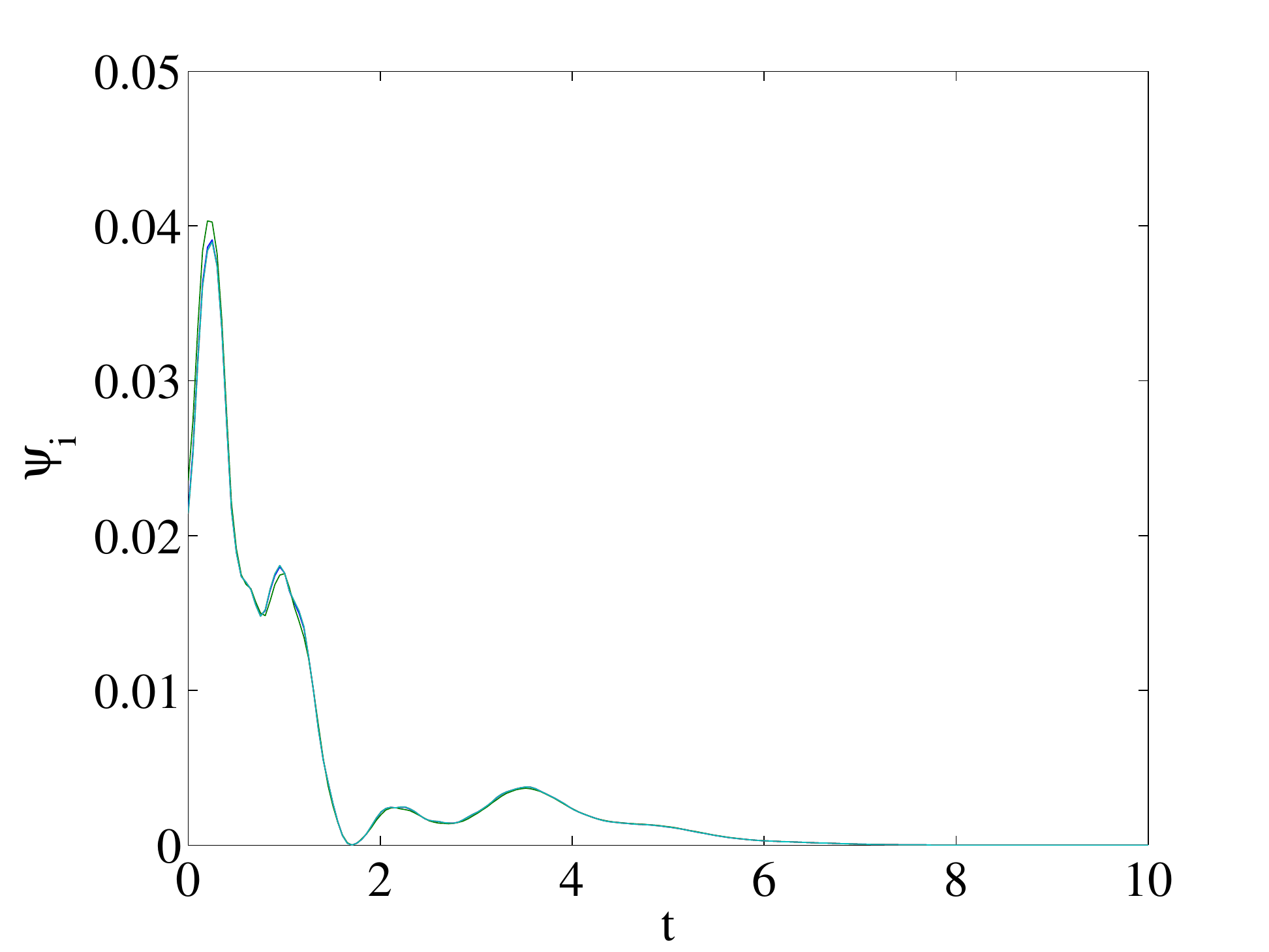}}
}
\centerline{
	\subfigure[Quadrotors total thrust inputs $f_{i}$]{
		\includegraphics[width=0.4\columnwidth]{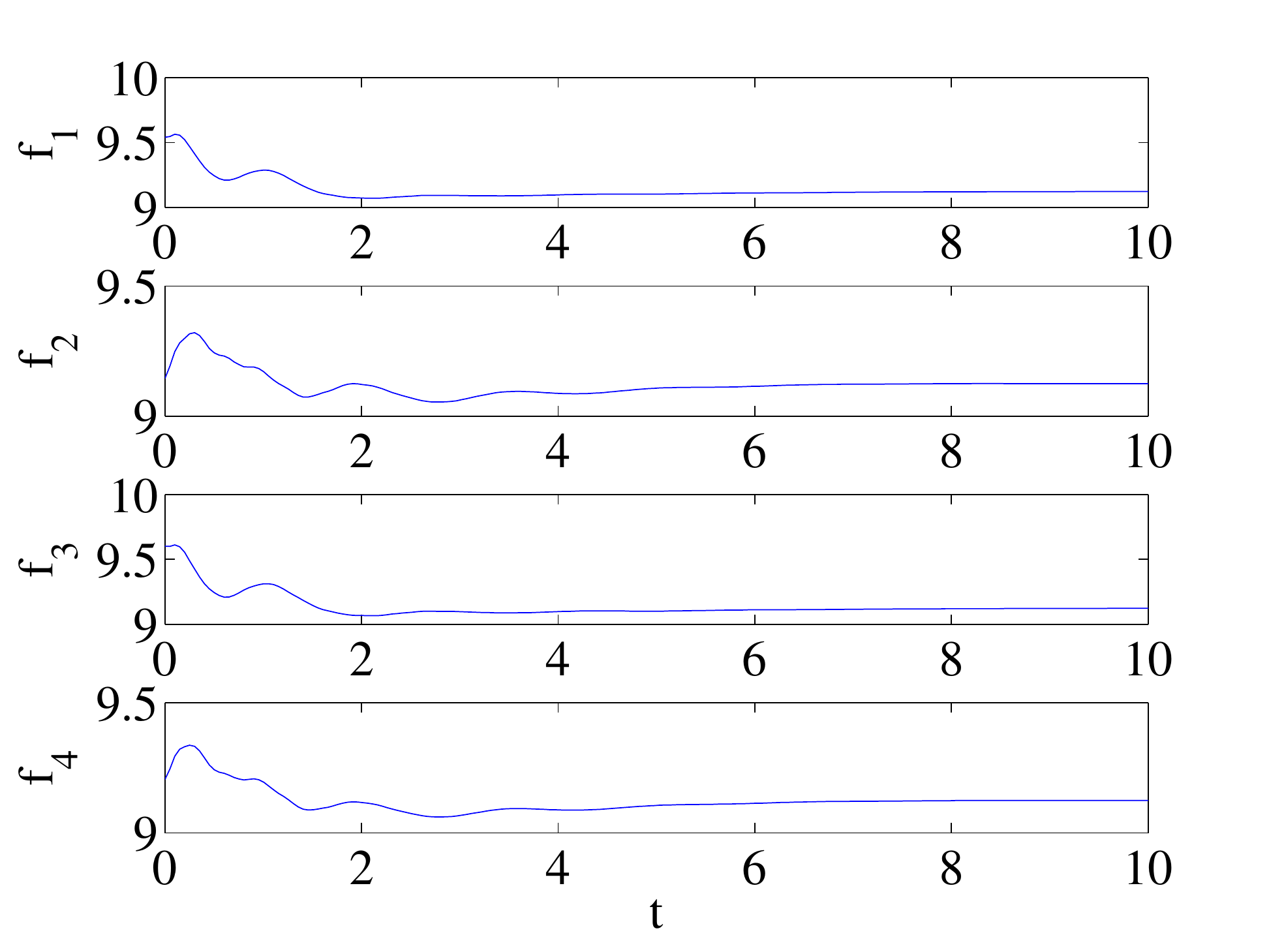}}
	\subfigure[Direction error $e_{q}$, and angular velocity error $e_{\omega}$ for the links]{
		\includegraphics[width=0.4\columnwidth]{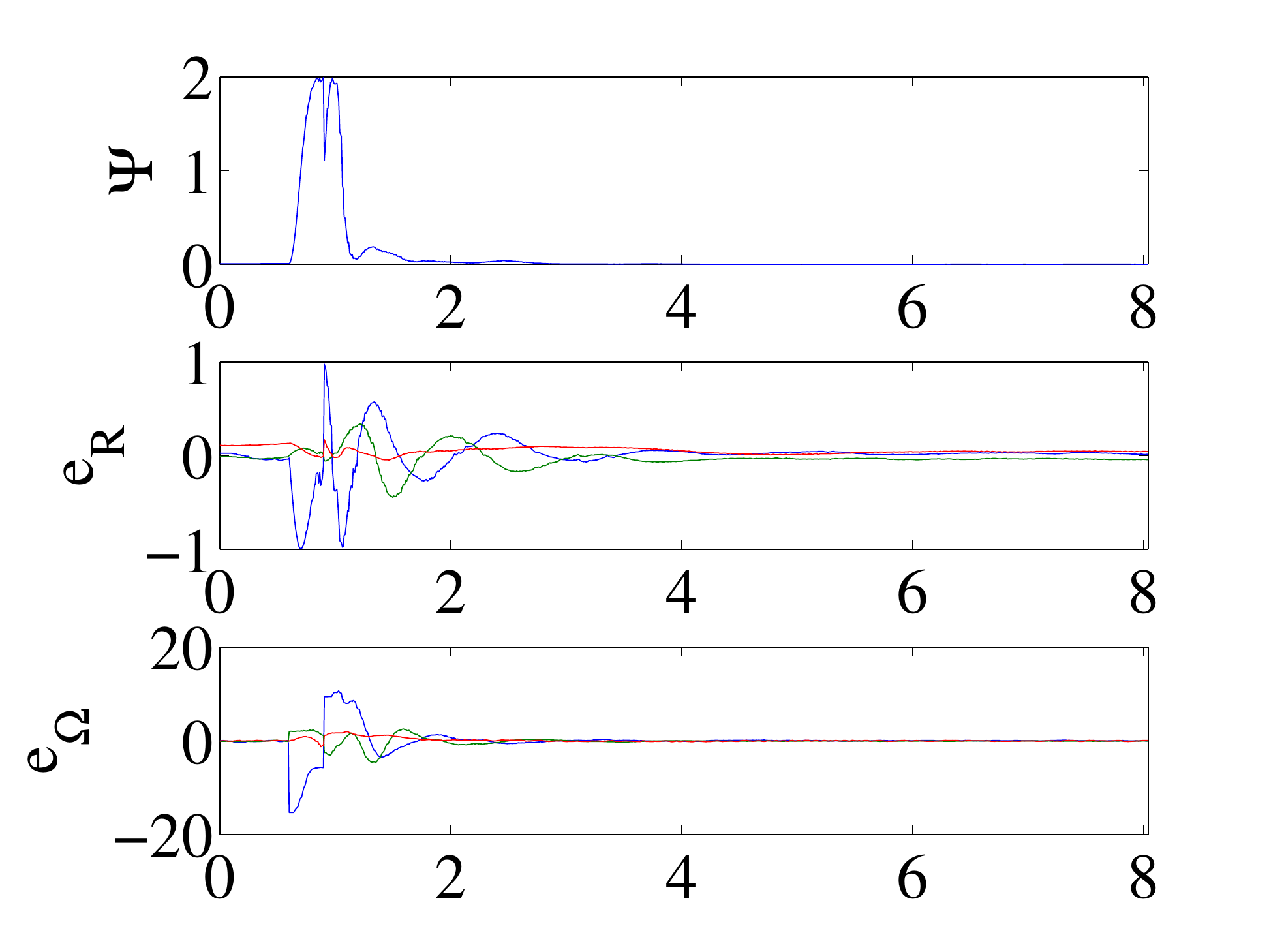}}
}
\caption{Stabilization of a rigid-body connected to multiple quadrotors}\label{fig:simresults1}
\end{figure}
Numerical simulation results are presented at Figure \ref{fig:simresults1}, which shows the position and velocity of the payload, and its tracking errors. We have also presented the link direction error defined as
\begin{align*}
e_{q}=\sum_{i=1}^{m}\sum_{j=1}^{n_{i}}{\|q_{ij}-e_{3}\|}.
\end{align*}
\begin{figure}
\centerline{
	\subfigure[3D perspective]{
		\includegraphics[width=0.8\columnwidth]{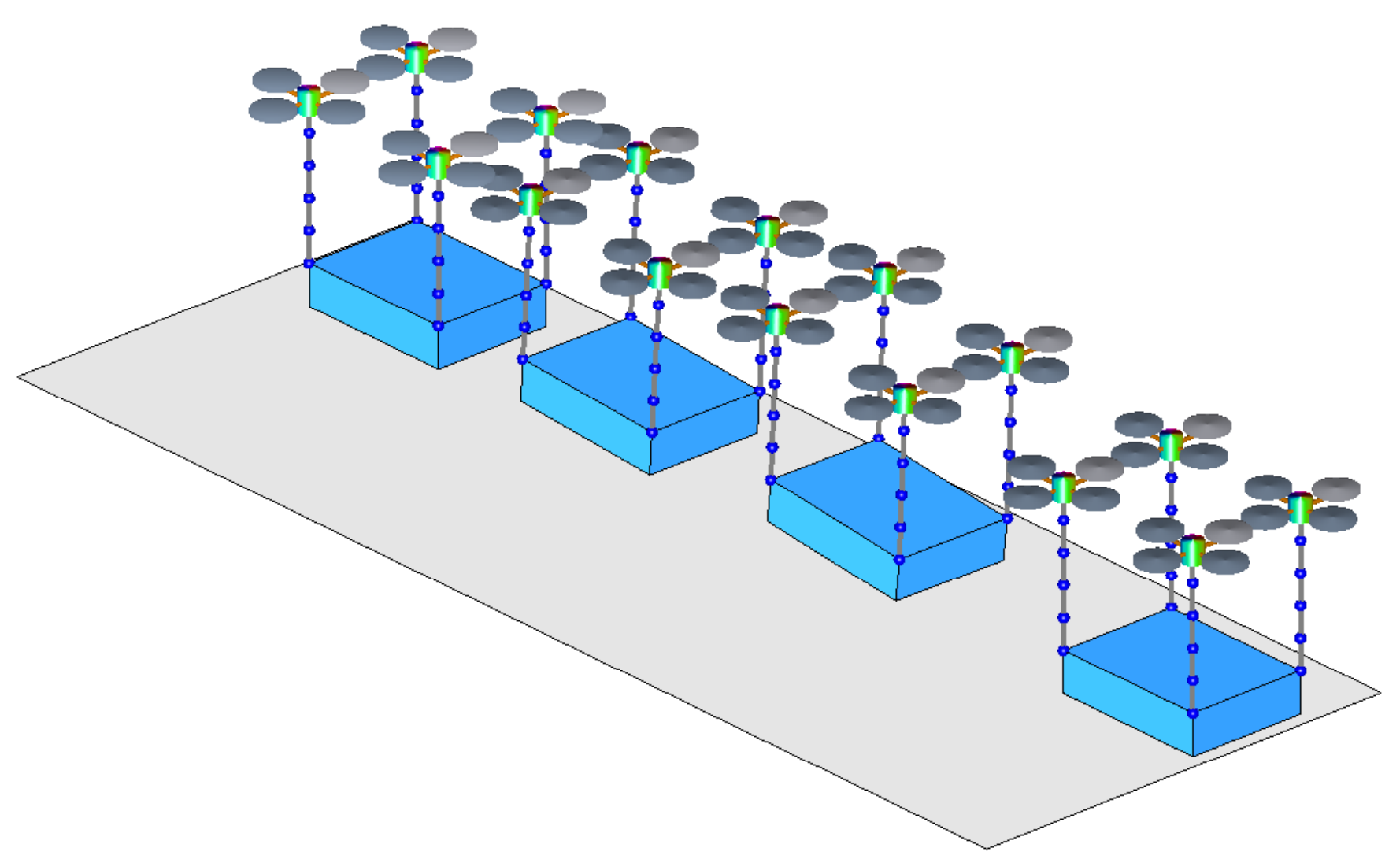}}
}
\centerline{
	\subfigure[Side view]{
		\includegraphics[width=0.8\columnwidth]{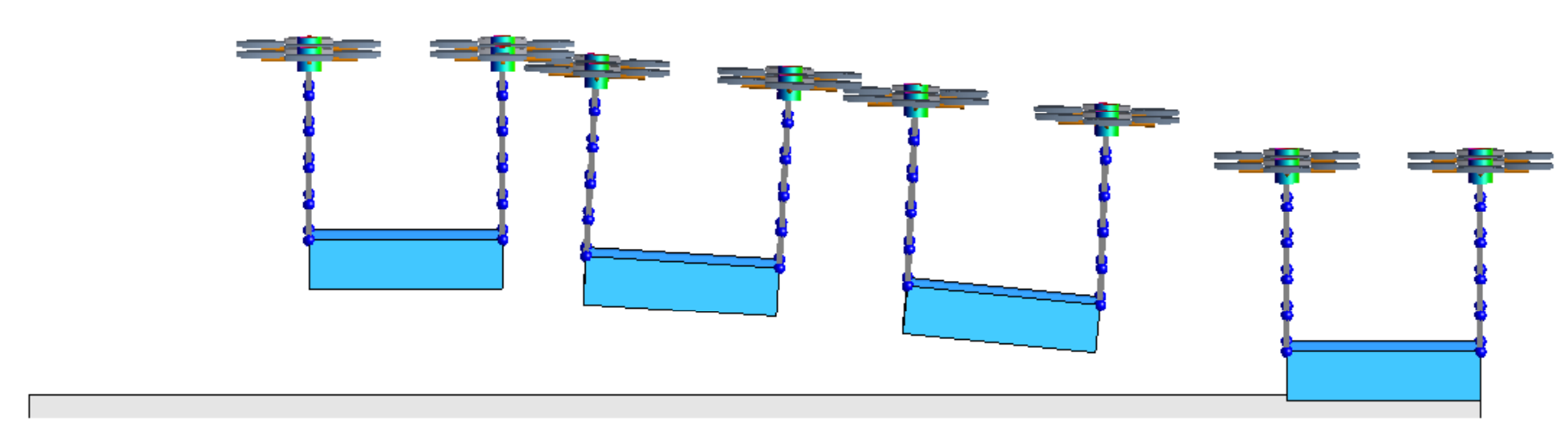}}
}
\centerline{
	\subfigure[Top view]{
		\includegraphics[width=0.8\columnwidth]{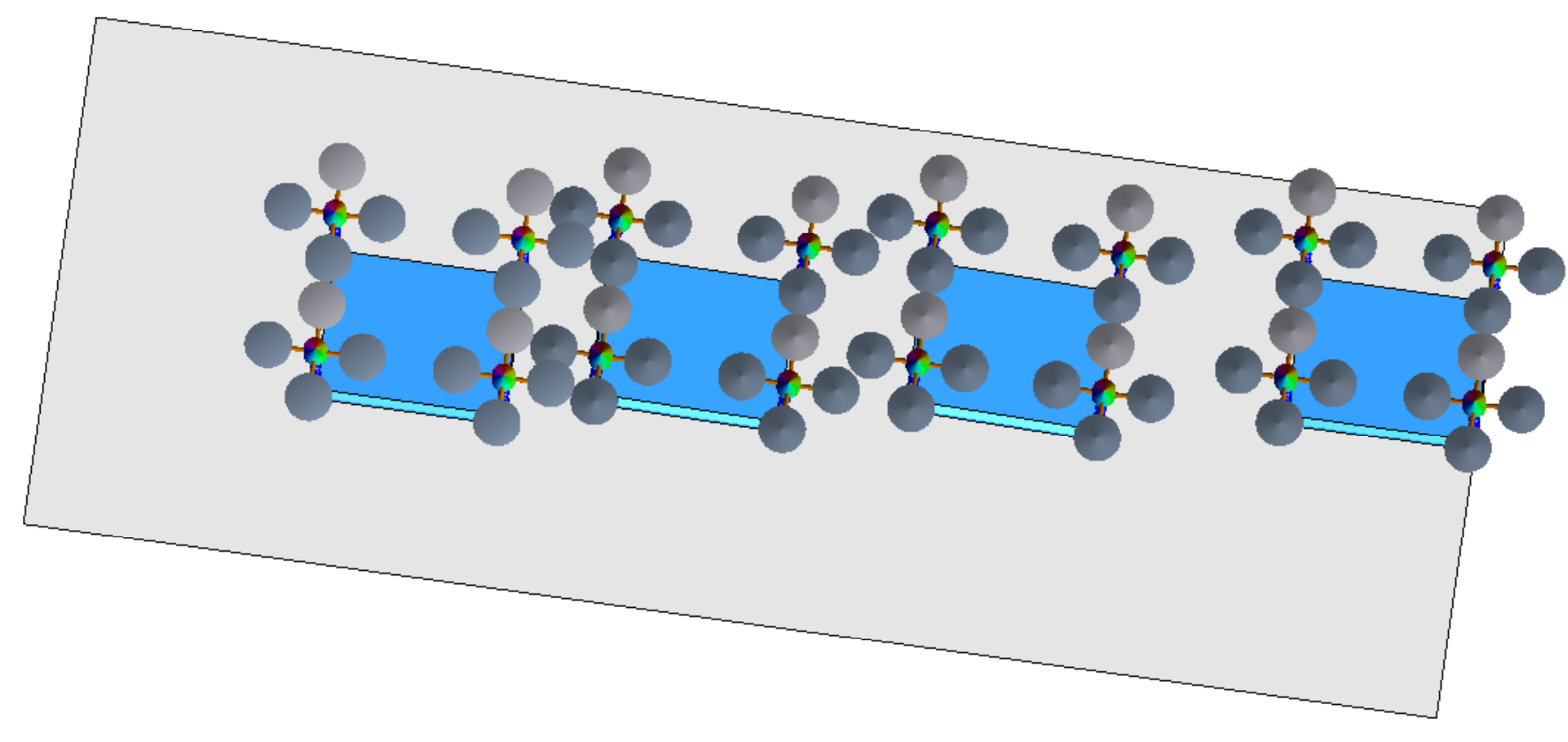}}
}
\caption{Snapshots of controlled maneuver}\label{fig:simresults1snap}
\end{figure}
\newpage
%%%%%%%%%%%%%%%%%%%%%%%%%%%%%%%%%%%%%%%%%%%%%%%%%%%%%
\subsubsection {\normalsize Payload Stabilization with Large Initial Attitude Errors}
In the second case, we consider large initial errors for the attitude of the payload and quadrotors. Initially, the rigid body is tilted in its $b_{1}$ axis by $30$ degrees, and the initial direction of the links are chosen such that two cables are curved along the horizontal direction. The initial conditions are given by
\begin{gather*}
x_{0}(0)=[2.4,\; 0.8,\; -1.0]^{T},\; v_{0}(0)=0_{3\times 1},\\
\omega_{ij}(0)=0_{3\times 1},\; \Omega_{i}(0)=0_{3\times 1}\\
R_{0}(0)=R_{x}(30^\circ),\; \Omega_{0}=0_{3\times 1},
\end{gather*}
where $R_x(30^\circ)$ denotes the rotation about the first axis by $30^\circ$. The initial attitude of quadrotors are chosen as
\begin{gather*}
R_{1}(0)=R_{y}(-35^\circ),\; R_{2}(0)=I_{3\times 3},\\ 
R_{3}(0)=R_{y}(-35^\circ),\; R_{4}(0)=I_{3\times 3}.
\end{gather*}
The properties of quadrotors and cables are identical to the previous case. The payload mass is $m=1.0\,\mathrm{kg}$ , and its length, width, and height are $1.0$, $1.2$, and $0.2\,\mathrm{m}$, respectively. 
\begin{figure}
\centerline{
	\subfigure[Payload position $x_0$:blue, $x_{0_{d}}$:red]{
		\includegraphics[width=0.4\columnwidth]{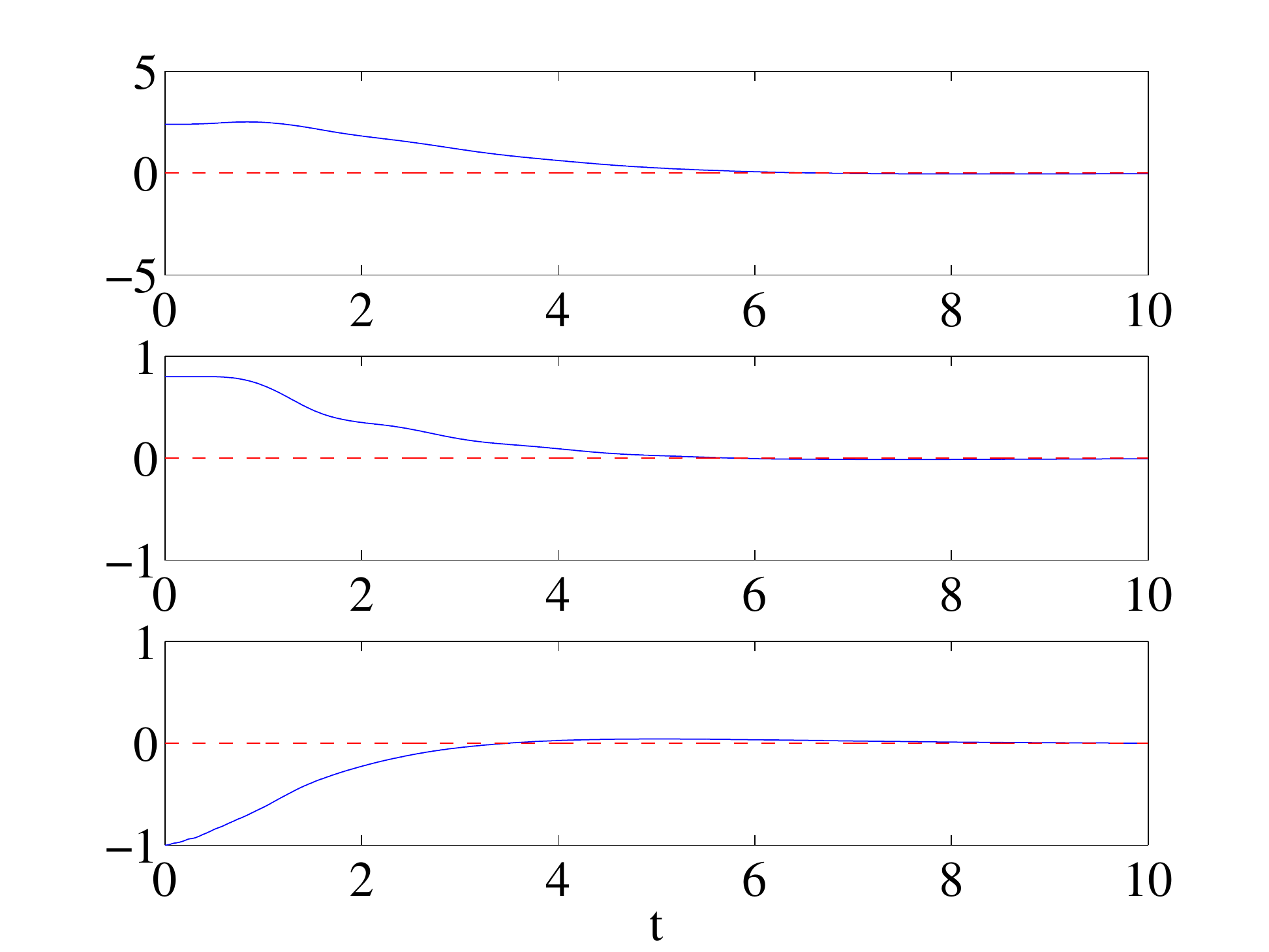}\label{fig:case2x0}}
	\subfigure[Payload velocity $v_0$:blue, $v_{0_{d}}$:red]{
		\includegraphics[width=0.4\columnwidth]{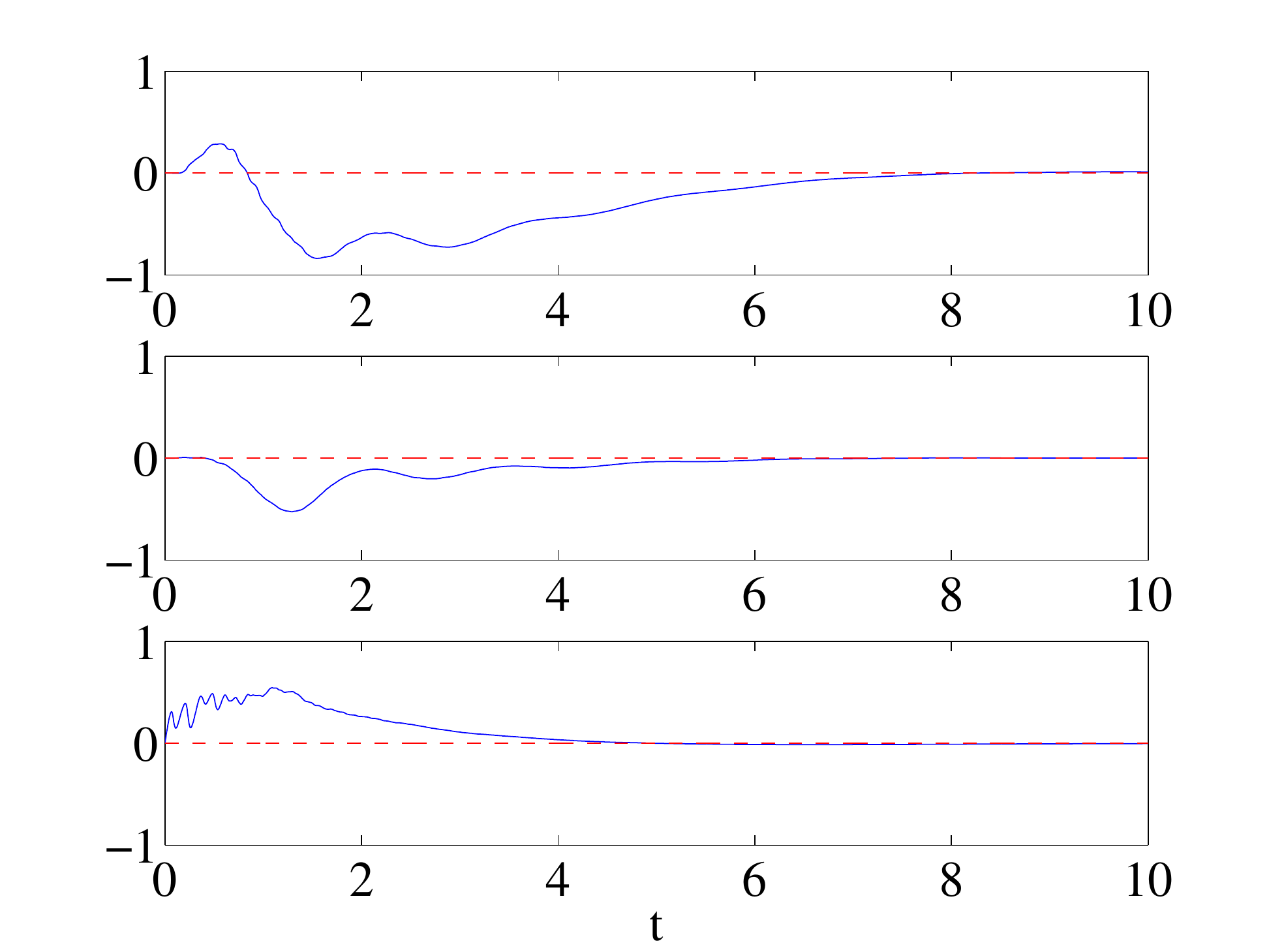}\label{fig:case2v0}}
}
\centerline{
	\subfigure[Payload angular velocity $\Omega_{0}$]{
		\includegraphics[width=0.4\columnwidth]{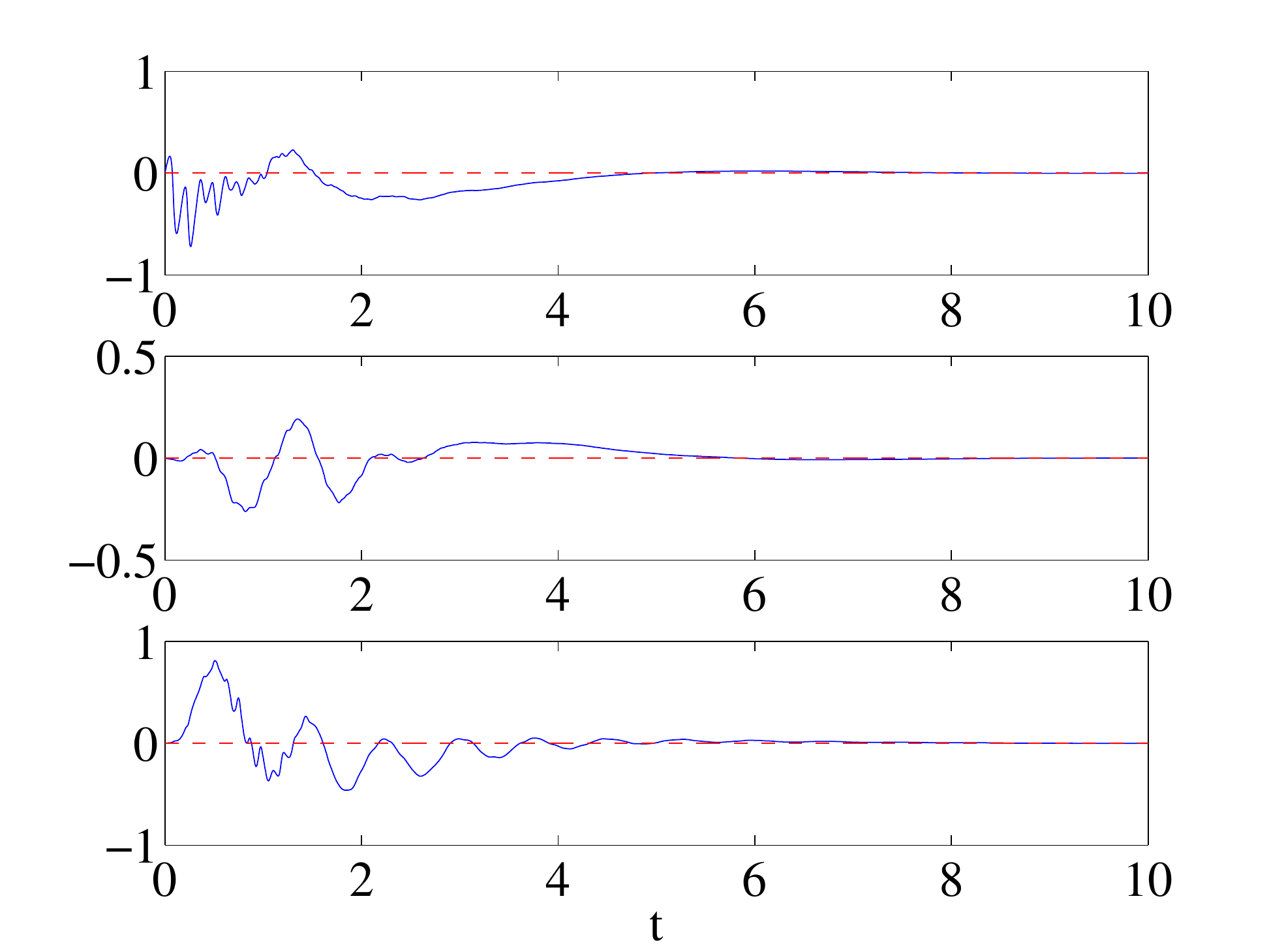}\label{fig:case2W0}}
	\subfigure[Quadrotors angular velocity errors $e_{\Omega_{i}}$]{
		\includegraphics[width=0.4\columnwidth]{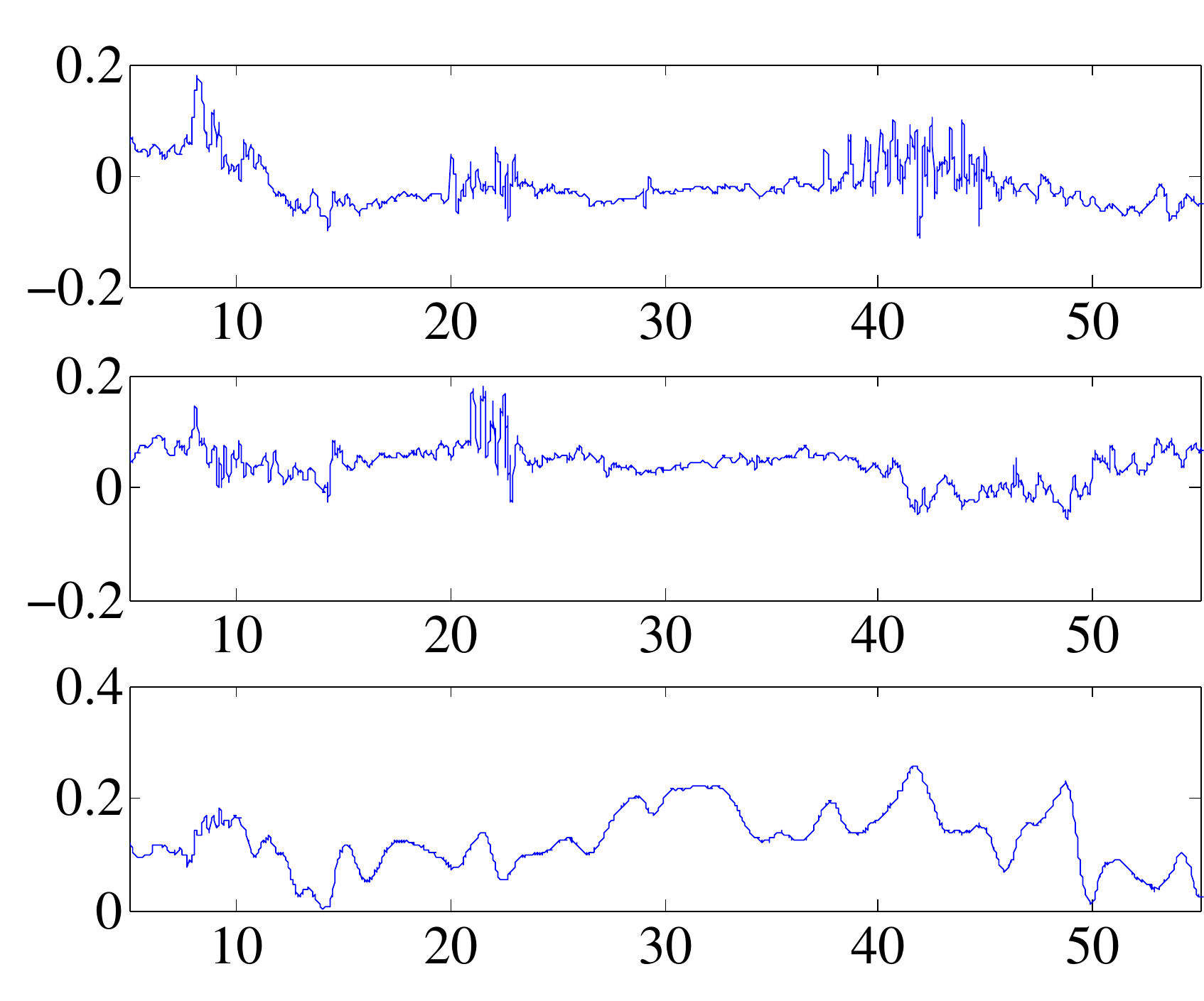}\label{fig:case2eW}}
}
\centerline{
	\subfigure[Payload attitude error $\psi_{0}$]{
		\includegraphics[width=0.4\columnwidth]{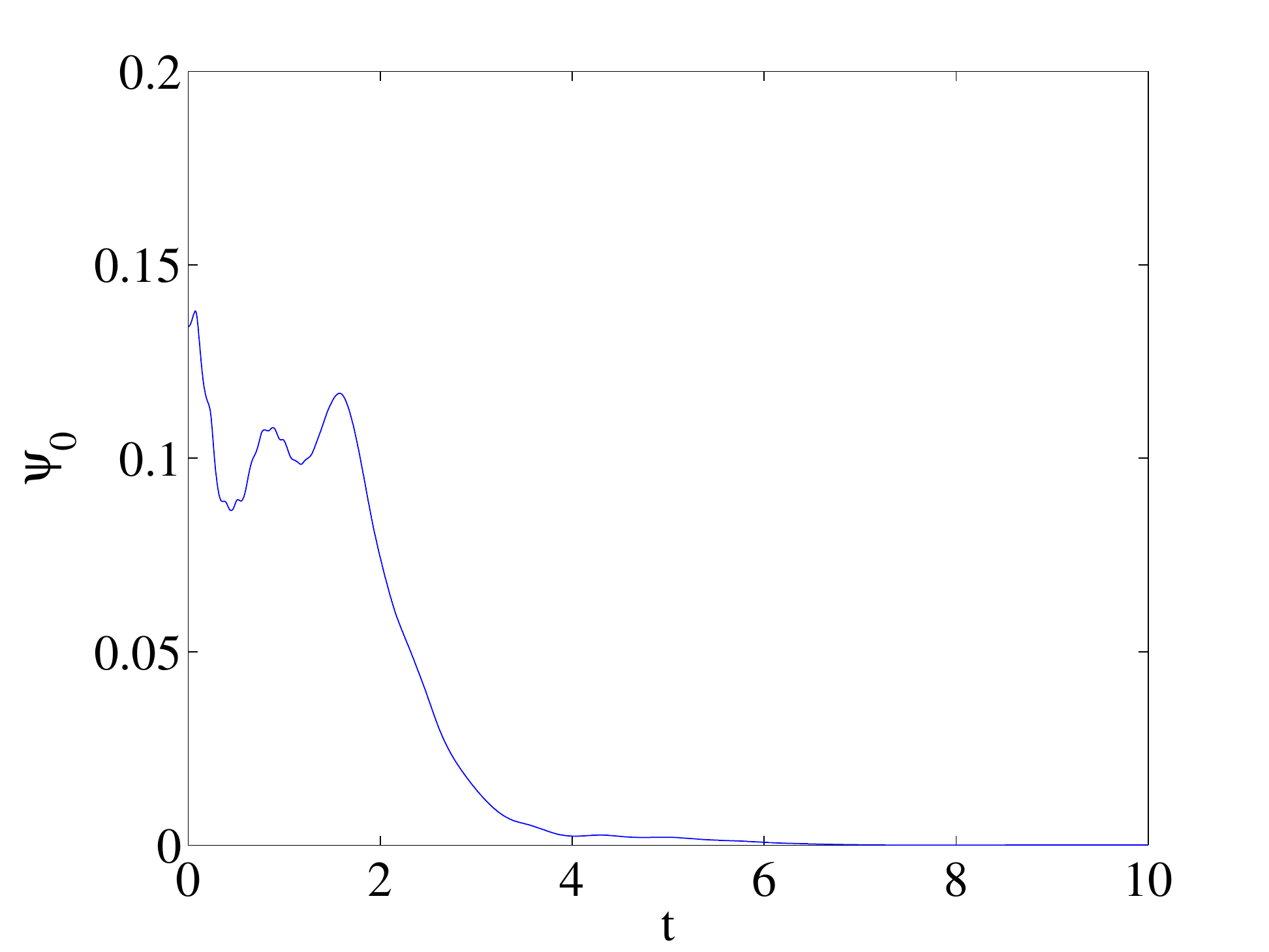}\label{fig:case2psi0}}
	\subfigure[Quadrotors attitude errors $\psi_{i}$]{
		\includegraphics[width=0.4\columnwidth]{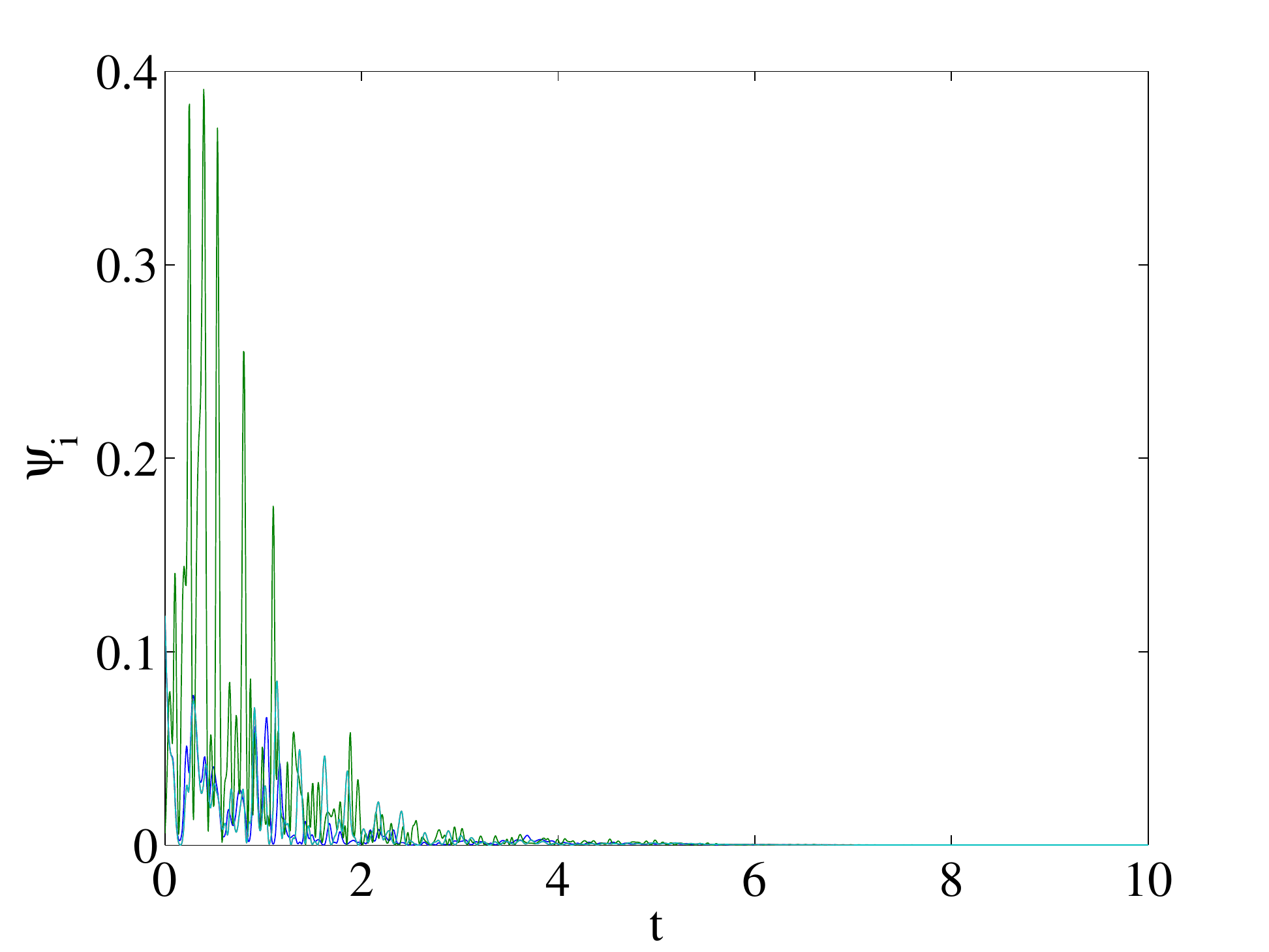}\label{fig:case2psii}}
}
\centerline{
	\subfigure[Quadrotors total thrust inputs $f_{i}$]{
		\includegraphics[width=0.4\columnwidth]{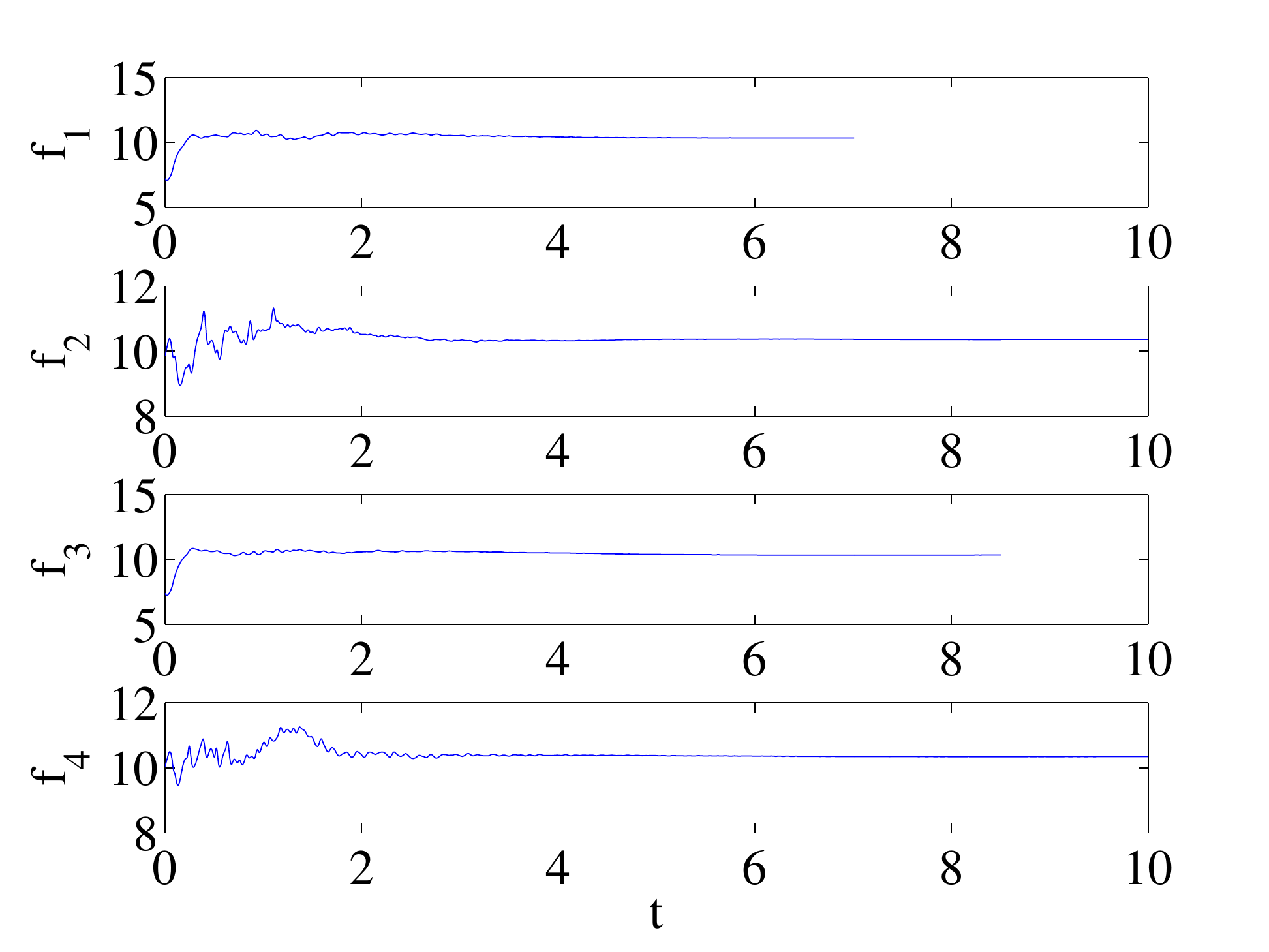}\label{fig:case2ui}}
	\subfigure[Direction error $e_{q}$, and angular velocity error $e_{\omega}$ for the links]{
		\includegraphics[width=0.4\columnwidth]{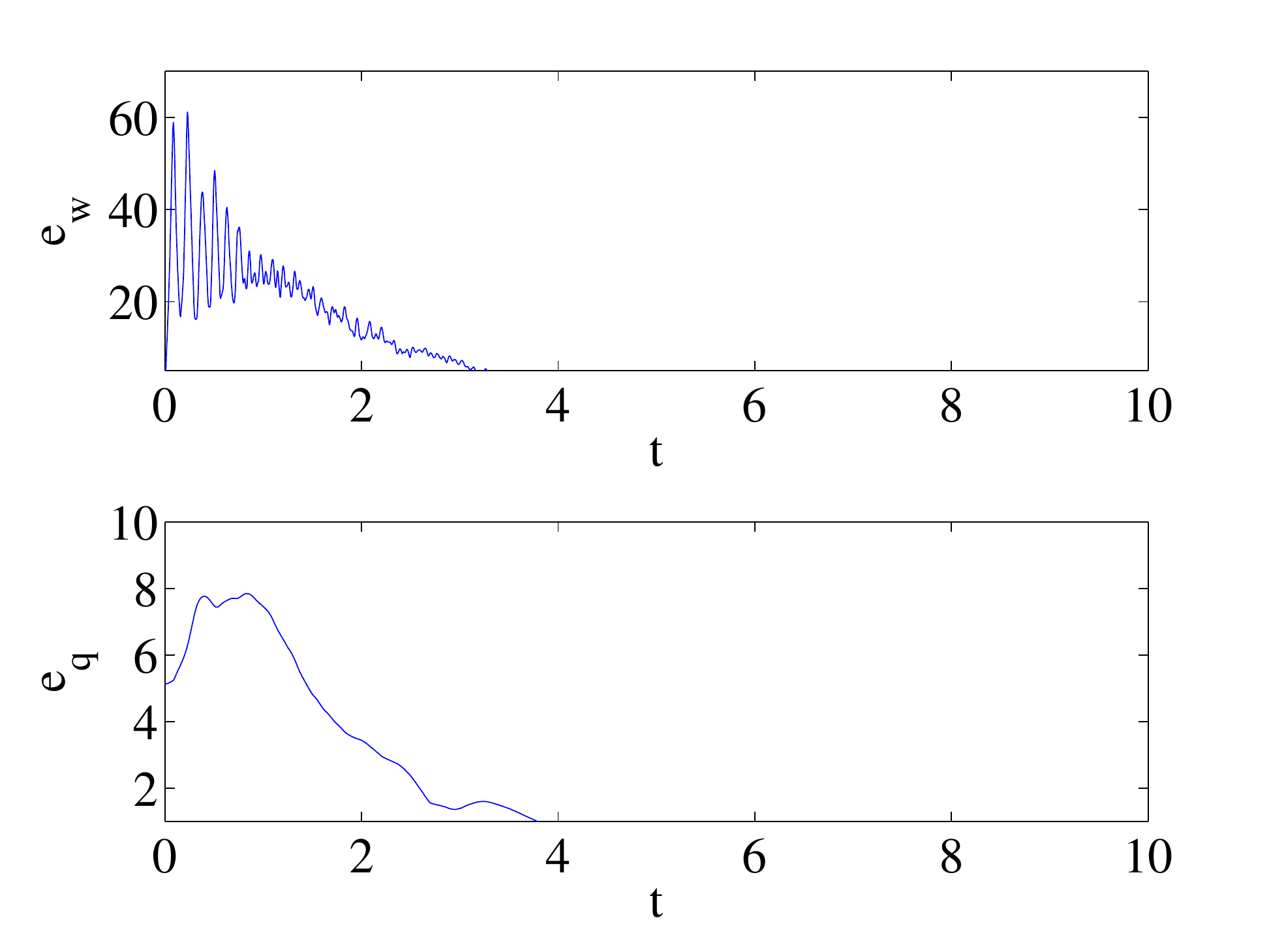}\label{fig:case2errors}}
}
\caption{Stabilization of a payload with multiple quadrotors connected with flexible cables.}\label{fig:simresults2}
\end{figure}
Figure \ref{fig:simresults2} illustrates the tracking errors, and the total thrust of each quadrotor. Snapshots of the controlled maneuvers is also illustrated at Figure \ref{fig:simresults3snap}. It is shown that the proposed controller is able to stabilize the payload and cables at their desired configuration even from the large initial attitude errors\footnote{A short animation of this numerical simulation is available at \url{http://youtu.be/j14tSuHd8oA}.}.
\begin{figure}
\centerline{
	\subfigure[$t=0$ Sec.]{
		\includegraphics[width=0.25\columnwidth]{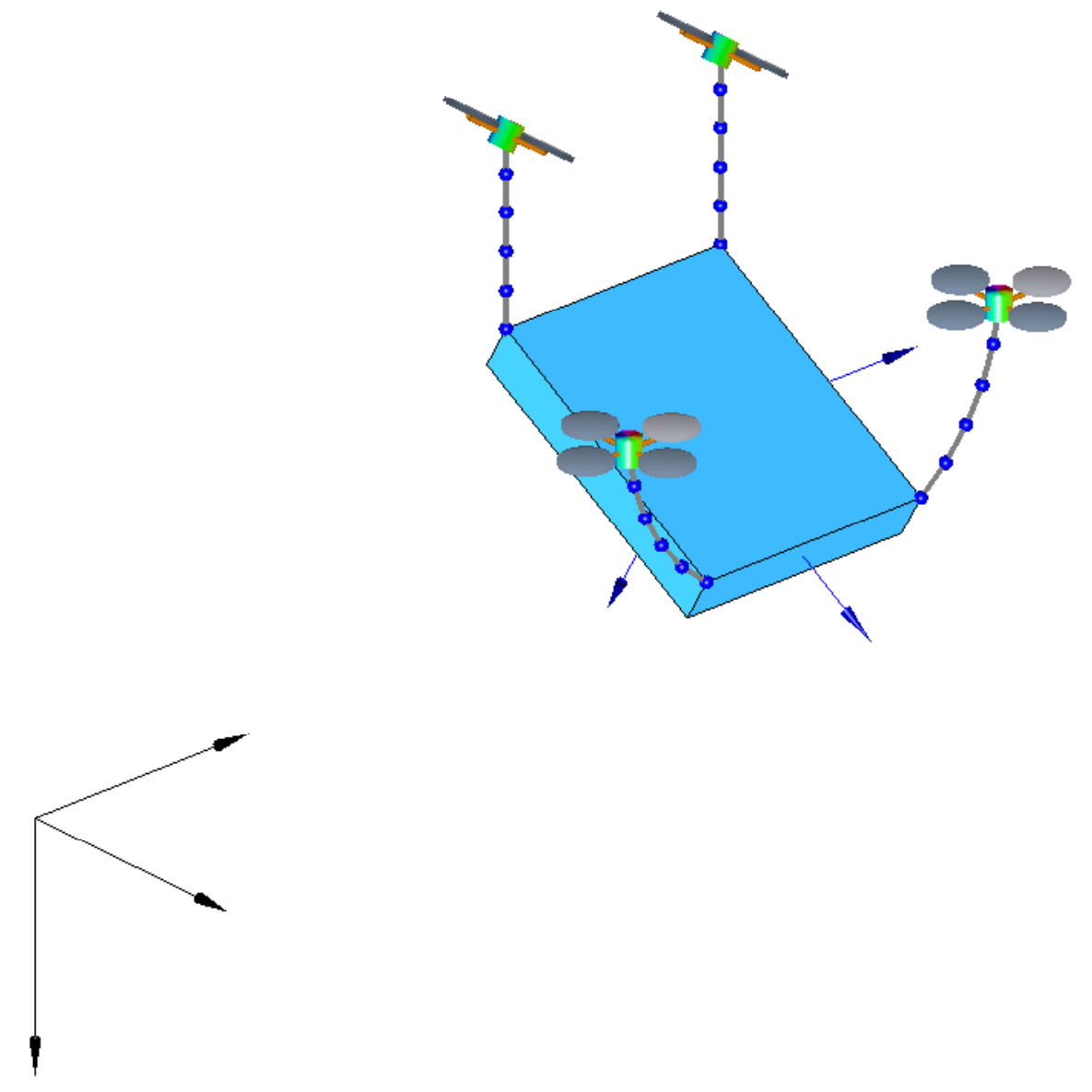}}
	\subfigure[$t=0.14$ Sec.]{
		\includegraphics[width=0.25\columnwidth]{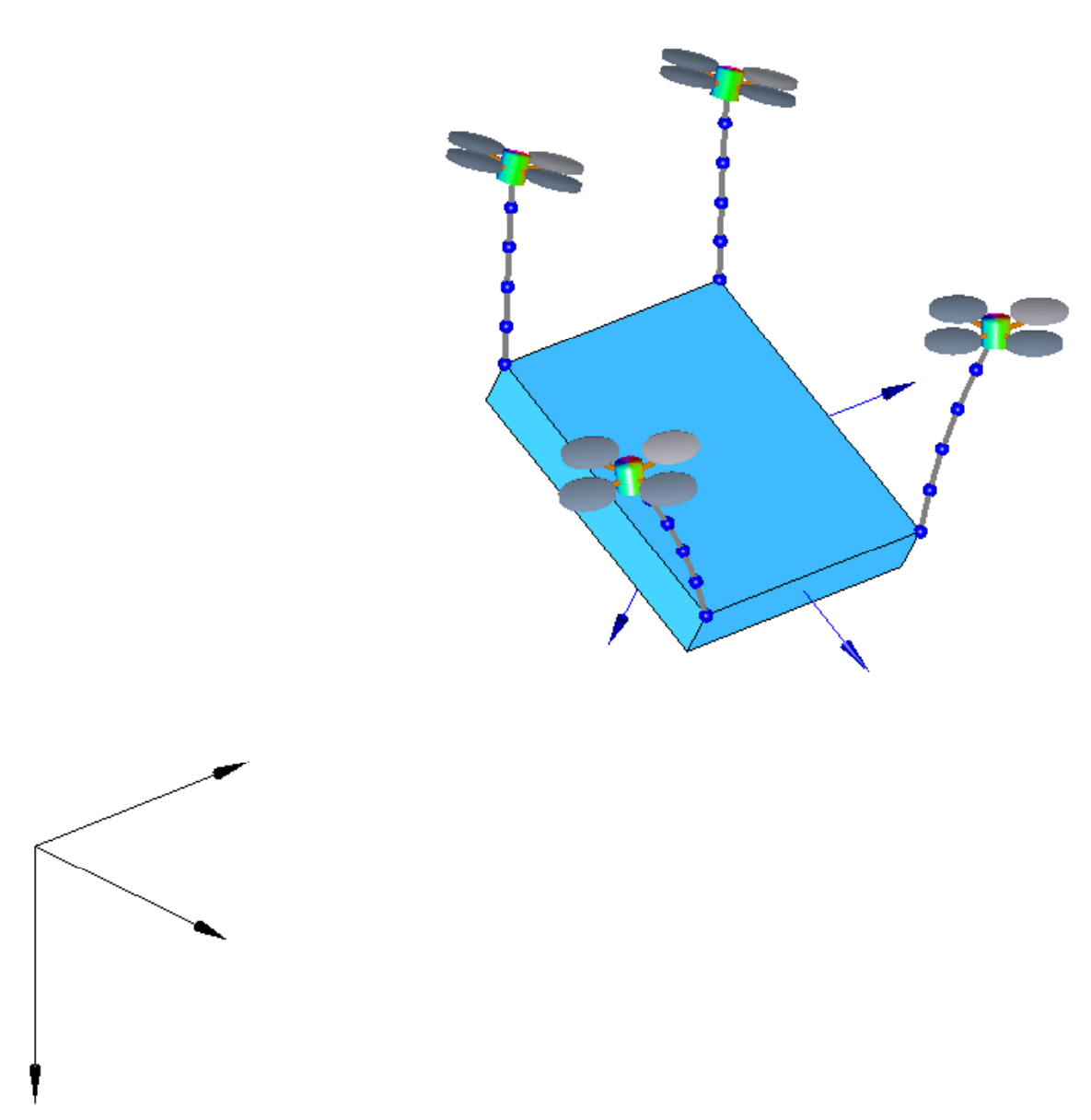}}
		\subfigure[$t=0.30$ Sec.]{
		\includegraphics[width=0.25\columnwidth]{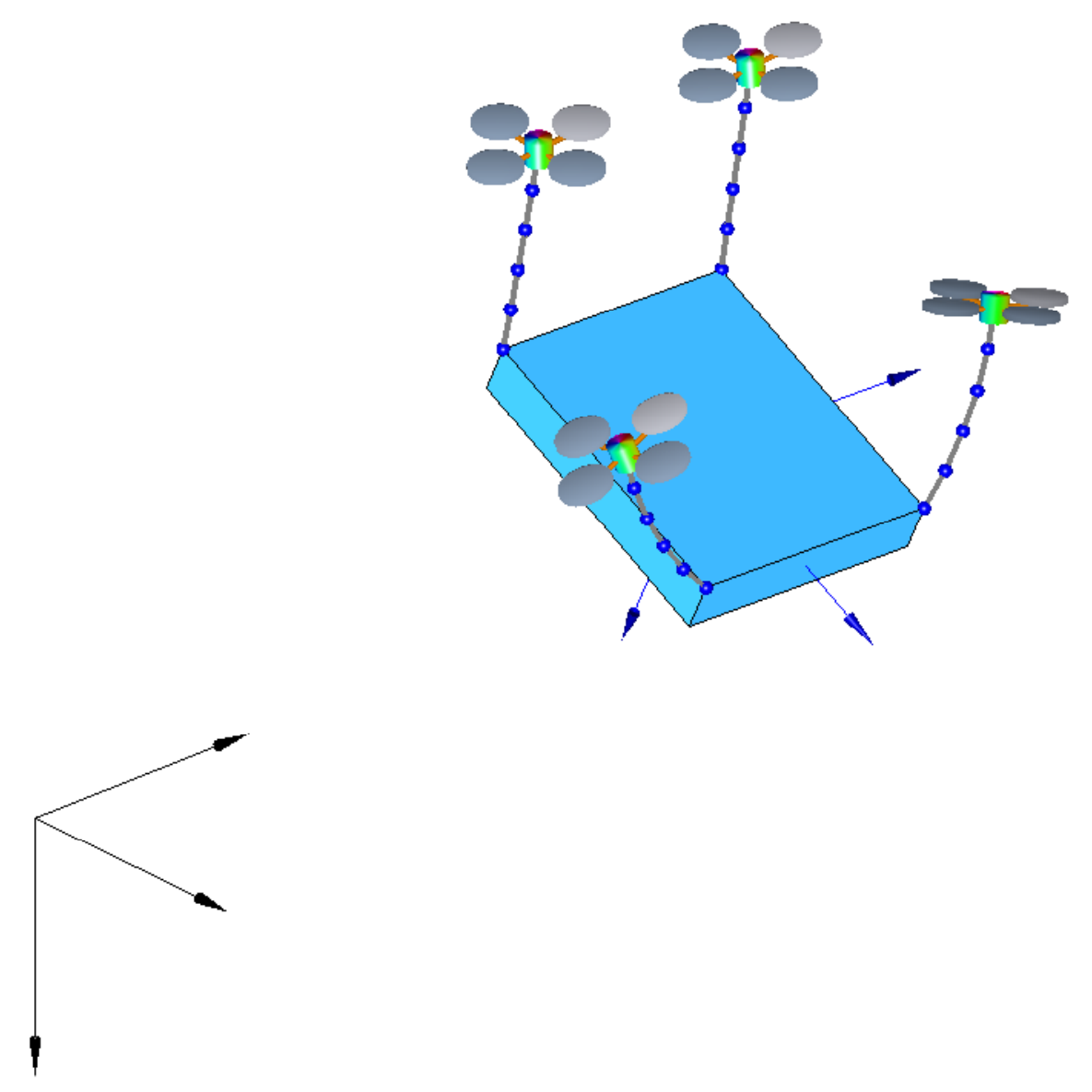}}
}
\centerline{
	\subfigure[$t=0.68$ Sec.]{
		\includegraphics[width=0.25\columnwidth]{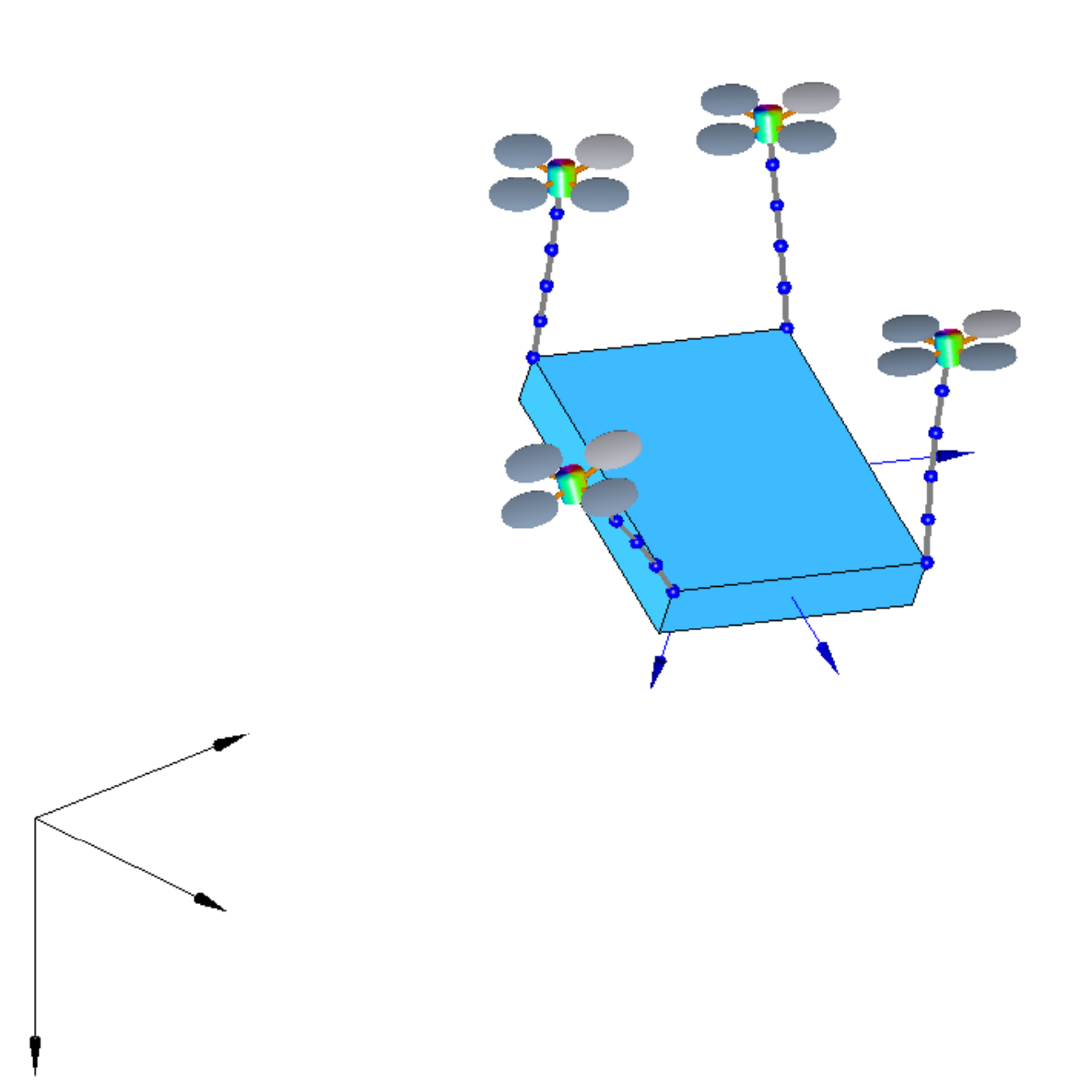}}
	\subfigure[$t=1.10$ Sec.]{
		\includegraphics[width=0.25\columnwidth]{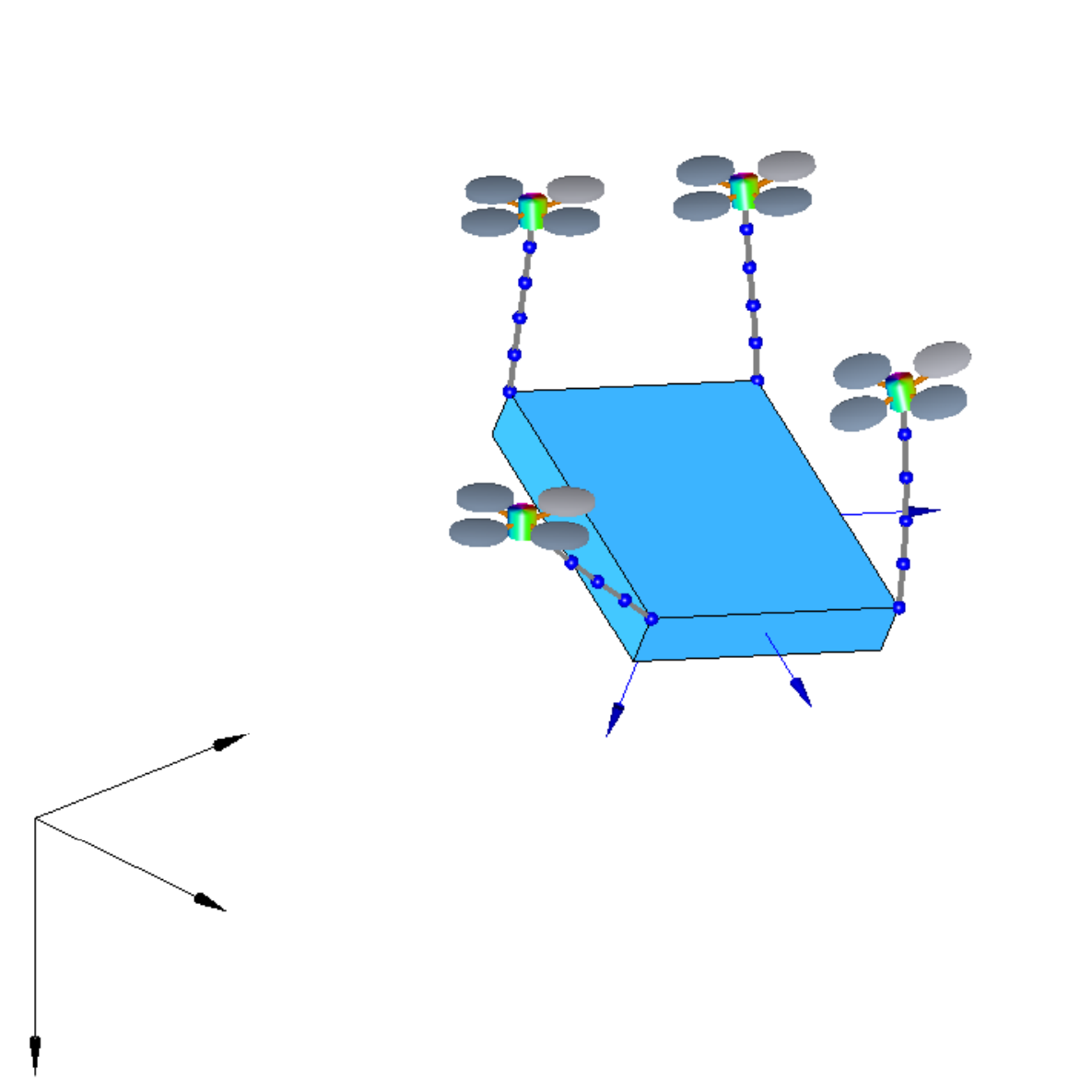}}
		\subfigure[$t=1.36$ Sec.]{
		\includegraphics[width=0.25\columnwidth]{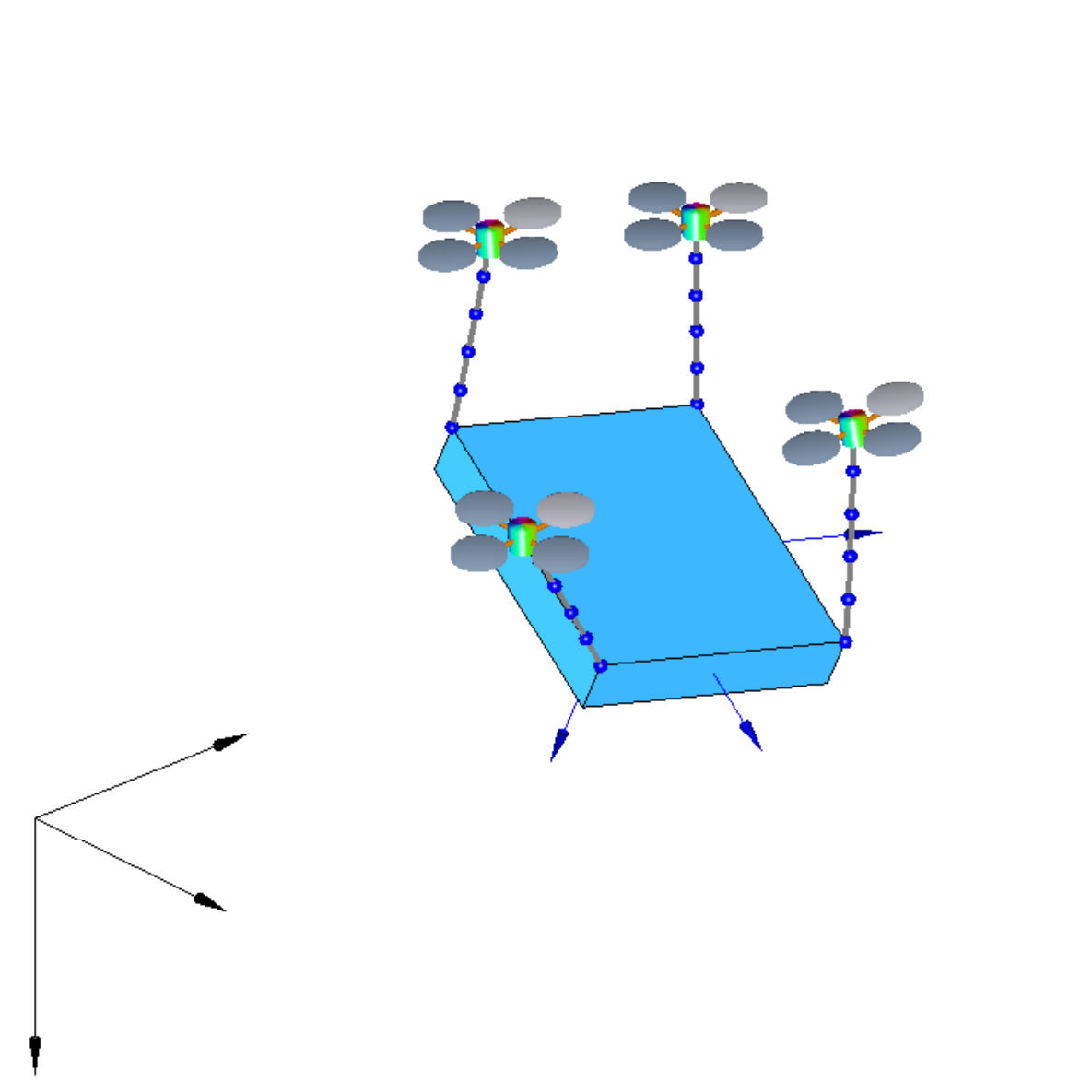}}
}
\centerline{
	\subfigure[$t=1.98$ Sec.]{
		\includegraphics[width=0.25\columnwidth]{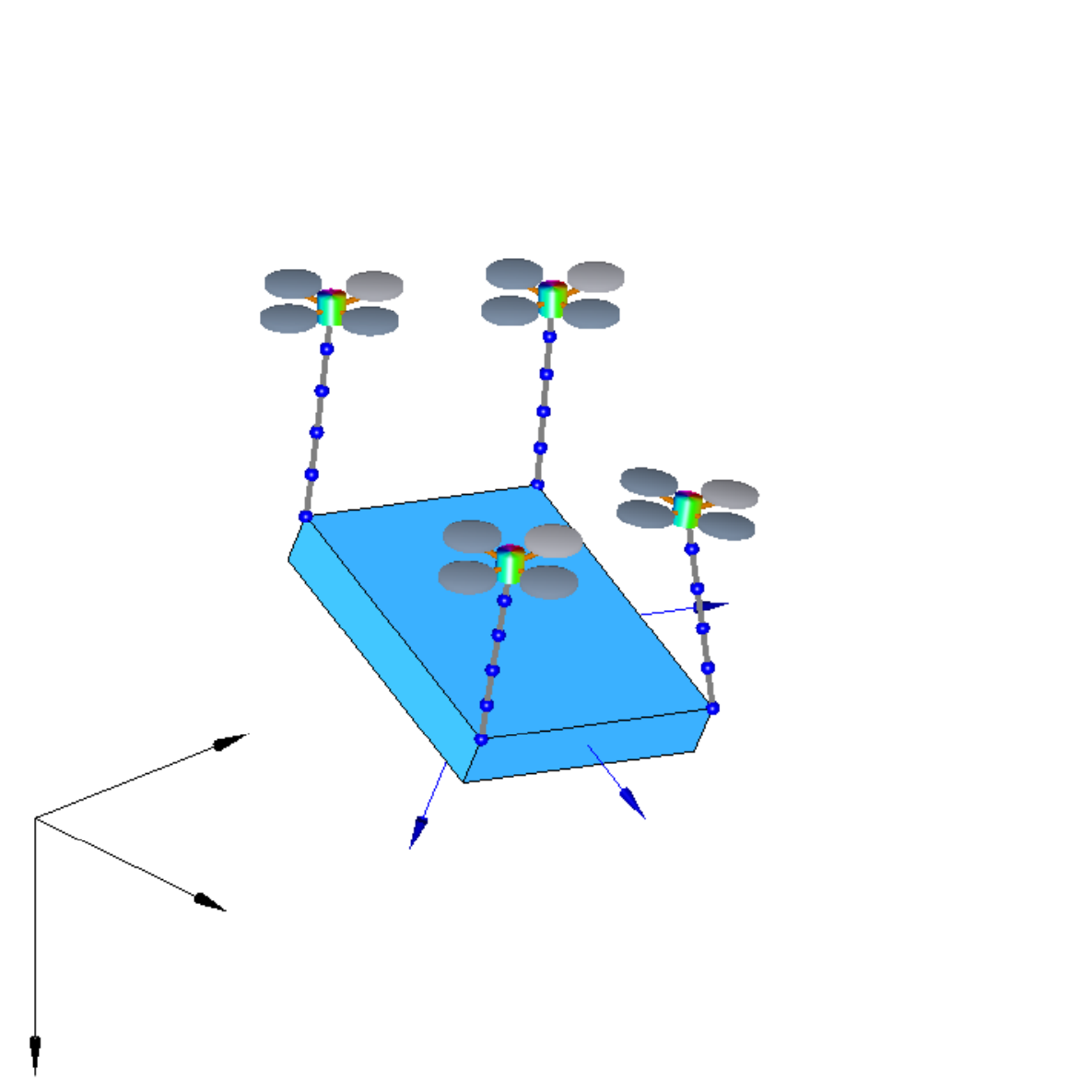}}
	\subfigure[$t=3.48$ Sec.]{
		\includegraphics[width=0.25\columnwidth]{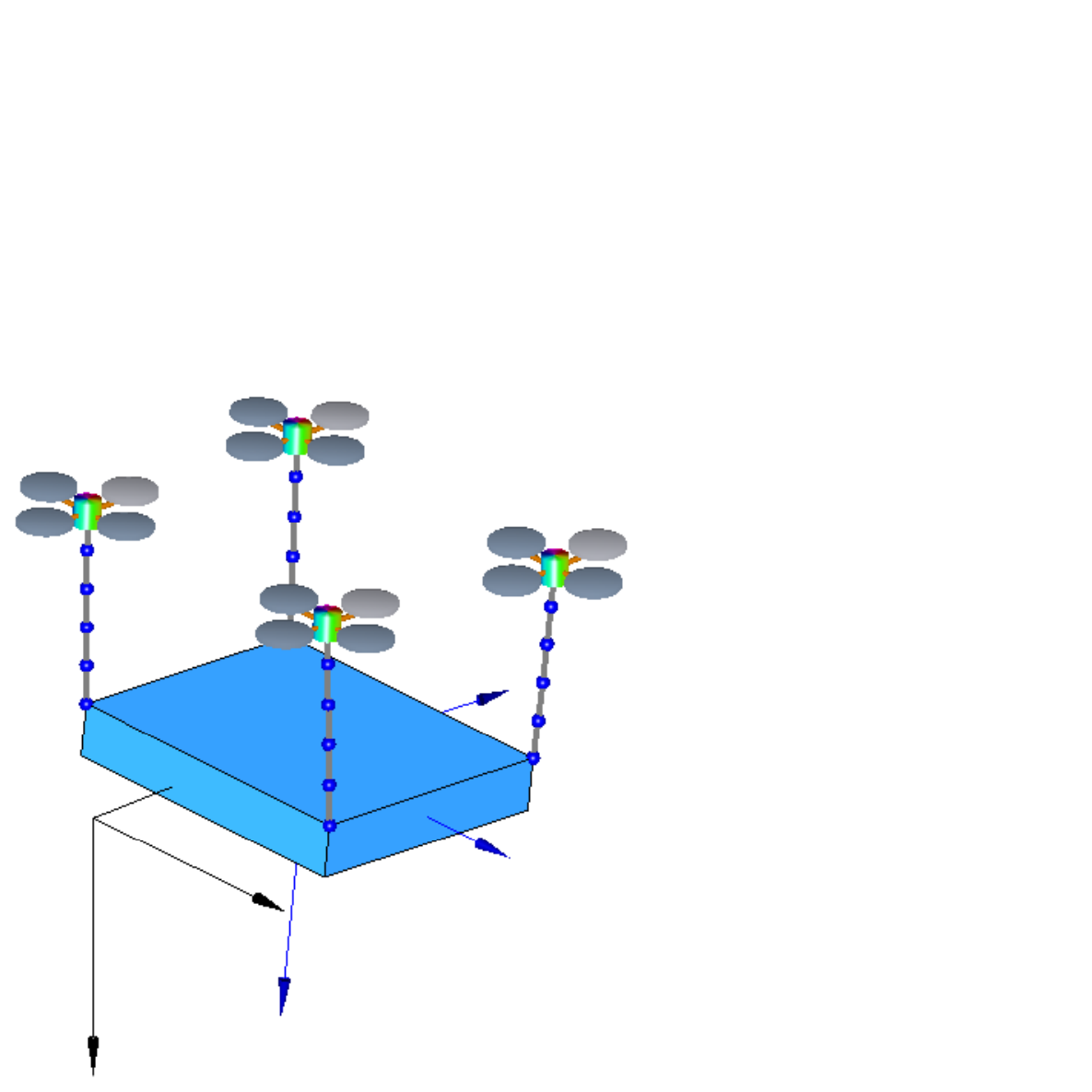}}
		\subfigure[$t=10$ Sec.]{
		\includegraphics[width=0.25\columnwidth]{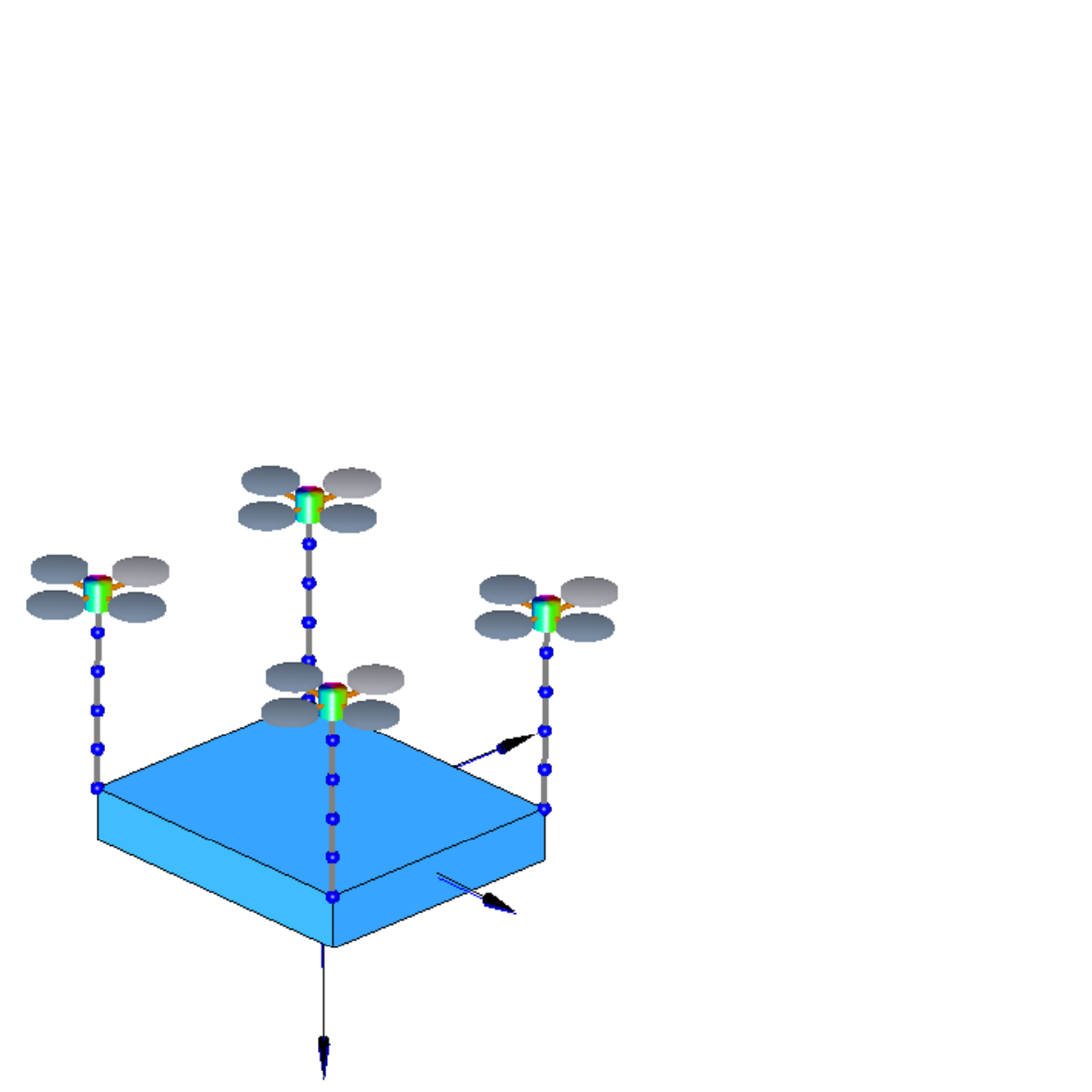}}
}
\caption{Snapshots of the controlled maneuver.}
\label{fig:simresults3snap}
\end{figure}
%%%%%%%%%%%%%%%%%%%%%%%%%%%%%%%%%%%%%%%%%%%%%%%%%%%%%
\newpage
\subsubsection {\normalsize Two Quadrotors Stabilizing a Rod}
The results of this chapter can be specialized to the following system. Consider two quadrotors transporting a rod to a fixed position as shown in Figure~\ref{fig:tworod}. 

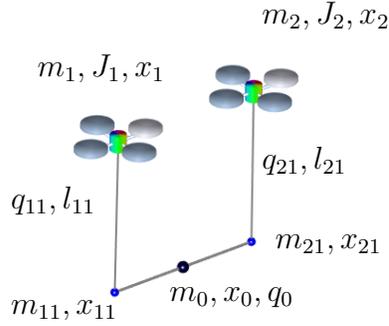
\begin{figure}[h]
\centerline{
		\includegraphics[width=0.2\columnwidth]{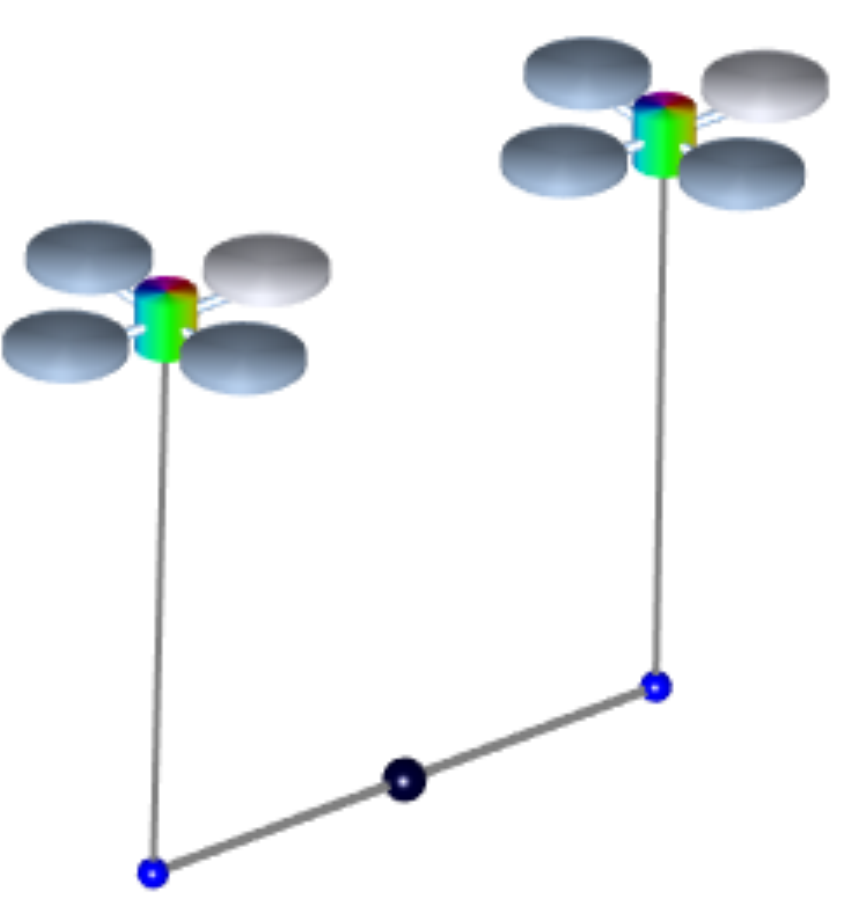}
		\put(-100,85){\shortstack[c]{$m_{1},J_{1},x_{1}$}}
		\put(-110,35){\shortstack[c]{$q_{11},l_{11}$}}
		\put(-110,-5){\shortstack[c]{$m_{11},x_{11}$}}
		\put(-50,0){\shortstack[c]{$m_{0},x_{0},q_{0}$}}
		\put(-10,20){\shortstack[c]{$m_{21},x_{21}$}}
		\put(-15,50){\shortstack[c]{$q_{21},l_{21}$}}
		\put(-15,105){\shortstack[c]{$m_{2},J_{2},x_{2}$}}
}
\caption{Two quadrotor with a rod}\label{fig:tworod}
\end{figure}

We compare stabilization of the rod with large initial attitude errors both for quadrotors and rod for two cases of with and without the integral term in presence of the fixed disturbances in both rotational and translational dynamics. 

The results of this numerical simulation is presented in Figure \ref{fig:simfinalwithout} and \ref{fig:simfinalwith} for without and with integral cases respectively. Snapshots of this maneuver with the integral term is also illustrated at Figure \ref{fig:simresults3snaptwoquad}.

\begin{figure}
\centerline{
	\subfigure[Rod position $x_0$:blue, $x_{0_{d}}$:red]{
		\includegraphics[width=0.4\columnwidth]{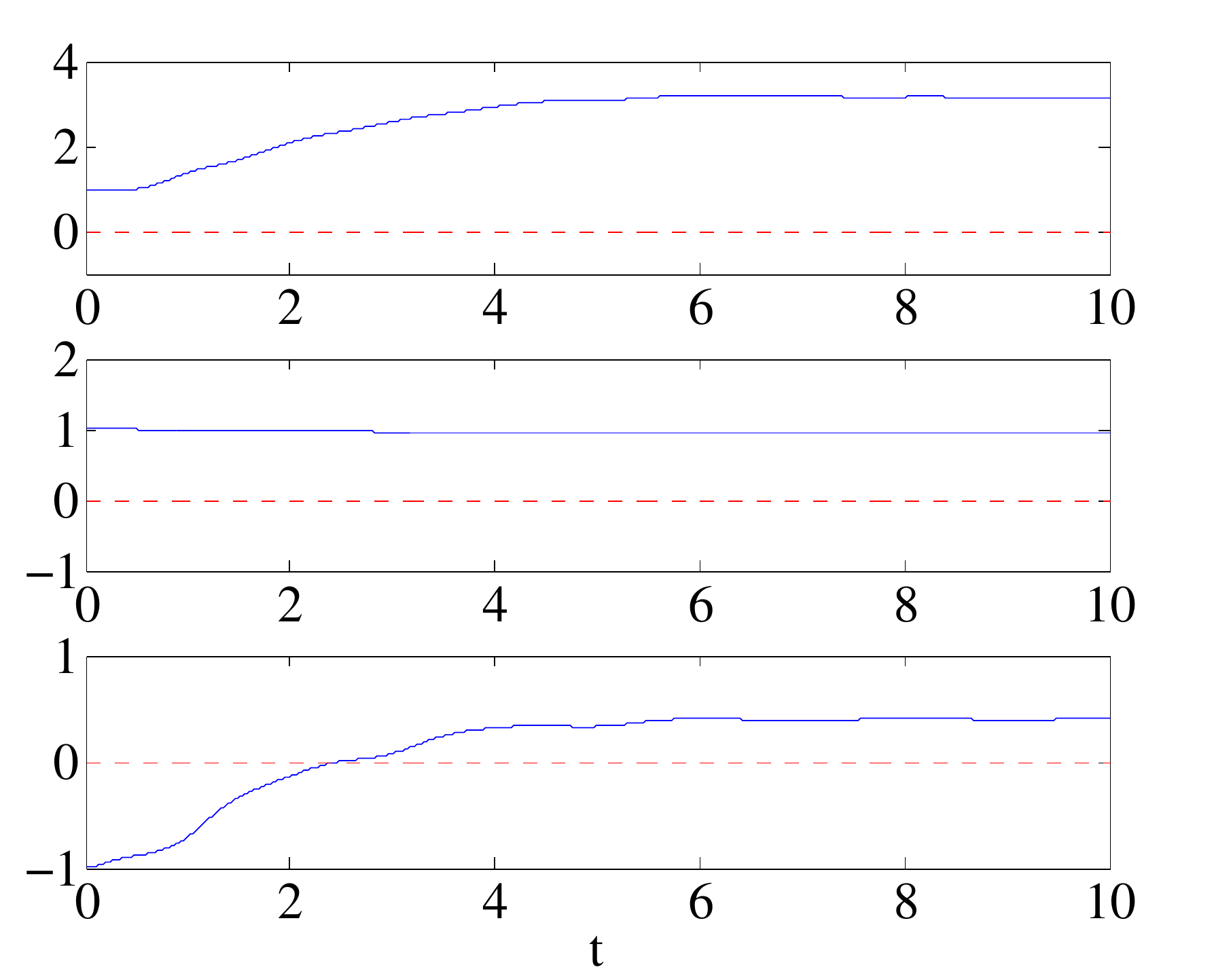}}
	\subfigure[Rod velocity $v_0$:blue, $v_{0_{d}}$:red]{
		\includegraphics[width=0.4\columnwidth]{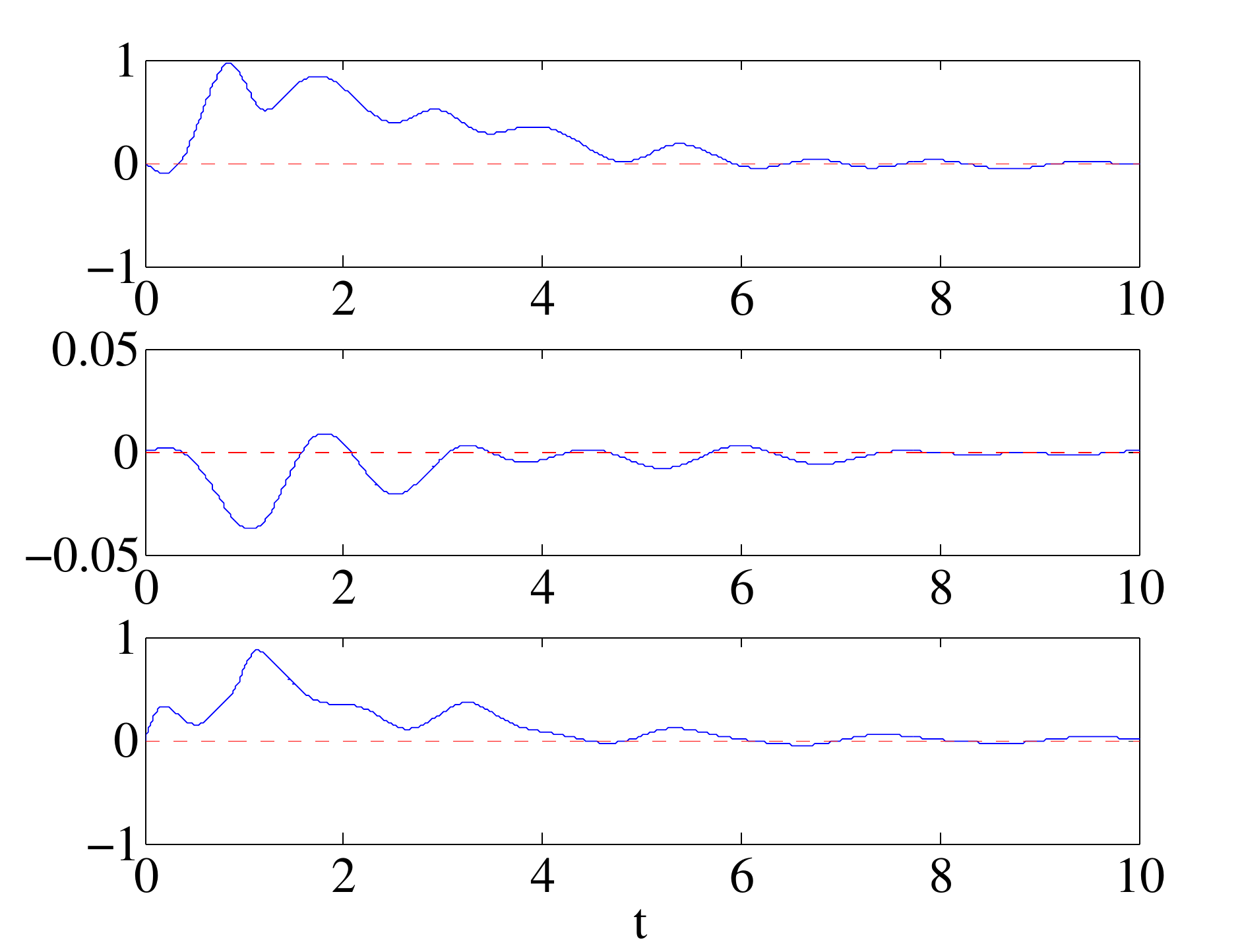}}
}
\centerline{
	\subfigure[Link errors $e_{\omega}, e_{q}$]{
		\includegraphics[width=0.4\columnwidth]{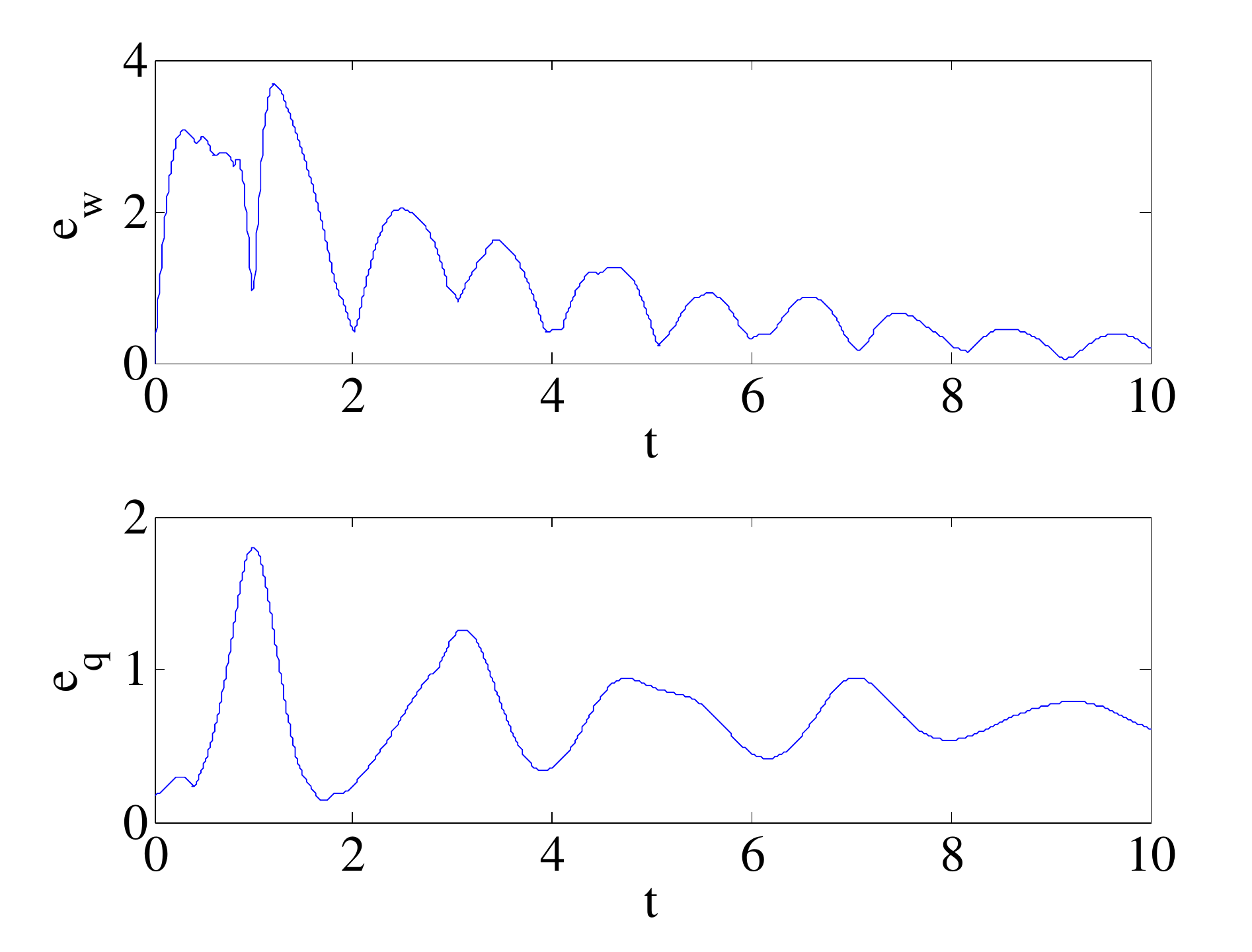}}
	\subfigure[Total Thrust $f_{i}$]{
		\includegraphics[width=0.4\columnwidth]{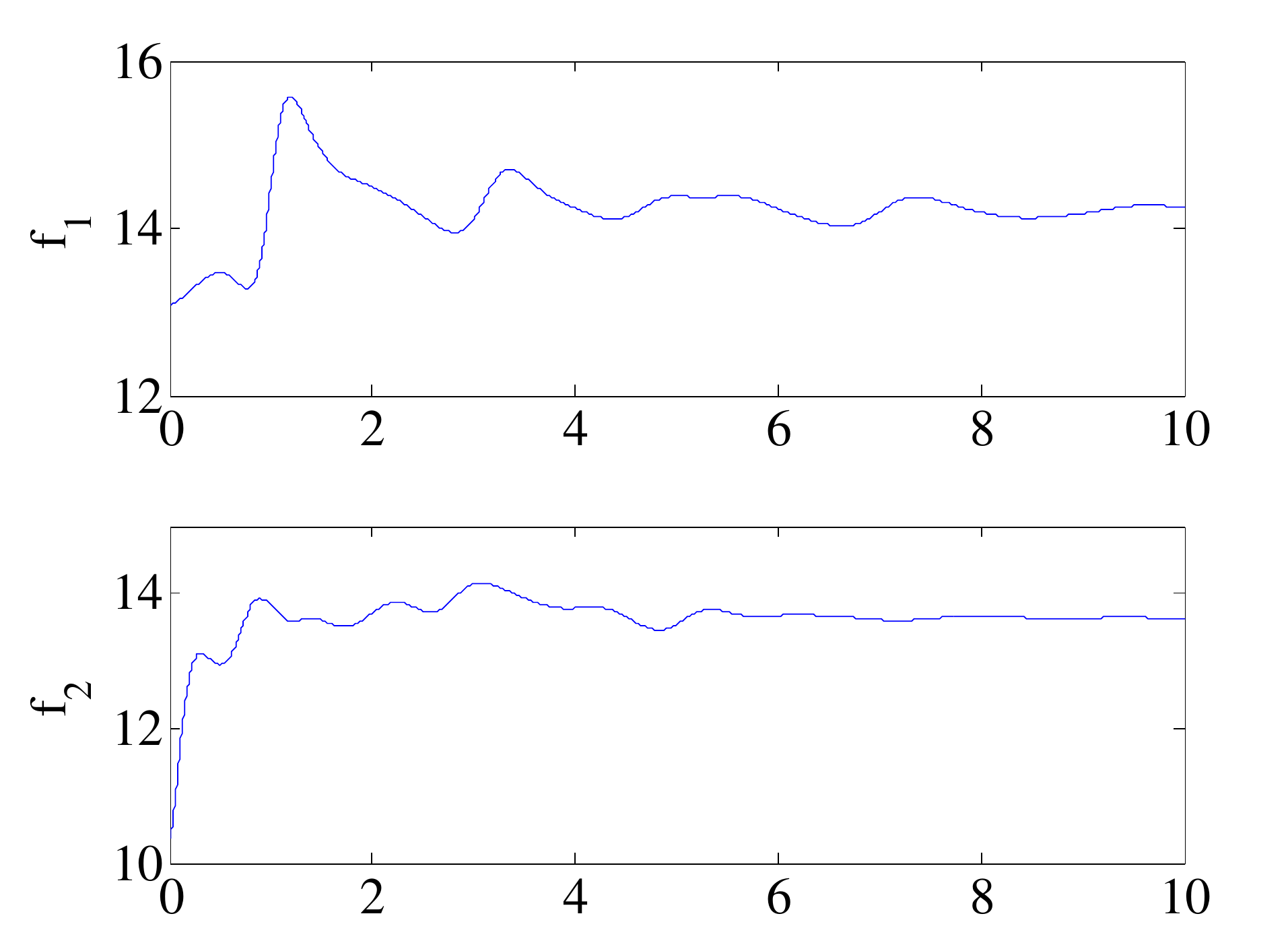}}
}
\centerline{
	\subfigure[Quads angular velocity errors $e_{\omega}$]{
		\includegraphics[width=0.4\columnwidth]{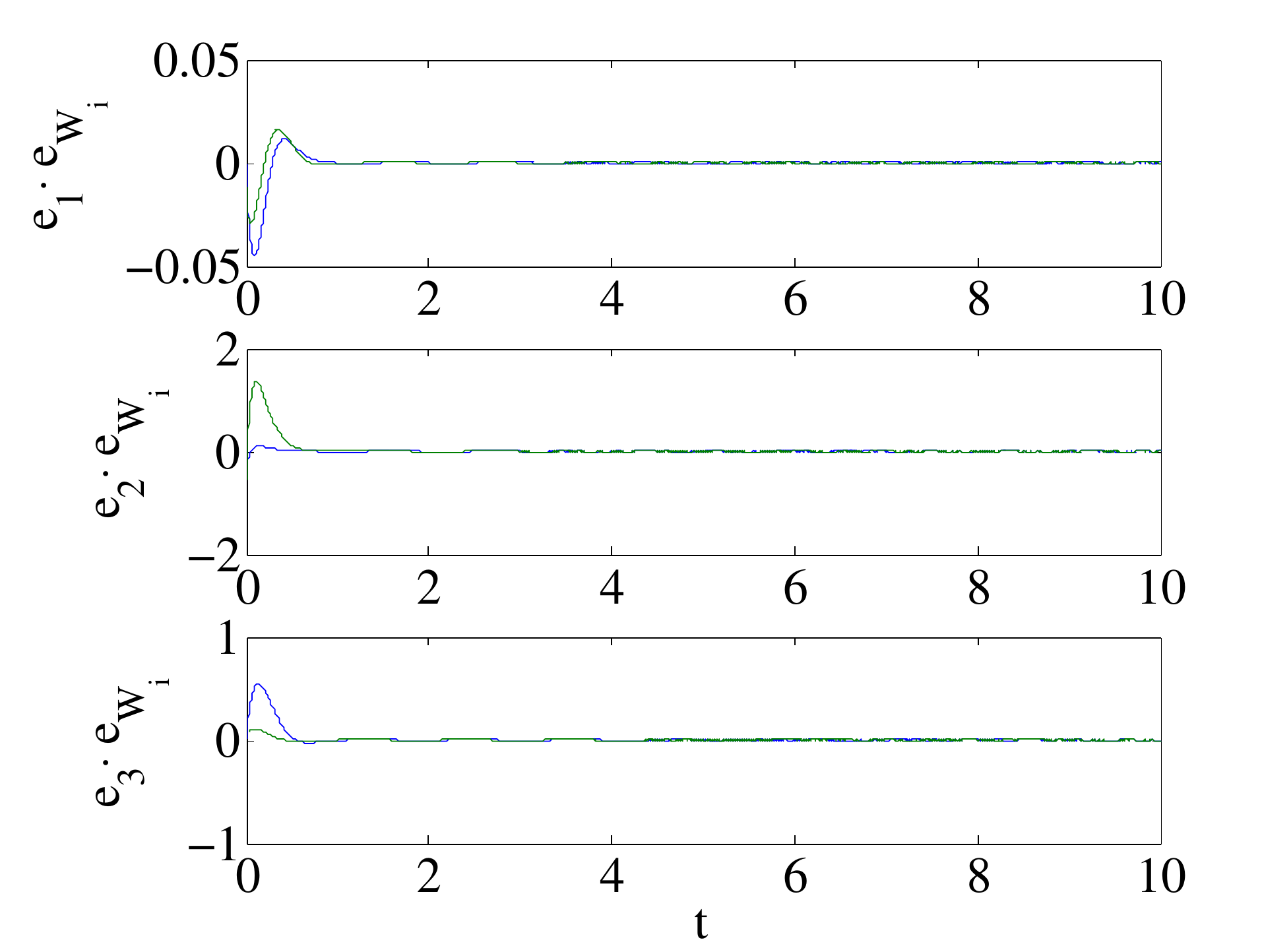}}
	\subfigure[Quads attitude errors $\psi_{i}$]{
		\includegraphics[width=0.4\columnwidth]{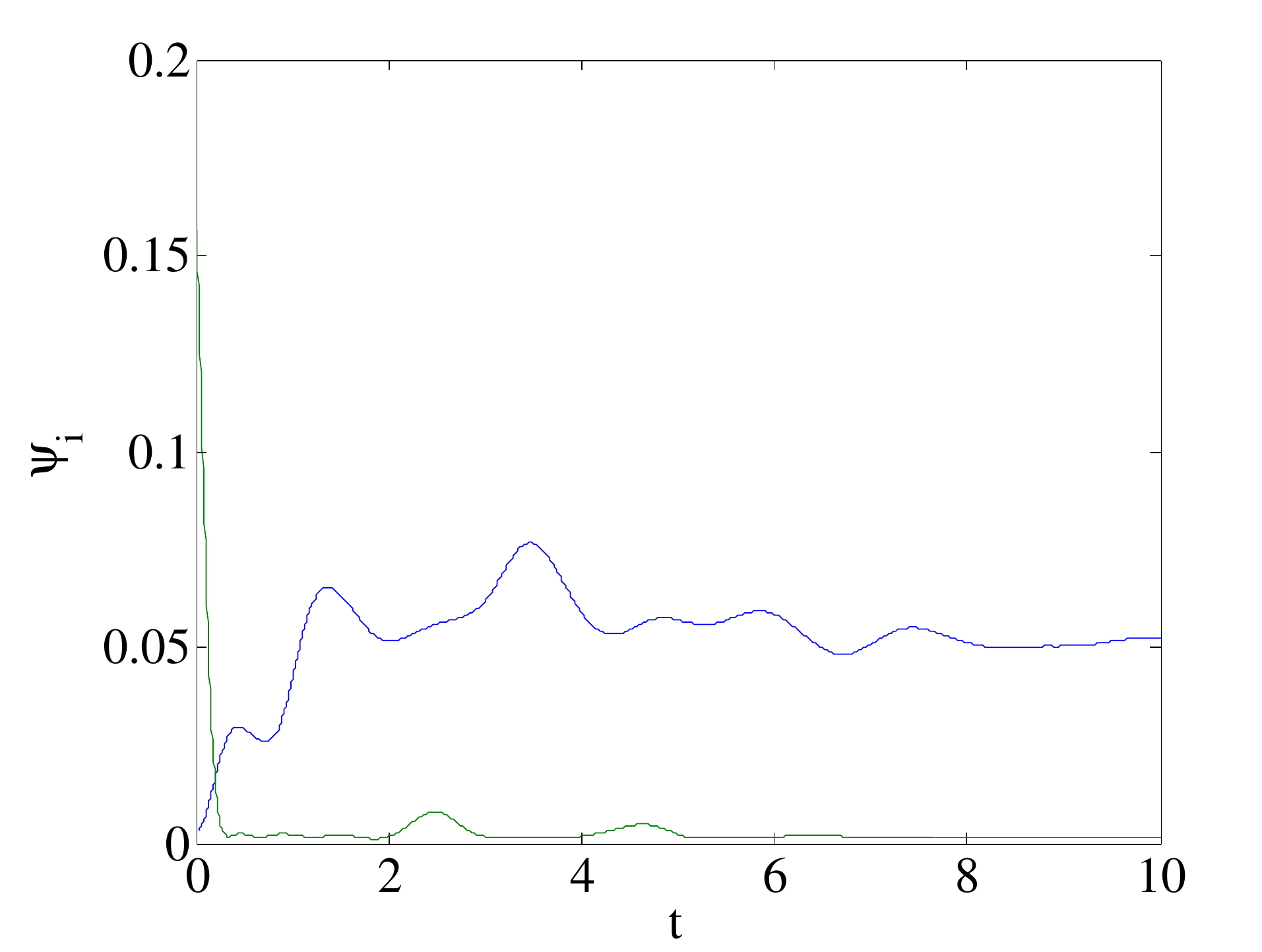}}
}
\caption{Stabilization of a rod with two quadrotors without Integral term.}
\label{fig:simfinalwithout}
\end{figure}

\begin{figure}
\centerline{
	\subfigure[Rod position $x_0$:blue, $x_{0_{d}}$:red]{
		\includegraphics[width=0.4\columnwidth]{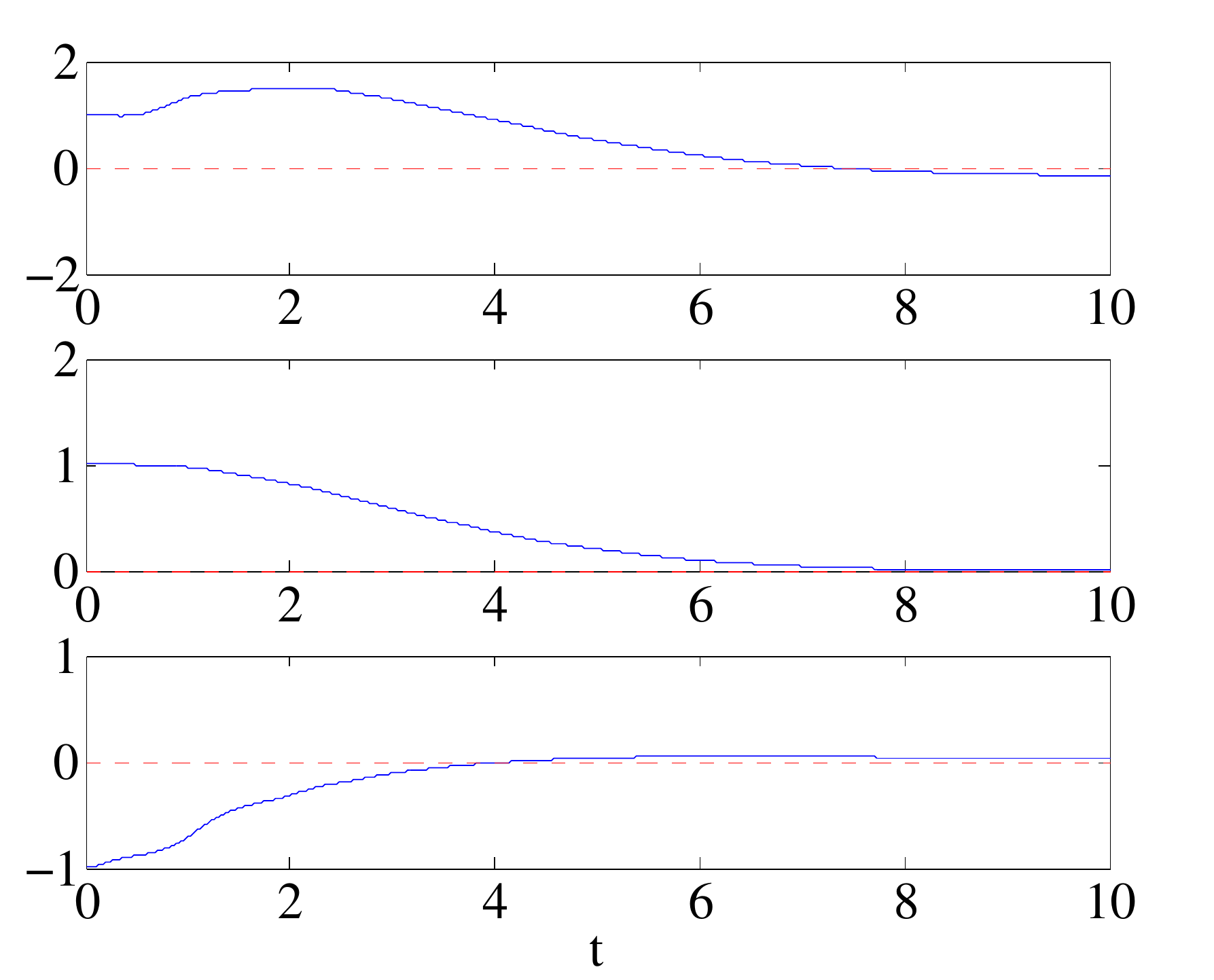}}
	\subfigure[Rod velocity $v_0$:blue, $v_{0_{d}}$:red]{
		\includegraphics[width=0.4\columnwidth]{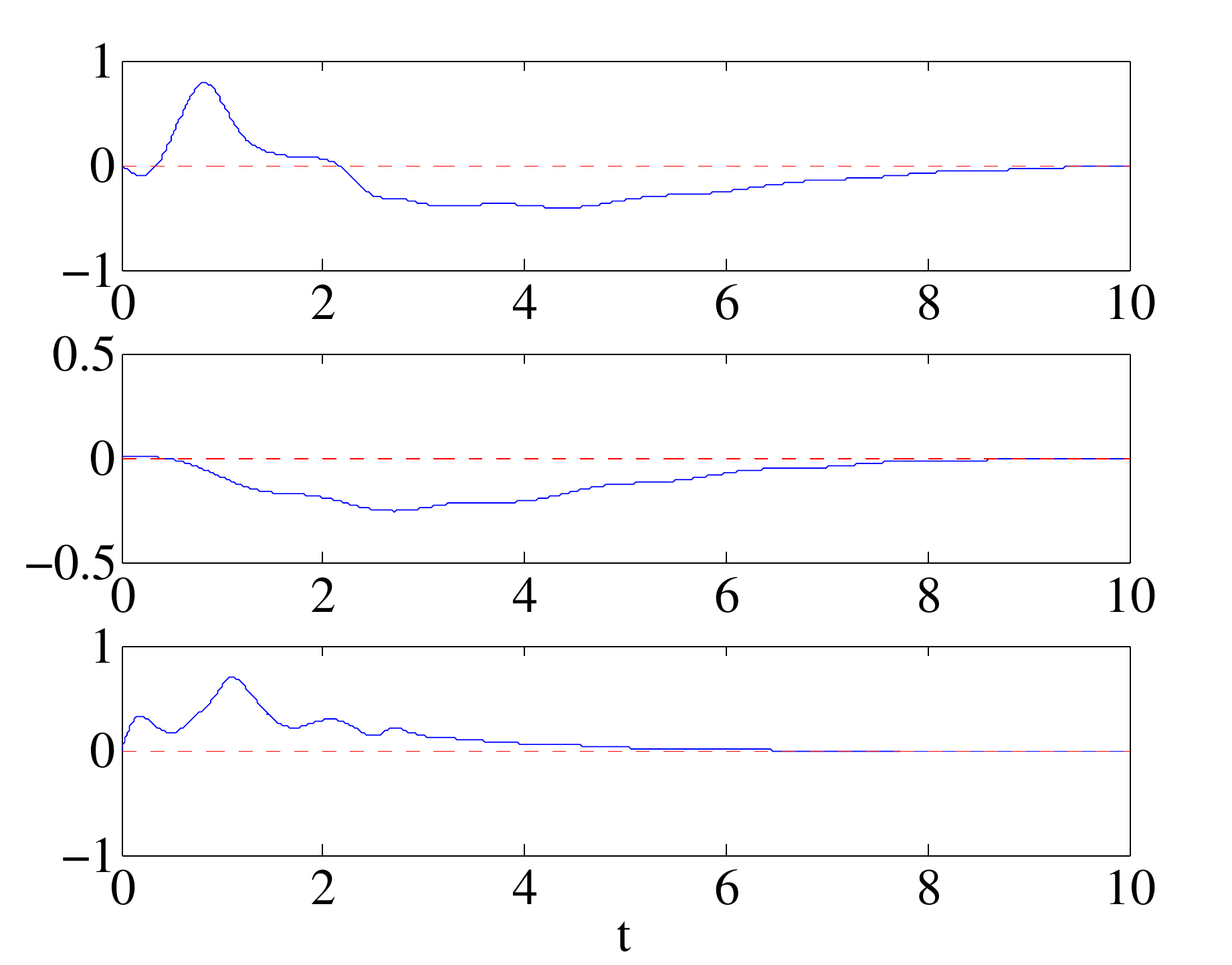}}
}
\centerline{
	\subfigure[Link errors $e_{\omega}, e_{q}$]{
		\includegraphics[width=0.4\columnwidth]{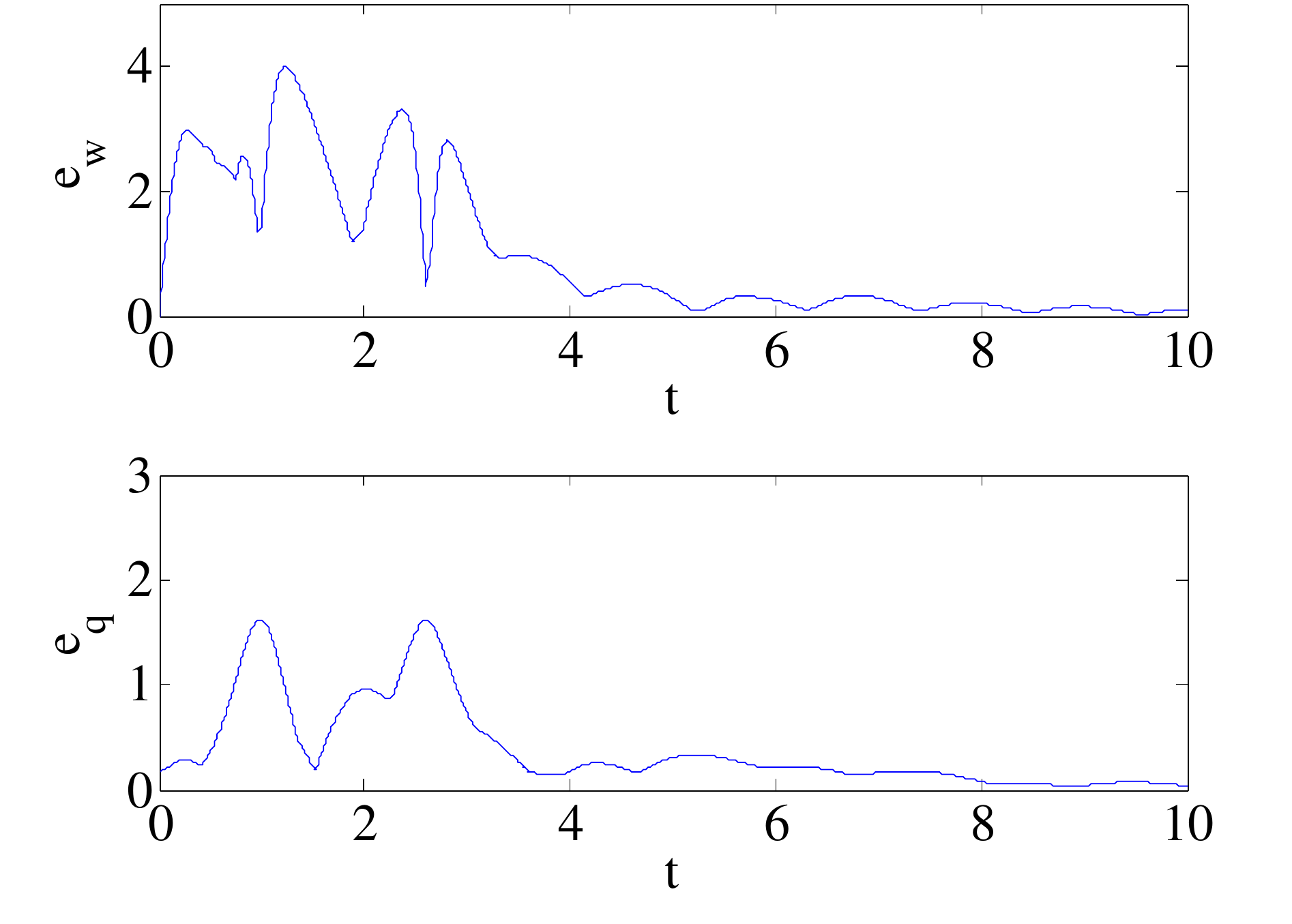}}
	\subfigure[Total Thrust $f_{i}$]{
		\includegraphics[width=0.4\columnwidth]{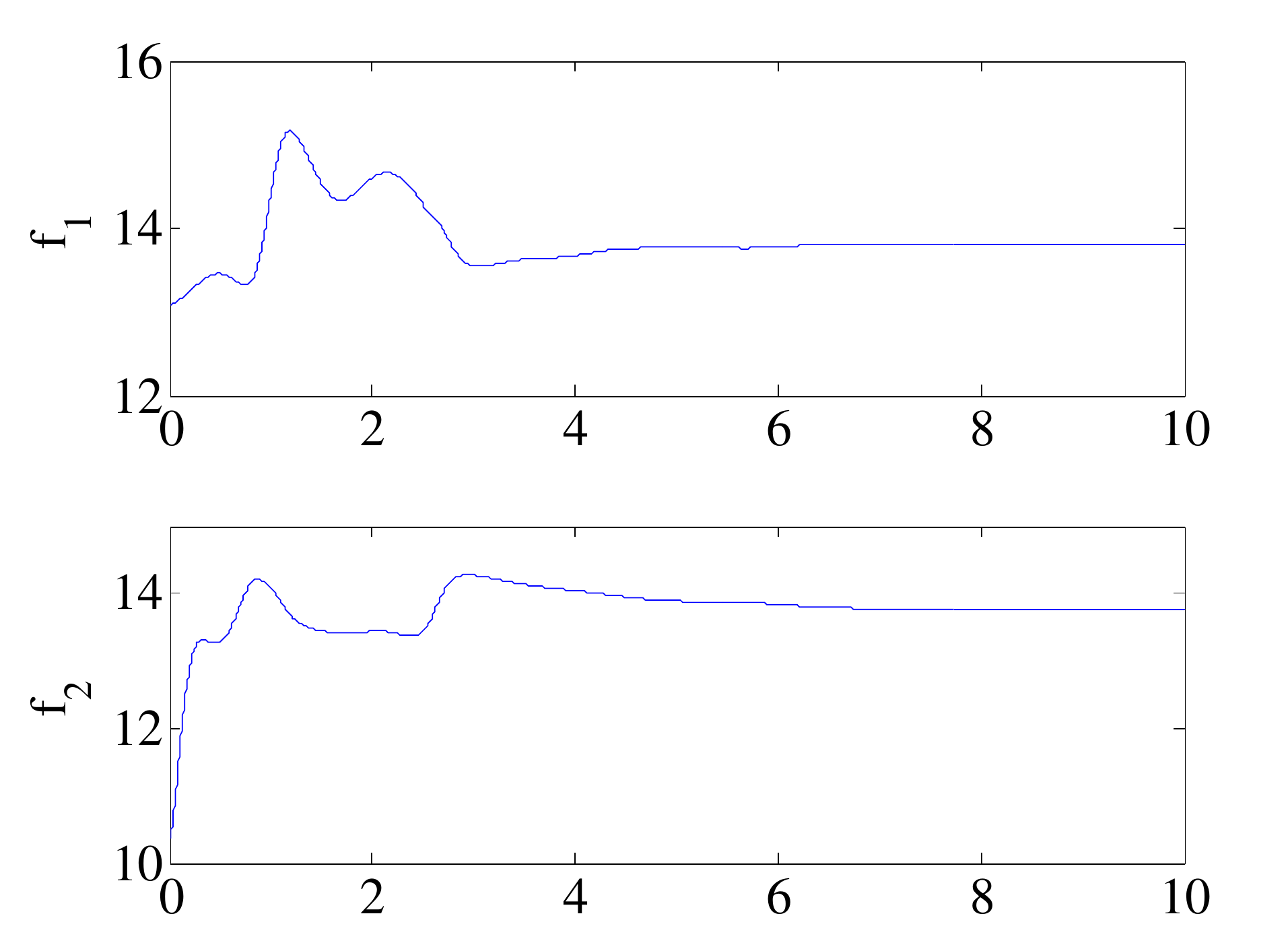}}
}
\centerline{
	\subfigure[Quads angular velocity errors $e_{\omega}$]{
		\includegraphics[width=0.4\columnwidth]{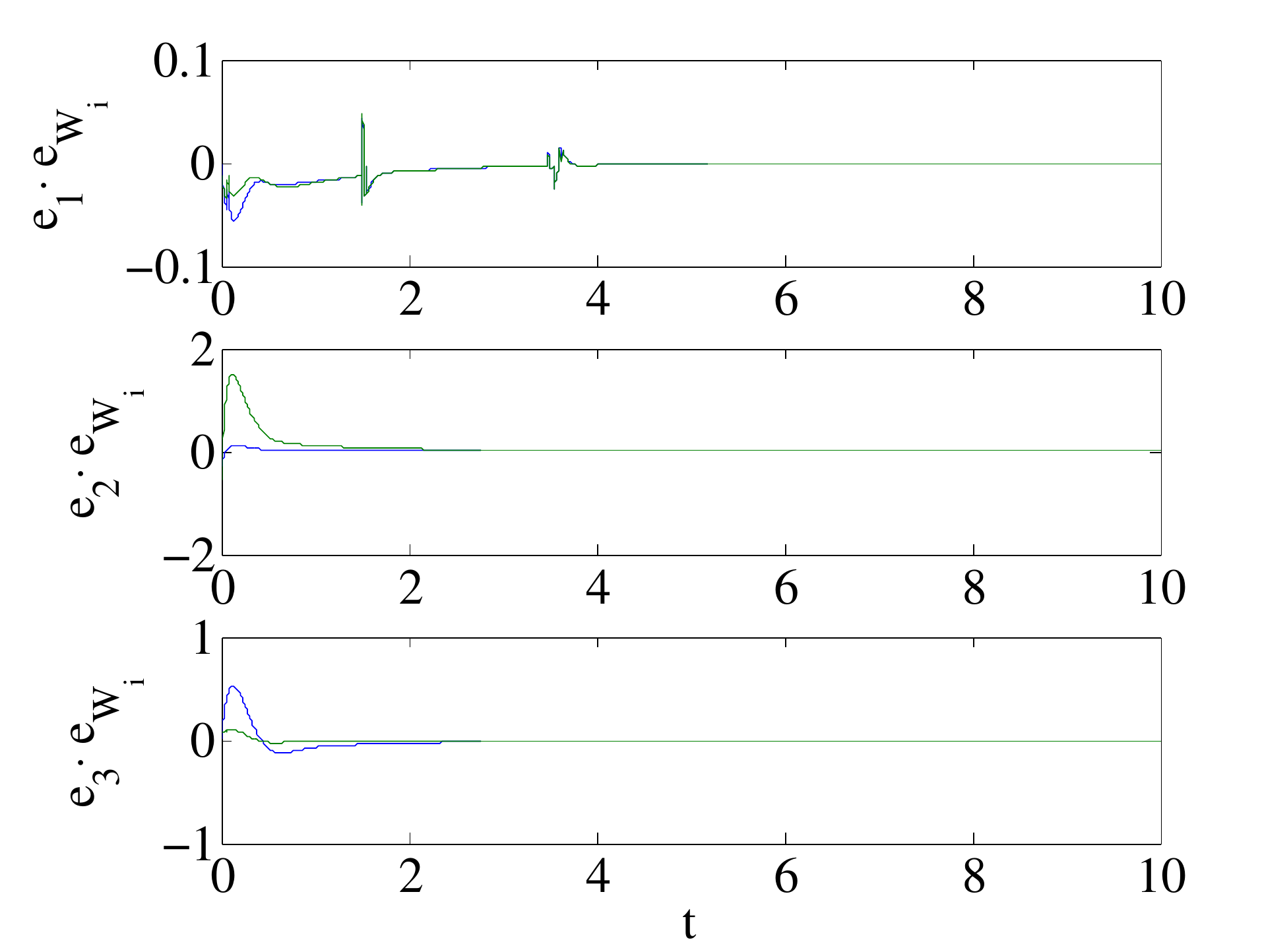}}
	\subfigure[Quads attitude errors $\psi_{i}$]{
		\includegraphics[width=0.4\columnwidth]{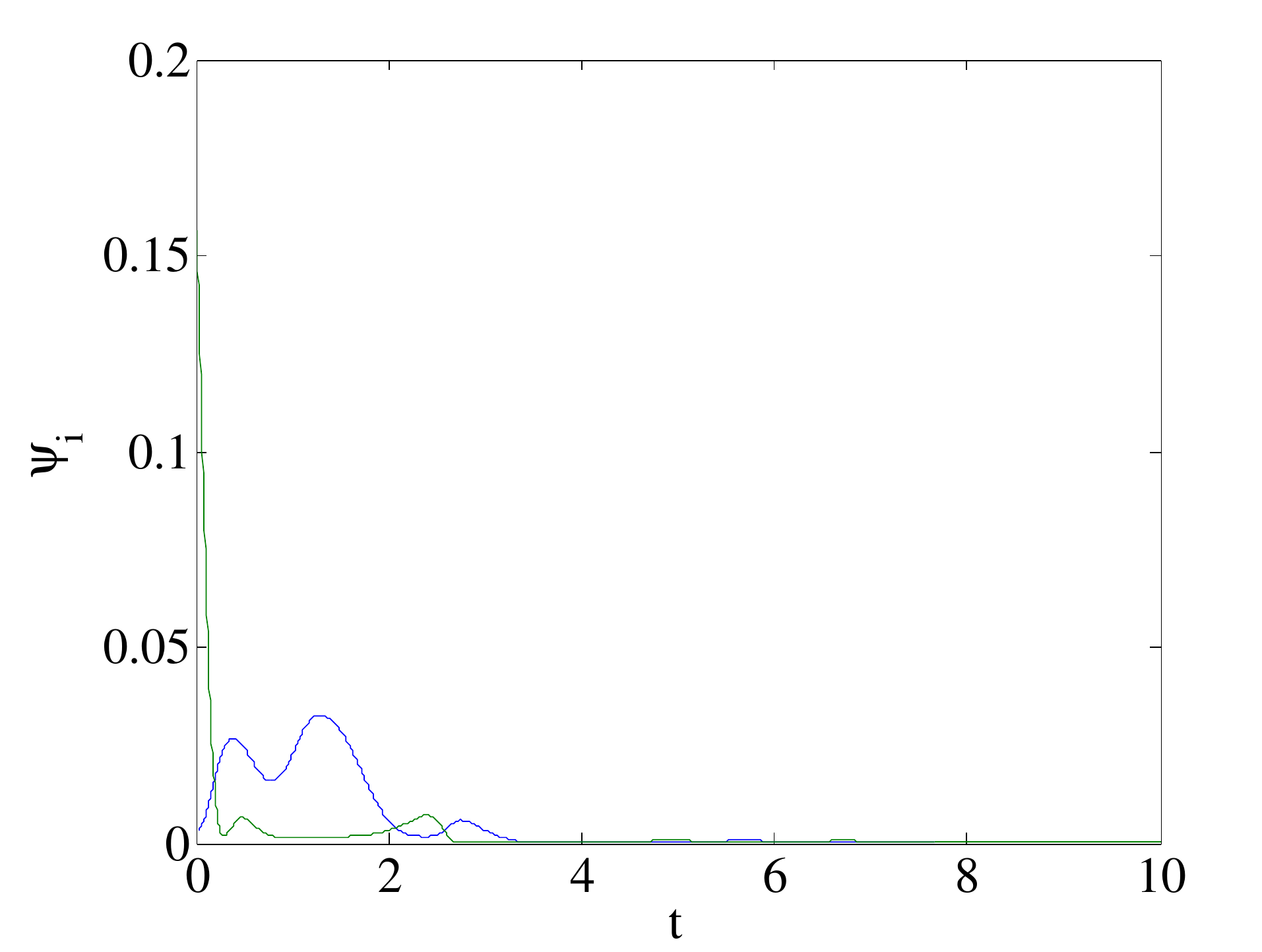}}
}
\caption{Stabilization of a rod with two quadrotors with Integral term.}
\label{fig:simfinalwith}
\end{figure}

\begin{figure}
\centerline{
	\subfigure[$t=0$ Sec.]{
		\includegraphics[width=0.25\columnwidth]{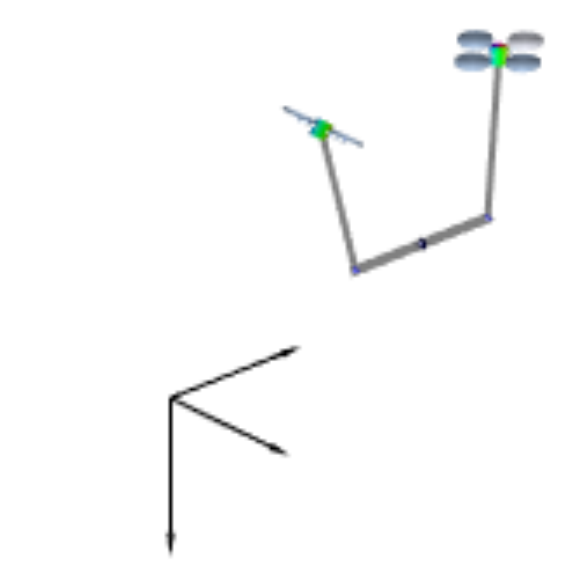}}
	\subfigure[$t=0.14$ Sec.]{
		\includegraphics[width=0.25\columnwidth]{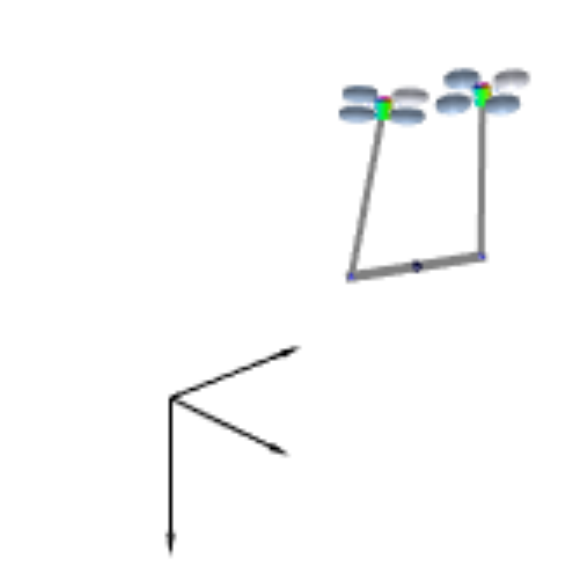}}
		\subfigure[$t=0.30$ Sec.]{
		\includegraphics[width=0.25\columnwidth]{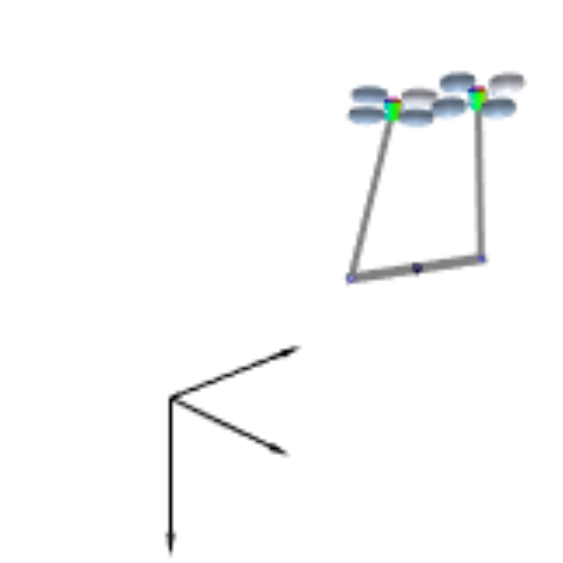}}
}
\centerline{
	\subfigure[$t=0.68$ Sec.]{
		\includegraphics[width=0.25\columnwidth]{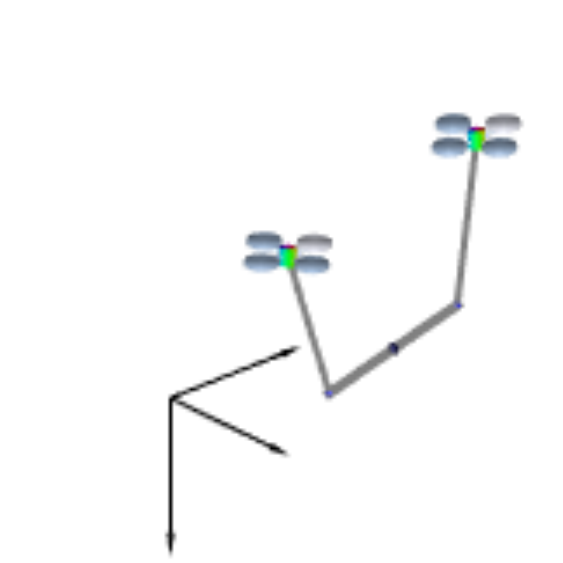}}
	\subfigure[$t=1.10$ Sec.]{
		\includegraphics[width=0.25\columnwidth]{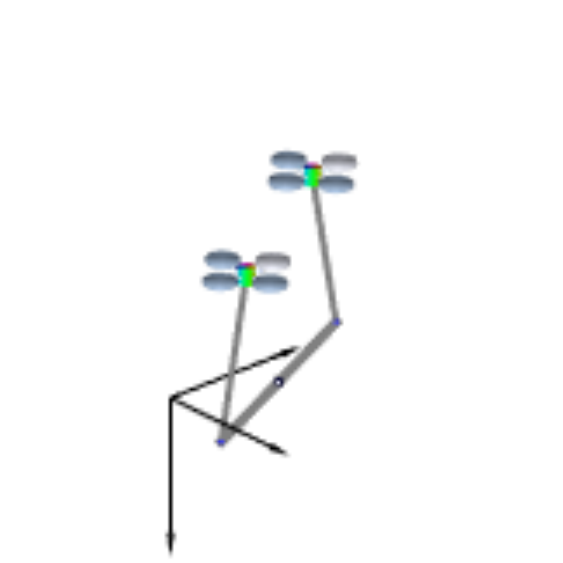}}
		\subfigure[$t=1.36$ Sec.]{
		\includegraphics[width=0.25\columnwidth]{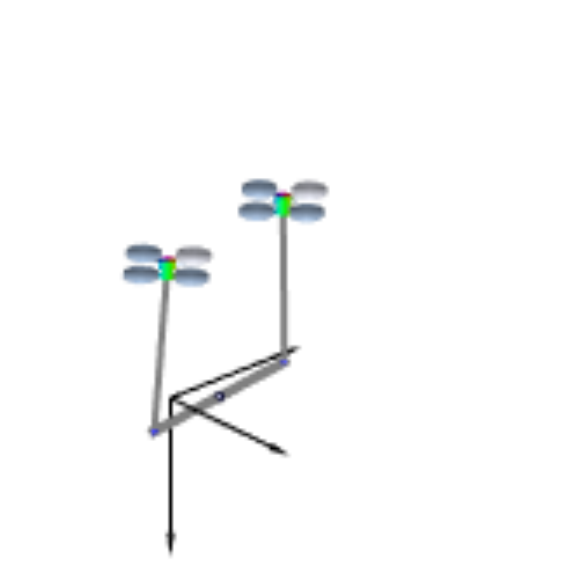}}
}
\centerline{
	\subfigure[$t=1.98$ Sec.]{
		\includegraphics[width=0.25\columnwidth]{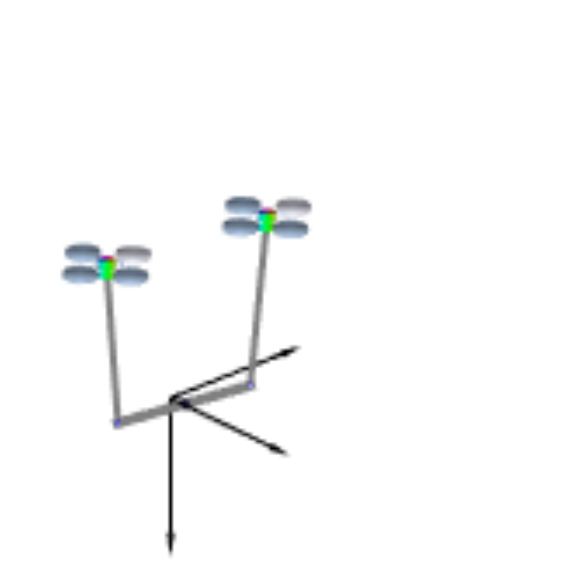}}
	\subfigure[$t=3.48$ Sec.]{
		\includegraphics[width=0.25\columnwidth]{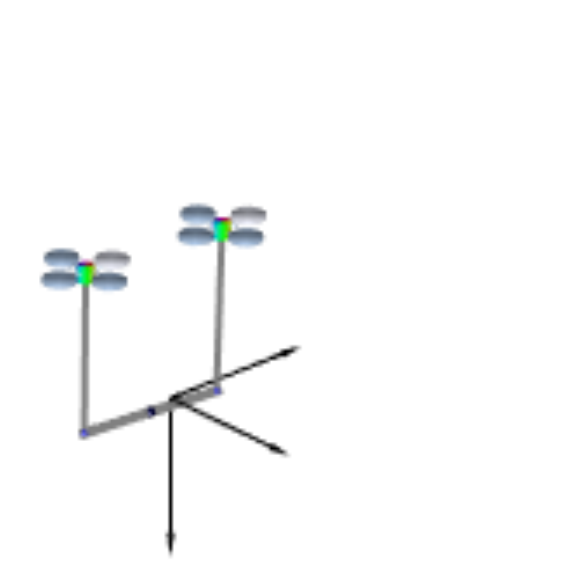}}
		\subfigure[$t=4.0$ Sec.]{
		\includegraphics[width=0.25\columnwidth]{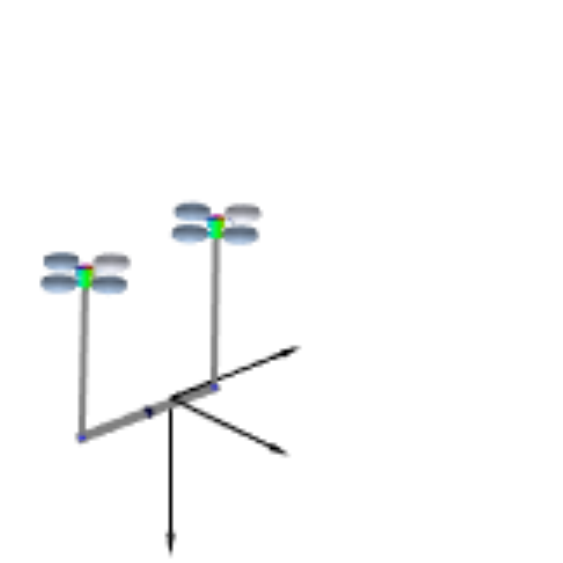}}
}
\caption{Snapshots of the controlled maneuver for two quadrotors stabilizing a rod.}
\label{fig:simresults3snaptwoquad}
\end{figure}

}

%%%%%%%%%%%%%%%%%%%%%%%%%%%%%%%%%%%%%%%%%%%%%%%%%%%%%
%%%%%%%%%%%%%%%%%%%%%%%%%%%%%%%%%%%%%%%%%%%%%%%%%%%%%
%%%%%%%%%%%%%%%%%%%%%%%%%%%%%%%%%%%%%%%%%%%%%%%%%%%%%
\newpage
\begin{singlespace}
\section{\protect \centering Chapter 5: Experiment}
\end{singlespace}
\setcounter{section}{5}
We develop an experimental testbed with quadrotor UAV's to validate our derivations and simulations. Two quadrotor UAV's are built from scratch by designing every parts followed by developing a software for real-time experiments. The hardware and software development for the experimental part of this dissertation is presented in the following. Finally, the experimental results for some aggressive trajectory tracking and payload transportation are provided.
\doublespacing
%%%%%%%%%%%%%%%%%%%%%%%%%%%%%%%%%%%%%%%%%%%%%%%%%%%%%
\subsection {\normalsize CAD Model and Calibration}
{\addtolength{\leftskip}{0.5in}
We developed an accurate CAD model as shown in Figure~\ref{fig:cadmodel} to identify several parameters of the quadrotor, such as moment of inertia and center of mass. Furthermore, a precise rotor calibration is performed for each rotor, with a custom-made thrust stand as shown in Figure~\ref{fig:stand} to determine the relation between the command in the motor speed controller and the actual thrust. For various values of motor speed commands, the corresponding thrust is measured, and those data are fitted with a second order polynomial. 

\begin{figure}[h]
\centerline{
	\subfigure[CAD Model]{
        \includegraphics[width=0.57\columnwidth]{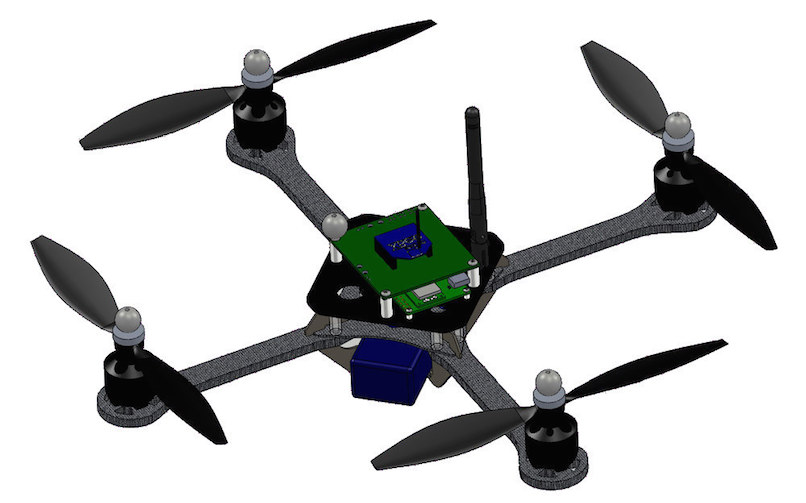}\label{fig:cadmodel}}
	\subfigure[Motor calibration setup]{
	\includegraphics[width=0.17\columnwidth]{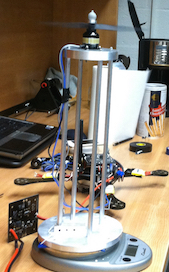}\label{fig:stand}}
}
\caption{Hardware development for a quadrotor UAV}
\end{figure}

}
%%%%%%%%%%%%%%%%%%%%%%%%%%%%%%%%%%%%%%%%%%%%%%%%%%%%%

\subsection {\normalsize Building Quadrotor}
{\addtolength{\leftskip}{0.5in}
The quadrotor UAV developed at the flight dynamics and control laboratory at the George Washington University is shown at Figure~\ref{fig:Quad}. The angular velocity is measured from inertial measurement unit (IMU) and the attitude is obtained from IMU data. Position of the UAV is measured from motion capture system (Vicon) as shown in Figure~\ref{fig:moca} and the velocity is estimated from the measurement. Ground computing system receives the Vicon data and send it to the UAV via XBee. The Gumstix is adopted as micro computing unit on the UAV.  

\begin{figure}[h]
\centerline{
\setlength{\unitlength}{0.1\columnwidth}\scriptsize
\begin{picture}(7,4)(0,0)
\put(0,0){\includegraphics[width=0.7\columnwidth]{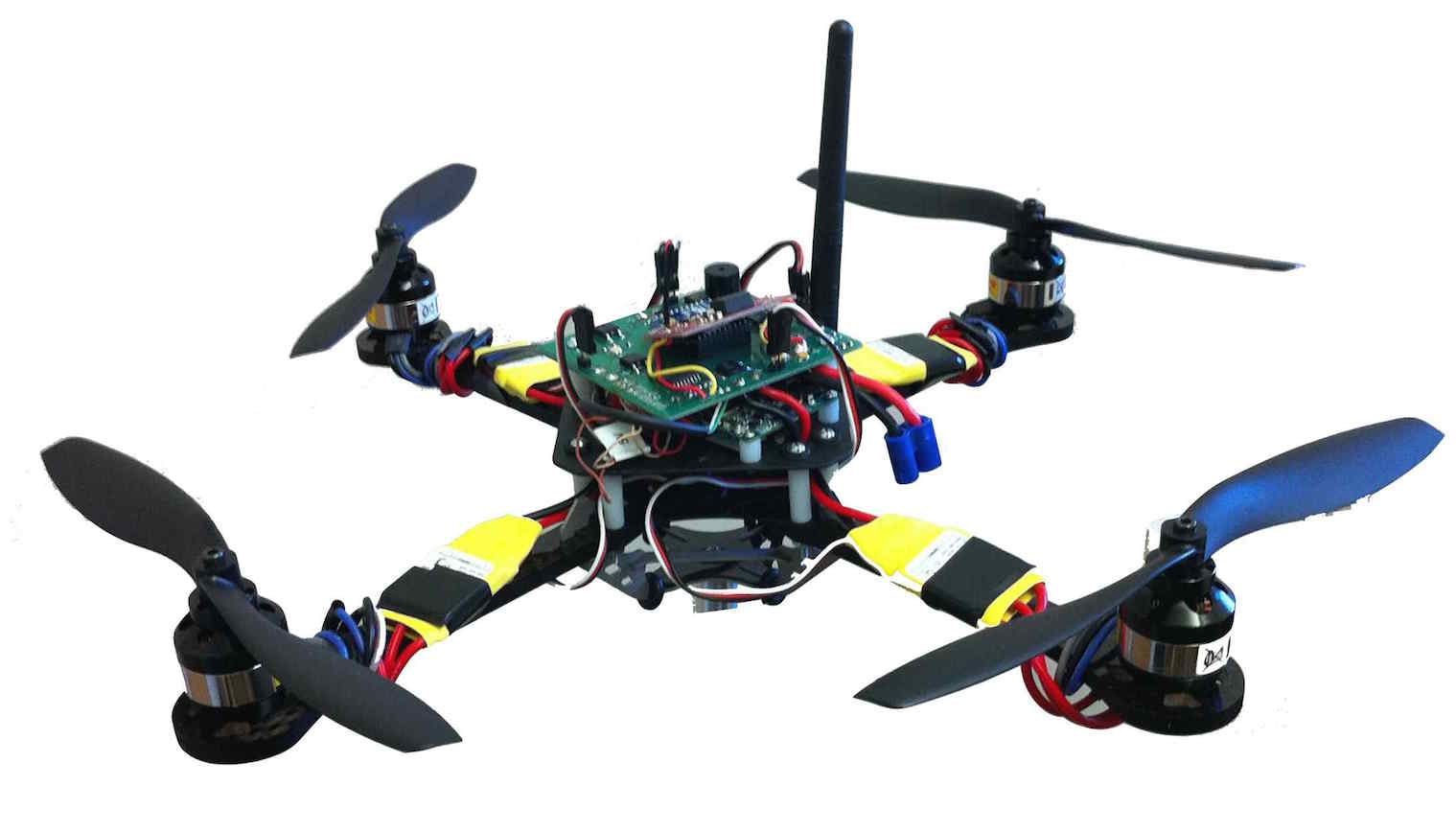}}
\put(1.95,3.2){\shortstack[c]{OMAP 600MHz\\Processor}}
\put(2.3,0){\shortstack[c]{Attitude sensor\\3DM-GX3\\ via UART}}
\put(0.85,1.4){\shortstack[c]{BLDC Motor\\ via I2C}}
\put(0.1,2.5){\shortstack[c]{Safety Switch\\XBee RF}}
\put(4.3,3.2){\shortstack[c]{WIFI to\\Ground Station}}
\put(5,2.0){\shortstack[c]{LiPo Battery\\11.1V, 2200mAh}}
\end{picture}
}
\caption{Hardware development}\label{fig:Quad}
\end{figure}
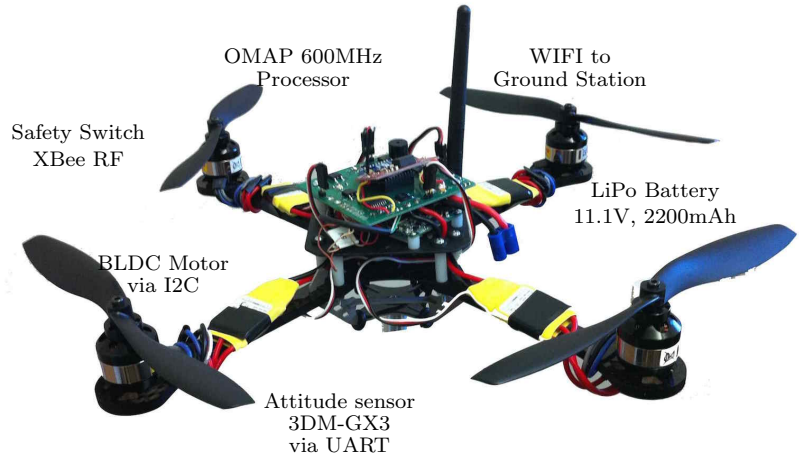

\begin{figure}[h]
\centerline{
        \includegraphics[width=0.57\columnwidth]{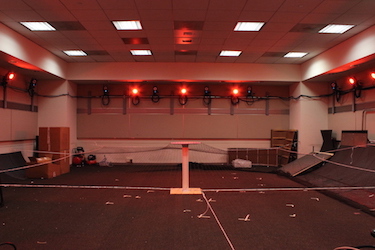}
}
\caption{Motion capture Laboratory}\label{fig:moca}
\end{figure}

\subsubsection {\normalsize Hardware Configuration}
Quadrotor is built using several parts as listed bellow. The motors are connected to the circuit board via speed controllers and I2C ports. The IMU is connected to the circuit board via serial ports and the computer module (GUMSTIX) handles the communications and controller computation.

\begin{itemize}
\item Gumstix Overo computer-in-module (OMAP 600MHz processor), running a non-realtime Linux operating system. It is connected to a ground station via WIFI (Figure~\ref{fig:gumstix}).
\item Microstrain 3DM-GX3 IMU, connected to Gumstix via UART (Figure~\ref{fig:IMU}).
\item BL-CTRL 2.0 motor speed controller, connected to Gumstix via I2C (Figure~\ref{fig:mc}).
\item Roxxy 2827-35 Brushless DC motors  (Figure~\ref{fig:motor}).
\item XBee RF module, connected to Gumstix via UART. (Figure~\ref{fig:xbee})
\end{itemize}
\begin{figure}[h]
\centerline{
	\subfigure[XBee]{
        \includegraphics[width=0.27\columnwidth]{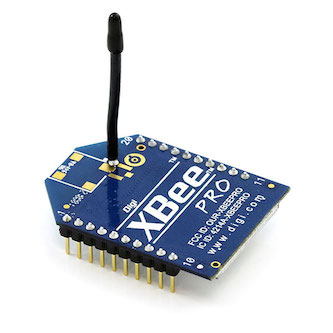}\label{fig:xbee}}
        \subfigure[IMU]{
        \includegraphics[width=0.27\columnwidth]{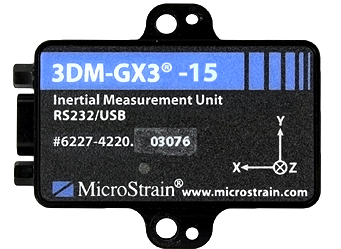}\label{fig:IMU}}
	\subfigure[Computer Module]{
	\includegraphics[width=0.37\columnwidth]{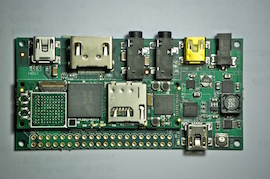}\label{fig:gumstix}}
}
\caption{Hardware parts for a quadrotor UAV}
\end{figure}
\begin{figure}[h]
\centerline{
	\subfigure[Motor speed controller]{
	\includegraphics[width=0.37\columnwidth]{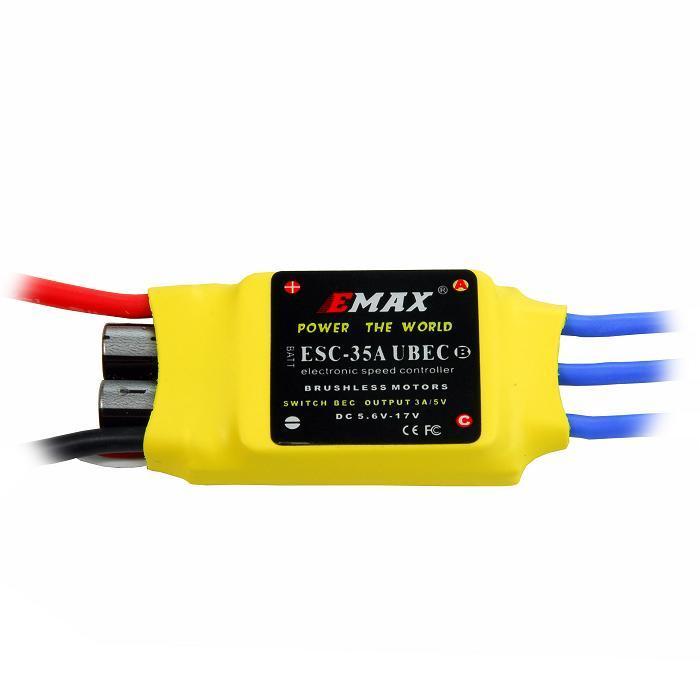}\label{fig:mc}}
	\subfigure[Brushless DC motors]{
	\includegraphics[width=0.37\columnwidth]{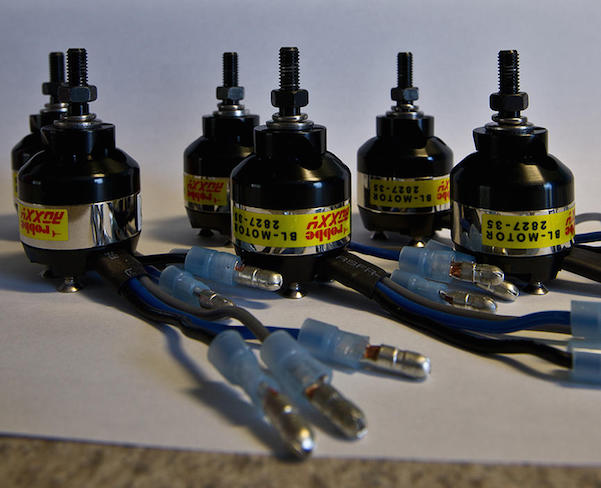}\label{fig:motor}}
}
\caption{Motors and speed controllers}
\end{figure}
Finally, these parts are all connected to the circuit board which has been custom designed and built for the experiment as illustrated in Figure~\ref{fig:hardfellow} and two quadrotors are prepared for the experiments as shown in Figure~\ref{fig:quadss}.
\begin{figure}[h]
\centerline{
	\includegraphics[width=0.6\columnwidth]{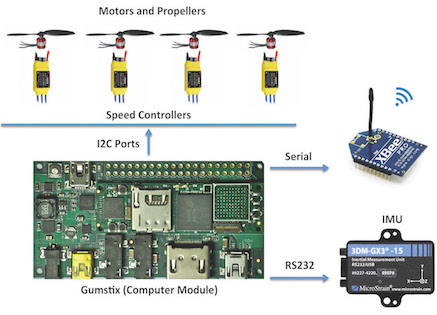}
}
\caption{Hardware connections}\label{fig:hardfellow}
\end{figure}
\begin{figure}[h]
\centerline{
	\subfigure[Quadrotor 1]{
	\includegraphics[width=0.37\columnwidth]{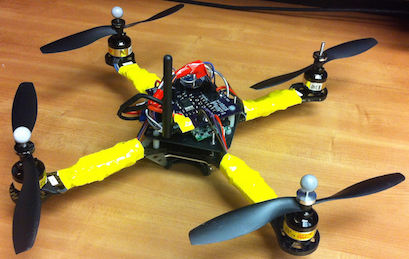}}
	\subfigure[Quadrotor 2]{
	\includegraphics[width=0.45\columnwidth]{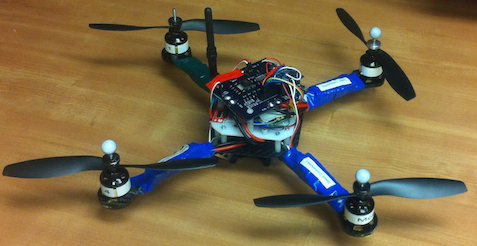}}
}
\caption{Quadrotors}\label{fig:quadss}
\end{figure}

}
{\addtolength{\leftskip}{0.5in}
\subsubsection {\normalsize Software Development}
Multi-threaded programming with C/C++ utilized to develop a complete and fast software which runs on the computer module and handles several tasks. It has three main threads, namely Vicon thread, IMU thread, and control thread. The Vicon thread receives the Vicon measurement and estimates linear velocity of the quadrotor. The IMU thread receives the IMU measurement and estimates the attitude. The last thread handles the control outputs at each time step. Also, control outputs are calculated at 120Hz which is fast enough to run any kind of aggressive maneuvers. Information flow of the system is illustrated in Figures \ref{fig:information_flow} and \ref{fig:information_flow2}.

\begin{figure}[h]
\centerline{
	\includegraphics[width=0.65\columnwidth]{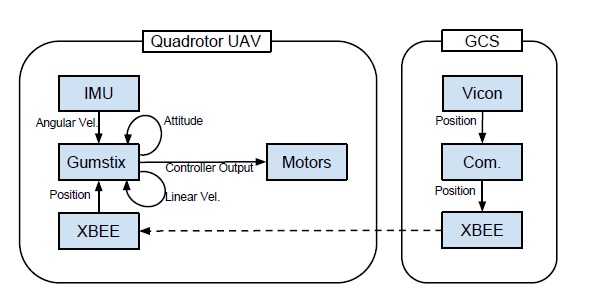}}
\caption{Information flow of overall system}\label{fig:information_flow}
\end{figure}
\begin{figure}[h]
\centerline{
	\includegraphics[width=0.85\columnwidth]{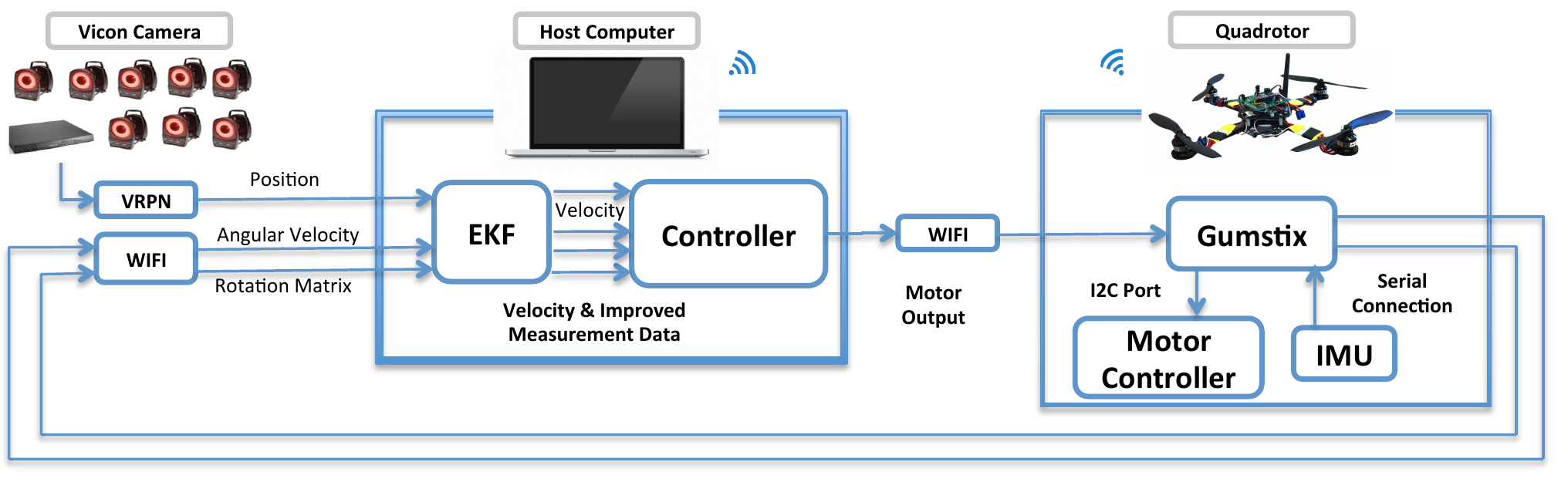}}
\caption{Information flow of overall system}\label{fig:information_flow2}
\end{figure}
We also need to tune the proportional, derivative, and integral gains for the experiment. By using the numerical simulation, we are able to find an estimate for this gains but they need to be tuned during the experiment in order to get the best performance of the controller.
 
}
%%%%%%%%%%%%%%%%%%%%%%%%%%%%%%%%%%%%%%%%%%%%%%%%%%%%%
\subsection {\normalsize Autonomous Trajectory Tracking}
{\addtolength{\leftskip}{0.5in}
\subsubsection {\normalsize Attitude Tracking Test}
Experimental results are provided for the attitude tracking control of a hardware system developed in FDCL. To test the attitude dynamics, it is attached to a spherical joint. As the center of rotation is below the center of gravity, there exists a destabilizing gravitational moment, and the resulting attitude dynamics is similar to an inverted rigid body pendulum. The control input is augmented with an additional term to eliminate the effects of the gravity.

The desired attitude command is described by using 3-2-1 Euler angles, i.e. $R_d(t)=R_d(\phi(t),\theta(t),\psi(t))$, where
$\phi(t) =  \frac{\pi}{9}\sin(\pi t)$, $\theta(t)=\frac{\pi}{9}\cos(\pi t)$, $\psi(t)=0$. This represents a combined rolling and pitching motion with a period of $2$ seconds. The results of the experiment are illustrated at Figures \ref{fig:QuadResult} and \ref{fig:Quadstand}. 
\begin{figure}[h]
\centerline{
	\subfigure[Attitude error function $\Psi$]{
		\includegraphics[width=0.35\columnwidth]{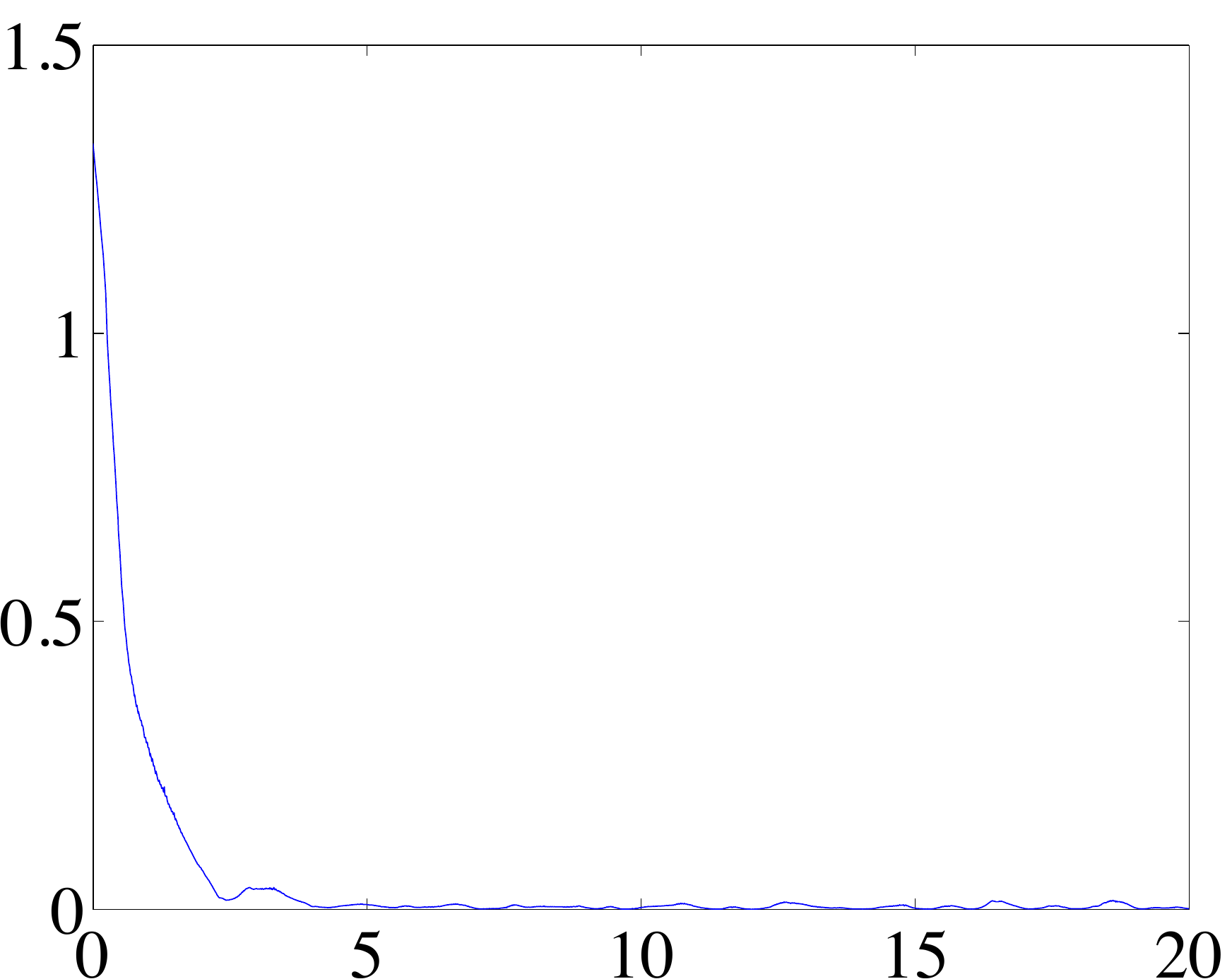}\label{fig:QuadPsi}}
	\subfigure[Attitude $R,R_d$]{
		\includegraphics[width=0.35\columnwidth]{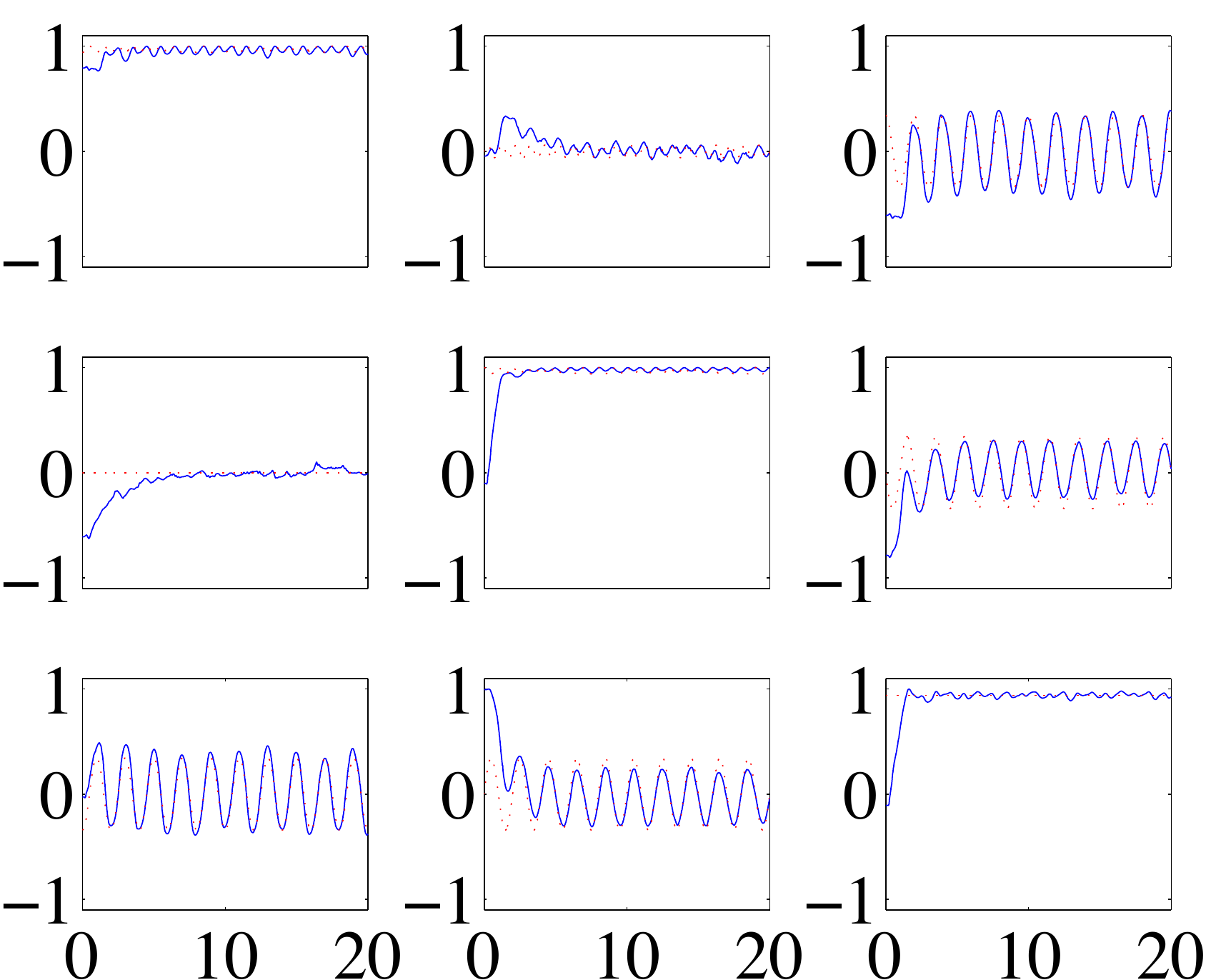}\label{fig:Quadx}}
}
\centerline{
	\subfigure[Angular velocity $\Omega,\Omega_d$ ($\mathrm{/sec}$)]{
		\includegraphics[width=0.35\columnwidth]{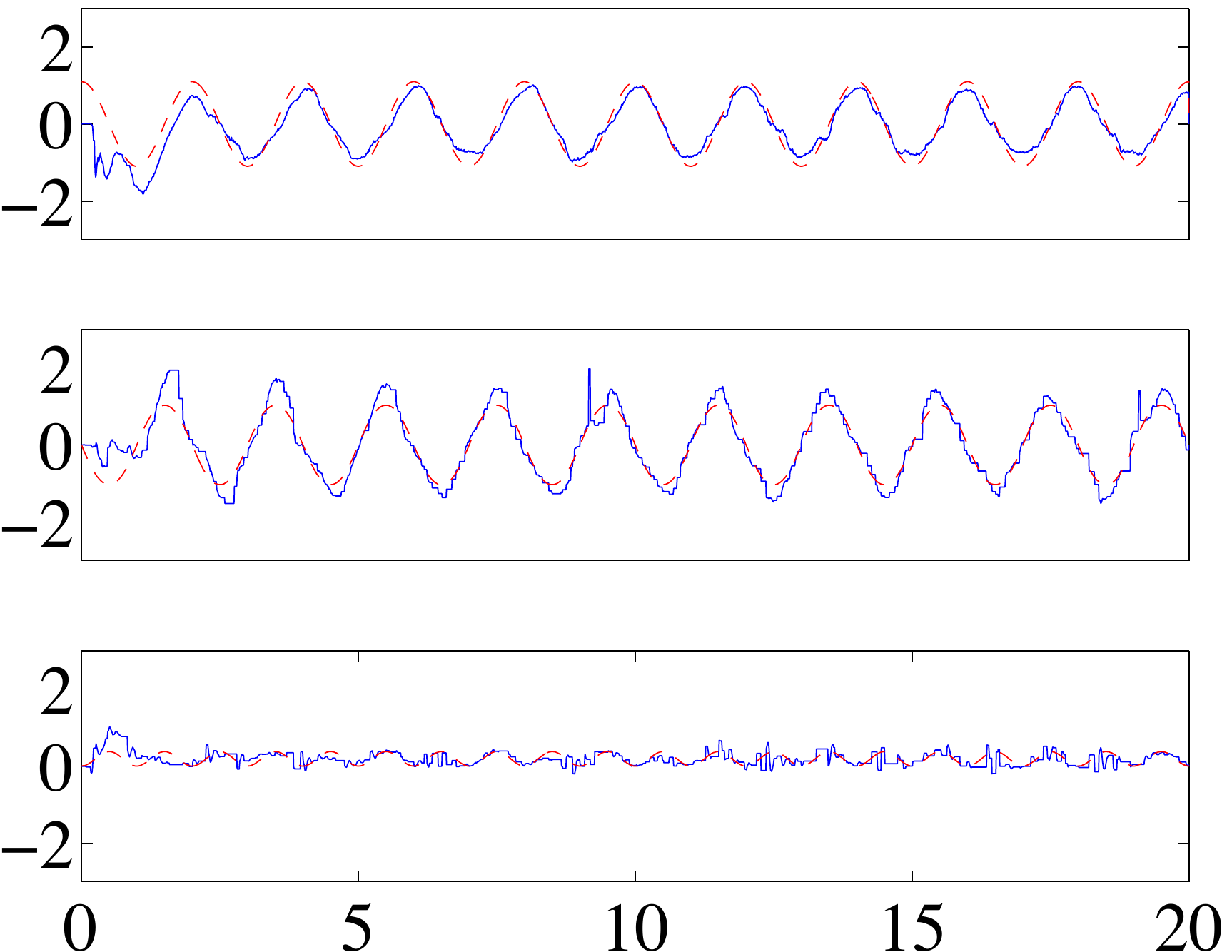}\label{fig:QuadW}}
	\subfigure[Thrust of each rotor ($\mathrm{N}$)]{
		\includegraphics[width=0.37\columnwidth]{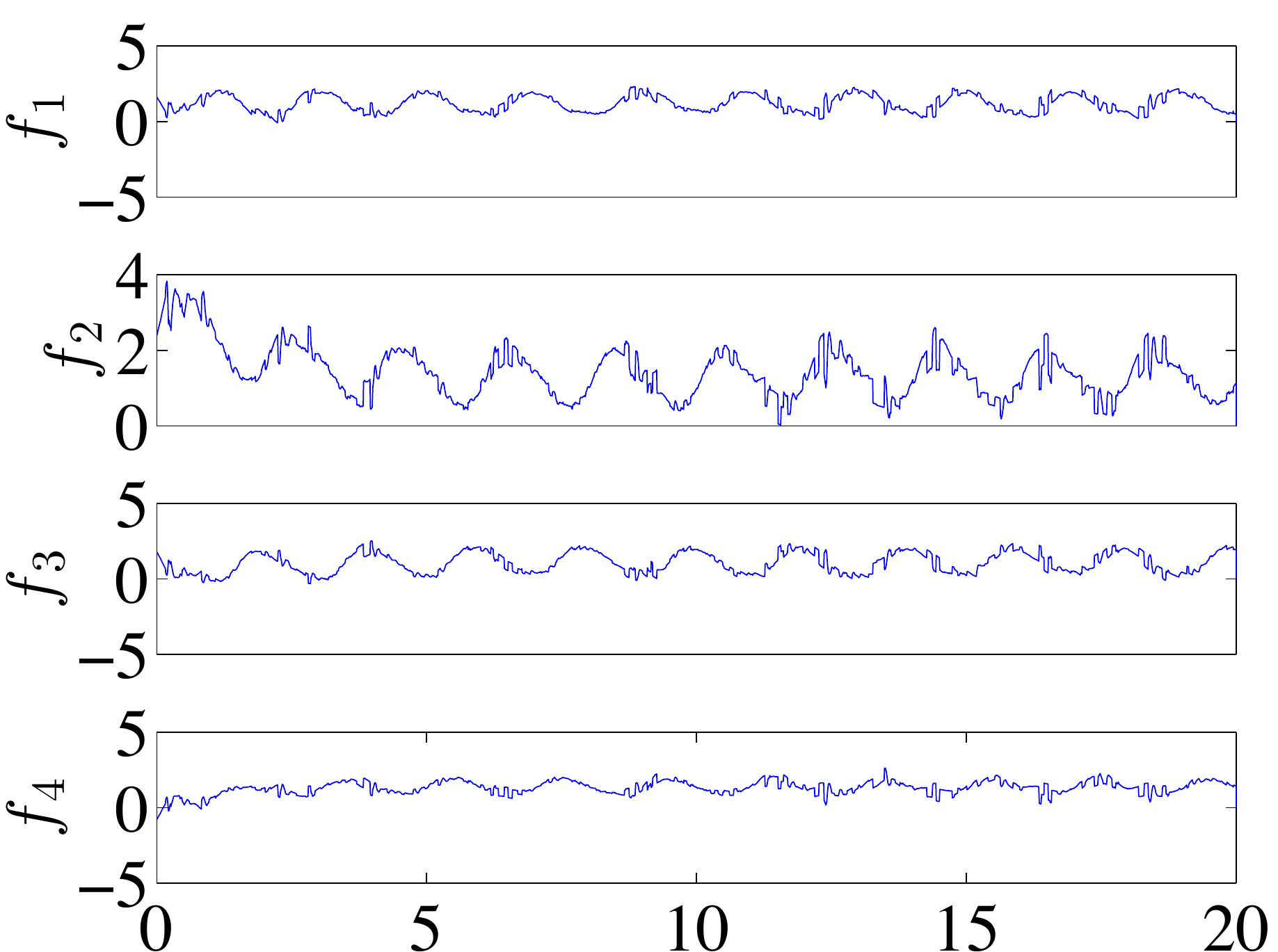}}
}
\caption{Attitude tracking test(dotted:desired, solid:actual)}\label{fig:QuadResult}
\end{figure}
\begin{figure}[h]
\centerline{
\includegraphics[width=0.35\columnwidth]{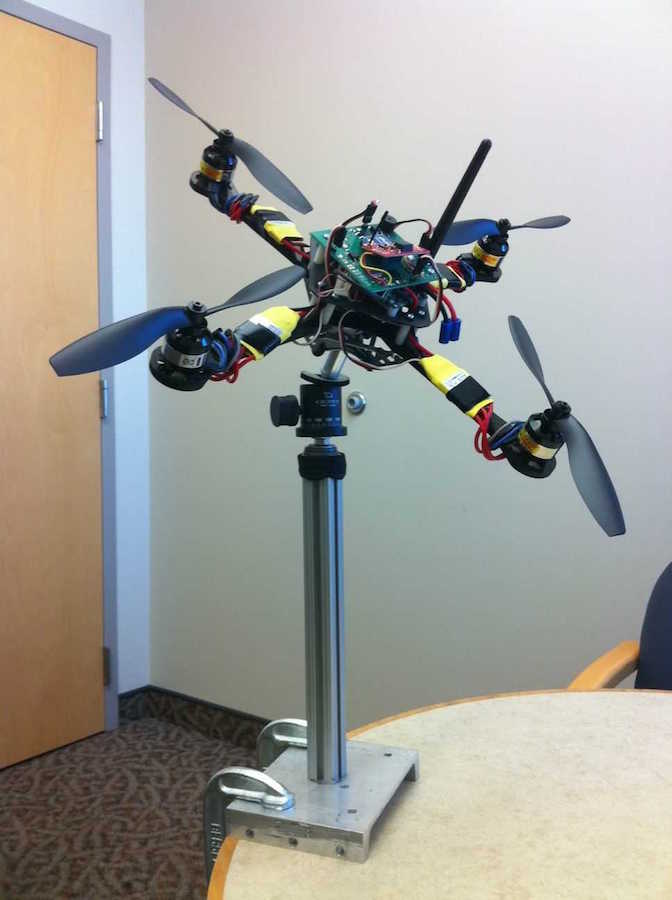}
}
\caption{Quadrotor on a stand}\label{fig:Quadstand}
\end{figure}

\newpage
%%%%%%%%%%%%%%%%%%%%%%%%%%%%%%%%%%%%%%%%%%%%%%%%%%%%%
\subsubsection {\normalsize Lissajous Curve Trajectory Tracking}
First, we consider a position tracking mode, where the desired trajectory is defined as the following Lissajous curve,
\begin{align*}
x_d (t) = \begin{cases}
x_o-\frac{t}{8}(x_o-x_i) &\hspace{-0.57cm} \mbox{if } 0 \leqq t < 8\\
[\sin(t-8)+\frac{\pi}{2}),\;\sin 2(t-8),\; -1.5]\,\mathrm{m} & \mbox{if } 8 \leqq t\\
\end{cases}.
\end{align*} 
The quadrotor takes-off from $x_o=[0.2,\,-2.8,\,-1.2]\mathrm{m}$ at $t=0\;\mathrm{sec}$ and flies to the initial position of the Lissajous curve trajectory where is $x_i=[1,\,0,\,1.5]\mathrm{m}$ by tracking a linear desired trajectory. Then, the quadrotor starts to follow the Lissajous curve trajectory at $t=8\;\mathrm{sec}$. There is about $0.15\;\mathrm{sec}$ of time delay from the Vicon motion capture system to the Gumstix. However, due to the robustness and stability properties of the proposed controller, position tracking performance shows satisfactory results as shown at Figure~\ref{fig:LJ} and \ref{fig:Lissajous_xyzz}.
\begin{figure}[h]
\centerline{		
\includegraphics[width=0.65\columnwidth]{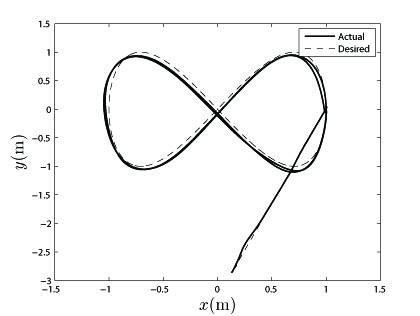}\label{fig:Lissajous_xyzzs}
}
\caption{Lissajous curve $x-y$ plane trajectory}\label{fig:Lissajous_xyzz}
\end{figure}
\begin{figure}
\centerline{
	\subfigure[Attitude error variables $\Psi,e_R,e_\Omega$]{
		\includegraphics[width=0.4\columnwidth]{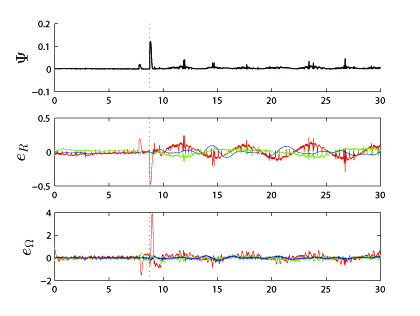}\label{fig:Lissajous_error}}
			\subfigure[Thrust of each rotor ($\mathrm{N}$)]{
				\includegraphics[width=0.4\columnwidth]{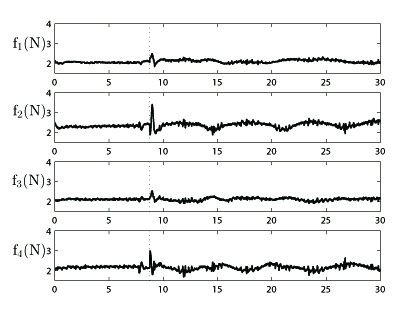}}
}
\centerline{
	\subfigure[Position (solid line) and desired (dotted line) $x,x_d$ ($\mathrm{m}$)]{
		\includegraphics[width=0.4\columnwidth]{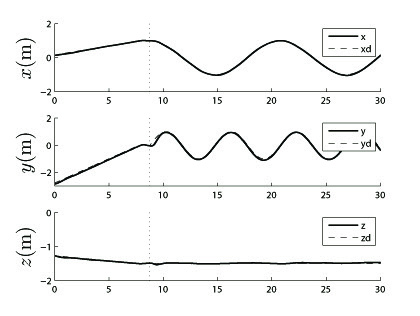}\label{fig:Lissajous_position}}
			\subfigure[Linear velocity ($\mathrm{m/sec}$)]{
				\includegraphics[width=0.4\columnwidth]{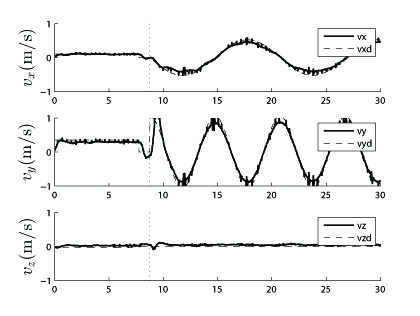}\label{fig:Lissajous_vel}}
}
\centerline{
	\subfigure[Eular angles ($\mathrm{rad}$)]{
		\includegraphics[width=0.4\columnwidth]{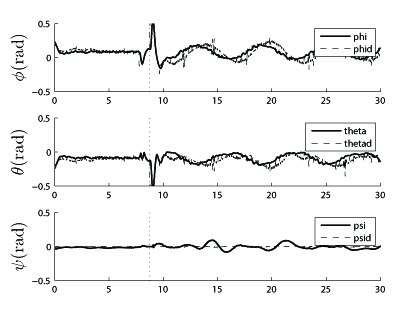}\label{fig:Lissajous_Eular}}
			\subfigure[Angular velocity $\Omega,\Omega_d$ ($\mathrm{rad/sec}$)]{
				\includegraphics[width=0.4\columnwidth]{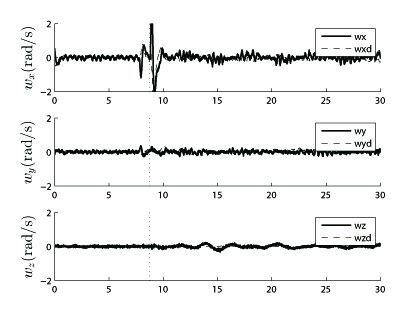}\label{fig:Lissajous_W}}
}
\caption{Lissajous curve tracking(dotted:desired, solid:actual)}\label{fig:LJ}
\end{figure}
\newpage
%%%%%%%%%%%%%%%%%%%%%%%%%%%%%%%%%%%%%%%%%%%%%%%%%%%%%
\subsubsection {\normalsize Flipping 360 Degree Maneuver}
Next, the proposed controller is validated with a flipping maneuver. The quadrotor takes off from a landing platform, increases altitude with constant speed to a constant point, flips $360$ degree about it $x$-axis. As presented in the numerical simulation section, this is a complex maneuver combining a nontrivial pitching maneuver with a yawing motion. It is achieved by concatenating the following two control modes of an attitude tracking same as presented in the numerical simulation to rotate the quadrotor
\begin{align*}
&R_d(t)= I+\sin(4 \pi t)\hat{e}_r+(1-\cos(4 \pi t))(e_r e_r^T-I),\; \Omega_d= 5\pi\cdot e_r.
\end{align*}
where $e_{r}=[1,\; 0,\;0]$, and a trajectory tracking mode to make it hover after completing the preceding rotation. As it is clear from the figures, the attitude control part which handles the rotation happens in almost $0.3$ seconds and then it switched to the position control mode to make the quadrotor stabilized and hovers to the desired position. Figure \ref{fig:ffgghhjjhhgg} and \ref{fig:snapflip} show the experimental results and snapshots of the flipping maneuver respectively.

There are several disturbances and modeling errors in this experimental setup, such as errors in mass properties of the quadrotor, processing and communication delay of the motion capture systems, and the dynamics of propellers. Therefore, these experimental illustrate robustness of the proposed control system with respect to various forms of uncertainties and disturbances\footnote{A short video of the experiments is available at \url{http://youtu.be/wtn9L6BsYiE}.}.

\begin{figure}
\centerline{
	\subfigure[Attitude error variables $\Psi,e_R,e_\Omega$]{
		\includegraphics[width=0.4\columnwidth]{errors.pdf}\label{fig:hover_error}}
		\subfigure[Thrust of each rotor ($\mathrm{N}$)]{
				\includegraphics[width=0.4\columnwidth]{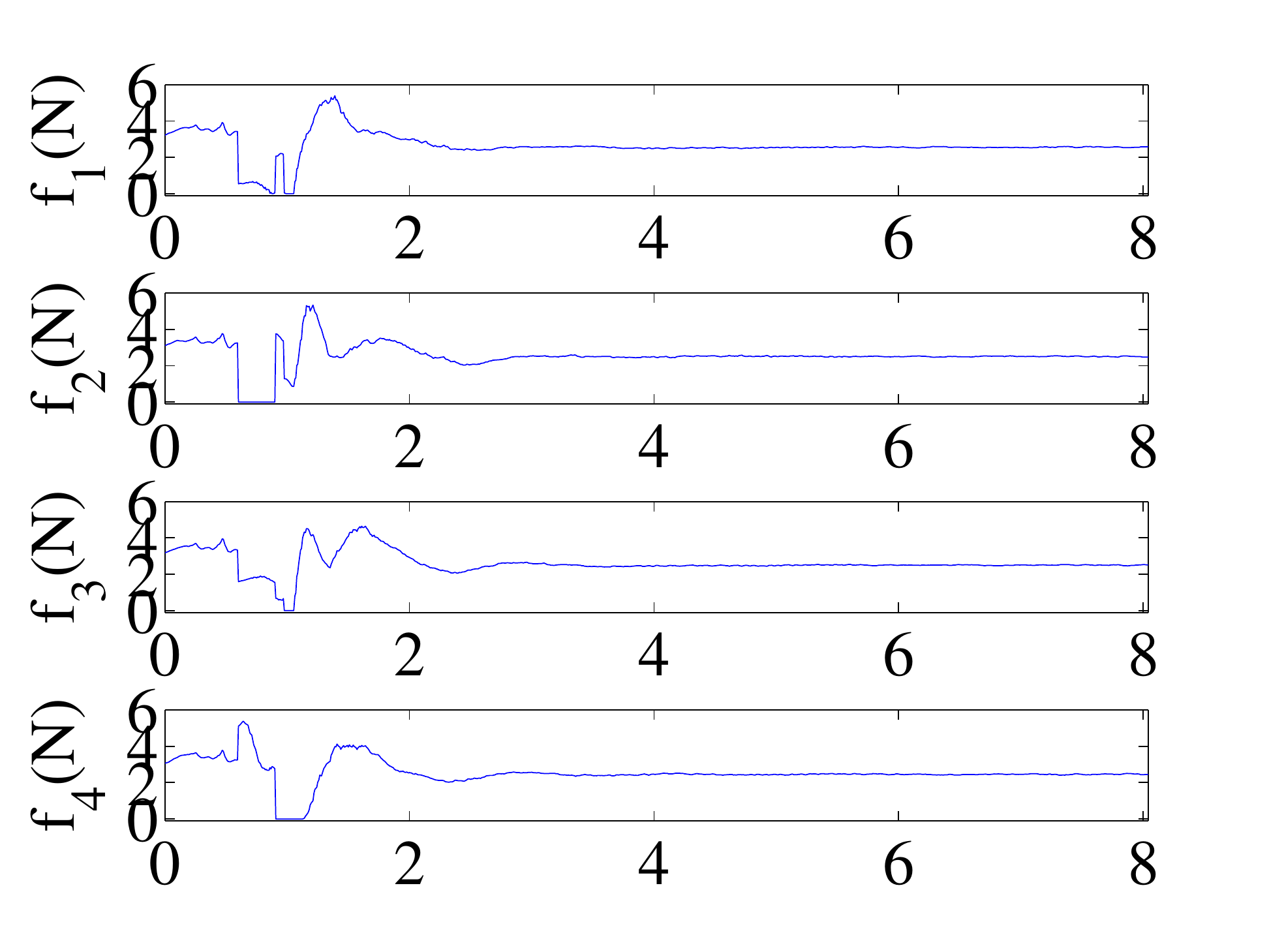}}
}
\centerline{
\subfigure[Position $x,x_d$ ($\mathrm{m}$)]{
		\includegraphics[width=0.4\columnwidth]{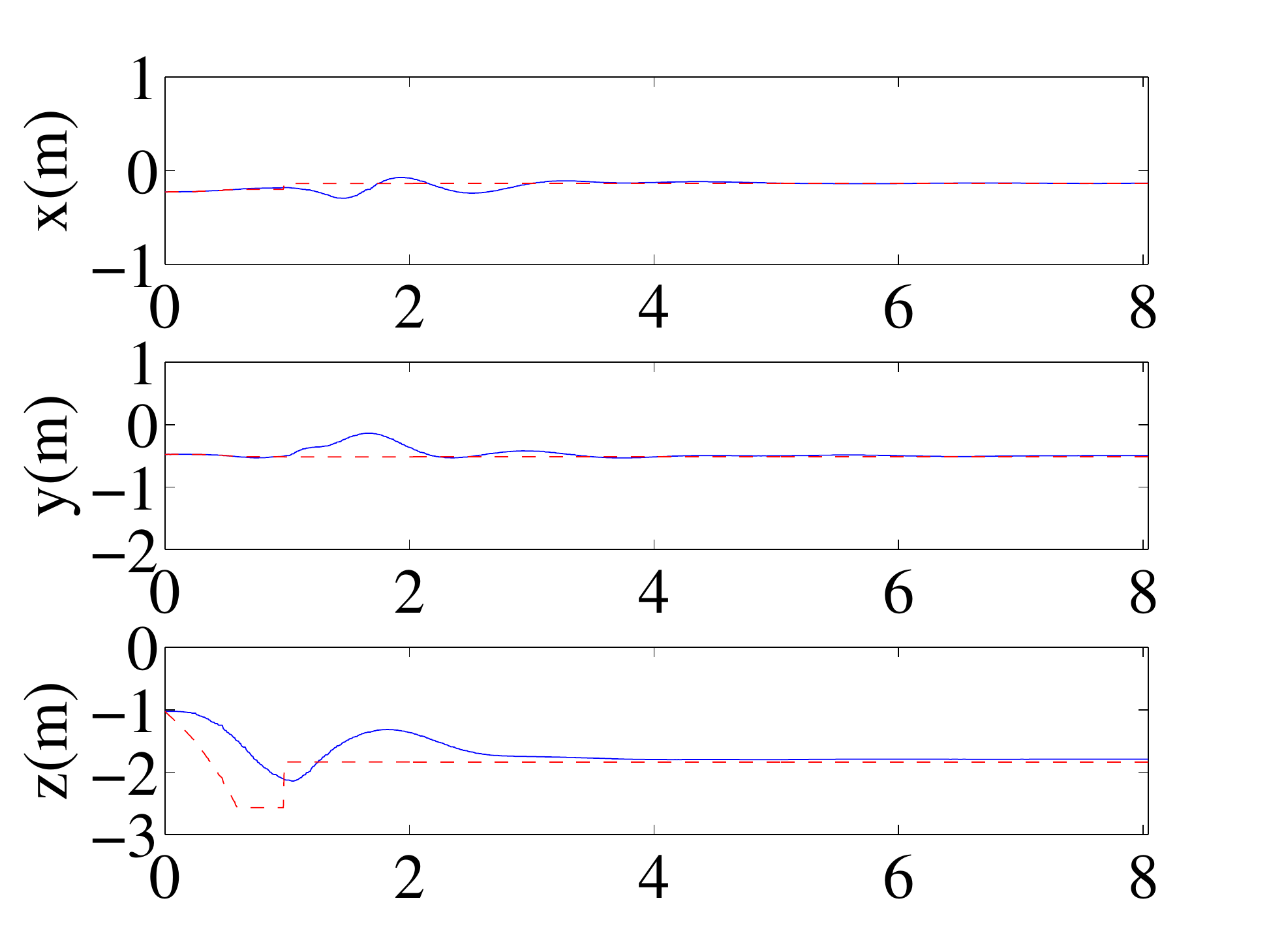}\label{fig:hover_position}}
			\subfigure[Linear velocity ($\mathrm{m/sec}$)]{
				\includegraphics[width=0.4\columnwidth]{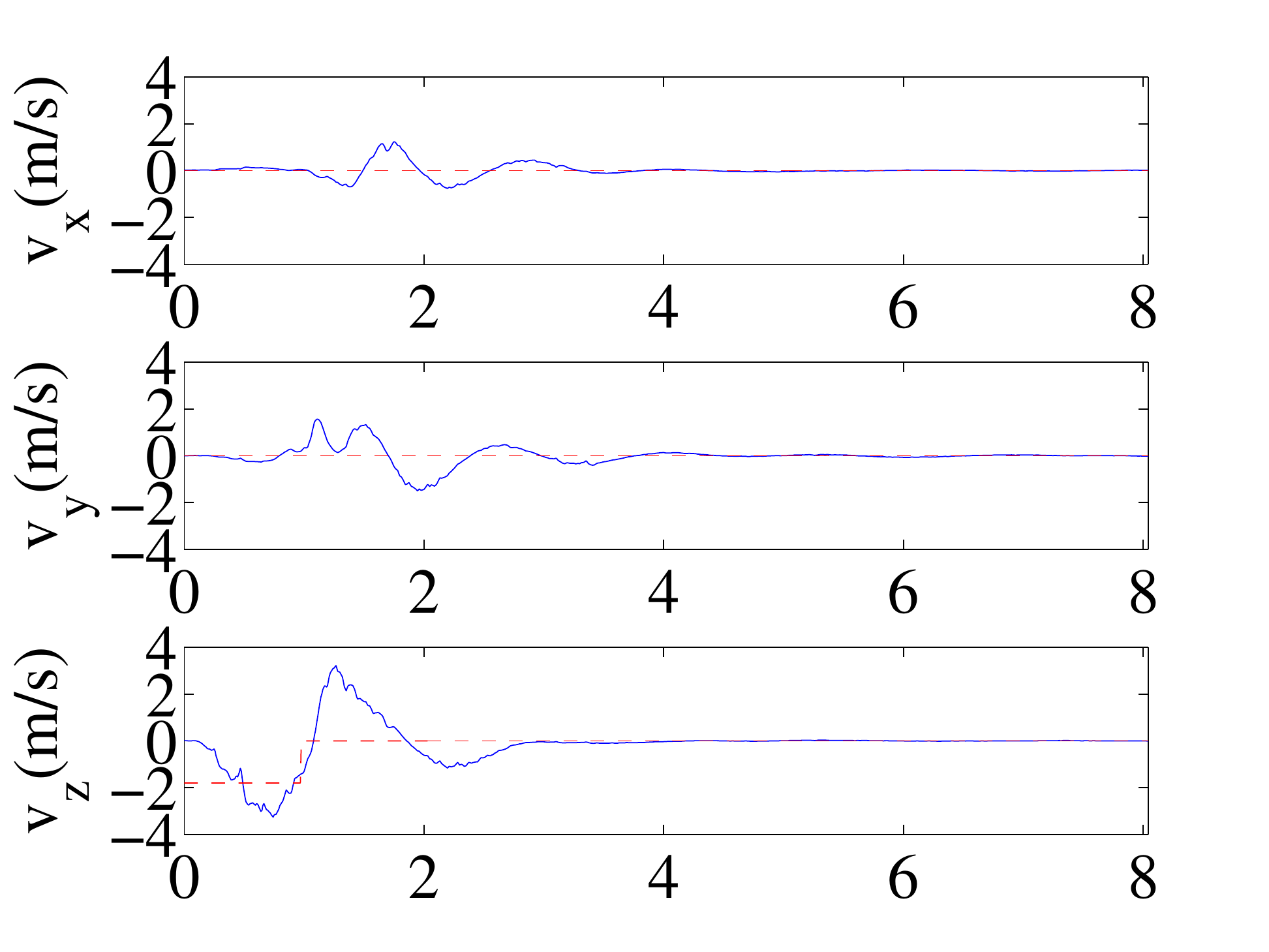}\label{fig:hover_vel}}
}
\centerline{
	\subfigure[Rotation Matrix ]{
		\includegraphics[width=0.4\columnwidth]{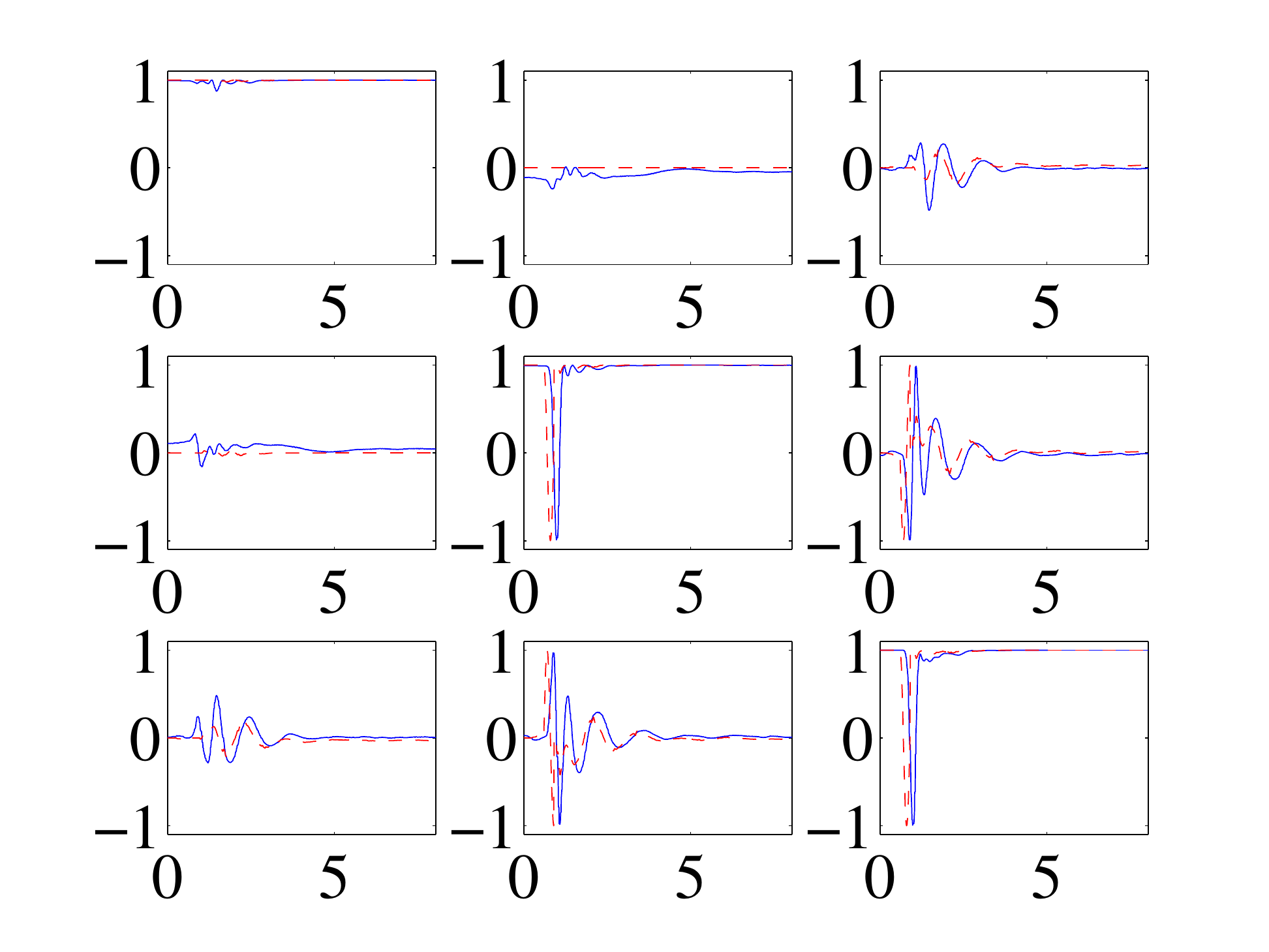}\label{fig:hover_Eular}}
			\subfigure[Angular velocity $\Omega,\Omega_d$ ($\mathrm{rad/sec}$)]{
				\includegraphics[width=0.4\columnwidth]{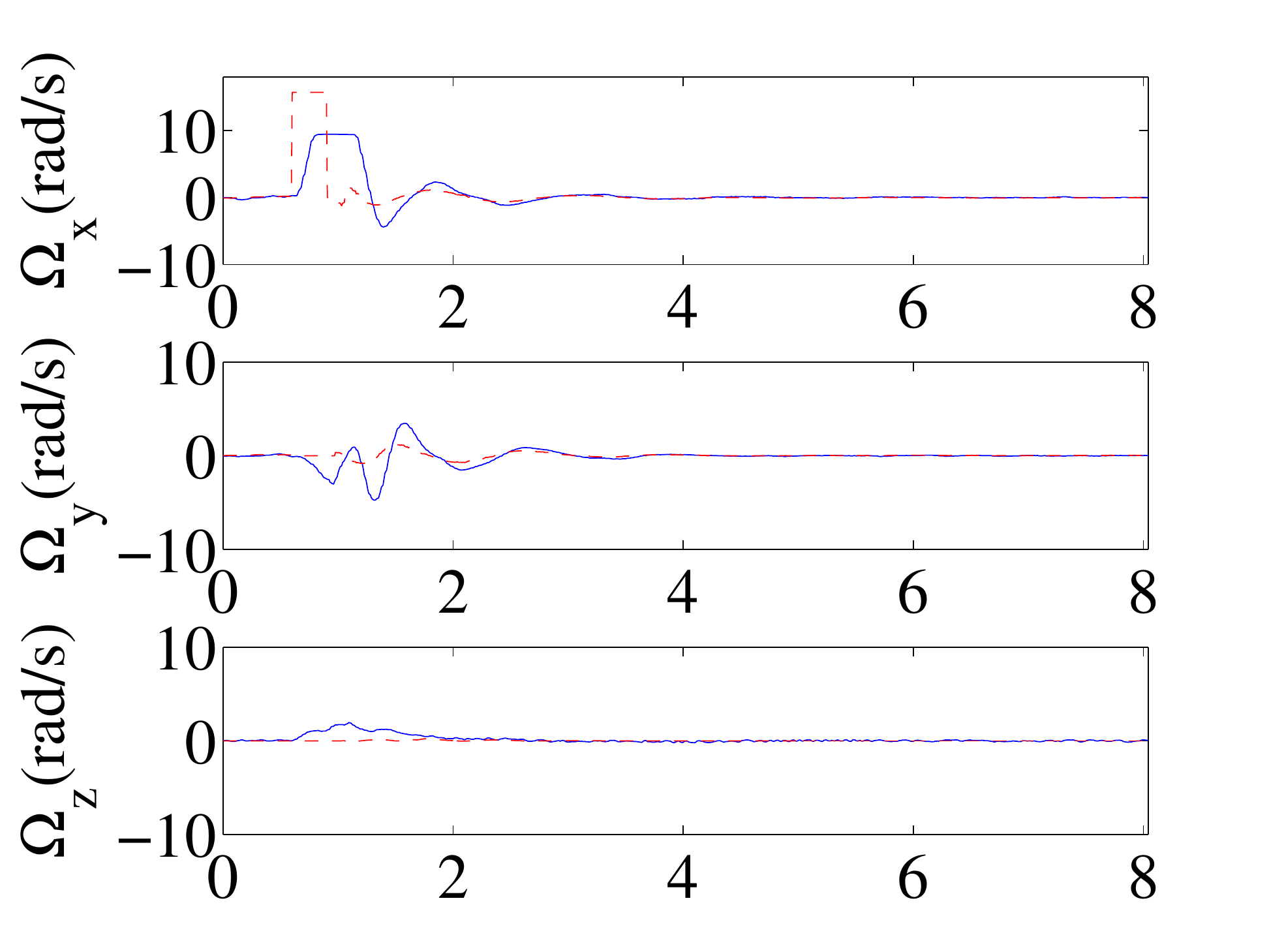}\label{fig:hover_W}}
}
\caption{Flipping flight test results (dotted:desired, solid:actual)}\label{fig:ffgghhjjhhgg}
\end{figure}
\begin{figure}[h]
\centerline{
	\subfigure[$t=0.0$ sec]{
		\includegraphics[width=0.250\columnwidth]{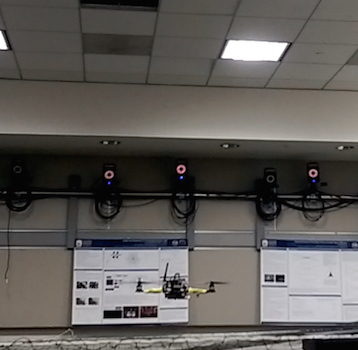}}
		\subfigure[$t=0.8756$ sec]{
				\includegraphics[width=0.250\columnwidth]{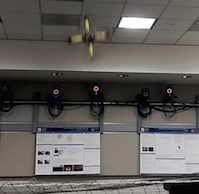}}
		\subfigure[$t=1.008$ sec]{
				\includegraphics[width=0.250\columnwidth]{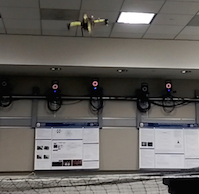}}
}
\centerline{
\subfigure[$t=1.079$ sec]{
		\includegraphics[width=0.250\columnwidth]{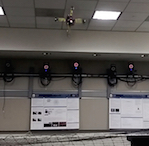}}
			\subfigure[$t=1.125$ sec]{
				\includegraphics[width=0.250\columnwidth]{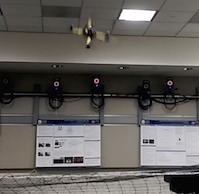}}
		\subfigure[$t=1.175$ sec]{
				\includegraphics[width=0.250\columnwidth]{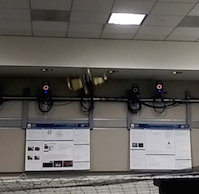}}
}
\centerline{
	\subfigure[$t=1.844$ sec]{
		\includegraphics[width=0.250\columnwidth]{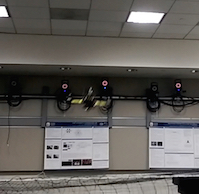}}
			\subfigure[$t=2.312$ sec]{
				\includegraphics[width=0.250\columnwidth]{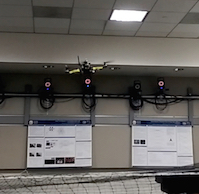}}
		\subfigure[$t=2.89$ sec]{
				\includegraphics[width=0.250\columnwidth]{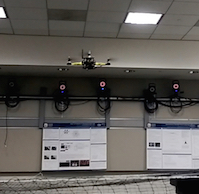}}
}
\caption{Snapshots for flipping maneuver.}
\label{fig:snapflip}
\end{figure}

}
\newpage
%%%%%%%%%%%%%%%%%%%%%%%%%%%%%%%%%%%%%%%%%%%%%%%%%%%%%
\subsection {\normalsize Payload Stabilization with one Quadrotor}
{\addtolength{\leftskip}{0.5in}
The weight of the entire UAV system is $0.791\mathrm{kg}$ including one battery.  A payload with mass of $m_1=0.036\ \mathrm{kg}$ is attached to the quadrotor via a cable of length $l_1=0.7\ \mathrm{m}$. The length from the center of the quadrotor to each motor rotational axis is $d=0.169\mathrm{m}$, the thrust to torque coefficient is $c_{{\tau}_f}=0.1056\mathrm{m}$ and the moment of inertia is $J=[0.56,0.56,1.05]\times 10^{-2}\,\mathrm{kgm^2}$. %The angular velocity is measured from inertial measurement unit (IMU) and the attitude is estimated from IMU data. Position of the UAV is measured from motion capture system (Vicon) and the velocity is estimated from the measurement. Ground computing system receives the Vicon data and send it to the UAV via XBee. The Gumstix is adopted as micro computing unit on the UAV. It has two main threads, Vicon thread and IMU thread. The Vicon thread receives the Vicon measurement and estimates linear velocity of the quadrotor and runs at 30Hz. In IMU thread, it receives the IMU measurement and estimates the angular velocity. Also, control outputs are calculated at 120Hz in this thread.

Two cases are considered and compared. For the first case, a position control system developed in~\cite{Farhad2013}, for quadrotor UAV that does not include the dynamics of the payload and the link, is applied to hover the quadrotor at the desired location, and the second case, the proposed control system is used.

\begin{figure}
\centerline{
\subfigure[Case I: quadrotor position control system~\cite{Farhad2013}]
{\includegraphics[width=0.7\columnwidth]{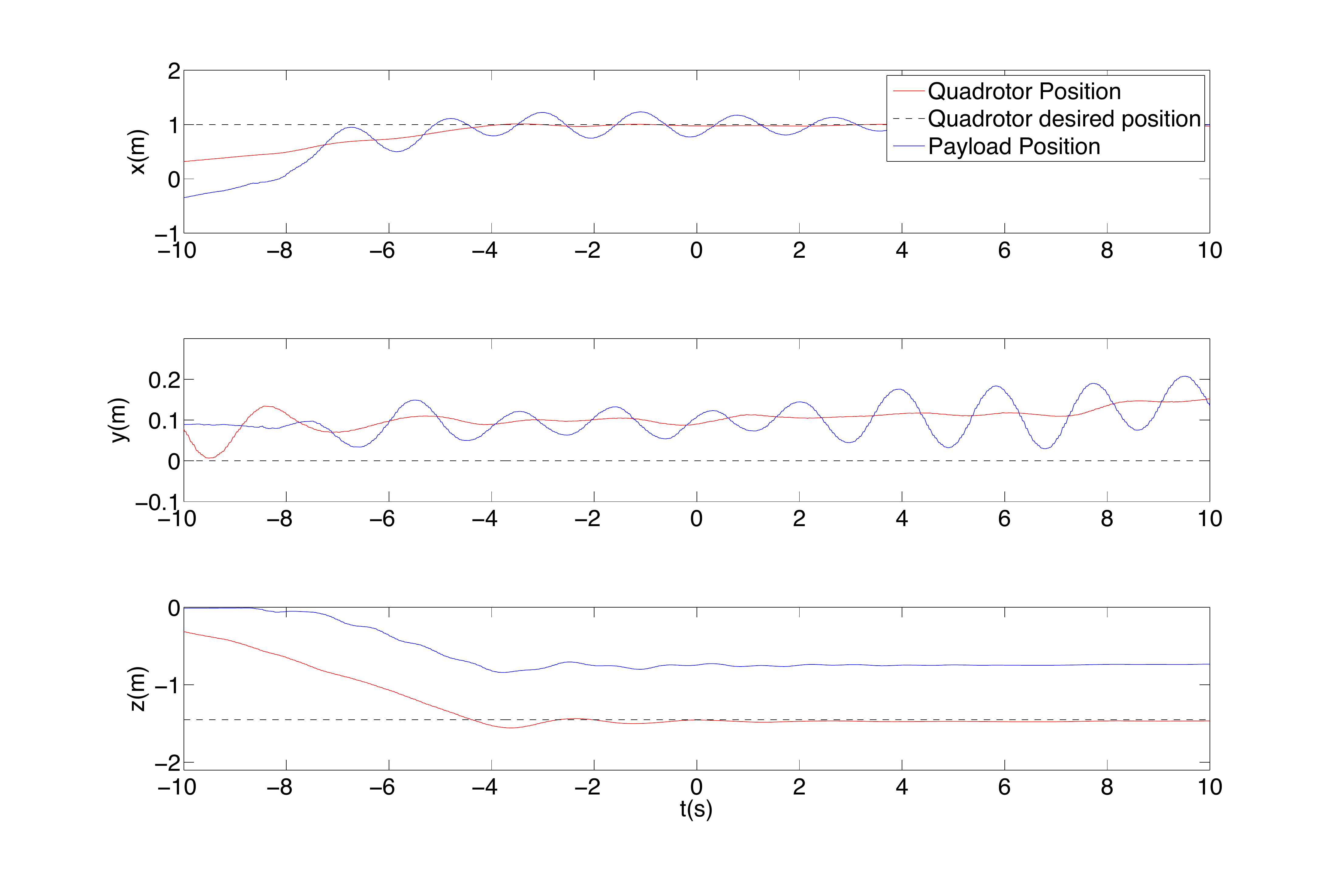}}}
\centerline{
\subfigure[Case II: proposed control system for quadrotor with suspended payload]
{\includegraphics[width=0.7\columnwidth]{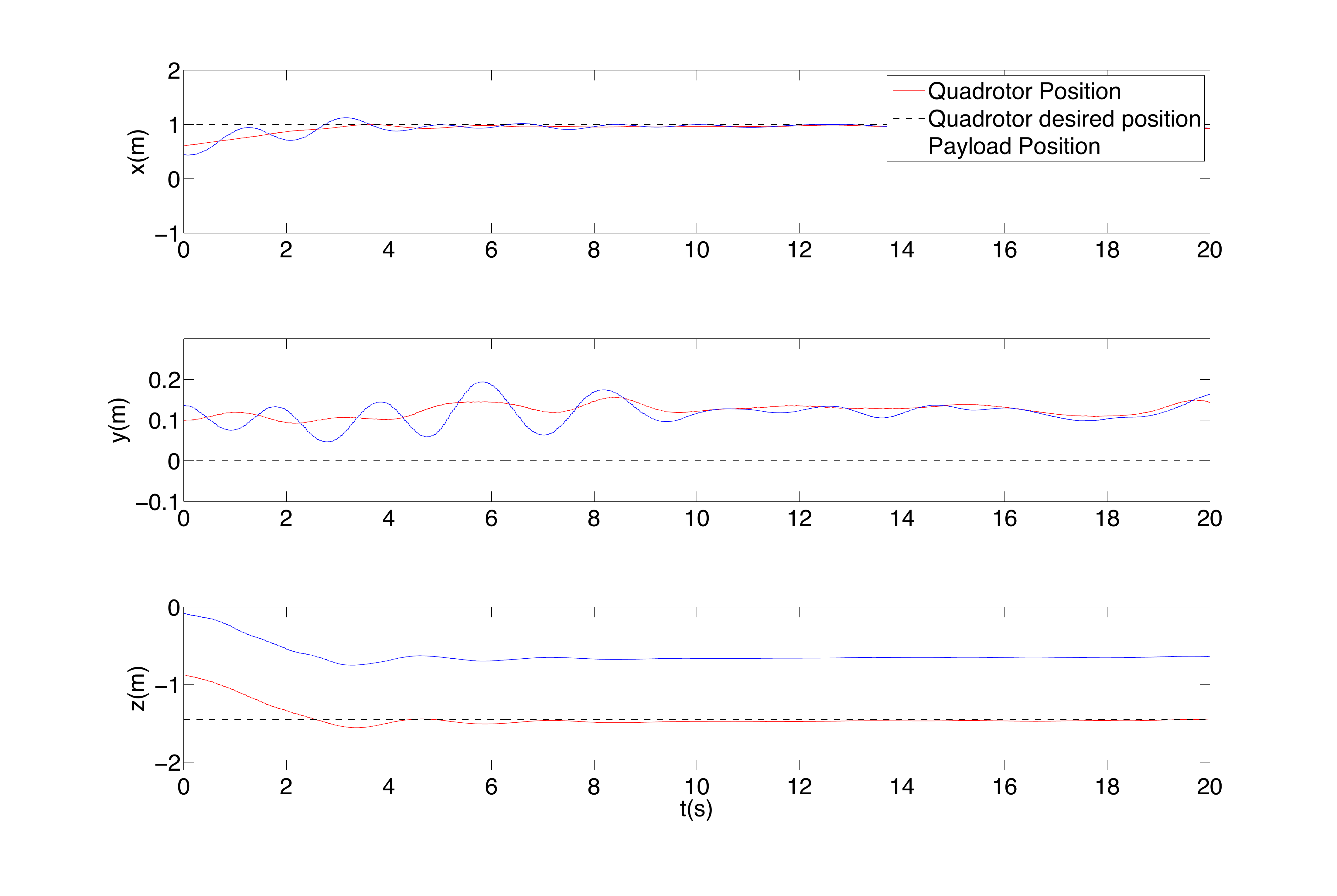}}}
\caption{Experimental results ($x_d$:black, $x$:red, $x+l_1q_1$:blue)}
\label{expresultsp}
\end{figure}

Experimental results are shown at Figures \ref{expresultsp} and \ref{expresultsq}. The position of the quadrotor and the payload is compared with the desired position of the quadrotor at Figure~\ref{expresultsp}, and the deflection angle of the link from the vertical direction are illustrated at Figure~\ref{expresultsq}. It is shown that the proposed control system reduces the undesired oscillation of the link effectively, compared with the quadrotor position control system. \footnote{A short video of the experiments is available at \url{http://youtu.be/RyTmWVbgt34}.}

\begin{figure}
\centerline{
	\includegraphics[width=0.4\columnwidth]{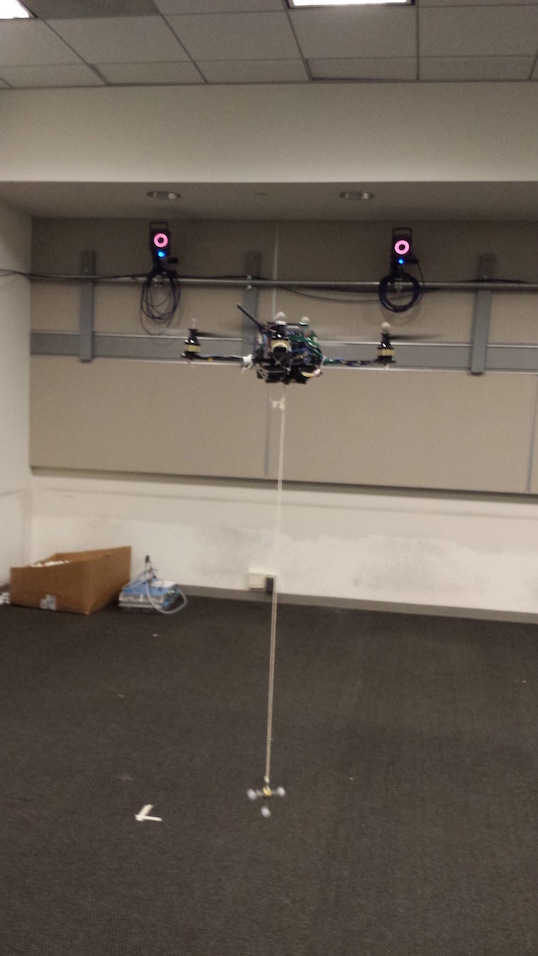}
}
\caption{Snapshot of a quadrotor UAV stabilizing a payload}\label{fig:QuadHW}
\end{figure}

\begin{figure}
\centerline{
\subfigure[Case I: quadrotor position control system~\cite{Farhad2013}]
{\includegraphics[width=0.4\columnwidth]{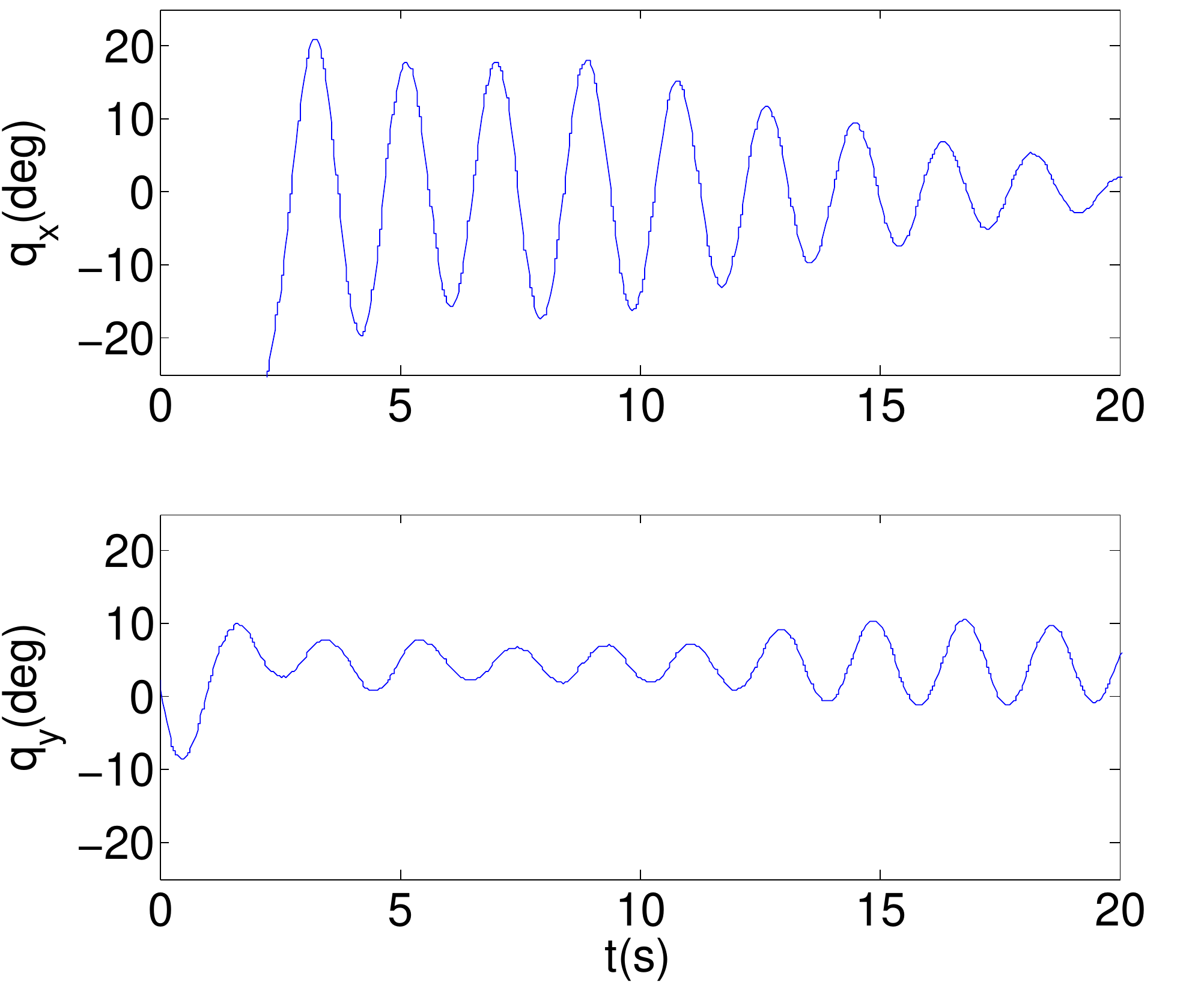}}
\subfigure[Case II: proposed control system for quadrotor with suspended payload]
{\includegraphics[width=0.4\columnwidth]{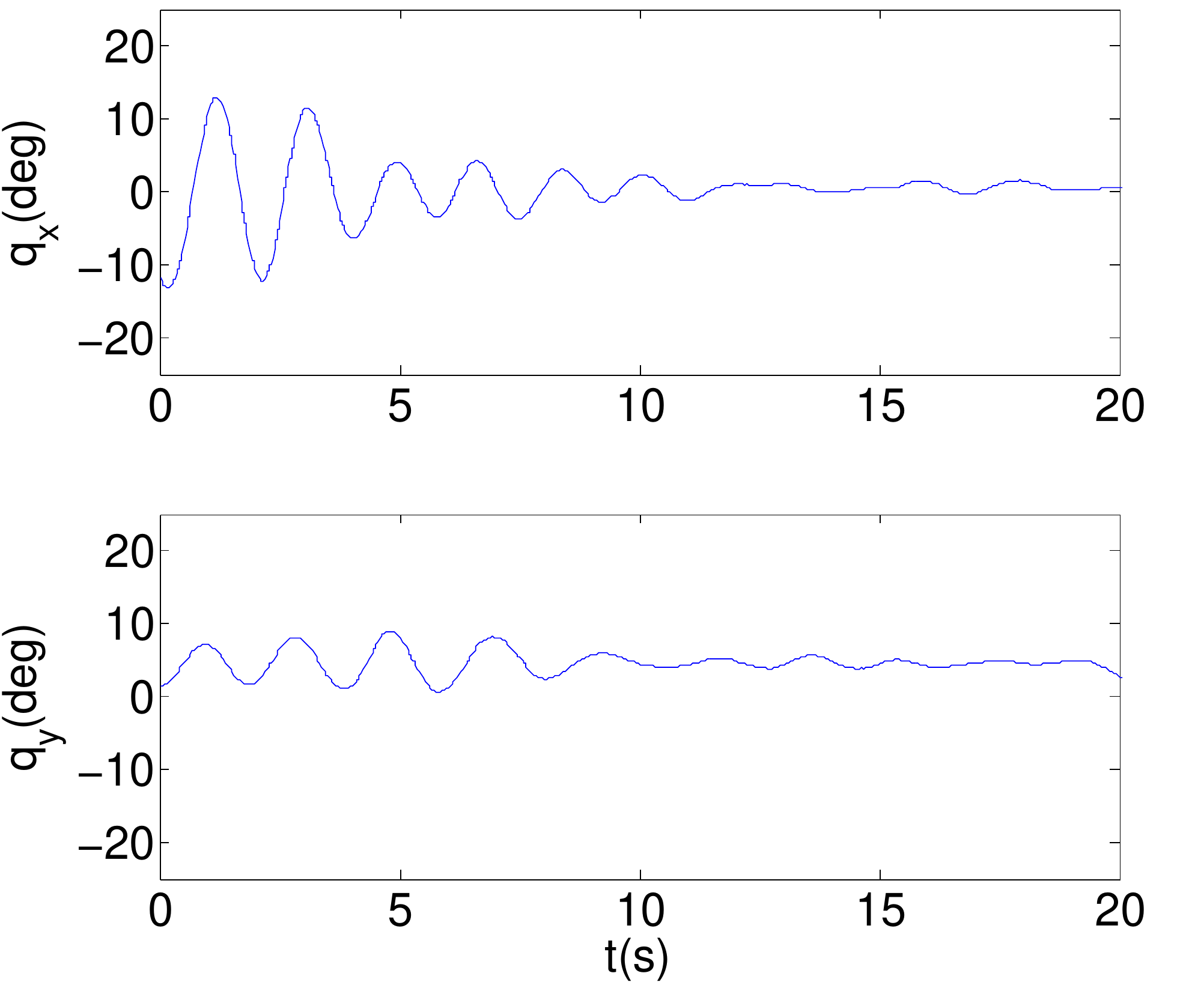}}
}
\caption{Experimental results: link deflection angles}
\label{expresultsq}
\end{figure}
There are several disturbances and modeling errors in this experimental setup, such as errors in mass properties of the quadrotor, processing and communication delay of the motion capture systems, and ground effects of air flow. Therefore, these experimental illustrate robustness of the proposed control system with respect to various forms of uncertainties and disturbances. Generalizing the presented control system for advanced nonlinear adaptive or robust controls are relegated to future investigation.

}

\newpage
\begin{singlespace}
\section{\protect \centering Chapter 6: Conclusions}
\end{singlespace}
\setcounter{section}{6}
We conclude our research and work of this dissertation here. In this section, first a summary of conclusions is given and it is followed by the future work.
\doublespacing
%%%%%%%%%%%%%%%%%%%%%%%%%%%%%%%%%%%%%%%%%%%%%%%%%%%%%
\subsection {\normalsize Conclusions}
{\addtolength{\leftskip}{0.5in}
A new nonlinear adaptive control system is proposed for tracking control of quadrotor unmanned aerial vehicles. It is developed directly on the special orthogonal group to avoid complexities and ambiguities that are associated with Euler-angles or quaternions, and the proposed adaptive control term guarantees almost global attractively for the tracking error variables in the existence of uncertainties.\\

Euler-Lagrange equations have been derived for multiple cooperative quadrotor UAVs and the chain pendulum to model each flexible cable transporting a rigid body in 3D space. These derivations developed in a remarkably compact form which allows us to choose arbitrary number and any configuration of the links. We developed a geometric nonlinear controller to stabilize the links below the quadrotor in the equilibrium position from any chosen initial condition. We expanded these derivations in such way that there is no need of using local angle coordinate. Rigorous mathematical stability proofs are given without any time-scale separation assumptions and these are verified by mathematical analysis, numerical simulations, and experiments for aggressive maneuvers concurrently.

}
\newpage
%%%%%%%%%%%%%%%%%%%%%%%%%%%%%%%%%%%%%%%%%%%%%%%%%%%%%
\subsection {\normalsize Challenges $\&$ Future Work}
{\addtolength{\leftskip}{0.5in}
Several limitations and challenges are considered for the proposed dynamic model, controller and experimental testbed. Some of the limitations, important hardware challenges, and future directions are listed bellow 
\begin{itemize}
{\addtolength{\leftskip}{0.1in}
\item Communication is very important in real-time experiments. We need high speed update rate in order to maintain an aggressive maneuver, and any failure in communication causes instability of the system and would lead to incidents.
\item Batteries used for the experiment can play an important role in a successful maneuver. They can be drained very soon and it affects the performance of the hardware and the gain tuning process. 
\item Different sensors are used in real-time experiments. For example, an IMU used to obtain the angular velocities, and Vicon Motion Capture system to obtain the position of the object. In both, sensors provide the noisy data which needed to be filtered or estimated to obtain smooth data. Furthermore, we experienced delay in receiving the data using the Vicon system. 
\item The flexible cables are considered into the dynamic of the system by modeling them as serially connected links. As more as links get used in the cable, the dynamic model becomes more accurate. However, additional sensors are needed to measure the deformation of the cable accurately. 
\item The proposed cable model via interconnected lumped masses is readily generalized to uniform mass distribution at each link. Based on the assumption that the diameter of the cable is much smaller than its length, the torsion of the bending moment are not considered in this dissertation.
\item Payload stabilization with single and multiple quadrotor UAVs to a fixed position are presented in this work. Dynamics model can be used and new controller can be extended and designed for tracking time-varying trajectories.
\item Experimental results for payload trajectory tracking is another future direction of this dissertation. There are several challenges in running experiments with multiple quadrotors simultaneously and tracking a desired trajectory for a payload with multiple quadrotor UAV's to validate the derivation for time-continuous controller.
\item Future directions include extending this research in an outdoor environment to identify the wind effects and verifying the performance of the controller in various weather situations. A GPS system can be installed on the quadrotors instead of the Motion Capture system in a laboratory to identify the position and velocity of the quadrotor.

}
\end{itemize}

}

%%%%%%%%%%%%%%%%%%%%%%%%%%%%%%%%%%%%%%%%%%%%%%%%%%%%%
%%%%%%%%%%%%%%%%%%%%%%%%%%%%%%%%%%%%%%%%%%%%%%%%%%%%%
%%%%%%%%%%%%%%%%%%%%%%%%%%%%%%%%%%%%%%%%%%%%%%%%%%%%%

\newpage
\section{\protect\centering Bibliography}
\doublespacing
\renewcommand{\refname}{}
\bibliographystyle{asmems4}
\bibliography{refrences}

%%%%%%%%%%%%%%%%%%%%%%%%%%%%%%%%%%%%%%%%%%%%%%%%%%%%%
%%%%%%%%%%%%%%%%%%%%%%%%%%%%%%%%%%%%%%%%%%%%%%%%%%%%%
%%%%%%%%%%%%%%%%%%%%%%%%%%%%%%%%%%%%%%%%%%%%%%%%%%%%%

\newpage
\renewcommand{\thesubsection}{A.\arabic{subsection} }
\setcounter{subsection}{0}
\section{\protect\centering Appendix- Proofs}

\subsection{\protect Proof for Proposition \ref{prop:propchap1_first}}\label{sec:pfchap1_first}{
The proofs of (i)-(iii) are available at~\cite[(Chap. 11)]{BulLew05}. To show (iv) and (v), let $Q=R_{d}^{T}R=\exp\hat{x}\in\SO$ for $x\in\Re^{3}$ from Rodrigues' formula. Using the MATLAB Symbolic Computation Tool, we find 
\begin{gather}
\Psi=\frac{1-\cos \|x\|}{2\|x\|^2}\sum_{(i,j,k)\in \CCC}{(g_{i}+g_{j})x_{k}^{2}},\nonumber\\
\|e_{R}\|^{2}=\frac{(1-\cos \|x\|)^{2}}{4\|x\|^{4}}\sum_{(i,j,k)\in \CCC}{(g_{i}-g_{j})^{2}x_{i}^{2}x_{j}^{2}}+\frac{\sin^{2}\|x\|}{4\|x\|^{2}}\sum_{(i,j,k)\in \CCC}{(g_{i}+g_{j})^{2}x_{k}^{2}}\nonumber,
\end{gather}
where $\CCC =\{(1,2,3),(2,3,1),(3,1,2)\}$. When $\Psi=0$, Eq.~\ref{eqn:PsiLB} is trivial. Assuming that $\Psi\neq0$, therefore $\|x\|\neq0$, an upper bound of $\frac{\|e_{R}\|^2}{\Psi}$ is given by
\begin{align}
\frac{\|e_{R}\|^{2}}{\Psi}\leq\frac{1}{2h_{1}}(1-\cos \|x\|)h_{2}+\frac{1}{2h_{1}}(1+\cos\|x\|)h_{3}\leq\frac{h_{2}+h_{3}}{h_{1}}\nonumber,
\end{align}
which shows Eq.~\ref{eqn:PsiLB}.

Next, we consider (v). When $\Psi=0$, Eq.~\ref{eqn:PsiUB} is trivial. Hereafter, we assume $\Psi\neq0$, therefore $R\neq R_{d}$. At the three remaining critical points of $\Psi$, the values of $\Psi$ are given by $g_1 + g_2$, $g_2 + g_3$, or $g_3 + g_1$. So, from the given bound $\Psi<\psi$, these three critical points are avoided, and we can guarantee that $e_{R}\neq0$ and $\|x\|<\pi$. An upper bound of $\frac{\Psi}{\|e_{R}\|^2}$ is given by
\begin{align}\label{eqn:salamsalamazizam}
\frac{\Psi}{\|e_{R}\|^{2}}\leq\frac{2(1-\cos\|x\|)}{\sin\|x\|^{2}}\frac{\sum_{\CCC}(g_{i}+h_{j})x_{k}^{2}/\|x\|^2}{\sum_{\CCC}(g_{i}+h_{j})^{2}x_{k}^{2}/\|x\|^2}\leq\frac{2}{1+\cos\|x\|}\frac{h_{4}}{h_{5}},
\end{align}
Also, an upper bound of $2-\psi$ is given by
\begin{align}
h_{1}-\psi<h_{1}-\Psi\leq h_{1}-\frac{1-\cos \|x\|}{2}h_{1}=\frac{h_{1}}{2}(1+\cos \|x\|).
\end{align}
Substituting this into Eq.~\ref{eqn:salamsalamazizam}, we have
\begin{align}
\frac{\Psi}{\|e_{R}\|^{2}}\leq\frac{h_{4}h_{1}}{h_{5}(h_{1}-\psi)}=h_{2}\nonumber,
\end{align}
which shows Eq.~\ref{eqn:PsiUB}.

For a given tracking command $(R_{d},\Omega_{d})$, and current attitude and angular velocity $(R,\Omega)$, we define an attitude error function $\Psi\  :\ SO(3)\times SO(3)\in R$ as
\begin{equation}
\Psi(R,R_{d})=\frac{1}{2}tr[I-R_{d}^{T}R].
\end{equation}
The desired angular velocity can be written as follows
\begin{equation}
\hat{\Omega}_{d}=R_{d}^{T}\dot{R}_{d}.
\end{equation}
We take the derivative to find the $e_{R}$ and $e_{\Omega}$
\begin{equation}
D\Psi(R,R_{d})=-\frac{1}{2}tr[d(R_{d}^{T}R)]=-\frac{1}{2}tr[R_{d}^{T}R\hat{\eta}].
\end{equation}
Using
\begin{equation}\label{eq50}
tr(\hat{x}A)=tr[A\hat{x}]=\frac{1}{2}tr[\hat{x}(A-A^{T})]=-x^{T}(A-A^{T})^{\vee},
\end{equation}
we can conclude
\begin{equation}
tr[R_{d}^{T}R\hat{\eta}]=-\eta(R_{d}^{T}R-R_{d}R^{T})^{\vee},
\end{equation}
\begin{equation}
D\Psi(R,R_{d})\cdot R\hat{\eta}=\frac{1}{2}(R_{d}^{T}R-R^{T}R_{d})^{\vee}\cdot\eta.
\end{equation}
We define
\begin{equation}
e_{R}=\frac{1}{2}(R_{d}^{T}R-R^{T}R_{d})^{\vee},
\end{equation}
and 
\begin{equation}
\hat{\Omega}_{d}=R_{d}^{T}\dot{R}_{d}\ \Rightarrow\ \dot{R}_{d}=R_{d}\hat{\Omega}_{d}.
\end{equation}
So we can write and simplify the following equation to find the $e_{\Omega}$
\begin{equation}
\dot{R}-\dot{R}_{d}(R^{T}R)=R\hat{\Omega}-R_{d}\hat{\Omega}_{d}R_{d}^{T}R=R(\hat{\Omega}-R^{T}R_{d}\hat{\Omega}_{d}R_{d}^{T}R).
\end{equation}
Comparing the above equation with $\dot{R}=R\hat{\Omega}$, we can result in the following equation
\begin{equation}
\hat{e}_{\Omega}=\hat{\Omega}-R^{T}R_{d}\hat{\Omega}_{d}R_{d}^{T}R,
\end{equation}
or
\begin{equation}
e_{\Omega}=\Omega-R^{T}R_{d}\Omega_{d}.
\end{equation}
Now, we try to find an expression for $\dot{\Psi}$ by taking time-derivative of $\Psi$
\begin{equation}
\dot{\Psi}=\frac{1}{2}tr[-(\frac{d}{dt}R_{d}^{T})R-R_{d}^{T}\dot{R}].
\end{equation}
Substituting $\dot{R}=R\hat{\Omega}$ and $\dot{R}_{d}^{T}=-\hat{\Omega}_{d}R_{d}^{T}$ into the above expression, we would have
\begin{equation}
\dot{\Psi}=-\frac{1}{2}tr[-\hat{\Omega}_{d}R_{d}^{T}R+R_{d}^{T}R\hat{\Omega}],
\end{equation}
and we can rewrite the above equation as follows
\begin{equation}
\dot{\Psi}=\frac{1}{2}e_{\Omega}^{T}(R_{d}^{T}R-R^{T}R_{d})^{\vee},
\end{equation}
so
\begin{equation}
\dot{\Psi}=e_{R}\cdot e_{\Omega}.
\end{equation}
}

%%%%%%%%%%%%%%%%%%%%%%%%%%%%%%%%%%%%%%%%%%%%%%%%%%%%%%%%
\subsection{\protect Proof for Proposition \ref{prop:propchap2_2}}\label{sec:pfchap2_2}{

We first find the error dynamics for $e_R,e_\Omega$, and define a Lyapunov function. Then, we find conditions on control parameters to guarantee the boundedness of tracking errors. Using \refeqn{EL3}, \refeqn{EL4}, \refeqn{M}, the time-derivative of $Je_\Omega$ can be written as
\begin{align}
J\dot e_\Omega & = \{Je_\Omega + d\}^\wedge e_\Omega - k_R e_R-k_\Omega e_\Omega+ \mathds{W}_{R}\tilde{\theta}_{R},\label{eqn:JeWdot}
\end{align}
where $d=(2J-\trs{J}I)R^TR_d\Omega_d\in\Re^3$ and  $\tilde{\theta}_{R}=\theta_{R}-\bar{\theta}_{R}$. The important property is that the first term of the right hand side is normal to $e_\Omega$, and it simplifies the subsequent Lyapunov analysis. Define a Lyapunov function $\mathcal{V}_2$ be 
\begin{align}
\mathcal{V}_2 & = \frac{1}{2} e_\Omega \cdot J e_\Omega + k_R\, \Psi(R,R_d)+c_2 e_R\cdot Je_\Omega+\frac{1}{2\gamma_{R}}\|\tilde{\theta}_{R}\|^2.\label{eqn:V2}
\end{align}
From \refeqn{PsiLB}, \refeqn{PsiUB}, the Lyapunov function $\mathcal{V}_2$ is bounded as
\begin{align}
z_2^T M_{21} z_2+\frac{1}{2\gamma_{R}}\|\tilde{\theta}_{R}\|^2
\leq \mathcal{V}_2 \leq z_2^T M_{22} z_2+\frac{1}{2\gamma_{R}}\|\tilde{\theta}_{R}\|^2,
\label{eqn:V2b}
\end{align}
where $z_2 =[\|e_R\|,\;\|e_\Omega\|]^T\in\Re^2$, and the matrices $M_{21},M_{22}$ are given by
\begin{align}
M_{21} = \frac{1}{2}\begin{bmatrix} k_R & -c_2\lambda_M \\ -c_2\lambda_M & \lambda_m  \end{bmatrix},\,
M_{22} = \frac{1}{2}\begin{bmatrix} \frac{2k_R}{2-\psi_2} & c_2\lambda_M \\ c_2\lambda_M & \lambda_{M}\end{bmatrix}.
%\label{eqn:M2}
\end{align}
From \refeqn{PsiUB}, the upper-bound of \refeqn{V2b} is satisfied in the following domain:
\begin{align}
D_2 = \{ (R,\Omega)\in \SO\times\Re^3\,|\, \Psi(R,R_d)<\psi_2<2\}.\label{eqn:D2}
\end{align}
From \refeqn{Psidot00}, \refeqn{JeWdot}, the time derivative of $\mathcal{V}_2$ along the solution of the controlled system is given by
\begin{align*}
\dot{\mathcal{V}}_2  =&
-k_\Omega\|e_\Omega\|^2  +e_{\Omega}^{T}\mathds{W}_{R}\tilde{\theta}_{R}+ c_2 \dot e_R \cdot Je_\Omega+ c_2 e_R \cdot J\dot e_\Omega +\frac{1}{\gamma_{R}}(\tilde{\theta})^{T}(\dot{\tilde{\theta}}).
\end{align*}
We have $\dot{\tilde{\theta}}_{R}=-\dot{\bar{\theta}}_{R}$. Substituting \refeqn{JeWdot}, the above equation becomes
\begin{align*}
\dot{\mathcal{V}}_2  =&
-k_\Omega\|e_\Omega\|^2  + c_2 \dot e_R \cdot Je_\Omega-c _2 k_R \|e_R\|^2 \\
&+ c_2 e_R \cdot ((Je_\Omega+d)^\wedge e_\Omega -k_\Omega e_\Omega)\\
&+\tilde{\theta}_{R}^{T}\mathds{W}_{R}^{T}(e_{\Omega}+c_{2}e_{R})-\frac{1}{\gamma_{R}}\tilde{\theta}_{R}^{T}\dot{\bar{\theta}}_{R}.
\end{align*}
By substituting the adaptive law given by \refeqn{eI}
\begin{align*}
\dot{\mathcal{V}}_2  =&
-k_\Omega\|e_\Omega\|^2  + c_2 \dot e_R \cdot Je_\Omega-c _2 k_R \|e_R\|^2 + c_2 e_R \cdot ((Je_\Omega+d)^\wedge e_\Omega -k_\Omega e_\Omega).
\end{align*}
Since $\|e_R\|\leq 1$, $\|\dot e_R\|\leq \|e_\Omega\|$, and $\|d\|\leq B_2$, we have
\begin{align}
\dot{\mathcal{V}}_2 \leq - z_2^T W_2 z_2,\label{eqn:dotV2}
\end{align}
where the matrix $W_2\in\Re^{2\times 2}$ is given by
\begin{align*}
W_2 = \begin{bmatrix} c_2k_R & -\frac{c_2}{2}(k_\Omega+B_2) \\ 
-\frac{c_2}{2}(k_\Omega+B_2) & k_\Omega-2c_2\lambda_M \end{bmatrix}.%\label{eqn:W2}
\end{align*}
The condition on $c_2$ given at \refeqn{c2} guarantees that all of matrices $M_{21}$, $W_2$ are positive definite. According to the LaSalle-Yoshizawa theorem~\cite[Theorem A.8]{Krstic95Kokotovic}, this implies that the zero equilibrium of tracking errors and the estimation error is stable in the sense of Lyapunov, and $e_R,e_\Omega\rightarrow 0$ as $t\rightarrow\infty$. Furthermore $\tilde{\theta}_{R}$ is uniformly bounded.}

\subsection{Proof for Proposition \ref{prop:propchap2_3}}\label{sec:pfchap2_3}{
We derive the tracking error dynamics and a Lyapunov function for the translational dynamics of a quadrotor UAV, and later it is combined with the stability analysis of the rotational dynamics. The subsequent analysis are developed in the domain $D_1$
\begin{align}
D_1=\{&(e_x,e_v,R,e_\Omega)\in\Re^3\times\Re^3\times \SO\times\Re^3\,|\,\nonumber\\
& \|e_x\|< e_{x_{\max}},\;\Psi< \psi_1 < 1\},\label{eqn:D}
\end{align}
Similar to \refeqn{PsiUB}, we can show that 
\begin{align}
\frac{1}{2} \norm{e_R}^2 \leq  \Psi(R,R_c) \leq \frac{1}{2-\psi_1} \norm{e_R}^2\label{eqn:eRPsi1}.
\end{align}

\subsubsection*{a) Translational Error Dynamics} The time derivative of the position error is $\dot e_x=e_v$. The time-derivative of the velocity error is given by
\begin{align}
m\dot e_v = m\ddot x -m\ddot x_d = mg e_3 - fRe_3 -m\ddot x_d+\mathds{W}_{x}\theta_{x}. \label{eqn:evdot0}
\end{align}
Consider the quantity $e_3^T R_c^T R e_3$, which represents the cosine of the angle between $b_3=Re_3$ and $b_{3_c}=R_ce_3$. Since $1-\Psi(R,R_c)$ represents the cosine of the eigen-axis rotation angle between $R_c$ and $R$, we have $e_3^T R_c^T R e_3\geq  1-\Psi(R,R_c)>0$ in $D_1$. Therefore, the quantity $\frac{1}{e_3^T R_c^T R e_3}$ is well-defined. To rewrite the error dynamics of $e_v$ in terms of the attitude error $e_R$, we add and subtract $\frac{f}{e_3^T R_c^T R e_3}R_c e_3$ to the right hand side of \refeqn{evdot0} to obtain
\begin{align}
m\dot e_v &  = mg e_3 -m\ddot x_d- \frac{f}{e_3^T R_c^T R e_3}R_c e_3 - X+\mathds{W}_{x}\theta_{x},\label{eqn:evdot1}
\end{align}
where $X\in\Re^3$ is defined by
\begin{align}
X=\frac{f}{e_3^T R_c^T R e_3}( (e_3^T R_c^T R e_3)R e_3 -R_ce_3).\label{eqn:X}
\end{align}
Let $A=-k_x e_x - k_v e_v -\mathds{W}_{x}\bar{\theta}_{x}-mg e_3 + m\ddot x_d$.
%be the desired control force for the translational dynamics. 
Then, from \refeqn{Rd3}, \refeqn{f}, we have ${b}_{3_c}=R_c e_3 = -A/\norm{A}$ and $f=-A\cdot Re_3$. By combining these, we obtain $f= (\norm{A}R_c e_3)\cdot R e_3$. Therefore, the third term of the right hand side of \refeqn{evdot1} can be written as
\begin{align*}
-  \frac{f}{e_3^T R_c^T R e_3} & R_c e_3 = -\frac{(\norm{A}R_c e_3)\cdot R e_3}{e_3^T R_c^T R e_3}\cdot - \frac{A}{\norm{A}}=A\\
& =-k_x e_x - k_v e_v -\mathds{W}_{x}\bar{\theta}_{x} -mg e_3 + m\ddot x_d.
\end{align*}
Substituting this into \refeqn{evdot1}, the error dynamics of $e_v$ can be written as
\begin{align}
m\dot e_v  = & -k_x e_x - k_v e_v -\mathds{W}_{x}\bar{\theta}_{x} - X+\mathds{W}_{x}{\theta}_{x}\nonumber \\
&=  -k_x e_x - k_v e_v +\mathds{W}_{x}\tilde{\theta}_{x} - X.\label{eqn:evdot}
\end{align}
where $\tilde{\theta}_{x}=\theta_{x}-\bar{\theta}_{x}$ is the estimation errors.

\subsubsection*{b) Lyapunov Candidate for Translation Dynamics}
Let a Lyapunov candidate $\mathcal{V}_1$ be
\begin{align}
\mathcal{V}_1 & = \frac{1}{2}k_x\|e_x\|^2  + \frac{1}{2} m \|e_v\|^2 + c_1 e_x\cdot me_v+ \frac{1}{2\gamma_{x}}\|\tilde{\theta}_{x}\|^2
\label{eqn:V1}.
\end{align}
The derivative of ${\mathcal{V}}_1$ along the solution of \refeqn{evdot} is given by
\begin{align}
\dot{\mathcal{V}}_1 
 =&  -(k_v-mc_1) \|e_v\|^2 
- c_1 k_x \|e_x\|^2 \nonumber\\
&-c_1 k_v e_x\cdot e_v+X\cdot \braces{ c_1 e_x + e_v}\nonumber\\
&+\mathds{W}_{x}\tilde{\theta}_{x}\cdot\{e_{v}+c_{1}e_{x}\}-\frac{1}{\gamma_{x}}\tilde{\theta}_{x}^{T}\dot{\bar{\theta}}_{x}.\label{eqn:V1dot03}
\end{align}
For the first case of the adaptive law given by \refeqn{adaptivelawx}, the last two terms of \refeqn{V1dot03} are cancelled out. Also for the second case~\cite{Ioannou96} that $\dot{\bar{\theta}}_{x}=\gamma_{x}(I-\frac{\bar{\theta}_{x}\bar{\theta}_{x}^{T}}{\bar{\theta}_{x}^{T}\bar{\theta}_{x}})\mathds{W}_{x}^{T}(e_{v}+c_{1}e_{x})$, we obtain
\begin{align}
\dot{\mathcal{V}}_1 
 =&  -(k_v-mc_1) \|e_v\|^2 
- c_1 k_x \|e_x\|^2 \nonumber\\
&-c_1 k_v e_x\cdot e_v+X\cdot \braces{ c_1 e_x + e_v}\nonumber\\
&+\tilde{\theta}_{x}^{T}\frac{\bar{\theta}_{x}\bar{\theta}_{x}^{T}}{\bar{\theta}_{x}^{T}\bar{\theta}_{x}}W_{x}^{T}(e_{v}+c_{1}e_{x}).
\end{align}
In the above equation, the last term on the right hand side is always negative since $\bar{\theta}_{x}^{T}\mathds{W}_{x}^T(e_{v}+c_{1}e_{x})>0$ and $\tilde{\theta}_{x}^T\bar{\theta}_{x}\leq 0$. Therefore for both cases of \refeqn{adaptivelawx}, we obtain
\begin{align}\label{eqn:V1dot0}
\dot{\mathcal{V}}_1 
 \leq&  -(k_v-mc_1) \|e_v\|^2 
- c_1 k_x \|e_x\|^2 -c_1 k_v e_x\cdot e_v+X\cdot \braces{ c_1 e_x + e_v}.
\end{align}
The last term of the above equation corresponds to the effects of the attitude tracking error on the translational dynamics. We find a bound of $X$, defined at \refeqn{X}, to show stability of the coupled translational dynamics and rotational dynamics in the subsequent Lyapunov analysis. Since $f=\|A\| (e_3^T R_c^T R e_3)$, we have
\begin{align*}
\norm{X}  \leq& \|A\|\,\| (e_3^T R_c^T R e_3)R e_3 -R_ce_3\|\\
 \leq&( k_x \|e_x\| + k_v \|e_v\| + B_{W_{x}} B_{\theta} +B_1) \times \| (e_3^T R_c^T R e_3)R e_3 -R_ce_3\|.
\end{align*}
The last term $\| (e_3^T R_c^T R e_3)R e_3 -R_ce_3\|$ represents the sine of the angle between $b_3=Re_3$ and $b_{c_3}=R_c e_3$, since $(b_{3_c}\cdot b_3)b_3 - b_{3_c} = b_{3}\times (b_3\times b_{3_c})$. The magnitude of the attitude error vector, $\|e_R\|$ represents the sine of the eigen-axis rotation angle between $R_c$ and $R$ (see \cite{LeeLeoPICDC10}). Therefore, $\| (e_3^T R_c^T R e_3)R e_3 -R_ce_3\| \leq \| e_R\|$ in $D_1$. It follows that 
\begin{align}
\| (e_3^T R_d^T R e_3)R e_3 & -R_de_3\| \leq \| e_R\| = \sqrt{\Psi(2-\Psi)}\nonumber\\
& \leq \braces{\sqrt{\psi_1 (2-\psi_1)}\triangleq\alpha}  <1.\label{eqn:eR_bound}
\end{align}
Therefore, $X$ is bounded by
\begin{align}
\norm{X} 
&\leq ( k_x \|e_x\| + k_v \|e_v\| + B_{W_{x}} B_{\theta} + B_1) \|e_R\| \nonumber\\
&\leq ( k_x \|e_x\| + k_v \|e_v\| + B_{W_{x}} B_{\theta} + B_1) \alpha.\label{eqn:XB}
\end{align}
Substituting \refeqn{XB} into \refeqn{V1dot0}, 
\begin{align}
\dot{\mathcal{V}}_1 
& \leq   -(k_v(1-\alpha)-mc_1) \|e_v\|^2 
- {c_1 k_x}(1-\alpha) \|e_x\|^2 \nonumber\\
& + {c_1k_v}(1+\alpha) \|e_x\|\|e_v\|\nonumber\\
& +  \|e_R\| \braces{(B_{W_{x}} B_{\theta} +B_1)({c_1} \|e_x\| + \|e_v\|)+k_x\|e_x\|\|e_v\|}.\label{eqn:V1dot1}
\end{align}
In the above expression for $\dot{\mathcal{V}}_1$, there is a third-order error term, namely $k_x\|e_R\|\|e_x\|\|e_v\|$. Using \refeqn{eR_bound}, it is possible to choose its upper bound as $k_x\alpha\|e_x\|\|e_v\|$ similar to other terms, but the corresponding stability analysis becomes complicated, and the initial attitude error should be reduced further. Instead, we restrict our analysis to the domain $D_1$ defined in \refeqn{D}, and its upper bound is chosen as $k_xe_{x_{\max}}\|e_R\|\|e_v\|$.
\subsubsection*{c) Lyapunov Candidate for the Complete System}
Let $\mathcal{V}=\mathcal{V}_1+\mathcal{V}_2$ be the Lyapunov candidate of the complete system. Define $z_1=[\|e_x\|,\;\|e_v\|]^T$, $z_2=[\|e_R\|,\;\|e_\Omega\|]^T\in\Re^2$, and 
\begin{align*}
\mathcal{V}_A = \frac{1}{2\gamma_{x}}\|\theta_{x}-\bar{\theta}_{x}\|^2+\frac{1}{2\gamma_{R}}\|\theta_{R}-\bar{\theta}_{R}\|^2.
\end{align*}
Using \refeqn{eRPsi1}, the bound of the Lyapunov candidate $\mathcal{V}$ can be written as
\begin{align}
z_1^T M_{11} z_1 + z_2^T M_{21} z_2& + \mathcal{V}_A
 \leq \mathcal{V} \leq z_1^T M_{12} z_1 + z_2^T M'_{22} z_2 +\mathcal{V}_A,\label{eqn:Vb}
\end{align}
where the matrices $M_{11},M_{12},M_{21},M_{22}$ are given by
\begin{gather*}
M_{11} = \frac{1}{2}\begin{bmatrix} k_x & -mc_1 \\ -mc_1 & m\end{bmatrix},\;
M_{12} = \frac{1}{2}\begin{bmatrix} k_x & mc_1 \\ mc_1 & m\end{bmatrix},\\
M_{21} = \frac{1}{2}\begin{bmatrix} k_R & -c_2\lambda_M \\ -c_2\lambda_M & \lambda_{m}  \end{bmatrix},\;
M_{22} = \frac{1}{2}\begin{bmatrix} \frac{2k_R}{2-\psi_1} & c_2\lambda_M \\ c_2\lambda_M & \lambda_{M}\end{bmatrix}.
\end{gather*}
Using \refeqn{dotV2} and \refeqn{V1dot1}, the time-derivative of $\mathcal{V}$ is given by
\begin{align}\label{eqn:eq53}
\dot{\mathcal{V}} & \leq -z_1^T W_1 z_1  + z_1^T W_{12} z_2 - z_2^T W_2 z_2\leq -z^T W z
\end{align}
where $z=[z_1,z_2]^T\in\Re^2$, and the matrices $W_1,W_{12},W_2\in\Re^{2\times 2}$ are defined at \refeqn{W1}-\refeqn{W2}. The matrix $W\in\Re^{2\times 2}$ is given by
\begin{align*}
W=\begin{bmatrix}
\lambda_{m}(W_1) & -\frac{1}{2}\|W_{12}\|_2\\
-\frac{1}{2}\|W_{12}\|_2 & \lambda_m(W_2)
\end{bmatrix}.
\end{align*}
The conditions given at \refeqn{c2}, \refeqn{c1b}, \refeqn{kRkWb} guarantee that all of matrices $M_{11},M_{12},M_{21},M_{22},W$ are positive definite. According to the LaSalle-Yoshizawa theorem~\cite[Theorem A.8]{Krstic95Kokotovic}, this implies that the zero equilibrium of tracking errors and the estimation errors is stable in the sense of Lyapunov, and tracking errors asymptotically converge to zero. Furthermore, estimation errors are uniformly bounded.}

\subsection{Proof for Proposition \ref{prop:propchap2_4}}\label{sec:pfchap2_4}
According to the proof of Proposition \ref{prop:propchap2_2}, the attitude tracking errors asymptotically decrease to zero, and therefore, they enter the region given by \refeqn{Psi0} in a finite time $t^*$, after which the results of Proposition \ref{prop:propchap2_3} can be applied to yield attractiveness. The remaining part of the proof is showing that the tracking error $z_1=[\|e_x\|,\|e_v\|]^T$ is bounded in $t\in[0, t^*]$. This is similar to the proof given at~\cite{LeeLeoAJC13}.

%%%%%%%%%%%%%%%%%%%%%%%%%%%%%%%%%%%%%%%%%%%%%%%%%%%%
\subsection{Proof for Proposition \ref{prop:propchap3_1}}\label{sec:pfchap3_1}
From \refeqn{KE} and \refeqn{PE}, the Lagrangian is given by
\begin{align}
L&=\frac{1}{2}M_{00}\|\dot{x}\|^{2}+\dot{x}\cdot\sum^{n}_{i=1}{M_{0i}\dot{q}_i}+\frac{1}{2}\sum^{n}_{i,j=1}{M_{ij}\dot{q}_{i}\cdot\dot{q}_{j}}\nonumber \\
&\quad+\sum^{n}_{i=1}\sum^{n}_{a=i}m_{a}gl_{i}e_{3}\cdot q_{i}+M_{00}ge_{3}\cdot x+\frac{1}{2}\Omega^{T}J\Omega,
\end{align}
The derivatives of the Lagrangian are given by
\begin{align*}
\D_x L & = M_{00} g e_3,\\
\D_{\dot x} L & = M_{00} \dot x + \sum_{i=1}^n M_{0i}\dot q_i,
\end{align*}
where $\D_x L$ represents the derivative of $L$ with respect to $x$. From the variation of the angular velocity given after \refeqn{delR}, we have
\begin{align}
\D_{\Omega}L\cdot\delta\Omega=J\Omega\cdot(\dot\eta+\Omega\times\eta)
=J\Omega\cdot\dot\eta- \eta\cdot(\Omega\times J\Omega).
\end{align}
Similarly from \refeqn{delqi}, the derivative of the Lagrangian with respect to $q_i$ is given by
\begin{align*}
\D_{q_i} L \cdot \delta q_i = \sum_{a=i}^n m_a gl_ie_3\cdot (\xi_i\times q_i) 
= -\sum_{a=i}^n m_a gl_i\hat e_3 q_i\cdot \xi_i.
\end{align*}
The variation of $\dot q_i$ is given by
\begin{align*}
\delta\dot q_i = \dot \xi_i\times q_i + \xi_i\times q_i.
\end{align*}
Using this, the derivative of the Lagrangian with respect to $\dot q_i$ is given by
\begin{align*}
&\D_{\dot q_i}L\cdot \delta \dot q_i  = (M_{i0}\dot x + \sum_{j=1}^n M_{ij}\dot q_j) \cdot \delta \dot q_i \\
& = (M_{i0}\dot x + \sum_{j=1}^n M_{ij}\dot q_j) \cdot (\dot \xi_i \times q + \xi_i \times \dot q_i)\\
& = \hat q_i (M_{i0}\dot x + \sum_{j=1}^n M_{ij}\dot q_j)\cdot \dot\xi_i +
\hat{\dot q}_i (M_{i0}\dot x + \sum_{j=1}^n M_{ij}\dot q_j)\cdot \xi_i.
\end{align*}
Let $\mathfrak{G}$ be the action integral, i.e., $\mathfrak{G}=\int_{t_0}^{t_f} L\,dt$. From the above expressions for the derivatives of the Lagrangian, the variation of the action integral can be written as
\begin{align*}
\delta \mathfrak{G} =& \int_{t_0}^{t_f} \{M_{00} \dot x + \sum_{i=1}^n M_{0i}\dot q_i\}\cdot \delta \dot x
+M_{00}ge_3\cdot\delta x\\
& +\sum_{i=1}^n \{\hat q_i (M_{i0}\dot x + \sum_{j=1}^n M_{ij}\dot q_j)\}\cdot \dot \xi_i\\
 &+ \sum_{i=1}^n\{\hat{\dot q}_i (M_{i0}\dot x + \sum_{j=1}^n M_{ij}\dot q_j)-\sum_{a=i}^n m_a gl_i\hat e_3 q_i \}\cdot \xi_i\\
 &+ J\Omega\cdot\dot\eta- \eta\cdot(\Omega\times J\Omega)\,dt.
\end{align*}
Integrating by parts and using the fact that variations at the end points vanish, this reduces to
\begin{align*}
\delta \mathfrak{G} =& \int_{t_0}^{t_f} \{M_{00}ge_3 -M_{00} \ddot x - \sum_{i=1}^n M_{0i}\ddot q_i\}\cdot \delta x \\
&+ \sum_{i=1}^n\{-\hat q_i (M_{i0}\ddot x + \sum_{j=1}^n M_{ij}\ddot q_j)-\sum_{a=i}^n m_a gl_i\hat e_3 q_i \}\cdot \xi_i\\
& - \eta\cdot(J\dot\Omega+\Omega\times J\Omega) \,dt.
\end{align*}
According to the Lagrange-d'Alembert principle, the variation of the action integral is equal to the negative of the virtual work done by the external force and moment, namely
\begin{align}\label{eqn:estef}
-\int_{t_0}^{t_f} (-fRe_3+\Delta_{x})\cdot\delta x + (M+\Delta_R)\cdot \eta\; dt,
\end{align}
and we obtain \refeqn{xddot} and \refeqn{Wdot}. As $\xi_i$ is perpendicular to $q_i$, we also have
\begin{gather}
-\hat q_i^2 (M_{i0}\ddot x + \sum_{j=1}^n M_{ij}\ddot q_j)+\sum_{a=i}^n m_a gl_i\hat q_i^2 e_3=0.\label{eqn:qddot0}
\end{gather}
Equation \refeqn{qddot0} is rewritten to obtain an explicit expression for $\ddot q_i$. As $q_i\cdot \dot q_i =0$, we have $\dot q_i \cdot \dot q_i +q_i\cdot \ddot q_i=0$. Using this, we have
\begin{align*}
-\hat q_i^2 \ddot q_i = -(q_i\cdot \ddot q_i )q_i + (q_i\cdot q_i)\ddot q_i =(\dot q_i \cdot \dot q_i) q_i + \ddot q_i.
\end{align*}
Substituting this equation into \refeqn{qddot0}, we obtain \refeqn{qddot}. This can be slightly rewritten in terms of the angular velocities. Since $\dot q_i = \omega_i\times q_i$ for the angular velocity $\omega_i$ satisfying $q_i\cdot\omega_i=0$, we have
\begin{align*}
    \ddot q_i & = \dot \omega_i \times q_i + \omega_i\times (\omega_i\times q_i) \\
    &= \dot \omega_i \times q_i - \|\omega_i\|^2q_i=-\hat q_i \dot\omega_i - \|\omega_i\|^2q_i.
\end{align*}
Using this and the fact that $\dot\omega_i\cdot q_i=0$, we obtain \refeqn{ELwm}.

%%%%%%%%%%%%%%%%%%%%%%%%%%%%%%%%%%%%%%%%%%
\subsection{Proof for Proposition \ref{prop:propchap3_2}}\label{sec:pfchap3_2}

The variations of $x,u$ and $q$ are given by \refeqn{delxLin} and \refeqn{delqLin}. From the kinematics equation $\dot q_i=\omega_i\times q_i$, $\delta\dot q_i$ is given by
\begin{align*}
\delta \dot q_i = \dot\xi_i \times e_3 =\delta\omega_i \times e_3 + 0\times (\xi_i\times e_3)=\delta\omega_i \times e_3.
\end{align*}
Since both sides of the above equation is perpendicular to $e_3$, this is equivalent to $e_3\times(\dot\xi_i\times e_3) = e_3\times(\delta\omega_i\times e_3)$, which yields
\begin{gather*}
\dot \xi - (e_3\cdot\dot\xi) e_3 = \delta\omega_i -(e_3\cdot\delta\omega_i)e_3.
\end{gather*}
Since $\xi_i\cdot e_3 =0$, we have $\dot\xi\cdot e_3=0$. As $e_3\cdot\delta\omega_i=0$ from the constraint, we obtain the linearized equation for the kinematics equation:
\begin{align}
\dot\xi_i = \delta\omega_i.\label{eqn:dotxii}
\end{align}
Substituting these into \refeqn{ELwm}, and ignoring the higher order terms, we obtain \refeqn{Lin}. See~\cite{FarhadDTLeeIJCAS} for details.
%%%%%%%%%%%%%%%%%%%%%%%%%%%%%%%%%%%%%%%%%%%%%%%%%%%
\subsection{Proof for Proposition \ref{prop:propchap3_3}}\label{sec:pfchap3_3}
We first show stability of the rotational dynamics, and later it is combined with the stability analysis of the translational dynamics of quad rotor and the rotational dynamics of links. 
\subsubsection*{a) Attitude Error Dynamics}
Here, attitude error dynamics for $e_{R}$, $e_{\Omega}$ are derived and we find conditions on control parameters to guarantee the boundedness of attitude tracking errors. The time-derivative of $Je_{\Omega}$ can be written as
\begin{align}
J\dot e_\Omega & = \{Je_\Omega + d\}^\wedge e_\Omega - k_R e_R-k_\Omega e_\Omega- k_I e_I + \Delta_R,\label{eqn:JeWdot}
\end{align}
where $d=(2J-\trs{J}I)R^TR_d\Omega_d\in\Re^3$~\cite{TFJCHTLeeHG}. The important property is that the first term of the right hand side is normal to $e_\Omega$, and it simplifies the subsequent Lyapunov analysis.

\subsubsection*{b) Stability for Attitude Dynamics}
Define a configuration error function on $\SO$ as follows
\begin{align}
\Psi= \frac{1}{2}\trs[I- R_c^T R].
\end{align}
We introduce the following Lyapunov function
\begin{align}
\begin{split}
\mathcal{V}_{2}=&\frac{1}{2}e_{\Omega}\cdot J\dot{e}_{\Omega}+k_{R}\Psi(R,R_{d})+c_{2}e_{R}\cdot e_{\Omega}+\frac{1}{2}k_{I}\|e_{I}-\frac{\Delta_R}{k_{I}}\|^{2}.
\end{split}
\end{align}
Consider a domain $D_{2}$ given by
\begin{align}
D_2 = \{ (R,\Omega)\in \SO\times\Re^3\,|\, \Psi(R,R_d)<\psi_2<2\}.\label{eqn:D2}
\end{align}
In this domain we can show that $\mathcal{V}_{2}$ is bounded as follows~\cite{TFJCHTLeeHG}
\begin{align}\label{eqn:ffff1}
\begin{split}
z_{2}^{T}M_{21}z_{2}&+\frac{k_{I}}{2}\|e_{I}-\frac{\Delta_R}{k_{I}}\|^{2}\leq\mathcal{V}_{2}\leq z_{2}^{T}M_{22}z_{2}+\frac{k_{I}}{2}\|e_{I}-\frac{\Delta_R}{k_{I}}\|^{2},
\end{split}
\end{align}
where $z_{2}=[\|e_{R}\|,\|e_{\Omega}\|]^{T}\in \Re^{2}$ and the matrices $M_{21}$, $M_{22}$ are given by
\begin{align}
M_{21}=\frac{1}{2}\begin{bmatrix}
k_{R}&-c_{2}\lambda_{M}\\
-c_{2}\lambda_{M}&\lambda_{m}
\end{bmatrix},M_{22}=\frac{1}{2}\begin{bmatrix}
\frac{2k_{R}}{2-\psi_{2}}&c_{2}\lambda_{M}\\
c_{2}\lambda_{M}&\lambda_{M}
\end{bmatrix},
\end{align}
The time derivative of $\mathcal{V}_2$ along the solution of the controlled system is given by
\begin{align*}
\dot{\mathcal{V}}_2  =&
-k_\Omega\|e_\Omega\|^2 - e_\Omega\cdot(k_Ie_I-\Delta_R) + c_2 \dot e_R \cdot Je_\Omega+ c_2 e_R \cdot J\dot e_\Omega + (k_Ie_I- \Delta_R) \dot e_I.
\end{align*}
We have $\dot e_I = c_2 e_R + e_\Omega$ from \refeqn{integralterm}. Substituting this and \refeqn{JeWdot}, the above equation becomes
\begin{align*}
\dot{\mathcal{V}}_2 =&
-k_\Omega\|e_\Omega\|^2  + c_2 \dot e_R \cdot Je_\Omega-c _2 k_R \|e_R\|^2 + c_2 e_R \cdot ((Je_\Omega+d)^\wedge e_\Omega -k_\Omega e_\Omega).
\end{align*}
We have $\|e_R\|\leq 1$, $\|\dot e_R\|\leq \|e_\Omega\|$~\cite{TFJCHTLeeHG}, and choose a constant $B_{2}$ such that $\|d\|\leq B_2$. Then we have
\begin{align}
\dot{\mathcal{V}}_2 \leq - z_2^T W_2 z_2,\label{eqn:dotV2}
\end{align}
where the matrix $W_2\in\Re^{2\times 2}$ is given by
\begin{align*}
W_2 = \begin{bmatrix} c_2k_R & -\frac{c_2}{2}(k_\Omega+B_2) \\ 
-\frac{c_2}{2}(k_\Omega+B_2) & k_\Omega-2c_2\lambda_M \end{bmatrix}.%\label{eqn:W2}
\end{align*}
The matrix $W_{2}$ is a positive definite matrix if 
\begin{align}\label{eqn:c2}
c_{2}<\min\{\frac{\sqrt{k_{R}\lambda_{m}}}{\lambda_{M}},\frac{4k_{\Omega}}{8k_{R}\lambda_{M}+(k_{\Omega}+B_{2})^{2}} \}.
\end{align}
This implies that
\begin{align}\label{eqn:eq2}
\dot{\mathcal{V}}_{2}\leq - \lambda_{m}(W_{2})\|z_{2}\|^{2},
\end{align} 
which shows stability of attitude dynamics.
%%%%%%%%%%%%%%%%%%%%%%%%%%%%%%%%%%%%%%%%
\subsubsection*{c) Translational Error Dynamics}
We derive the tracking error dynamics and a Lyapunov function for the translational dynamics of a quadrotor UAV and the dynamics of links. Later it is combined with the stability analyses of the rotational dynamics. This proof is based on the Lyapunov method presented in Theorem 3.6 and 3.7~\cite{Kha96}. From \refeqn{delxLin}, \refeqn{xddot}, \refeqn{Lin}, and \refeqn{fi}, the linearized equation of motion for the controlled full dynamic model is given by
\begin{align}\label{eqn:salam}
\Mb\ddot \xb  + \Gb\xb=\Bb(-fRe_{3}-M_{00}ge_{3})+\g(\xb,\dot{\xb})+\Bb\Delta_{x},
\end{align}
and $\g(\xb,\dot{\xb})$ is higher order term. The subsequent analyses are developed in the domain $D_{1}$
\begin{align}
D_1=\{&(\xb,\dot{\xb},R,e_\Omega)\in\Re^{2n+3}\times\Re^{2n+3}\times \SO\times\Re^3\,|\, \Psi< \psi_1 < 1\}.\label{eqn:D}
\end{align}
In the domain $D_{1}$, we can show that 
\begin{align}
\frac{1}{2} \norm{e_R}^2 \leq  \Psi(R,R_c) \leq \frac{1}{2-\psi_1} \norm{e_R}^2\label{eqn:eRPsi1}.
\end{align}
Consider the quantity $e_{3}^{T}R_{c}^{T}Re_{3}$, which represents the cosine of the angle between $b_{3}=Re_{3}$ and $b_{3_{c}}=R_{c}e_{3}$. Since $1-\Psi(R,R_{c})$ represents the cosine of the eigen-axis rotation angle between $R_{c}$ and $R$, we have $e_{3}^{T}R_{c}^{T}Re_{3}\geq 1-\Psi(R,R_{c})>0$ in $D_{1}$. Therefore, the quantity $\frac{1}{e_{3}^{T}R_{c}^{T}Re_{3}}$ is well-defined. We add and subtract $\frac{f}{e_{3}^{T}R_{c}^{T}Re_{3}}R_{c}e_{3}$ to the right hand side of \refeqn{salam} to obtain
\begin{align}\label{eqn:taghall}
\Mb\ddot \xb  + \Gb\xb=&\Bb(\frac{-f}{e_{3}^{T}R_{c}^{T}Re_{3}}R_{c}e_{3}-X-M_{00}ge_{3}+\Delta_{x})+\g(\xb,\dot{\xb}),
\end{align}
where $X\in \Re^{3}$ is defined by
\begin{align}\label{eqn:Xdef}
X=\frac{f}{e_{3}^{T}R_{c}^{T}Re_{3}}((e_{3}^{T}R_{c}^{T}Re_{3})Re_{3}-R_{c}e_{3}).
\end{align}
The first term on the right hand side of \refeqn{taghall} can be written as 
\begin{align}
-\frac{f}{e_{3}^{T}R_{c}^{T}Re_{3}}R_{c}e_{3}=-\frac{(\|A\|R_{c}e_{3})\cdot Re_{3}}{e_{3}^{T}R_{c}^{T}Re_{3}}\cdot -\frac{A}{\|A\|}=A.
\end{align}
Substituting this and \refeqn{A} into \refeqn{taghall}
\begin{align}
\Mb\ddot \xb  + \Gb\xb=\Bb(-K_{x}\xb-K_{\dot{x}}\dot{\xb}-K_{z}\sat_{\sigma}(e_{\xb})-X+\Delta_{x})+\g(\xb,\dot{\xb}),
\end{align}
This can be rearranged as
\begin{align}
\begin{split}
\ddot{\xb}=&-(\Mb^{-1}\Gb+\Mb^{-1}\Bb K_{x})\xb-(\Mb^{-1}\Bb K_{\dot{x}})\dot{\xb}\\
&-\Mb^{-1}\Bb X-\Mb^{-1}\Bb K_{z}\sat_{\sigma}(e_{\xb})+\Mb^{-1}\g(\xb,\dot{\xb})+\Mb^{-1}\Bb\Delta_{x}.
\end{split}
\end{align}
Using the definitions for $\mathds{A}$, $\mathds{B}$, and $z_{1}$ presented before, the above expression can be rearranged as
\begin{align}\label{eqn:zdot1}
\dot{z}_{1}=&\mathds{A} z_{1}+\mathds{B}(-\Bb X+\g(\xb,\dot{\xb})-\Bb K_{z}\sat_{\sigma}(e_{\xb})+\Bb\Delta_{x}).
\end{align}
\subsubsection*{d) Lyapunov Candidate for Translation Dynamics}
From the linearized control system developed at section 3, we use matrix $P$ to introduce the following Lyapunov candidate for translational dynamics
\begin{align}
\mathcal{V}_{1}=z_{1}^{T}Pz_{1}+2\int_{p_{eq}}^{e_{\xb}}{(\Bb K_{z}\satr_{\sigma}(\mu)-\Bb\Delta_{x})}\cdot d \mu.
\end{align}
The last integral term of the above equation is positive definite about the equilibrium point $e_{\xb}=p_{eq}$ where
\begin{align}
p_{eq}=[\frac{\Delta_{x}}{k_{z}},0,0,\cdots],
\end{align}
if 
\begin{align}
\delta< k_z\sigma,
\end{align}
considering the fact that $\sat_{\sigma}{y}=y$ if $y<\sigma$. The time derivative of the Lyapunov function using the Leibniz integral rule is given by
\begin{align}\label{eqn:devrr}
\dot{\mathcal{V}_{1}}=\dot{z}_{1}^{T}Pz_{1}+z_{1}^{T}P\dot{z}_{1}+2\dot{e}_{\xb}\cdot(\Bb K_{z}\satr_{\sigma}(e_{\xb})-\Bb\Delta_{x}).
\end{align}
Since $\dot{e}_{\xb}^{T}=((P\mathds{B})^{T}z_{1})^{T}=z_{1}^{T}P\mathds{B}$ from \refeqn{exterm}, the above expression can be written as
\begin{align}\label{eqn:devvcv}
\dot{\mathcal{V}_{1}}=\dot{z}_{1}^{T}Pz_{1}+z_{1}^{T}P\dot{z}_{1}+2z_{1}^{T}P\mathds{B}(\Bb K_{z}\satr_{\sigma}(e_{\xb})-\Bb\Delta_{x}).
\end{align}
Substituting \refeqn{zdot1} into \refeqn{devvcv}, it reduces to
\begin{align}\label{eqn:beforsimp}
\dot{\mathcal{V}}_{1}=z_{1}^{T}(\mathds{A}^{T}P+P\mathds{A})z_{1}+2z_{1}^{T}P\mathds{B}(-\Bb X+\g(\xb,\dot{\xb})).
\end{align}
Let $c_{3}=2\|P\mathds{B}\Bb\|_{2}\in\Re$ and using $\mathds{A}^{T}P+P\mathds{A}=-Q$, we have
\begin{align}\label{eqn:test}
\dot{\mathcal{V}}_{1}\leq-z_{1}^{T}Qz_{1}+c_{3}\|z_{1}\|\|X\|+2z_{1}^{T}P\mathds{B}g(\xb,\dot{\xb}).
\end{align}
The second term on the right hand side of the above equation corresponds to the effects of the attitude tracking error on the translational dynamics. We find a bound of $X$, defined at \refeqn{Xdef}, to show stability of the coupled translational dynamics and rotational dynamics in the subsequent Lyapunov analysis. Since 
\begin{align}
f=\|A\|(e_{3}^{T}R_{c}^{T}Re_{3}), 
\end{align}
we have
\begin{align}\label{eqn:ssstr}
\|X\|\leq\|A\|\|(e_{3}^{T}R_{c}^{T}Re_{3})Re_{3}-R_{c}e_{3}\|.
\end{align}
The last term $\|(e_{3}^{T}R_{c}^{T}Re_{3})Re_{3}-R_{c}e_{3}\|$ represents the sine of the angle between $b_{3}=Re_{3}$ and $b_{3_{c}}=R_{c}e_{3}$, since $(b_{3_{c}}\cdot b_{3})b_{3}-b_{3_{c}}=b_{3}\times(b_{3}\times b_{3_{c}})$. The magnitude of the attitude error vector, $\|e_{R}\|$ represents the sine of the eigen-axis rotation angle between $R_{c}$ and $R$. Therefore, $\|(e_{3}^{T}R_{c}^{T}Re_{3})Re_{3}-R_{c}e_{3}\|\leq\|e_{R}\|$ in $D_{1}$. It follows that
\begin{align}
\begin{split}
\|(e_{3}^{T}R_{c}^{T}Re_{3})Re_{3}-R_{c}e_{3}\|&\leq\|e_{R}\|=\sqrt{\Psi(2-\Psi)}\\
&\leq\{\sqrt{\psi_1(2-\psi_1)}\triangleq\alpha\}<1,
\end{split}
\end{align}
therefore
\begin{align}
\begin{split}
\|X\|&\leq \|A\|\|e_{R}\|\\
&\leq\|A\|\alpha.
\end{split}
\end{align}
We also use the following properties
\begin{align}
\lambda_{\min}(Q)\|z_{1}\|^{2}\leq z_{1}^{T}Qz_{1}.
\end{align}
Note that $\lambda_{\min}(Q)$ is real and positive since $Q$ is symmetric and positive definite. Then, we can simplify \refeqn{test} as given
\begin{align}\label{eqn:ssddss}
\dot{\mathcal{V}}_{1} \leq -\lambda_{\min}(Q) \|z_{1}\| ^{2}+c_{3} \|z_{1}\| \|A\|\|e_{R}\|+2z_{1}^{T}P\mathds{B}\g(\xb,\dot{\xb}).
\end{align}
We find an upper boundary for
\begin{align}\label{eqn:AA}
A=-K_{x} \xb-K_{\dot{x}}\dot\xb -K_{z}\satr_{\sigma}(e_{\xb})+ M_{00}ge_3,
\end{align}
by defining
\begin{align}
\|M_{00}ge_{3}\|\leq B_{1},
\end{align} 
for a given positive constant $B_{1}$. We use the following properties where for any matrix $A\in\Re^{m\times n}$ 
\begin{align}
\|A\|_{2}\leq \sqrt{mn}\|A\|_{\max},
\end{align}
where $\|A\|_{\max}=\max\{a_{mn}\}$. The third term on the right hand side of \refeqn{AA} can be bounded as
\begin{align}
\|-K_{z}\satr_{\sigma}(e_{\xb})\|\leq \|K_{z}\|\|\satr_{\sigma}(e_{\xb})\|\leq\|K_{z}\| \sqrt{2n+3}\sigma,
\end{align}
where
\begin{align*}
\|K_{z}\|\leq \sqrt{3(2n+3)}\|K_{z}\|_{\max},
\end{align*}
and similarly 
\begin{align*}
&\|K_{\xb}\|\leq \sqrt{3(2n+3)}\|K_{\xb}\|_{\max},\\
&\|K_{\dot{\xb}}\|\leq \sqrt{3(2n+3)}\|K_{\dot{\xb}}\|_{\max}.
\end{align*}
We define $K_{max}, K_{z_{m}}\in\Re$
\begin{align*}
&K_{\max}=\max\{\|K_{\xb}\|_{\max},\|K_{\dot{\xb}}\|_{\max}\}\sqrt{3(2n+3)}, \\
&K_{z_{m}}=\sqrt{3}(2n+3)\|K_{z}\|_{\max},
\end{align*}
and then the upper bound of $A$ is given by
\begin{align}
\|A\| & \leq K_{\max}(\|\xb\|+\|\dot{\xb}\|)+\sigma K_{z_{m}}+B_{1}\\
&\leq 2K_{\max}\|z_{1}\|+(B_{1}+\sigma K_{z_{m}}),\label{eqn:normA}
\end{align}
and substitute \refeqn{normA} into \refeqn{ssddss}
\begin{align}\label{eqn:eq1}
\begin{split}
\dot{\mathcal{V}}_{1} \leq& -(\lambda_{\min}(Q)-2c_{3}K_{\max}\alpha) \|z_{1}\| ^{2}\\
&+c_{3}(B_{1}+\sigma K_{z_{m}})\|z_{1}\|\|e_{R}\|+2z_{1}^{T}P\mathds{B}\g(\xb,\dot{\xb}).
\end{split}
\end{align}
%%%%%%%%%%%%%%%%%%%%%%%%%%%%%%%%%%%%%%%%
\subsubsection*{e) Lyapunov Candidate for the Complete System}
Let $\mathcal{V}=\mathcal{V}_{1}+\mathcal{V}_{2}$ be the Lyapunov function for the complete system. The time derivative of $\mathcal{V}$ is given by
\begin{align}
\dot{\mathcal{V}}=\dot{\mathcal{V}}_{1}+\dot{\mathcal{V}}_{2}.
\end{align}
Substituting \refeqn{eq1} and \refeqn{eq2} into the above equation
\begin{align}
\begin{split}
\dot{\mathcal{V}}\leq& -(\lambda_{\min}(Q)-2c_{3}K_{\max}\alpha) \|z_{1}\| ^{2}+2z_{1}^{T}P\mathds{B}\g(\xb,\dot{\xb})\\
&+c_{3}(B_{1}+\sigma K_{z_{m}}) \|z_{1}\|\|e_{R}\|-\lambda_{m}(W_{2})\|z_{2}\|^{2},
\end{split}
\end{align}
and using $\|e_{R}\|\leq \|z_{2}\|$, it can be written as
\begin{align}\label{eqn:finalsimp}
\begin{split}
\dot{\mathcal{V}}\leq& -(\lambda_{\min}(Q)-2c_{3}K_{\max}\alpha) \|z_{1}\| ^{2}+2z_{1}^{T}P\mathds{B}\g(\xb,\dot{\xb})\\
&+c_{3}(B_{1}+\sigma K_{z_{m}})\|z_{1}\|\|z_{2}\|-\lambda_{m}(W_{2})\|z_{2}\|^{2}.
\end{split}
\end{align}
Also, the $2z_{1}^{T}P\mathds{B}\g(\xb,\dot{\xb})$ term in the above equation is indefinite. The function $\g(\xb,\dot{\xb})$ satisfies
\begin{align}
\frac{\|\g(\xb,\dot{\xb})\|}{\|z_{1}\|}\rightarrow 0\quad \mbox{as}\quad \|z_{1}\|\rightarrow 0.
\end{align}
Then, for any $\gamma>0$ there exists $r>0$ such that
\begin{align}
\|\g(\xb,\dot{\xb})\|<\gamma\|z_{1}\|,\quad \forall\|z_{1}\|<r,
\end{align}
so,
\begin{align}
2z_{1}^{T}P\mathds{B}\g(\xb,\dot{\xb})\leq 2\gamma\|P\|_{2}\|z_{1}\|^{2}.
\end{align}
Substituting the above equation into \refeqn{finalsimp}
\begin{align}
\begin{split}
\dot{\mathcal{V}}\leq &-(\lambda_{\min}(Q)-2c_{3}K_{\max}\alpha) \|z_{1}\| ^{2}+2\gamma\|P\|_{2}\|z_{1}\|^{2}\\
&+c_{3}(B_{1}+\sigma K_{z_{m}})\|z_{1}\|\|z_{2}\|-\lambda_{m}(W_{2})\|z_{2}\|^{2},
\end{split}
\end{align}
we obtain
\begin{align}
\dot{\mathcal{V}}\leq-z^{T}Wz+2\gamma\|P\|_{2}\|z_{1}\|^{2},
\end{align}
where $z=[z_{1},z_{2}]^{T}\in\Re^{2}$ and
\begin{align}
W=\begin{bmatrix}
\lambda_{\min}(Q)-2c_{3}K_{\max}\alpha&-\frac{c_{3}(B_{1}+\sigma K_{z_{m}})}{2}\\
-\frac{c_{3}(B_{1}+\sigma K_{z_{m}})}{2}&\lambda_{m}(W_{2})
\end{bmatrix}.
\end{align}
By using $\|z_{1}\|\leq\|z\|$, we obtain
\begin{align}
\dot{\mathcal{V}}\leq -(\lambda_{\min}(W)-2\gamma\|P\|_{2})\|z\|^{2}.
\end{align}
Choosing $\gamma<(\lambda_{\min}(W))/2\|P\|_{2}$, ensures that $\dot{\mathcal{V}}$ is negative semi-definite. This implies that the zero equilibrium of tracking errors is stable in the sense of Lyapunov and $\mathcal{V}$ is non-increasing. Therefore all of error variables $z_{1}$, $z_{2}$ and integral control terms $e_{I}$, $e_{\xb}$ are uniformly bounded. Also, from Lasalle-Yoshizawa theorem~\cite[Thm 3.4]{Kha96}, we have $z\rightarrow 0$ as $t\rightarrow \infty$.
%%%%%%%%%%%%%%%%%%%%%%%%%%%%%%%%%%%%%%%%%%%%%%%%%%%%
\subsection{Proof for Proposition \ref{prop:propchap4_1}}\label{sec:pfchap4_1}
\subsubsection*{a) Kinetic Energy}
The kinetic energy of the whole system is composed of the kinetic energy of quadrotors, cables and the rigid body, as
\begin{align}
T=&\frac{1}{2}m_{0}\|\dot{x}_{0}\|^{2}+\sum_{i=1}^{n}\sum_{j=1}^{n_{i}}{\frac{1}{2}m_{ij}\|\dot{x}_{ij}\|^{2}}+\frac{1}{2}\sum_{i=1}^{n}{m_{i}\|\dot{x}_{i}\|^{2}}+\frac{1}{2}\sum_{i=1}^{n}{\Omega_{i}\cdot J_{i}\Omega_{i}}+\frac{1}{2}\Omega_{0}\cdot J_{0}\Omega_{0}.
\end{align}
Substituting the derivatives of \refeqn{xi} and \refeqn{xij} into the above expression we have
\begin{align}
T=&\frac{1}{2}m_{0}\|\dot{x}_{0}\|^{2}+\sum_{i=1}^{n}\sum_{j=1}^{n_{i}}{\frac{1}{2}m_{ij}\|\dot{x}_{0}+\dot{R}_{0}\rho_{i}-\sum_{a=j+1}^{n_{i}}{l_{ia}\dot{q}_{ia}}\|^{2}} \nonumber \\
&+\frac{1}{2}\sum_{i=1}^{n}{m_{i}\|\dot{x}_{0}+\dot{R}_{0}\rho_{i}-\sum_{a=1}^{n_{i}}{l_{ia}\dot{q}_{ia}}\|^{2}} +\frac{1}{2}\sum_{i=1}^{n}{\Omega_{i}\cdot J_{i}\Omega_{i}}+\frac{1}{2}\Omega_{0}\cdot J_{0}\Omega_{0}.
\end{align}
We expand the above expression as follow
\begin{align}\label{eqn:kineticbs}
T=&\frac{1}{2}(m_{0}\|\dot{x}_{0}\|^{2}+\sum_{i=1}^{n}\sum_{j=1}^{n_{i}}{m_{ij}\|\dot{x}_{0}\|^{2}}+\sum_{i=1}^{n}{m_{i}\|\dot{x}_{0}\|^2}) \nonumber\\
&+\frac{1}{2}\sum_{i=1}^{n}(\sum_{j=1}^{n_{i}}{m_{ij}\|\dot{R}_{0}\rho_{i}\|^2}+m_{i}\|\dot{R}_{0}\rho_{i}\|^2) \nonumber\\
&+\sum_{i=1}^{n}(\sum_{j=1}^{n_{i}}{m_{ij}\dot{x}_{0}\cdot\dot{R}_{0}\rho_{i}}+m_{i}\dot{x}_{0}\cdot\dot{R}_{0}\rho_{i}) \nonumber\\
&+\frac{1}{2}\sum_{i=1}^{n}(\sum_{j=1}^{n_{i}}{m_{ij}\|\sum_{a=j+1}^{n_{i}}{l_{ia}\dot{q}_{ia}}\|^2}+{m_{i}\|\sum_{a=1}^{n_{i}}{l_{ia}\dot{q}_{ia}}\|^2}) \nonumber\\
&-\sum_{i=1}^{n}(\sum_{j=1}^{n_{i}}{m_{ij}\dot{x}_{0}}\cdot\sum_{a=j+1}^{n_{i}}{{l_{ia}\dot{q}_{ia}}}+\dot{x}_{0}\cdot\sum_{a=1}^{n_{i}}{l_{ia}\dot{q}_{ia}}) \nonumber\\
&-\sum_{i=1}^{n}(\sum_{j=1}^{n_{i}}{m_{ij}\dot{R}_{0}\rho_{i}}\cdot\sum_{a=j+1}^{n_{i}}{l_{ia}\dot{q}_{ia}}+m_{i}\dot{R}_{0}\rho_{i}\cdot\sum_{a=1}^{n_{i}}{l_{ia}\dot{q}_{ia}}) \nonumber\\
&+\frac{1}{2}\sum_{i=1}^{n}{\Omega_{i}\cdot J_{i}\Omega_{i}}+\frac{1}{2}\Omega_{0}\cdot J_{0}\Omega_{0},
\end{align}
and substituting \refeqn{def1}, \refeqn{def3}, it is rewritten as
\begin{align}\label{eqn:kinetic}
T=&\frac{1}{2}M_{T}\|\dot{x}_{0}\|^2+\frac{1}{2}\sum_{i=1}^{n}{M_{iT}\|\dot{R}_{0}\rho_{i}\|^{2}}+\sum_{i=1}^{n}({M_{iT}\dot{x}_{0}\cdot\dot{R}_{0}\rho_{i}}) \nonumber\\
&+\sum_{i=1}^{n}\sum_{j,k=1}^{n_{i}}{M_{0ij}l_{ik}\dot{q}_{ij}\cdot\dot{q}_{ik}}-\sum_{i=1}^{n}({\dot{x}_{0}\cdot\sum_{j=1}^{n_{i}}{M_{0ij}l_{ij}\dot{q}_{ij}}}) \nonumber\\
&-\sum_{i=1}^{n}({\dot{R}_{0}\rho_{i}}\cdot\sum_{j=1}^{n_{i}}{M_{0ij}l_{ij}\dot{q}_{ij}}) \nonumber\\
&+\frac{1}{2}\sum_{i=1}^{n}{\Omega_{i}\cdot J_{i}\Omega_{i}}+\frac{1}{2}\Omega_{0}\cdot J_{0}\Omega_{0}.
\end{align}
%%%%%%%%%%%%%%%%%%%%%%%%%%%%%%%%%%%%%%%%%%%%%%%%%%%%%%%%%%%%%%%
%%%%%%%%%%%%%%%%%%%%%%%%%%%%%%%%%%%%%%%%%%%%%%%%%%%%%%%%%%%%%%%
\subsubsection*{b) Potential Energy}
We can derive the potential energy expression by considering the gravitational forces on each part of system as given
\begin{align}
V=-m_{0}ge_{3}\cdot x_{0}-\sum_{i=1}^{n}{m_{i}ge_{3}}\cdot x_{i}-\sum_{i=1}^{n}\sum_{j=1}^{n_{i}}{m_{ij}ge_{3}}\cdot x_{ij}.
\end{align}
Using \refeqn{xi} and \refeqn{xij}, we obtain
\begin{align}
V=&-m_{0}ge_{3}\cdot x_{0}-\sum_{i=1}^{n}{m_{i}ge_{3}}\cdot (x_{0}+R_{0}\rho_{i}-\sum_{a=1}^{n_{i}}{l_{ia}q_{ia}}) \nonumber\\
&-\sum_{i=1}^{n}\sum_{j=1}^{n_{i}}{m_{ij}ge_{3}}\cdot (x_{0}+R_{0}\rho_{i}-\sum_{a=j+1}^{n_{i}}{l_{ia}q_{ia}}),
\end{align}
and utilizing \refeqn{def3}, we can simplify the potential energy as
\begin{align}
V=-M_{T}ge_{3}\cdot x_{0}-\sum_{i=1}^{n}{M_{iT}ge_{3}\cdot R_{0}\rho_{i}}+\sum_{i=1}^{n}\sum_{j=1}^{n_{i}}{M_{0ij}l_{ij}q_{ij}\cdot e_{3}}.
\end{align}
%%%%%%%%%%%%%%%%%%%%%%%%%%%%%%%%%%%%%%%%%%%%%%%%%%%%%%%%%%%%%%%
%%%%%%%%%%%%%%%%%%%%%%%%%%%%%%%%%%%%%%%%%%%%%%%%%%%%%%%%%%%%%%%
\subsubsection*{c) Derivatives of Lagrangian}
We develop the equation of motion for the Lagrangian $L=T-V$. The derivatives of the Lagrangian are given by
\begin{align}
&D_{\dot{x}_{0}}L=M_{T}\dot{x}_{0}+\sum_{i=1}^{n}{M_{iT}\dot{R}_{0}\rho_{i}}-\sum_{i=1}^{n}\sum_{j=1}^{n_{i}}{M_{0ij}l_{ij}\dot{q}_{ij}},\\
&D_{x_{0}}L=M_{T}ge_{3},\\
&D_{\dot{q}_{ij}}L=\sum_{i=1}^{n}\sum_{j=1}^{n_{i}}{M_{0ij}l_{ik}\dot{q}_{ik}}-\sum_{i=1}^{n}{M_{0ij}l_{ij}(\dot{x}_{0}}+{\dot{R}_{0}\rho_{i}}),\\
&D_{q_{ij}}L=-\sum_{i=1}^{n}{M_{0ij}l_{ij}e_{3}},
\end{align}
where $D_{\dot x_0}$ denote the derivative with respect to $\dot x_0$, and other derivatives are defined similarly. We also have
\begin{align}
D_{\Omega_{0}}L=&J_{0}\Omega_{0}+\sum_{i=1}^{n}{M_{iT}\hat{\rho}_{i}R_{0}^{T}\dot{x}_{0}}-\sum_{i=1}^{n}\sum_{j=1}^{n_{i}}{M_{0ij}l_{ij}\hat{\rho}_{i}R_{0}^{T}\dot{q}_{ij}}-\sum_{i=1}^{n}{M_{iT}\hat{\rho}_{i}^{2}\Omega_{0}},
\end{align}
\begin{align}
&D_{\Omega_{0}}L=\bar{J}_{0}\Omega_{0}+\sum_{i=1}^{n}{\hat{\rho}_{i}R_{0}^{T}(M_{iT}\dot{x}_{0}-\sum_{j=1}^{n_{i}}{M_{0ij}l_{ij}\dot{q}_{ij}})},\\
&D_{\Omega_{i}}L=\sum_{i=1}^{n}{J_{i}\Omega_{i}},
\end{align}
where $\bar{J}_{0}$ is defined as
\begin{align}
\bar{J}_{0}=J_{0}-\sum_{i=1}^{n}{M_{iT}\hat{\rho}_{i}^{2}}.
\end{align}
The derivation of the Lagrangian with respect to $R_{0}$ is given by
\begin{align}
D_{R_{0}}L\cdot\delta R_{0}=&\sum_{i=1}^{n}{M_{iT}R_{0}\hat{\eta}_{0}\hat{\Omega}_{0}\rho_{i}\cdot\dot{x}_{0}}\nonumber\\
&-\sum_{i=1}^{n}{R_{0}\hat{\eta}_{0}\hat{\Omega}_{0}\rho_{i}\cdot\sum_{j=1}^{n_{i}}{M_{0ij}l_{ij}\dot{q}_{ij}}}\nonumber\\
&+\sum_{i=1}^{n}{M_{iT}ge_{3}\cdot R_{0}\hat{\eta}_{0}\rho_{i}},
\end{align}
which can be rewritten as
\begin{align}
&D_{R_{0}}L\cdot\delta R_{0}=d_{R_{0}}\cdot \eta_{0},
\end{align}
where
\begin{align}
d_{R_{0}}=\sum_{i=1}^{n}&(((\widehat{\hat{\Omega}_{0}\rho_{i}}R_{0}^{T}(M_{iT}\dot{x}_{0})-\sum_{j=1}^{n_{i}}{M_{0ij}l_{ij}\dot{q}_{ij}})+M_{iT}g\hat{\rho}_{i}R_{0}^{T}e_{3})).
\end{align}
\subsubsection*{d) Lagrange-d'Alembert Principle}
Consider $\mathfrak{G}=\int_{t_{0}}^{t_{f}}{L}$ be the action integral. Using the equations derived in previous section, the infinitesimal variation of the action integral can be written as
\begin{align}
\delta \mathfrak{G}=&\int_{t_{0}}^{t_{f}}D_{\dot{x}_{0}}L\cdot\delta\dot{x}_{0}+D_{x_{0}}\cdot\delta x_{0} \nonumber\\
&+D_{\Omega_{0}}L(\dot{\eta}_{0}+\Omega_{0}\times \eta_{0})+d_{R_{0}}L\cdot\eta_{0} \nonumber\\
&+\sum_{i=1}^{n}\sum_{j=1}^{n_{i}}{D_{\dot{q}_{ij}}L(\dot{\xi}_{ij}\times q_{ij}+\xi_{ij}\times\dot{q}_{ij})}\nonumber\\
&+\sum_{i=1}^{n}\sum_{j=1}^{n_{i}}{D_{q_{ij}}L\cdot(\xi_{ij}\times q_{ij})} \nonumber\\
&+\sum_{i=1}^{n}{D_{\Omega_{i}}L\cdot(\dot{\eta}_{i}+\Omega_{i}\times\eta_{i})}.
\end{align}
The total thrust at the $i$-th quadrotor with respect to the inertial frame is denoted by $u_{i}=-f_{i}R_{i}e_{3}\in\Re^{3}$ and the total moment at the $i$-th quadrotor is defined as $M_{i}\in\Re^{3}$. The corresponding virtual work is given by
\begin{align}
\delta W=\int_{t_{0}}^{t_{f}}&{\sum_{i=1}^{n}{u_{i}\cdot\{\delta x_{0}+R_{0}\hat{\eta}_{0}\rho_{i}-\sum_{j=1}^{n_{i}}{l_{ij}\dot{\xi}_{ij}\times q_{ij}}}}\} +M_{i}\cdot \eta_{i}\; dt.
\end{align}
According to Lagrange-d Alembert principle, we have $\delta \mathfrak{G}=-\delta W$ for any variation of trajectories with fixes end points. By using integration by parts and rearranging, we obtain the following Euler-Lagrange equations
\begin{gather}
\frac{d}{dt}D_{\dot{x}_{i}}L-D_{x_{0}}L=\sum_{i=1}^{n}{u_{i}},\\
\frac{d}{dt}D_{\Omega_{0}}+\Omega_{0}\times D_{\Omega_{0}}-d_{R_{0}}=\sum_{i=1}^{n}{\hat{\rho}_{i}R_{0}^{T}u_{i}},\\
\hat{q}_{ij}\frac{d}{dt}D_{\dot{q}_{ij}}L-\hat{q}_{ij}D_{q_{i}}L=-l_{ij}\hat{q}_{ij}u_{i},\\
\frac{d}{dt}D_{\Omega_{i}}L+\Omega_{i}\times D_{\Omega_{i}}L=M_{i}.
\end{gather}
Substituting the derivatives of Lagrangians into the above expression and rearranging, the equations of motion are given by \refeqn{EOMM1}, \refeqn{EOMM2}, \refeqn{EOMM3}, \refeqn{EOMM4}.

\subsection{Proof for Proposition \ref{prop:propchap4_2}}\label{sec:pfchap4_2}
The variations of $x$ and $q$ are given by \refeqn{xlin} and \refeqn{qlin}. From the kinematics equation $\dot q_{ij}=\omega_{ij}\times q_{ij}$ and
\begin{align*}
\delta \dot q_{ij} = \dot\xi_{ij} \times e_3 =\delta\omega_{ij} \times e_3 + 0\times (\xi_{ij}\times e_3)=\delta\omega_{ij} \times e_3.
\end{align*}
Since both sides of the above equation is perpendicular to $e_3$, this is equivalent to $e_3\times(\dot\xi_{ij}\times e_3) = e_3\times(\delta\omega_{ij}\times e_3)$, which yields
\begin{gather*}
\dot \xi_{ij} - (e_3\cdot\dot\xi_{ij}) e_3 = \delta\omega_{ij} -(e_3\cdot\delta\omega_{ij})e_3.
\end{gather*}
Since $\xi_{ij}\cdot e_3 =0$, we have $\dot\xi_{ij}\cdot e_3=0$. As $e_3\cdot\delta\omega_{ij}=0$ from the constraint, we obtain the linearized equation for the kinematics equation of the link
\begin{align}
\dot\xi_{ij} = \delta\omega_{ij}.\label{eqn:dotxii}
\end{align}
The infinitesimal variation of $R_{0}\in\SO$ in terms of the exponential map
\begin{align}
\delta R_{0} = \frac{d}{d\epsilon}\bigg|_{\epsilon = 0} R_{0}\exp (\epsilon \hat\eta_{0}) = R_{0}\hat\eta_{0},\label{eqn:delR0}
\end{align}
for $\eta_{0}\in\Re^3$. Substituting these into \refeqn{EOMM1}, \refeqn{EOMM2}, and \refeqn{EOMM3}, and ignoring the higher order terms, we obtain the following sets of linearized equations of motion 
\begin{gather}
M_{T}\delta \ddot{x}_{0}-\sum_{i=1}^{n}{M_{iT}\hat{\rho}_{i}}\delta\dot{\Omega}_{0}+\sum_{i=1}^{n}\sum_{j=1}^{n_{i}}M_{0ij}l_{ij}\hat{e}_{3}C(C^{T}\ddot{\xi}_{ij})=\sum_{i=1}^{n}{\delta u_{i}}\\
\sum_{i=1}^{n}{M_{iT}\hat{\rho}_{i}\delta\ddot{x}_{0}}+\bar{J}_{0}\delta\dot{\Omega}_{0}+\sum_{i=1}^{n}\sum_{j=1}^{n_{i}}M_{0ij}l_{ij}\hat{\rho}_{i}\hat{e}_{3}C(C^{T}\ddot{\xi}_{ij})+\sum_{i=1}^{n}\frac{m_{0}}{n}g\hat{\rho}_{i}\hat{e}_{3}\eta_{0}=\sum_{i=1}^{n}{\hat{\rho}_{i}\delta u_{i}}\\
-M_{0ij}C^{T}\hat{e}_{3}\delta\ddot{x}_{0}+M_{0ij}C^{T}\hat{e}_{3}\hat{\rho}_{i}\delta\dot{\Omega}_{0}+\sum_{k=1}^{n_{i}}{M_{0ij}l_{ik}I_{2}(C^{T}\ddot{\xi}_{ij})}\nonumber\\
=-C^{T}\hat{e}_{3}\delta u_{i}+(-M_{iT}-\frac{m_{0}}{n}+M_{0ij})ge_{3} I_{2}(C^{T}\xi_{ij})\\
\dot{\eta}_{i}=\delta\Omega_{i},\; \dot{\eta}_{0}=\delta\Omega_{0},\; J_{i}\delta\Omega_{i}=\delta M_{i},
\end{gather}
which can be written in a matrix form as presented in \refeqn{EOMLin}. See~\cite{FarhadDTLeeIJCAS} for detailed derivations for a similar dynamic system. We used $C^{T}\hat{e}_{3}^{2}C=-I_{2}$ to simplify these derivations. 

\subsection{Proof for Proposition \ref{prop:propchap4_3}}\label{sec:pfchap4_3}
We first show stability of the rotational dynamics of each quadrotor, and later it is combined with the stability analysis for the remaining parts.
\subsubsection*{a) Attitude Error Dynamics}
Here, attitude error dynamics for $e_{R_{i}}$, $e_{\Omega_{i}}$ are derived and we find conditions on control parameters to guarantee the boundedness of attitude tracking errors. The time-derivative of $J_{i}e_{\Omega_{i}}$ can be written as
\begin{align}
J_{i}\dot e_{\Omega_{i}} & = \{J_{i}e_{\Omega_{i}} + d_{i}\}^\wedge e_{\Omega_{i}} - k_R e_{R_{i}}-k_\Omega e_{\Omega_{i}},\label{eqn:JeWdot}
\end{align}
where $d_{i}=(2J_{i}-\trs{[J_{i}]I})R_{i}^TR_{i_{d}}\Omega_{i_d}\in\Re^3$~\cite{Farhad2013}. The important property is that the first term of the right hand side is normal to $e_{\Omega_{i}}$, and it simplifies the subsequent Lyapunov analysis.

\subsubsection*{b) Stability for Attitude Dynamics}
Define a configuration error function on $\SO$ as follows
\begin{align}
\Psi_{i}= \frac{1}{2}\trs[I- R_{{i}_c}^T R_{i}].
\end{align}
We introduce the following Lyapunov function
\begin{align}
\mathcal{V}_{2}=\sum_{i=1}^{n}{\mathcal{V}_{2_{i}}},
\end{align}
where 
\begin{align}
\mathcal{V}_{2_i}=\frac{1}{2}e_{\Omega_{i}}\cdot J_{i}\dot{e}_{\Omega_{i}}+k_{R}\Psi_{i}(R_{i},R_{d_{i}})+c_{2_i}e_{R_{i}}\cdot e_{\Omega_{i}}.
\end{align}
Consider a domain $D_{2}$ given by
\begin{align}
D_2 = \{ (R_{i},\Omega_{i})\in \SO\times\Re^3\,|\, \Psi_{i}(R_{i},R_{d_{i}})<\psi_{2_i}<2\}.\label{eqn:D2}
\end{align}
In this domain we can show that $\mathcal{V}_{2}$ is bounded as follows~\cite{Farhad2013}
\begin{align}\label{eqn:ffff1}
\begin{split}
z_{2_i}^{T}M_{i_{21}}z_{2_i}\leq\mathcal{V}_{2_i}\leq z_{2_i}^{T}M_{i_{22}}z_{2_i},
\end{split}
\end{align}
and $z_{2_i}=[\|e_{R_{i}}\|,\|e_{\Omega_{i}}\|]^{T}\in \Re^{2}$. Matrices $M_{i_{21}}$, $M_{i_{22}}$ are given by
\begin{align}
M_{i_{21}}=&\frac{1}{2}\begin{bmatrix}
k_{R}&-c_{2_i}\lambda_{M_{i}}\\
-c_{2_i}\lambda_{M_{i}}&\lambda_{m_{i}}
\end{bmatrix},\nonumber\\
M_{i_{22}}=&\frac{1}{2}\begin{bmatrix}
\frac{2k_{R}}{2-\psi_{2_i}}&c_{2_i}\lambda_{M_{i}}\\
c_{2_i}\lambda_{M_{i}}&\lambda_{M_{i}}\nonumber
\end{bmatrix},
\end{align}
The time derivative of $\mathcal{V}_2$ along the solution of the controlled system is given by
\begin{align*}
\dot{\mathcal{V}}_2  =&
\sum_{i=1}^{n}-k_\Omega\|e_{\Omega_{i}}\|^2 + c_{2_i} \dot e_{R_{i}} \cdot J_{i}e_{\Omega_{i}}+ c_{2_i} e_{R_{i}} \cdot J_{i}\dot e_{\Omega_{i}}.
\end{align*}
Substituting \refeqn{JeWdot}, the above equation becomes
\begin{align*}
\dot{\mathcal{V}}_2 =&
\sum_{i=1}^{n}-k_\Omega\|e_{\Omega_{i}}\|^2  + c_{2_i} \dot e_{R_{i}} \cdot J_{i}e_{\Omega_{i}}-c_{2_i} k_R \|e_{R_{i}}\|^2 \\
&+ c_{2_i} e_{R_{i}} \cdot ((J_{i}e_{\Omega_{i}}+d_{i})^\wedge e_{\Omega_{i}} -k_\Omega e_{\Omega_{i}}).
\end{align*}
We have $\|e_{R_{i}}\|\leq 1$, $\|\dot e_{R_{i}}\|\leq \|e_{\Omega_{i}}\|$~\cite{TFJCHTLeeHG}, and choose a constant $B_{2_i}$ such that $\|d_{i}\|\leq B_{i_2}$. Then we have
\begin{align}
\dot{\mathcal{V}}_2 \leq -\sum_{i=1}^{n} z_{2_i}^T W_{2_i} z_{2_i},\label{eqn:dotV2}
\end{align}
where the matrix $W_{2_i}\in\Re^{2\times 2}$ is given by
\begin{align*}
W_{2_i} = \begin{bmatrix} c_{2_i}k_R & -\frac{c_{2_i}}{2}(k_\Omega+B_{2_i}) \\ 
-\frac{c_{2_i}}{2}(k_\Omega+B_{2_i}) & k_\Omega-2c_{2_i}\lambda_{M_{i}} \end{bmatrix}.
\end{align*}
The matrix $W_{2_i}$ is a positive definite matrix if 
\begin{align}\label{eqn:c2}
c_{2_i}<\min\{\frac{\sqrt{k_{R}\lambda_{m_i}}}{\lambda_{M_i}},\frac{4k_{\Omega}}{8k_{R}\lambda_{M_i}+(k_{\Omega}+B_{i_2})^{2}} \}.
\end{align}
This implies that
\begin{align}\label{eqn:eq2}
\dot{\mathcal{V}}_{2}\leq -\sum_{i=1}^{n} {\lambda_{m}(W_{2_i})\|z_{2_i}\|^{2}},
\end{align} 
which shows stability of the attitude dynamics of quadrotors.
%%%%%%%%%%%%%%%%%%%%%%%%%%%%%%%%%%%%%%%%
\subsubsection*{c) Error Dynamics of the Payload and Links}
We derive the tracking error dynamics and a Lyapunov function for the translational dynamics of a payload and the dynamics of links. Later it is combined with the stability analyses of the rotational dynamics. From \refeqn{EOMM1}, \refeqn{EOMLin}, \refeqn{Ai}, and \refeqn{fi}, the equation of motion for the controlled dynamic model is given by
\begin{align}\label{eqn:salam}
\Mb\ddot \xb  + \Gb\xb=\Bb(u-u^{*})+\g(\xb,\dot{\xb}),
\end{align}
where
\begin{align}
u=\begin{bmatrix}
u_{1}\\
u_{2}\\
\vdots\\
u_{n}
\end{bmatrix},\; u^{*}=\begin{bmatrix}
-(M_{1T}+\frac{m_{0}}{n})ge_{3}\\
-(M_{2T}+\frac{m_{0}}{n})ge_{3}\\
\vdots\\
-(M_{nT}+\frac{m_{0}}{n})ge_{3}
\end{bmatrix},
\end{align}
and $\g(\xb,\dot{\xb})$ corresponds to the higher order terms. As $u_i=-f_iR_ie_3$ for the full dynamic model, $\delta u=u-u^*$ is given by
\begin{align}\label{eqn:deltauu}
\delta u=
\begin{bmatrix}
-f_{1}R_{1}e_{3}+(M_{1T}+\frac{m_{0}}{n})ge_{3}\\
-f_{2}R_{2}e_{3}+(M_{2T}+\frac{m_{0}}{n})ge_{3}\\
\vdots\\
-f_{n}R_{n}e_{3}+(M_{nT}+\frac{m_{0}}{n})ge_{3}
\end{bmatrix}.
\end{align}
The subsequent analyses are developed in the domain $D_{1}$
\begin{align}
D_1=\{&(\xb,\dot{\xb},R_i,e_{\Omega_i})\in\Re^{D_{\xb}}\times\Re^{D_{\xb}}\times \SO\times\Re^3\,|\, \Psi_{i}< \psi_{1_i} < 1\}.\label{eqn:D}
\end{align}
In the domain $D_{1}$, we can show that 
\begin{align}
\frac{1}{2} \norm{e_{R_{i}}}^2 \leq  \Psi_i(R_i,R_{c_i}) \leq \frac{1}{2-\psi_{1_i}} \norm{e_{R_i}}^2\label{eqn:eRPsi1}.
\end{align}
Consider the quantity $e_{3}^{T}R_{c_i}^{T}R_{i}e_{3}$, which represents the cosine of the angle between $b_{3_i}=R_{i}e_{3}$ and $b_{3_{c_i}}=R_{c_i}e_{3}$. Since $1-\Psi_i(R_i,R_{c_i})$ represents the cosine of the eigen-axis rotation angle between $R_{c_i}$ and $R_i$, we have $e_{3}^{T}R_{c_i}^{T}Re_{3}\geq 1-\Psi_i(R_i,R_{c_i})>0$ in $D_{1}$. Therefore, the quantity $\frac{1}{e_{3}^{T}R_{c_i}^{T}R_i e_{3}}$ is well-defined. We add and subtract $\frac{f_i}{e_{3}^{T}R_{c_i}^{T}R_i e_{3}}R_{c_i}e_{3}$ to the right hand side of \refeqn{deltauu} to obtain
\begin{align}\label{eqn:hallaa}
\delta u=
\begin{bmatrix}
\frac{-f_1}{e_{3}^{T}R_{c_1}^{T}R_1 e_{3}}R_{c_1}e_{3}-X_1+(M_{1T}+\frac{m_{0}}{n})ge_{3}\\
\frac{-f_2}{e_{3}^{T}R_{c_2}^{T}R_2 e_{3}}R_{c_2}e_{3}-X_2+(M_{2T}+\frac{m_{0}}{n})ge_{3}\\
\vdots\\
\frac{-f_n}{e_{3}^{T}R_{c_n}^{T}R_n e_{3}}R_{c_n}e_{3}-X_n+(M_{nT}+\frac{m_{0}}{n})ge_{3}
\end{bmatrix}.
\end{align}
where $X_i \in \Re^{3}$ is defined by
\begin{align}\label{eqn:Xdef}
X_i=\frac{f_i}{e_{3}^{T}R_{c_i}^{T}R_i e_{3}}((e_{3}^{T}R_{c_i}^{T}R_i e_{3})R_i e_{3}-R_{c_i}e_{3}).
\end{align}
Using 
\begin{align}
-\frac{f_i}{e_{3}^{T}R_{c_i}^{T}R_i e_{3}}R_{c_i}e_{3}=-\frac{(\|A_i\|R_{c_i}e_{3})\cdot R_i e_{3}}{e_{3}^{T}R_{c_i}^{T}R_i e_{3}}\cdot -\frac{A_i}{\|A_i\|}=A_i,
\end{align}
the equation \refeqn{hallaa} becomes
\begin{align}
\delta u=
\begin{bmatrix}
A_{1}-X_1+(M_{1T}+\frac{m_{0}}{n})ge_{3}\\
A_{2}-X_2+(M_{2T}+\frac{m_{0}}{n})ge_{3}\\
\vdots\\
A_{n}-X_n+(M_{nT}+\frac{m_{0}}{n})ge_{3}
\end{bmatrix}.
\end{align}
Substituting \refeqn{Ai} into the above equation, \refeqn{salam} becomes
\begin{align}
\Mb\ddot \xb  + \Gb\xb=\Bb(-K_{\xb}\xb-K_{\dot{\xb}}\dot{\xb}-X)+\g(\xb,\dot{\xb}),
\end{align}
where $X=[X_{1}^T,\; X_{2}^T,\; \cdots,\; X_{n}^T]^{T}\in\Re^{3n}$.
It is rewritten in the following matrix form
\begin{align}\label{eqn:zdot1}
\dot{z}_{1}=\mathds{A} z_{1}+\mathds{B}(\Bb X+\g(\xb,\dot{\xb})),
\end{align}
where $z_{1}=[\xb,\dot{\xb}]^{T}\in\Re^{2D_{\xb}}$ and
\begin{align}
\mathds{A}=\begin{bmatrix}
0&I\\
-\Mb^{-1}(\Gb+\Bb K_{\xb})&-\Mb^{-1}\Bb K_{\dot{\xb}}
\end{bmatrix},
\mathds{B}=\begin{bmatrix}
0\\
\Mb^{-1}
\end{bmatrix}.
\end{align}
We can also choose $K_{\xb}$ and $K_{\dot{\xb}}$ such that $\mathds{A}\in\Re^{2D_{x}\times 2D_{x}}$ is Hurwitz. Then for any positive definite matrix $Q\in\Re^{2D_{\xb}\times 2D_{\xb}}$, there exist a positive definite and symmetric matrix $P\in\Re^{2D_{\xb}\times 2D_{\xb}}$ such that $\mathds{A}^{T}P+P\mathds{A}=-Q$ according to~\cite[Thm 3.6]{Kha96}.

\subsubsection*{d) Lyapunov Candidate for Simplified Dynamics}
From the linearized control system developed at section 3, we use matrix $P$ to introduce the following Lyapunov candidate for translational dynamics
\begin{align}
\mathcal{V}_{1}=z_{1}^{T}Pz_{1}.
\end{align}
The time derivative of the Lyapunov function using the Leibniz integral rule is given by
\begin{align}\label{eqn:devrr}
\dot{\mathcal{V}_{1}}=\dot{z}_{1}^{T}Pz_{1}+z_{1}^{T}P\dot{z}_{1}.
\end{align}
Substituting \refeqn{zdot1} into above expression
\begin{align}\label{eqn:beforsimp}
\dot{\mathcal{V}}_{1}=z_{1}^{T}(\mathds{A}^{T}P+P\mathds{A})z_{1}+2z_{1}^{T}P\mathds{B}(\Bb X+\g(\xb,\dot{\xb})).
\end{align}
Let $c_{3}=2\|P\mathds{B}\Bb\|_{2}\in\Re$ and using $\mathds{A}^{T}P+P\mathds{A}=-Q$, we have
\begin{align}\label{eqn:test}
\dot{\mathcal{V}}_{1}\leq-z_{1}^{T}Qz_{1}+c_{3}\|z_{1}\|\|X\|+2z_{1}^{T}P\mathds{B}\g(\xb,\dot{\xb}).
\end{align}
The second term on the right hand side of the above equation corresponds to the effects of the attitude tracking error on the translational dynamics. We find a bound of $X_{i}$, defined at \refeqn{Xdef}, to show stability of the coupled translational dynamics and rotational dynamics in the subsequent Lyapunov analysis. Since 
\begin{align}
f_{i}=\|A_{i}\|(e_{3}^{T}R_{c_i}^{T}R_i e_{3}), 
\end{align}
we have
\begin{align}\label{eqn:ssstr}
\|X_i\|\leq\|A_i\|\|(e_{3}^{T}R_{c_i}^{T}R_i e_{3})R_i e_{3}-R_{c_i}e_{3}\|.
\end{align}
The last term $\|(e_{3}^{T}R_{c_i}^{T}R_i e_{3})R_i e_{3}-R_{c_i}e_{3}\|$ represents the sine of the angle between $b_{3_i}=R_i e_{3}$ and $b_{3_{c_i}}=R_{c_i}e_{3}$, since $(b_{3_{c_i}}\cdot b_{3_i})b_{3_i}-b_{3_{c_i}}=b_{3_i}\times(b_{3_i}\times b_{3_{c_i}})$. The magnitude of the attitude error vector, $\|e_{R_i}\|$ represents the sine of the eigen-axis rotation angle between $R_{c_i}$ and $R_i$. Therefore, $\|(e_{3}^{T}R_{c_i}^{T}R_i e_{3})R_i e_{3}-R_{c_i}e_{3}\|\leq\|e_{R_i}\|$ in $D_{1}$. It follows that
\begin{align}
\|(e_{3}^{T}R_{c_i}^{T}R_i e_{3})R_i e_{3}-R_{c_i}e_{3}\|&\leq\|e_{R_i}\|=\sqrt{\Psi_i(2-\Psi_i)}\nonumber\\
&\leq\{\sqrt{\psi_{1_i}(2-\psi_{1_i})}\triangleq\alpha_i\}<1,
\end{align}
therefore
\begin{align}
\|X_i\|&\leq \|A_i\|\|e_{R_i}\|\nonumber\\
&\leq\|A_i\|\alpha_i.
\end{align}
We find an upper boundary for
\begin{align}\label{eqn:AA}
A_i=-K_{\xb} \xb-K_{\dot{\xb}}\dot\xb +u_{i}^{*},
\end{align}
by defining
\begin{align}
\|u_{i}^{*}\|\leq B_{1_i},
\end{align} 
for a given positive constant $B_{1}$. defining $K_{max}\in\Re$
\begin{align*}
&K_{\max}=\max\{\|K_{\xb}\|,\|K_{\dot{\xb}}\|\},
\end{align*}
and then the upper bound of $A$ is given by
\begin{align}
\|A_i\| & \leq K_{\max}(\|\xb\|+\|\dot{\xb}\|)+B_{1}\nonumber\\
&\leq 2K_{\max}\|z_{1}\|+B_{1}.\label{eqn:normA}
\end{align}
Using the above steps we can show that
\begin{align}
\|X\|&\leq \sum_{i=1}^{n}((2K_{\max}\|z_{1}\|+B_{1})\|e_{R_{i}}\|)\nonumber\\
&\leq (2K_{\max}\|z_{1}\|+B_{1})\alpha,
\end{align}
where $\alpha=\sum_{i=1}^{n}\alpha_{i}$. Then, we can simplify \refeqn{test} as 
\begin{align}\label{eqn:eq1}
\dot{\mathcal{V}}_{1} \leq& -(\lambda_{\min}(Q)-2c_{3}K_{\max}\alpha) \|z_{1}\| ^{2}+\sum_{i=1}^{n}c_{3}B_{1}\|z_{1}\|\|e_{R_{i}}\|+2z_{1}^{T}P\mathds{B}\g(\xb,\dot{\xb}).
\end{align}
%%%%%%%%%%%%%%%%%%%%%%%%%%%%%%%%%%%%%%%%
\subsubsection*{e) Lyapunov Candidate for the Complete System}
Let $\mathcal{V}=\mathcal{V}_{1}+\mathcal{V}_{2}$ be the Lyapunov function for the complete system. The time derivative of $\mathcal{V}$ is given by
\begin{align}
\dot{\mathcal{V}}=\dot{\mathcal{V}}_{1}+\dot{\mathcal{V}}_{2}.
\end{align}
Substituting \refeqn{eq1} and \refeqn{eq2} into the above equation
\begin{align}
\dot{\mathcal{V}}\leq& -(\lambda_{\min}(Q)-2c_{3}K_{\max}\alpha) \|z_{1}\| ^{2}+2z_{1}^{T}P\mathds{B}\g(\xb,\dot{\xb})\nonumber\\
&+\sum_{i=1}^{n}c_{3}B_{1} \|z_{1}\|\|e_{R_{i}}\|-\sum_{i=1}^{n}\lambda_{m}(W_{2_i})\|z_{2_i}\|^{2},
\end{align}
and using $\|e_{R_{i}}\|\leq \|z_{2_i}\|$, it can be written as
\begin{align}\label{eqn:finalsimp}
\dot{\mathcal{V}}\leq& -(\lambda_{\min}(Q)-2c_{3}K_{\max}\alpha) \|z_{1}\| ^{2}+2z_{1}^{T}P\mathds{B}\g(\xb,\dot{\xb})\nonumber\\
&+\sum_{i=1}^{n}c_{3}B_{1} \|z_{1}\|\|z_{2_i}\|-\sum_{i=1}^{n}\lambda_{m}(W_{2_i})\|z_{2_i}\|^{2}.
\end{align}
The $2z_{1}^{T}P\mathds{B}\g(\xb,\dot{\xb})$ term in the above equation is indefinite. The function $\g(\xb,\dot{\xb})$ satisfies
\begin{align*}
\frac{\|\g(\xb,\dot{\xb})\|}{\|z_{1}\|}\rightarrow 0\quad \mbox{as}\quad \|z_{1}\|\rightarrow 0.
\end{align*}
Then, for any $\gamma>0$ there exists $r>0$ such that
\begin{align*}
\|\g(\xb,\dot{\xb})\|<\gamma\|z_{1}\|\quad \forall\|z_{1}\|<r.
\end{align*}
Therefore
\begin{align}
2z_{1}^{T}P\mathds{B}\g(\xb,\dot{\xb})\leq 2\gamma\|P\|_{2}\|z_{1}\|^{2}.
\end{align}
Substituting the above inequality into \refeqn{finalsimp}
\begin{align}
\dot{\mathcal{V}}\leq& -(\lambda_{\min}(Q)-2c_{3}K_{\max}\alpha) \|z_{1}\| ^{2}+2\gamma\|P\|_{2}\|z_{1}\|^{2}\nonumber\\
&+\sum_{i=1}^{n}c_{3}B_{1} \|z_{1}\|\|z_{2_i}\|-\sum_{i=1}^{n}\lambda_{m}(W_{2_i})\|z_{2_i}\|^{2},
\end{align}
and rearranging
\begin{align}
\dot{\mathcal{V}}\leq -\sum_{i=1}^{n}(&\frac{\lambda_{\min}(Q)-2c_{3}K_{\max}\alpha}{n} \|z_{1}\| ^{2}\nonumber\\
&-c_{3}B_{1} \|z_{1}\|\|z_{2_i}\|+\lambda_{m}(W_{2_i})\|z_{2_i}\|^{2})\nonumber\\
&+2\gamma\|P\|_{2}\|z_{1}\|^{2},
\end{align}
we obtain
\begin{align}
\dot{\mathcal{V}}\leq-\sum_{i=1}^{n}(\zb_{i}^{T}W_{i}\zb_{i})+2\gamma\|P\|_{2}\|z_{1}\|^{2},
\end{align}
where $\zb_{i}=[\|z_{1}\|,\|z_{2_{i}}\|]^{T}\in\Re^{2}$ and
\begin{align}
W_i=\begin{bmatrix}
\frac{\lambda_{\min}(Q)-2c_{3}K_{\max}\alpha}{n}&-\frac{c_{3}B_{1_i}}{2}\\
-\frac{c_{3}B_{1_i}}{2}&\lambda_{m}(W_{2_i})
\end{bmatrix}.
\end{align}
By using $\|z_{1}\|\leq\|\zb_i\|$, we obtain
\begin{align}
\dot{\mathcal{V}}\leq -\sum_{i=1}^{n}(\lambda_{\min}(W_i)-\frac{2\gamma\|P\|_{2}}{n})\|\zb_{i}\|^{2}.
\end{align}
Choosing $\gamma<n(\lambda_{\min}(W_i))/2\|P\|_{2}$, and
\begin{align}
\lambda_{m}(W_{2_i})>\frac{n\|\frac{c_{3}B_{1_i}}{2}\|^2}{\lambda_{\min}(Q)-2c_{3}K_{\max}\alpha},
\end{align}
ensures that $\dot{\mathcal{V}}$ is negative definite. Then, the zero equilibrium is exponentially stable.

\end{document}